\documentclass[11pt,notitlepage]{amsart} 
\usepackage[utf8]{inputenc}
\usepackage[all,cmtip]{xy}
\usepackage{amssymb,amsmath,amsthm}
\usepackage{mathtools, mathdots}
\usepackage{fancyhdr} 
\usepackage{mathrsfs, calligra, bbm}
\usepackage[top=1.5in, bottom=1.5in, left=1in, right=1in]{geometry}
\usepackage{hyperref}
\usepackage{tikz-cd}
\usetikzlibrary{decorations.pathmorphing}
\usepackage{dirtytalk}
\usepackage{tikz}
\usepackage{enumitem}
\usepackage[T1]{fontenc}
\usepackage{graphicx}
\usepackage{subcaption}
\usepackage{array}
\usepackage{graphicx}
\usepackage[toc,page]{appendix}
\usepackage{stfloats}
\usepackage[strict=true, style=british]{csquotes}
\usepackage[british]{babel}

\usepackage{amsfonts}
\usepackage{scalerel}
\usepackage{latexsym}
\usepackage{lscape}
\usepackage{epstopdf}
\usetikzlibrary{calc}
\usepackage{color}
\usepackage{multirow}  
\usepackage{scalefnt}
\usepackage{relsize} 
\usepackage[bbgreekl]{mathbbol} 
\usepackage{amscd,epsf,verbatim,mathtools,framed}
\DeclareSymbolFontAlphabet{\mathbb}{AMSb} 
\DeclareSymbolFontAlphabet{\mathbbl}{bbold}

\newcounter{oldtocdepth}

\let\emptyset\varnothing

\DeclareMathAlphabet{\mathpzc}{OT1}{pzc}{m}{it}

\theoremstyle{definition}
\newtheorem{Definition}{Definition}[subsection]
\newtheorem{Example}[Definition]{Example}
\newtheorem{Convention}[Definition]{Convention}
\newtheorem{Proposition-Definition}[Definition]{Proposition-Definition}

\newtheorem{Question}[Definition]{Question}
\newtheorem{Assumption}[Definition]{Assumption}
\newtheorem{Remark}[Definition]{Remark}
\newtheorem*{Remark*}{Remark}

\newtheorem{Theorem}[Definition]{Theorem}
\newtheorem{Lemma}[Definition]{Lemma}
\newtheorem{Proposition}[Definition]{Proposition}
\newtheorem{Corollary}[Definition]{Corollary}

\newtheorem*{Theorem*}{Theorem}

\newcommand\scalemath[2]{\scalebox{#1}{\mbox{\ensuremath{\displaystyle #2}}}}

\DeclareMathOperator\A{\mathbf{A}}

\DeclareMathOperator\C{\mathbf{C}}

\DeclareMathOperator\F{\mathbf{F}}

\DeclareMathOperator\Q{\mathbf{Q}}
\DeclareMathOperator\R{\mathbf{R}}

\DeclareMathOperator\Z{\mathbf{Z}}

\DeclareMathOperator\bfita{\textbf{\textit{a}}}
\DeclareMathOperator\bfitj{\textbf{\textit{j}}}
\DeclareMathOperator\bfits{\textbf{\textit{s}}}
\DeclareMathOperator\bfitu{\textbf{\textit{u}}}

\DeclareMathOperator\bfitw{\textbf{\textit{w}}}

\DeclareMathOperator\bfitz{\textbf{\textit{z}}}

\DeclareMathOperator\bbA{\mathbb{A}}

\DeclareMathOperator\bbG{\mathbb{G}}
\DeclareMathOperator\bbH{\mathbb{H}}

\DeclareMathOperator\bbT{\mathbb{T}}
\DeclareMathOperator\bbX{\mathbb{X}}

\DeclareMathOperator\calA{\mathcal{A}}
\DeclareMathOperator\calB{\mathcal{B}}

\DeclareMathOperator\calD{\mathcal{D}}
\DeclareMathOperator\calE{\mathcal{E}}

\DeclareMathOperator\calG{\mathcal{G}}
\DeclareMathOperator\calH{\mathcal{H}}

\DeclareMathOperator\calN{\mathcal{N}}
\DeclareMathOperator\calO{\mathcal{O}}
\DeclareMathOperator\calP{\mathcal{P}}

\DeclareMathOperator\calS{\mathcal{S}}

\DeclareMathOperator\calU{\mathcal{U}}
\DeclareMathOperator\calV{\mathcal{V}}
\DeclareMathOperator\calW{\mathcal{W}}
\DeclareMathOperator\calX{\mathcal{X}}
\DeclareMathOperator\calY{\mathcal{Y}}
\DeclareMathOperator\calZ{\mathcal{Z}}

\DeclareMathOperator\scrA{\mathscr{A}}

\DeclareMathOperator\scrD{\mathscr{D}}
\DeclareMathOperator\scrE{\mathscr{E}}
\DeclareMathOperator\scrF{\mathscr{F}}
\DeclareMathOperator\scrG{\mathscr{G}}
\DeclareMathOperator\scrH{\mathscr{H}}

\DeclareMathOperator\scrM{\mathscr{M}}
\DeclareMathOperator\scrN{\mathscr{N}}
\DeclareMathOperator\scrO{\mathscr{O}}
\DeclareMathOperator\scrP{\mathscr{P}}

\DeclareMathOperator\scrT{\mathscr{T}}

\DeclareMathOperator\scrV{\mathscr{V}}
\DeclareMathOperator\scrW{\mathscr{W}}

\DeclareMathOperator\frakH{\mathfrak{H}}

\DeclareMathOperator\frakU{\mathfrak{U}}
\DeclareMathOperator\frakV{\mathfrak{V}}

\DeclareMathOperator\fraka{\mathfrak{a}}

\DeclareMathOperator\frakm{\mathfrak{m}}
\DeclareMathOperator\frakn{\mathfrak{n}}

\DeclareMathOperator\frakp{\mathfrak{p}}

\DeclareMathOperator\fraks{\mathfrak{s}}
\DeclareMathOperator\frakt{\mathfrak{t}}

\DeclareMathOperator\frakz{\mathfrak{z}}

\DeclareMathOperator\GL{GL}
\DeclareMathOperator\GSp{GSp}

\DeclareMathOperator\SL{SL}

\DeclareMathOperator\Aut{Aut}
\DeclareMathOperator\End{End}
\DeclareMathOperator\Hom{Hom}

\DeclareMathOperator\sheafHom{\scrH\!\!\textit{om}}

\DeclareMathOperator\adicFL{\mathcal{F}\!\mathbf{\ell}}

\DeclareMathOperator\adicIW{\mathcal{IW}}
\DeclareMathOperator\adicIw{\mathcal{I}\!\mathit{w}}

\DeclareMathOperator\alg{alg}

\DeclareMathOperator\an{an}

\DeclareMathOperator\BS{BS}

\DeclareMathOperator\coker{coker}

\DeclareMathOperator\cts{cts}
\DeclareMathOperator\cusp{cusp}

\DeclareMathOperator\cyc{cyc}

\DeclareMathOperator\diag{diag}

\DeclareMathOperator\dR{dR}

\DeclareMathOperator\ES{ES}

\DeclareMathOperator\et{\mathrm{\acute{e}t}}

\DeclareMathOperator\Fil{Fil}

\DeclareMathOperator\Fl{Fl}

\DeclareMathOperator\fs{fs}
\DeclareMathOperator\Gal{Gal}

\DeclareMathOperator\Gr{Gr}

\DeclareMathOperator\hst{hst}
\DeclareMathOperator\HT{HT}

\DeclareMathOperator\image{image}

\DeclareMathOperator\Isom{Isom}
\DeclareMathOperator\Iw{Iw}
\DeclareMathOperator\ket{\mathrm{k\acute{e}t}}

\DeclareMathOperator\Lie{Lie}

\DeclareMathOperator\one{\mathbbm{1}}
\DeclareMathOperator\oneanti{\breve{\one}}
\DeclareMathOperator\opp{opp}
\DeclareMathOperator\Par{par}
\DeclareMathOperator\pr{pr}
\DeclareMathOperator\proet{\mathrm{pro\acute{e}t}}

\DeclareMathOperator\proj{proj}
\DeclareMathOperator\proket{\mathrm{prok\acute{e}t}}

\DeclareMathOperator\rig{rig}
\DeclareMathOperator\Res{Res}

\DeclareMathOperator\Si{Si}
\DeclareMathOperator\Spa{Spa}

\DeclareMathOperator\Spec{Spec}

\DeclareMathOperator\Sym{Sym}

\DeclareMathOperator\sms{ss}

\DeclareMathOperator\std{std}

\DeclareMathOperator\tor{tor}

\DeclareMathOperator\trans{^{\mathtt{t}}\!}

\DeclareMathOperator\univ{univ}

\DeclareMathOperator\wt{wt}

\DeclareMathOperator\Alg{{\mathrm{Alg}}}
\DeclareMathOperator\Mod{{\mathrm{Mod}}}
\DeclareMathOperator\Groups{{\mathrm{Group}}}

\DeclareMathOperator\Sets{{\mathrm{Sets}}}

\DeclareMathOperator\Ban{{\mathrm{Ban}}}

\DeclareMathOperator\bfalpha{\boldsymbol{\alpha}}
\DeclareMathOperator\bfbeta{\boldsymbol{\beta}}
\DeclareMathOperator\bfgamma{\boldsymbol{\gamma}}
\DeclareMathOperator\bfdelta{\boldsymbol{\delta}}
\DeclareMathOperator\bfepsilon{\boldsymbol{\varepsilon}}

\DeclareMathOperator\bfnu{\boldsymbol{\nu}}

\DeclareMathOperator\bftau{\boldsymbol{\tau}}

\DeclareMathOperator\llbrack{\![\![\!}
\DeclareMathOperator\rrbrack{\!]\!]}

\DeclareMathOperator\bla{\boldsymbol{\langle}}
\DeclareMathOperator\bra{\boldsymbol{\rangle}}

\title{Overconvergent Eichler--Shimura morphisms for $\mathrm{GSp}_4$}
\author{Hansheng Diao, Giovanni Rosso, and Ju-Feng Wu}
\date{}

\begin{document}

\begin{abstract} 
We construct explicit Eichler--Shimura morphisms for families of overconvergent Siegel modular forms of genus two. These can be viewed as $p$-adic interpolations of the Eichler--Shimura decomposition of Faltings--Chai for classical Siegel modular forms. In particular, we are able to $p$-adically interpolate the entire decomposition, extending our previous work on the $H^0$-part. The key new inputs are the higher Coleman theory of Boxer--Pilloni and a theory of pro-Kummer \'etale cohomology with supports.
\end{abstract}

\maketitle
\thispagestyle{empty}
\tableofcontents

\section{Introduction}

\subsection{Background}
Consider the Poincar\'e upper-half plane $\bbH$ equipped with a left action by $\SL_2(\Z)$ via the M\"obius transformation. Consider a congruence subgroup $\Gamma \subset \SL_2(\Z)$, and let $X(\C) \coloneq \Gamma \backslash \bbH$ be the (complex analytic) modular curve of level $\Gamma$. The classical Eichler--Shimura decomposition reads as follows. 

\begin{Theorem}[Eichler--Shimura decomposition]\label{thm: ES decomposition}
    For $k\in \Z_{\geq 0}$, let $M_{k+2}(\Gamma)$ (reps., $S_{k+2}(\Gamma)$) be the space of modular forms (resp., cuspforms) of weight $k+2$ and level $\Gamma$. Then there is a Hecke-equivariant decomposition \[
        H^1(X(\C), \Sym^k \C^2) = M_{k+2}(\Gamma) \oplus \overline{S_{k+2}(\Gamma)},
    \]
    where $\overline{\bullet}$ stands for the complex conjugation.
\end{Theorem}

Theorem \ref{thm: ES decomposition} has an arithmetic incarnation which we now explain. The complex analytic modular curve $X(\C)$ admits a structure of an algebraic curve $X$ over $\Q$, which classifies elliptic curves with $\Gamma$-level structures. Let $\overline{X}$ be the compactification of $X$ which classifies generalised elliptic curves, and let $\pi \colon E^{\univ} \rightarrow \overline{X}$ be the universal semiabelian scheme over $\overline{X}$ with the identity section $e$. Consider the line bundle $\underline{\omega} := e^* \Omega_{E^{\univ}/\overline{X}}^1$. For $k\in \Z_{\geq 0}$, it is well-known that $M_{k+2}(\Gamma) = H^0(\overline{X}, \underline{\omega}^{\otimes k+2})\otimes_{\Q} \C$. We have the following theorem of Faltings \cite{FaltingsHT}.

\begin{Theorem}[$p$-adic Eichler--Shimura decomposition]\label{thm: p-adic ES decomposition}
    Let $p$ be a prime number and let $k\in \Z_{\geq 0}$. There exists a Hecke- and Galois-equivariant \footnote{Throughout the article, Galois-equivariance is always respect to the action of $\Gal_{\Q_p}$, unless specified.} split short exact sequence \[
        \scalemath{0.9}{ 0 \rightarrow  H^1(\overline{X}_{\Q_p}, \underline{\omega}^{-k})\otimes_{\Q_p}\C_p(k) \xrightarrow{\mathrm{ES}_k^{\vee}}  H^1_{\et}(X_{\C_p}, \Sym^k\Q_p^2)\otimes_{\Q_p}\C_p \xrightarrow{\mathrm{ES}_k} H^0(\overline{X}_{\Q_p}, \underline{\omega}^{\otimes k+2})\otimes_{\Q_p}\C_p(-1) \rightarrow 0 },
    \]
    where the Galois actions on the coherent cohomology groups are trivial. 
\end{Theorem}

Inspired by the groundbreaking work on $p$-adic families of modular forms by Hida, Coleman, and Coleman--Mazur, etc., it is natural to explore the possibility of $p$-adically interpolating the aforementioned results. More precisely, can we establish arrows $\mathrm{ES}_k$ and $\mathrm{ES}_k^{\vee}$ for a general \emph{$p$-adic weight} $\kappa$, or even for a \emph{family} of $p$-adic weights, so that they $p$-adically interpolate the arrows in Theorem \ref{thm: p-adic ES decomposition} in an appropriate sense? Indeed, this question has been extensively studied in recent years:

\begin{itemize}
    \item The first result in this direction was due to Andreatta--Iovita--Stevens (\cite{AIS-2015}), where they established an overconvergent Eichler--Shimura morphism. It maps from the so-called \emph{overconvergent cohomology group} (which can be viewed as a certain $p$-adic interpolation of the \'etale cohomology group $H^1_{\et}(X_{\C_p}, \Sym^k\Q_p^2)$) to the space of overconvergent modular forms. That is, they established a $p$-adic variation of the arrow $\mathrm{ES}_k$.  
    \item The method of Andreatta--Iovita--Stevens has been extended to study automorphic forms on Shimura curves (\cite{BSG-2017, BSG-2021}).
    \item In \cite{CHJ-2017}, Chojecki--Hansen--Johansson developed a perfectoid method to construct the overconvergent Eichler--Shimura morphism. They are able to (re)construct the morphisms of Andreatta--Iovita--Stevens, but for automorphic forms on compact Shimura curves. Their method makes use of the perfectoid Shimura varieties constructed by Scholze, as well as the Hodge--Tate period map \cite{Scholze-2015}.  
    \item The first result establishing the $p$-adic variation of $\mathrm{ES}_k^{\vee}$ was due to J. E. Rodr\'iguez Camargo (\cite{Rodriguez}). The key ingredient in his work is the higher Coleman theory on modular curves established by Boxer--Pilloni (\cite{BP-HCTModularCurve}).
\end{itemize}

The present paper concerns the generalisation of this question to Siegel modular forms. In the Siegel case, there is still a classical Eichler--Shimura decomposition which we would like to $p$-adically interpolate. However, it turns out the Siegel case is much more involved compared with the elliptic case. We shall present our main results in \S \ref{subsection: main results}.

\subsection{Main results}\label{subsection: main results}
We start by setting up some notations. Let $p$ be a prime number. Let $\Gamma = \prod_{\ell\neq p}\Gamma_{\ell} \subset \GSp_4(\A_{\Q}^{\infty, p})$ be a neat open compact subgroup, which serves as our tame level. We denote by $N$ the product of primes $\ell$ such that $\Gamma_{\ell}$ is not spherical. For every $n\geq 1$, consider the \emph{strict Iwahori subgroup} $\Iw_{\GSp_4, n}^+\subset \GSp_4(\Z_p)$ which consists of those matrices that are congruent to diagonal matrices modulo $p^n$. We will take $\Gamma_n = \Gamma\Iw_{\GSp_4, n}^+\subset \GSp_4(\widehat{\Z})$ to be our level structure. We work with the strict Iwahori level because our construction requires taking transposes of matrices, while the usual Iwahori subgroup is not preserved under transposition. Note that there is no harm working with the strict Iwahori level since the space of classical finite-slope forms is independent of the level structure at $p$ (cf. Proposition \ref{Proposition: finite-slope part of classical cohomology is independent to the level}). 

For every $n\in \Z_{\geq 0}$, let $X_n$ denote the Siegel threefold of level $\Gamma_n$; it is an algebraic variety over $\Q$ which classifies principally polarised abelian varieties with $\Gamma_n$-level structures. By fixing a choice of cone decomposition, each $X_n$ admits a toroidal compactification $X_n^{\tor}$ and the compactifications are compatible when we vary $n$.  There is a tautological semiabelian scheme $\pi \colon G_n^{\univ} \rightarrow X_n^{\tor}$ with identity section $e$. Consider $\underline{\omega}_n \coloneq e^* \Omega^1_{G_n^{\univ}/X_n^{\tor}}$. This is a vector bundle on $X_n^{\tor}$ of rank 2. When the level $\Gamma_n$ is clear from the context, we simply write $\underline{\omega}$ instead of $\underline{\omega}_n$. For any $k = (k_1, k_2)\in \Z^2$ with $k_1\geq k_2$, consider \[
    \underline{\omega}^k \coloneq \Sym^{k_1-k_2} \underline{\omega} \otimes (\det \underline{\omega})^{\otimes k_2}
\]
which is the \emph{classical automorphic sheaf of weight $k$}.

Moreover, let $H$ be the Levi subgroup of the Siegel parabolic subgroup of $\GSp_4$ and let $W^H$ be a set of representatives of the quotient of the Weyl groups $W_{\GSp_4}/W_H$. We follow \cite{Faltings-Chai} to choose these representatives so that $W^H = \{\bfitw_0 = \one_4, \bfitw_1, \bfitw_2, \bfitw_3\}$ where the Weyl elements are indexed by their length. See \S \ref{subsection: the group GSp4} for more details.

The following theorem of Faltings--Chai \cite[Chapter VI, Theorem 6.2]{Faltings-Chai} can be viewed as an analogue of Theorem \ref{thm: p-adic ES decomposition}.

\begin{Theorem}[$p$-adic Eichler--Shimura decomposition for $\GSp_4$]\label{thm: Faltings-Chai}
    Let $k = (k_1, k_2)\in \Z^2$ such that $k_1\geq k_2>0$. Let $V_k$ be the $\GSp_4$-representation of highest weight $k$ and let $V_k^{\vee}$ be its dual. Then there exists a Hecke- and Galois-stable $4$-step filtration $\Fil_{\ES}^{\bullet}$ on $H_{\et}^3(X_{n, \C_p}, V_k^{\vee})\otimes_{\Q_p}\C_p$, whose graded pieces give rise to a Hecke- and Galois-equivariant decomposition \begin{equation}\label{eq: FC's ES}
        \begin{array}{crl}
            H_{\et}^3(X_{n, \C_p}, V_k^{\vee})\otimes_{\Q_p}\C_p \cong & & H^0(X_{n, \Q_p}^{\tor}, \underline{\omega}^{k+(3,3)})\otimes_{\Q_p}\C_p(-3)\\
            & \oplus & H^1(X_{n, \Q_p}^{\tor}, \underline{\omega}^{\bfitw_3^{-1}\bfitw_2k+(3, 1)})\otimes_{\Q_p}\C_p(k_2-2)\\
            & \oplus &  H^2(X_{n, \Q_p}^{\tor}, \underline{\omega}^{\bfitw_3^{-1}\bfitw_1k+(2,0)})\otimes_{\Q_p}\C_p(k_1-1)\\
            & \oplus &  H^3(X_{n, \Q_p}^{\tor}, \underline{\omega}^{\bfitw_3^{-1} k})\otimes_{\Q_p}\C_p(k_1+k_2).
        \end{array}
    \end{equation}
\end{Theorem}

Our goal is to $p$-adically interpolate the decomposition in Theorem \ref{thm: Faltings-Chai}. To achieve this goal, we must move to the world of $p$-adic geometry. Firstly, let $\calX_n$ and $\calX_n^{\tor}$ be the rigid analytic varieties (viewed as adic spaces over $\Spa(\C_p, \calO_{\C_p})$) associated with $X_{n, \C_p}$ and $X_{n, \C_p}^{\tor}$. We have morphisms \[
    \begin{tikzcd}
        \calX_{\Gamma(p^{\infty})}^{\tor} \arrow[r, "\pi_{\HT}"] \arrow[d, "h_n"] & \adicFL\\
        \calX_n^{\tor}
    \end{tikzcd}
\]
where \begin{itemize}
    \item $\calX_{\Gamma(p^{\infty})}^{\tor}$ is the (toroidally compactified) perfectoid Siegel modular variety studied in \cite{Pilloni-Stroh}, 
    \item $\adicFL$ is the adic space over $\Spa(\C_p, \calO_{\C_p})$ associated with the flag variety $\Fl= P_{\mathrm{Si}} \backslash \GSp_4$, where $P_{\mathrm{Si}}$ is the Siegel parabolic subgroup, 
    \item $\pi_{\HT}$ is the Hodge--Tate period map studied in \cite{Pilloni-Stroh}, 
    \item $h_n$ is the natural projection map. 
\end{itemize}
Note that $h_n \colon \calX_{\Gamma(p^{\infty})}^{\tor} \rightarrow \calX_n^{\tor}$ is a Galois pro-Kummer \'etale cover (in the sense of \cite{Diao}) with Galois group $\Iw_{\GSp_4, n}^+$.

Secondly, our construction involves studying various $\bfitw$-loci (and open subspaces of such) of the Siegel modular varieties. Using the Bruhat decomposition $\Fl = \bigsqcup_{\bfitw \in W^H} \Fl_{\bfitw}$, we consider various loci $\Fl_{\F_p, \bfitw}$, $\Fl_{\F_p, \leq \bfitw}$, and $\Fl_{\F_p, \geq\bfitw}$ which yield loci $\adicFL_{\bfitw}$, $\adicFL_{\leq \bfitw}$, and $\adicFL_{\geq \bfitw}$ by taking tubular neighbourhoods. We also need to consider certain open subsets $\adicFL_{\bfitw, (r,s)}$ of $\adicFL_{\bfitw}$ for $r,s\in \Q_{\geq 0}$. \footnote{In particular, $\adicFL_{\bfitw, (0,0)}$ is precisely $\adicFL_{\bfitw}$.} Pulling back these loci via the Hodge--Tate period map, we obtain the corresponding loci $\calX_{n, \bfitw}^{\tor}$, $\calX_{n, \leq \bfitw}^{\tor}$, $\calX_{n, \geq \bfitw}^{\tor}$, and $\calX_{n, \bfitw, (r,s)}^{\tor}$ on the Siegel modular varieties. See \S \ref{subsection: flag variety} and \S \ref{subsection: overconvergent Siegel modular forms: perfectoid} for more details.

These loci yield a stratification 
\[
        \calX_{n}^{\tor} = \overline{\calX_{n, \leq \bfitw_3}^{\tor}} \supsetneq \overline{\calX_{n, \leq \bfitw_2}^{\tor}} \supsetneq \overline{\calX_{n, \leq \bfitw_1}^{\tor}} \supsetneq \overline{\calX_{n, \leq \bfitw_0}^{\tor}}
    \]
of $\calX_n^{\tor}$ where $\overline{\calX_{n, \leq \bfitw}^{\tor}}$ denotes the closure of $\calX_{n, \leq \bfitw}^{\tor}$ in $\calX_n^{\tor}$. Figure \ref{Fig: cartoon} illustrates the corresponding strata. The dashed lines (resp., solid lines) roughly indicate where the strata are open (resp., closed). The arrows around the $2\times 2$ box demonstrate the dynamics of the $U_p$-operator. For example, on $\overline{\calX_{n, \leq \bfitw_2}^{\tor}}\smallsetminus\overline{\calX_{n, \leq \bfitw_1}^{\tor}}$, the $U_p$-operator moves the points outward in one direction, but inward in the other direction.

\begin{figure}[ht]
    \centering
    \begin{tikzpicture}
        \draw (0,0) -- (10,0) -- (10,4) -- (0,4) -- (0,0);
        \draw (0,2) -- (10,2);
        \draw (5,0) -- (5,4);
        \draw (2.5,0.7) node[anchor = south]{$\calX_{n}^{\tor} \smallsetminus \overline{\calX_{n, \leq \bfitw_2}^{\tor}}$};
        \draw (2.5,2.7) node[anchor = south]{$\overline{\calX_{n, \leq \bfitw_2}^{\tor}} \smallsetminus \overline{\calX_{n, \leq \bfitw_1}^{\tor}}$};
        \draw (7.5,2.7) node[anchor = south]{$\overline{\calX_{n, \leq \bfitw_1}^{\tor}} \smallsetminus \overline{\calX_{n, \leq \bfitw_0}^{\tor}}$};
        \draw (7.5,0.7) node[anchor = south]{$\overline{\calX_{n, \leq \bfitw_0}^{\tor}}$};
        \draw[->] (-0.5, 3.5) -- (-0.5, 0.5);
        \draw[->] (9.5, 4.5) -- (0.5, 4.5);
        \draw[->] (10.5, 0.5) -- (10.5, 3.5);
        \draw[->] (9.5, -0.5) -- (0.5, -0.5);
        \draw[dashed] (0,1.95) -- (4.95, 1.95) -- (4.95, 0);
        \draw[dashed] (4.95, 4) -- (4.95, 2);
        \draw[dashed] (5.05, 2.05) -- (10, 2.05);
        \draw (-0.5, 2) node[anchor=east]{$U_p$};
        \draw (5, -0.5) node[anchor=north]{$U_p$};
        \draw (10.5, 2) node[anchor=west]{$U_p$};
        \draw (5, 4.5) node[anchor=south]{$U_p$};
    \end{tikzpicture}
    \caption{Stratification of $\calX_{n}^{\tor}$}
    \label{Fig: cartoon}
\end{figure}

Thirdly, we need the notion of families of $p$-adic weights. Let $\calW$ be the \emph{weight space} which parameterises $p$-adic weights (cf. \S \ref{subsection: weight space and analytic representations}). Then, by a \emph{family of $p$-adic weights}, we mean an affinoid open $\calU = \Spa(R_{\calU}, R_{\calU}^{\circ}) \hookrightarrow \calW$; we denote by $(R_{\calU}, \kappa_{\calU})$ (or just $\kappa_{\calU}$) the corresponding weight character.

Now, we are ready to $p$-adically interpolate the objects on both sides of \eqref{eq: FC's ES}. On the side of coherent cohomology groups, for a suitable $r\in \Q_{\geq 0}$, we can define the \emph{$(\bfitw, r)$-overconvergent automorphic sheaves} $\underline{\omega}_{n, r}^{\bfitw_3^{-1}\bfitw\kappa_{\calU}}$ on $\calX_{n, \bfitw, (r,r)}^{\tor}$ following a similar construction as in \cite{DRW} (cf. \S \ref{subsection: overconvergent Siegel modular forms: perfectoid}). More precisely, sections of $\underline{\omega}_{n, r}^{\bfitw_3^{-1}\bfitw\kappa_{\calU}}$ consist of functions on $\calX_{\Gamma(p^{\infty}), \bfitw, (r,r)}^{\tor}$ which are invariant under the action of $\Iw_{\GSp_4, n}^+$ up to a certain automorphy factor. Indeed, when $\bfitw=\bfitw_3$, the sheaf $\underline{\omega}_{n, r}^{\bfitw_3^{-1}\bfitw\kappa_{\calU}}=\underline{\omega}_{n, r}^{\kappa_{\calU}}$ is precisely the overconvergent automorphic sheaf constructed in \emph{loc. cit.} whose global sections give rise to the space of overconvergent Siegel modular forms.\footnote{This also explains the notation `$\bfitw_3^{-1}\bfitw\kappa_{\calU}$' which is designed to match up with the notation in \cite{DRW}.} Following \cite{BP-HigherColeman}, we would like to study (variants of) the cohomology groups of the complex
\begin{equation}\label{eq: BP complex in intro}
    R\Gamma_{\calZ_{n, \bfitw}}(\calX_{n, \bfitw, (r,r)}^{\tor}, \,\,\underline{\omega}_{n, r}^{\bfitw_3^{-1}\bfitw \kappa_{\calU}}),
\end{equation}
where $\calZ_{n, \bfitw}$ is a certain suitable support condition depending on $\bfitw$ and $n$.\footnote{For technical reasons, in the main body of the paper, besides $\calX_{n, \bfitw, (r, r)}^{\tor}$, we will also look at the locus $\calX_{n, \bfitw}^{\tor, \bfitu_p}$ (see \eqref{eq: some technical loci} for its definition) following the spirit of \cite{BP-HigherColeman}. In fact, there is a quasi-isomorphism $R\Gamma_{\calZ_{n, \bfitw}}(\calX_{n, \bfitw, (r,r)}^{\tor}, \underline{\omega}_{n, r}^{\bfitw_3^{-1}\bfitw \kappa_{\calU}}) \cong R\Gamma_{\calZ_{n, \bfitw}}(\calX_{n, \bfitw}^{\tor, \bfitu_p}, \underline{\omega}_{n, r}^{\bfitw_3^{-1}\bfitw \kappa_{\calU}})$ due to \eqref{eq: change of ambient spaces}. } According to the classicality results proved in \cite[Theorem 5.12.3]{BP-HigherColeman}, the complex indeed $p$-adically interpolates the coherent cohomology groups of the classical automorphic sheaves. Recall that on certain strata (cf. Figure \ref{Fig: cartoon}) the $U_p$ operator may move points outward. The support condition remedies this discrepancy. In particular, the $U_p$ operator indeed act on these cohomology groups with support.

On the other hand, to $p$-adically interpolate the \'etale cohomology groups in \eqref{eq: FC's ES}, we consider the modules of distributions $D_{\kappa_{\calU}}^r$ of Ash--Stevens. These modules of distributions are designed to $p$-adically interpolate $V_k^{\vee}$'s. For our purpose, we further consider the associated sheaf of $\widehat{\scrO}_{\calX_{n, \proket}^{\tor}}$-modules $\scrO\!\!\scrD_{\kappa_{\calU}}^r$ on the pro-Kummer \'etale site $\calX_{n, \proket}^{\tor}$ and consider the pro-Kummer \'etale cohomology groups $H^3_{\proket}(\calX_n^{\tor}, \scrO\!\!\scrD_{\kappa_{\calU}}^r)$ (cf. \S \ref{section: OC}). In order to construct an explicit filtration of $H^3_{\proket}(\calX_n^{\tor}, \scrO\!\!\scrD_{\kappa_{\calU}}^r)$ interpolating the filtration in Theorem \ref{thm: Faltings-Chai}, we need a theory of \emph{pro-Kummer \'etale cohomology with support}. This is a key new input of our paper which is discussed in \S \ref{section: cohomology with supports}. In particular, there is a spectral sequence 
\begin{equation*}
    E_1^{i,j} = H_{\overline{\calX_{n, \leq \bfitw_{3-j}}^{\tor}}\smallsetminus \overline{\calX_{n, \leq \bfitw_{3-j-1}}^{\tor}}, \proket}^{i+j}(\calX_n^{\tor} \smallsetminus \overline{\calX_{n, \leq \bfitw_{3-j-1}}^{\tor}}, \scrO\!\!\scrD_{\kappa_{\calU}}^r) \Rightarrow H_{\proket}^{i+j}(\calX_n^{\tor}, \scrO\!\!\scrD_{\kappa_{\calU}}^r)
\end{equation*}
which allows us to compute the desired pro-Kummer \'etale cohomology group in terms of various cohomology groups with supports.

Finally, putting everything together, we would like to relate the aforementioned pro-Kummer \'etale cohomology groups (with or without supports) to the cohomology groups of the complex \eqref{eq: BP complex in intro}. The key is to construct Hecke- and Galois-equivariant morphisms
\begin{equation}\label{eq: ES map at the level of proket sheaves; intro}
    \mathrm{ES}_{\kappa_{\calU}}^{\bfitw, r} \colon \scrO\!\!\scrD_{\kappa_{\calU}}^r \rightarrow \widehat{\underline{\omega}}_{n, r}^{\bfitw_3^{-1}\bfitw\kappa_{\calU}}(\bfitw\kappa_{\calU}^{\cyc})
\end{equation}
of sheaves on the pro-Kummer \'etale site $\calX_{n, \bfitw, (r, r), \proket}^{\tor}$. Here, $\widehat{\underline{\omega}}_{n, r}^{\bfitw_3^{-1}\bfitw\kappa_{\calU}}$ is the completed pullback of $\underline{\omega}_{n, r}^{\bfitw_3^{-1}\bfitw\kappa_{\calU}}$ to the pro-Kummer \'etale site, and $\bfitw\kappa_{\calU}^{\cyc}$ stands for the `cyclotomic twist' of $\kappa_{\calU}$ defined by
\[
    \bfitw\kappa_{\calU}^{\cyc} = \left\{ \begin{array}{ll}
        0, & \text{if }\bfitw = \bfitw_3 \\
        \kappa_{\calU, 2}(\chi_{\cyc}), & \text{if }\bfitw = \bfitw_2\\
        \kappa_{\calU, 1}(\chi_{\cyc}), & \text{if }\bfitw = \bfitw_1\\
        \kappa_{\calU,1}(\chi_{\cyc})\kappa_{\calU, 2}(\chi_{\cyc}) & \text{if }\bfitw = \bfitw_0 = \one_4
    \end{array} \right. 
\] 
where $\kappa_{\calU} = (\kappa_{\calU, 1}, \kappa_{\calU, 2})$ and $\chi_{\cyc}: \mathrm{Gal}_{\Q_p}\rightarrow \Z_p^{\times}$ stands for the $p$-adic cyclotomic character. When $\bfitw = \bfitw_3$, the morphism $\mathrm{ES}_{\kappa_{\calU}}^{\bfitw, r}$ is the same as the one studied in \cite{DRW}.

Our main constructions are summarised in the following theorem.

\begin{Theorem}[Theorem \ref{Theorem: big OES diagram}]\label{thm: main thm intro}
    The morphisms $\mathrm{ES}_{\kappa_{\calU}}^{\bfitw, r}$ induces a natural Hecke- and Galois-equivariant diagram \[
    \begin{tikzcd}
            \scalemath{1}{ H^3_{\proket}(\calX_n^{\tor}, \scrO\!\!\scrD_{\kappa_{\calU}}^r)^{\fs} } \arrow[r] & \scalemath{1}{ H^0(\calX_{n, \bfitw_3, (r,r)}^{\tor}, \underline{\omega}_{n, r}^{\kappa_{\calU} + (3,3)})^{\fs}(-3) }\\
            \scalemath{1}{ H^3_{\overline{\calX_{n, \leq \bfitw_2}^{\tor}}, \proket}(\calX_n^{\tor}, \scrO\!\!\scrD_{\kappa_{\calU}}^r)^{\fs} } \arrow[r] \arrow[u]  & \scalemath{1}{ H_{\calZ_{n, \bfitw_2}}^1(\calX_{n, \bfitw_2, (r,r)}^{\tor}, \underline{\omega}_{n, r}^{\bfitw_3^{-1}\bfitw_2\kappa_{\calU} + (3,1)})^{\fs}(\kappa_{\calU, 2} - 2) }\\
            \scalemath{1}{ H^3_{\overline{\calX_{n, \leq \bfitw_1}^{\tor}}, \proket}(\calX_n^{\tor}, \scrO\!\!\scrD_{\kappa_{\calU}}^r)^{\fs} } \arrow[r] \arrow[u]  & \scalemath{1}{ H_{\calZ_{n, \bfitw_1}}^2(\calX_{n, \bfitw_1, (r,r)}^{\tor}, \underline{\omega}_{n, r}^{\bfitw_3^{-1}\bfitw_1\kappa_{\calU} + (2,0)})^{\fs}(\kappa_{\calU, 1} - 1) }\\
            \scalemath{1}{ H_{\overline{\calX_{n, \one_4}^{\tor}}, \proket}^3(\calX_n^{\tor}, \scrO\!\!\scrD_{\kappa_{\calU}}^r)^{\fs} } \arrow[r]\arrow[u]  & \scalemath{1}{ H_{\calZ_{n, \one_4}}^3(\calX_{n, \one_4, (r,r)}^{\tor}, \underline{\omega}_{n, r}^{\bfitw_3^{-1}\kappa_{\calU}})^{\fs}(\kappa_{\calU, 1}+\kappa_{\calU, 2}) }
        \end{tikzcd}
    \]
    where the superscript `$\bullet^{\fs}$' stands for `taking the finite-slope part'. 
\end{Theorem}

The horizontal arrows in the diagram are referred to as the \emph{overconvergent Eichler--Shimura morphisms}, as indicated in the title of the article. The cohomology groups appearing on the left half of the diagram give rise to a filtration of $H^3_{\proket}(\calX_n^{\tor}, \scrO\!\!\scrD_{\kappa_{\calU}}^r)^{\fs}$ which $p$-adically interpolates the filtration $\Fil^{\bullet}_{\ES}$ in Theorem \ref{thm: Faltings-Chai}, while the cohomology groups on the right half of the diagram $p$-adically interpolate the cohomology groups of classical automorphic sheaves. 

Can we do better? One might hope to achieve an interpolation of the Eichler--Shimura decomposition itself, rather than merely interpolating the filtration. That is to ask when do the cohomology groups on the right half of the diagram coincide with the graded pieces of the filtration (maybe after further taking the `small-slope part'). Indeed, we are able to prove this locally at a \emph{nice-enough} point (cf. Definition \ref{Defn: nice enough}; also see Assumption \ref{Assumption: Multiplicity One for cuspidal automorphic representations} and Remark \ref{Remark:1dimensional}) on the middle-degree eigenvariety $\calE$ constructed in \S \ref{subsection: eigenvarieties}.

\begin{Theorem}[Theorem \ref{Theorem: overconvergent Eichler--Shimura decomposition}]\label{Theorem: OES decomposition, intro}
Let $\calE$ be the middle degree eigenvariety and let $\mathrm{wt}:\calE\rightarrow \calW$ be the weight map. Let $\Pi$ be a nice-enough automorphic representation for $\GSp_4$ which defines a point $x_{\Pi}$ on $\calE$. Then there exists an affinoid neighbourhood $\calV\subset \calE$ of $x_{\Pi}$ such that
\begin{enumerate}
\item[(i)] $\calV$ is a connected component of $\mathrm{wt}^{-1}(\calU)$ where $\calU=\Spa(R_{\calU}, R_{\calU}^{\circ})\subset \calW$ is an affinoid subspace corresponding to a family of $p$-adic weights $(R_{\calU}, \kappa_{\calU})$;
\item[(ii)] There exists $h\in \Q_{\geq 0}$ such that $(R_{\calU}, \kappa_{\calU})$ is slope-$h$-adapted (see Theorem \ref{Theorem: overconvergent Eichler--Shimura decomposition}); 
\item[(iii)]The decreasing filtration $\Fil_{\mathrm{ES}, \calV}^{\bullet}$ on $e_{\calV}H_{\proket}^3(\calX_n^{\tor}, \scrO\!\!\scrD_{\kappa_{\calU}}^r)^{\leq h}$ defined by 
\begin{itemize}
   \item $\Fil^0_{\mathrm{ES}, \calV} \coloneq e_{\calV}H_{\proket}^3(\calX_n^{\tor}, \scrO\!\!\scrD_{\kappa_{\calU}}^r)^{\leq h}$;
   \item $\Fil_{\ES, \calV}^{3-i}  \coloneq e_{\calV} \image\left( H^3_{\overline{\calX_{n, \leq \bfitw_i}^{\tor}}, \proket}(\calX_n^{\tor}, \scrO\!\!\scrD_{\kappa_{\calU}}^r)^{\leq h} \rightarrow H_{\proket}^3(\calX_n^{\tor}, \scrO\!\!\scrD_{\kappa_{\calU}}^r)^{\leq h} \right)$ for $i=0,1,2$;
\item $\Fil^4_{\mathrm{ES}, \calV} \coloneq 0$
\end{itemize} 
is Hecke- and Galois-stable, where $e_{\calV}$ is the idempotent operator corresponding to $\calV$ and `$\leq h$' stands for the slope $\leq h$-part.
\item[(iv)] The graded pieces of the filtration $\Fil_{\mathrm{ES}, \calV}^{\bullet}$ admit canonical Hecke- and Galois-equivariant isomorphisms
\[
        \Gr_{\mathrm{ES}, \calV}^{3-i} \cong e_{\calV}H_{\calZ_{n, \bfitw_i}}^{3-i}(\calX^{\tor}_{n, \bfitw_i, (r, r)}, \underline{\omega}_{n, r}^{\bfitw_3^{-1}\bfitw_i\kappa_{\calU}+k_{\bfitw_i}})^{\leq h}(\bfitw_i\kappa_{\calU}^{\cyc} - i), 
    \] 
of $R_{\calU}\widehat{\otimes}\C_p$-modules, where \[
        k_{\bfitw_i} = \left\{ \begin{array}{cc}
            (3,3), & i=3 \\
            (3,1), & i=2 \\
            (2,0), & i=1 \\
            (0,0), & i=0
        \end{array}\right. .
        \]
\end{enumerate}
Moreover, there is a Hecke- and Galois-equivariant decomposition 
    \[
        e_{\calV}H_{\proket}^3(\calX_n^{\tor}, \scrO\!\!\scrD_{\kappa_{\calU}}^r)^{\leq h} \cong \bigoplus_{i=0}^{3} e_{\calV}H_{\calZ_{n, \bfitw_i}}^{3-i}(\calX^{\tor}_{n, \bfitw_i, (r,r)},\,\,\underline{\omega}_{n, r}^{\bfitw_3^{-1}\bfitw_i\kappa_{\calU}+k_{\bfitw_i}})^{\leq h}(\bfitw_i\kappa_{\calU}^{\cyc} - i)
    \] of $R_{\calU}\widehat{\otimes}\C_p$-modules, specialising to the Eichler--Shimura decompositions in Theorem \ref{thm: Faltings-Chai}. 
\end{Theorem}

\begin{Remark}
A key contribution of Theorem \ref{Theorem: OES decomposition, intro} is that it determines the Hodge--Tate--Sen weights in the $p$-adic interpolation of overconvergent cohomology groups: the weights are precisely $\bfitw_i\kappa_{\calU}^{\cyc} - i$ for $i=0,1,2,3$. We pin down these weights when we calculate the Tate twists in the Hecke- and Galois-equivariant morphisms (\ref{eq: ES map at the level of proket sheaves; intro}) on each stratum of the stratification (cf. Figure \ref{Fig: cartoon}). In particular, our method is completely different from the one in \cite{Faltings-Chai} (cf. Theorem \ref{thm: Faltings-Chai}).
\end{Remark}

As an application of Theorem \ref{Theorem: OES decomposition, intro}, we prove the following.

\begin{Corollary}[Corollary \ref{Corollary: etaleness of the weight map} and \ref{Corollary: big Galois representation}]\label{Corollary: big Galois repl intro}
    Let $\Pi$, $x_{\Pi}$, $\calV$, $\kappa_{\calU}$, and $R_{\calU}$ be as in Theorem \ref{Theorem: OES decomposition, intro}. Then we have:
     \begin{enumerate}
        \item[(1)] The weight map $\wt \colon \calE \rightarrow \calW$ is \'etale at $x_{\Pi}$. 
        \item[(2)] There exists a family of Galois representations \[
            \rho_{\calV}
         \colon  \Gal_{\Q} \rightarrow \GL_4(R_{\calU})
        \]
        attached to $\calV$ such that \begin{enumerate}
            \item[(i)] $\rho_{\calV}$ is unramified at $\ell \nmid Np$ and the characteristic polynomial of the geometric Frobenius at $\ell$ agrees with the Hecke polynomial at $\ell$; 
            \item[(ii)] $\rho_{\calV}|_{\Gal_{\Q_p}}$ admits a Galois-stable decreasing filtration and has Hodge--Tate--Sen weight $(-3, \kappa_{\calU, 2}-2, \kappa_{\calU, 1}-1, \kappa_{\calU, 1}+\kappa_{\calU, 2})$, where the ordering respects the indices of the graded pieces of the filtration. 
        \end{enumerate}
    \end{enumerate}
\end{Corollary}

The upshot of Corollary \ref{Corollary: big Galois repl intro} is that our new construction of the big Galois representations does not use Galois determinants.

\begin{Remark}
    In his thesis, J. E. Rodr\'iguez Camargo (\cite{Rodriguez-phd}) obtained a similar result for the completed cohomology groups (\`a la Emerton) using BGG resolution. In contrast, we study the overconvergent cohomology groups (\`a la Ash--Stevens) and our techniques are essentially different. The method of Rodriguez Camargo is expected to have some implications in modularity lifting questions, while our method is more suitable for constructing new $p$-adic $L$-functions over the eigenvarieties (see, for example, \cite{LPSZ} and \cite[\S 3.2]{LZ-BKGSp4}). We also expect applications in the study of geometry of eigenvarieties (for example, generalising the Halo conjecture in \cite{diaoyao} to the Siegel case).
\end{Remark}

\begin{Remark}
We expect the constructions and results in this article to generalise to more general Shimura varieties, at least to the case of Shimura varieties of PEL-type.
\end{Remark}

\subsection{Outline of the paper}
This article is organised as follows.

In \S \ref{section: flag variety}, we study the adic flag variety $\adicFL$ in details. In \S \ref{subsection: the group GSp4}, \S \ref{subsection: flag variety}, and \S \ref{subsection: vector bundles and torsors over the flag variety}, we introduce various $\bfitw$-loci on $\adicFL$ as well as sheaves on such. These materials are highly inspired by \cite[\S 3]{BP-HigherColeman}, yet we provide detailed and concrete computations. We prove a simple multiplicity-one property for algebraic representations for $\GSp_4$ in \S \ref{subsection: lemma for algebraic representations}. In \S \ref{subsection: weight space and analytic representations}, we define the notion of $p$-adic weight space and analytic representations. To wrap up the section, we introduce the notion of \emph{pseudoautomophic sheaves} on the flag variety in \S \ref{subsection: pseudoautomorphic sheaves}. Via the Hodge--Tate period map, these sheaves are closely related to the automorphic sheaves on the Siegel modular varieties studied in \S \ref{section: automorphic}. These sheaves play a central role in the construction of the morphisms $\mathrm{ES}_{\kappa_{\calU}}^{\bfitw, r}$.

The purpose of \S \ref{section: automorphic} is to study the classical and overconvergent automorphic sheaves on various loci on the Siegel modular variety. This generalises our previous construction in \cite{DRW}. We provide two different ways to construct the sheaves: one through the perfectoid method (\S \ref{subsection: overconvergent Siegel modular forms: perfectoid}) and the other uses analytic torsors (\S \ref{subsection: overconvergent Siegel modular forms: torsors}). A comparison of these two constructions is given by Theorem \ref{Theorem: comparison theorem for automorphic sheaves}. In \S \ref{subsection: Hecke operators}, we discuss the Hecke operators acting on the cohomology of these automorphic sheaves (with or without supports). In \S \ref{subsection: proket coh with support for classical aut sheaves}, we prove a classicality result for pro-Kummer \'etale cohomology groups with support. Again, a major part of this section is inspired by the work of Boxer--Pilloni, yet we spell out the details. 

We introduce the overconvergent cohomology groups in \S \ref{section: OC}. As a starter, \S \ref{subsection: Betti overconvergent cohomology} is a quick review of the modules of analytic functions and distributions of Ash--Stevens. These modules serve as coefficients in the Betti cohomology of the Siegel threefolds. In \S \ref{subsection: Kummer and pro-Kummer \'etale overconvergent cohomology}, we discuss how to view these Betti cohomology groups as certain (pro-)Kummer \'etale cohomology groups, using a similar technique developed in \cite{DRW}. The novelty of this section is \S \ref{subsection: pro-Kummer \'etale overconvergent cohomology with supports} where we further study pro-Kummer \'etale cohomology groups with support conditions coming from various stratifications on the Siegel threefolds. We also discuss the Hecke operators on those cohomology groups.

Finally, in \S \ref{section: OES}, we construct the overconvergent Eichler--Shimura morphisms and prove the main theorems. We start in \S \ref{subsection: classical ES} with an alternative perspective to understand the classical Eichler--Shimura decomposition of Faltings--Chai. These observations inspire our main construction in \S \ref{subsection: OES} and will be useful when we study the decompositions around a nice-enough point on the eigenvariety. In \S \ref{subsection: OES}, we construct the morphisms $\mathrm{ES}_{\kappa_{\calU}}^{\bfitw, r}$ and prove the main theorem. It is important to study the behavior of these morphisms when specialising at classical weights. This is treated in \S \ref{subsection: OES at classical weights}. The purpose of \S \ref{subsection: eigenvarieties} is to establish some preliminary results on eigenvarieties. In particular, we show that the middle-degree equidimensional eigenvariety (\`a la Hansen) is isomorphic to the equidimensional eigenvariety considered in \cite{BP-HigherColeman} (see Proposition \ref{Proposition: comparison of eigenvarieties}). In \S \ref{subsection: ES filtration on eigenvariety}, we prove the decomposition result around a nice-enough point on the eigenvariety. As an application, we provide a new construction of the big Galois representations. Finally, in \S \ref{subsection: non-neat level}, we provide a strategy to deal with non-neat levels (for example, paramodular levels). 

In the appendix, we introduce a cohomology theory with supports on the analytic, Kummer \'etale, and pro-Kummer \'etale sites of a locally noetherian fs log adic space. Although this approach does not lead to a full six-functor formalism, it is sufficient for our purpose.

\section*{Acknowledgement}
J.-F.W. would like to thank George Boxer, Kevin Buzzard, and Fred Diamond for helpful conversations after he presented this work at the London Number Theory Seminar. He would also like to thank David Loeffler for constructive comments on the preliminary draft of this paper, leading to an improvement of Theorem \ref{Theorem: OES decomposition, intro}; he thanks David Loeffler and Sarah Zerbes for their hospitality during his visit to FernUni Schweiz and ETH Zürich. 

During the preparation of this work, H. D. was partially supported by the National Key R\&D Program of China No.~2023YFA1009703 and No.~2021YFA1000704, and the National Natural Science Foundation of China No.~12422101; G.R. was partly funded by the NOVA-FRQNT-CRSNG grant~325940 and the NSERC grant RGPIN-2018-04392; J.-F.W. was supported by the ERC Consolidator grant ‘Shimura varieties and the BSD conjecture’ and Taighde \'{E}ireann -- Research Ireland under Grant number IRCLA/2023/849 (HighCritical).

\section*{Conventions and notations}
Throughout this article, we fix the following. \begin{itemize}
    \item $p\in \Z_{> 0}$ is an odd prime number. 
    \item $N\in \Z_{\geq 3}$ is an integer coprime to $p$.
    \item We fix once and forever an algebraic closure $\overline{\Q}_p$ of $\Q_p$ and an algebraic isomorphism $\C_p\cong \C$, where $\C_p$ is the $p$-adic completion of $\overline{\Q}_p$. We write $\Gal_{\Q_p}$ for the absolute Galois group $\Gal(\overline{\Q}_p/\Q_p)$. We also fix the $p$-adic absolute value on $\C_p$ so that $|p|=p^{-1}$.
    \item For any $r\in \Q_{\geq 0}$, we denote by `$p^r$' an element in $\C_p$ with absolute value $p^{-r}$. All constructions in the paper will not depend on such choices.
    \item For $n\in \Z_{\geq 1}$ and any ring $R$, we denote by $M_n(R)$ the set of all $n$ by $n$ matrices with entries in $R$.
    \item Matrices are often denoted by bold greek letters (\emph{e.g.}, $\bfalpha$, $\bfgamma$, $\bftau$). The transpose of a matrix $\bfalpha$ is denoted by $\trans\bfalpha$.
    \item For any $n\in \Z_{\geq 1}$, we denote by $\one_n$ the $n\times n$ identity matrix and denote by $\oneanti_n$ the $n\times n$ anti-diagonal matrix whose non-zero entries are $1$; \emph{i.e.,} \[
\one_n=\begin{pmatrix} 1& & \\ & \ddots & \\ & &1\end{pmatrix}\quad\text{ and }\quad\oneanti_n=\begin{pmatrix} & & 1\\ & \iddots & \\ 1 & &\end{pmatrix}.
    \]
    \item We adopt the language of almost mathematics. In particular, for an $\calO_{\C_p}$-module $M$, we denote by $M^a$ the associated almost $\calO_{\C_p}$-module with respect to the maximal ideal $\mathfrak{m}_{\C_p}$.
        \item For a topological space $T$ and a subset $S\subset T$, we denote by $\overline{S}$ (resp., $\mathring{S}$) the closure of $S$ in $T$ (resp., the interior of $S$ in $T$).
    \item Throughout the paper, the completed tensor symbol `$\widehat{\otimes}$' without subscript stands for either the \emph{complete tensor product} or the \emph{mixed complete tensor product} following the convention of \cite[Convention 2.2]{CHJ-2017}. 
    \item We freely use the terminologies in \cite[\S 2.4]{BP-HigherColeman}. In particular, given a complete Tate algebra $(R, R^+)$ of finite type over $(\Q_p, \Z_p)$,  we adopt the following notations. 
    \begin{itemize}
        \item Let $\Ban(R)$ denote the category of Banach $R$-modules;
        \item Let $\mathrm{C}(\Ban(R))$ denote the category of complexes of Banach $R$-modules and let $\mathrm{K}(\Ban(R))$ (resp., $\mathrm{D}(\Ban(R))$) denote the corresponding homotopy category (resp., derived category);\footnote{ Note that the category of Banach $R$-modules is not abelian. The derived category of Banach $R$-modules is actually defined as the localisation of the homotopy category of Banach $R$-modules with respect to the \emph{strict} quasi-isomorphisms.} 
        \item Let $\mathrm{C}^{\proj}(\Ban(R))$ denote the category of bounded complexes of \emph{projective} Banach $R$-modules (i.e., those Banach $R$-modules that have (Pr)). Let $\mathrm{K}^{\proj}(\Ban(R))$ denote the corresponding homotopy category;\footnote{There is a natural functor $\mathrm{K}^{\proj}(\Ban(R))\rightarrow \mathrm{D}(\Ban(R))$ which is fully faithful.}
        \item Let $\mathrm{Pro}_{\Z_{\geq 0}}(\mathrm{K}^{\proj}(\Ban(R)))$ denote the category of projective systems of complexes $\{ K_i\}_{i\in \Z_{\geq 0}}$ in $\mathrm{K}^{\proj}(\Ban(R))$ such that the $K_i$'s have non-zero cohomology in a uniformly bounded range of degrees. Objects in $\mathrm{Pro}_{\Z_{\geq 0}}(\mathrm{K}^{\proj}(\Ban(R)))$ are simply denoted by $\lim_i K_i$, instead of ``$\lim_i$''$K_i$ as in \cite[\S 2.4]{BP-HigherColeman}. There is a natural functor $\mathrm{Pro}_{\Z_{\geq 0}}(\mathrm{K}^{\proj}(\Ban(R)))\rightarrow \mathrm{D}(R)$ by forgetting the topology and `taking the limit'.
    \end{itemize}
    Moreover, we follow \cite[\S 2.4]{BP-HigherColeman} for the notions of \emph{compact morphisms} between such objects, and follow \cite[\S 6.1]{BP-HigherColeman} for the corresponding slope theory. Also see Definition \ref{Definition: potent compact operators}, Proposition-Definition \ref{Prop-Defn: finite slope part}, and Proposition-Definition \ref{Prop-Defn: finite slope part pro}.
    \item We adopt the language of \emph{Banach sheaves} (over an adic space) from \cite[\S 2.5]{BP-HigherColeman}.
    \item In principle, 
    symbols in calligraphic font (\emph{e.g.}, $\calX, \calY, \calZ$) are reserved for adic spaces; and symbols in script font (\emph{e.g.}, $\scrO, \scrF, \scrE$) are reserved for sheaves (over various geometric objects). 
\end{itemize}
\section{The flag variety}\label{section: flag variety}
In this section, we study the properties of the flag variety for $\GSp_4$ that we will use in the subsequent sections. Many of the ingredients are taken from \cite{BP-HigherColeman} with a special focus on the algebraic group $\GSp_4$. 

\subsection{Preliminaries on \texorpdfstring{$\GSp_4$}{GSp4}}\label{subsection: the group GSp4}

Let $V\coloneq\Z^{4}$ be equipped with an alternating pairing \begin{equation}\label{eq: symplectic pairing on standard rep}
    \bla\cdot, \cdot\bra:V\times V\rightarrow \Z, \quad (\vec{v}, \vec{v}')\mapsto \trans\vec{v}\begin{pmatrix} & -\oneanti_2\\ \oneanti_2 & \end{pmatrix}\vec{v}',
\end{equation}
where we view elements in $V$ as column vectors. In particular, if $e_1, ..., e_{4}$ is the standard basis for $V$, then \[
    \bla e_i, e_j\bra=\left\{\begin{array}{ll}
        -1, & \text{if }i<j\text{ and } j=5-i\\
        1, & \text{if }i>j\text{ and }j=5-i\\
        0, & \text{else}
    \end{array}\right. .
\]
We define the algebraic group $\GSp_{4}$ to be the subgroup of $\GL_{4}$ that preserves this pairing up to a unit. In other words, for any ring $R$, \[
    \GSp_{4}(R)\coloneq\left\{\bfgamma\in \GL_{4}(R): \trans\bfgamma \begin{pmatrix} & -\oneanti_2\\ \oneanti_2\end{pmatrix}\bfgamma = \varsigma(\bfgamma)\begin{pmatrix} & -\oneanti_2\\ \oneanti_2\end{pmatrix}\text{ for some }\varsigma(\bfgamma)\in R^{\times}\right\}.
\] Equivalently, for any $\bfgamma=\begin{pmatrix}\bfgamma_a & \bfgamma_b\\ \bfgamma_c & \bfgamma_d\end{pmatrix}\in \GL_{4}$, $\bfgamma\in \GSp_{4}$ if and only if \[
    \trans\bfgamma_a\oneanti_2\bfgamma_c=\trans\bfgamma_c\oneanti_2\bfgamma_a, \quad \trans\bfgamma_b\oneanti_2\bfgamma_d=\trans\bfgamma_d\oneanti_2\bfgamma_b, \text{ and }\trans\bfgamma_a\oneanti_2\bfgamma_d-\trans\bfgamma_c\oneanti_2\bfgamma_b=\varsigma(\bfgamma)\oneanti_2
\] for some $\varsigma(\bfgamma)\in \bbG_m$.

Due to our choice of the symplectic pairing, we may consider the Borel subgroup $B_{\GSp_{4}}$ defined by the upper triangular matrices in $\GSp_{4}$. We then have the Levi decomposition \[
    B_{\GSp_{4}} = T_{\GSp_{4}}N_{\GSp_{4}},
\] where \begin{enumerate}
    \item[$\bullet$] $T_{\GSp_{4}}$ is the maximal torus given by the diagonal matrices in $\GSp_{4}$; and 
    \item[$\bullet$] $N_{\GSp_{4}}$ is the unipotent radical given by the upper triangular matrices in $\GSp_{4}$ whose diagonal entries are all $1$.
\end{enumerate}

\begin{Remark}\label{Remark: appearance of T_GSp}
\normalfont 
By the definition of $\GSp_{4}$, one easily checks that elements in $T_{\GSp_{4}}$ are of the form \[
    \bftau = \diag(\tau_1,  \tau_2, \tau_0\tau_2^{-1}, \tau_0\tau_1^{-1})
\] for some $\tau_0,  \tau_1,  \tau_2\in \bbG_m$. Consequently, there is a natural isomorphism \[
    T_{\GSp_{4}} \xrightarrow{\cong} \bbG_m^{3}, \quad \diag(\tau_1,  \tau_2, \tau_0\tau_2^{-1},  \tau_0\tau_1^{-1}) \mapsto (\tau_1,  \tau_2; \tau_0).
\]
\end{Remark}

The subgroups $B_{\GSp_{4}}$ and $N_{\GSp_{4}}$ admit their opposite counterpart. That is, we have the opposite Borel $B_{\GSp_{4}}^{\opp}$ given by the lower triangular matrices in $\GSp_{4}$, the corresponding opposite unipotent radical $N_{\GSp_{4}}^{\opp}$ and the Levi decomposition \[
    B_{\GSp_{4}}^{\opp} = N_{\GSp_{4}}^{\opp} T_{\GSp_{4}}.
\]
We use similar notations for those subgroups of $\GL_2$. In particular, we have the upper triangular Borel $B_{\GL_2}$, the corresponding unipotent radical $N_{\GL_2}$, and the maximal torus $T_{\GL_2}$ consists of diagonal matrices.

Let $H \coloneq \GL_2 \times \bbG_m$. This algebraic group can be embedded into $\GSp_4$ via \[
    H \hookrightarrow \GSp_4, \quad (\bfgamma, \bfepsilon)\mapsto \begin{pmatrix}\bfgamma & \\ & \bfepsilon \oneanti_2\trans \bfgamma^{-1}\oneanti_2\end{pmatrix}.
\]
Denote by $T_H = T_{\GL_2} \times \bbG_m$ the maximal torus of diagonal matrices in $H$. We arrive at a natural identification
\[
    T_{\GSp_{4}} \cong \bbG_m^3 \cong T_H.
\]

Let $\bbX = \Hom(T_{\GSp_4}, \bbG_m)$ be the character group of $T_{\GSp_4}$. The isomorphism $T_{\GSp_4} \cong \bbG_m^3$ yields an identification \begin{equation}\label{eq: identification of the character group}
    \Z^{3} \xrightarrow{\cong} \bbX, \quad (k_1, k_2; k_0)\mapsto \left( \bftau = \diag(\tau_1, \tau_2, \tau_0\tau_2^{-1}, \tau_0 \tau_1^{-1}) \mapsto \prod_{i=0}^2 \tau_i^{k_i}\right).
\end{equation}
Under this isomorphism, we denote by $x_1, x_2, x_0$ the basis for $\bbX$ that corresponds to the standard basis for $\Z^3$. Note that, due to the identification $T_{\GSp_4} \cong T_H$, we may also view $\bbX$ as the character group of $T_H$. In what follows, we will often consider the embedding $\Z^2 \xrightarrow{(k_1, k_2)\mapsto (k_1, k_2; 0)} \Z^3$ and view elements in $\Z^2$ as characters in $\bbX$. 

Let $\Phi_{\GSp_4} \subset \bbX$ (resp., $\Phi_H \subset \bbX$) be the root system of $\GSp_4$ (resp., $H$) with respect to the choice of the torus $T_{\GSp_{4}}$ (resp., $T_H$). We can explicitly describe $\Phi_{\GSp_4}$ and $\Phi_H$ as follows: \begin{align*}
    \Phi_{\GSp_{4}} & = \{ \pm(x_1-x_2), \pm(x_1+x_2-x_0), \pm(2x_1-x_0), \pm(2x_2-x_0)\},\\
    \Phi_H & = \{ \pm(x_1-x_2), \pm x_2, \pm x_0\}.
\end{align*}
Moreover, due to our choice of the Borel subgroups, we have the corresponding positive roots \begin{align*}
    \Phi_{\GSp_{4}}^+ & = \{ x_1-x_2, x_1+x_2-x_0, 2x_1-x_0, 2x_2-x_0\},\\
    \Phi_H^+ & = \{ x_1-x_2\} = \Phi_{\GSp_4}^+ \cap \Phi_H.
\end{align*}
Furthermore, we define \[
    \begin{array}{ccc}
        \Phi_{\GSp_{4}}^{-} \coloneq \Phi_{\GSp_{4}}\smallsetminus \Phi_{\GSp_{4}}^+, & \Phi_{H}^- \coloneq \Phi_H \smallsetminus \Phi_{H}^+, \\ \\
        \Phi^H \coloneq \Phi_{\GSp_{4}}\smallsetminus \Phi_H, & \Phi^{+, H} \coloneq \Phi^+_{\GSp_{4}}\smallsetminus \Phi_H^+, &  \Phi^{-, H} \coloneq -\Phi^{+, H}.
    \end{array}
\]

The character group $\bbX$ carries an action of the Weyl group $W_{\GSp_{4}}$ (resp., $W_H$), where $W_{\GSp_4}$ (resp., $W_H$) is defined to be the quotient of the normaliser of $T_{\GSp_4}$ (resp., $T_H$) in $\GSp_4$ (resp., $H$) by $T_{\GSp_{4}}$ (resp., $T_H$). Explicitly, this action can be described as follows: for a given $\bfitw\in W_{\GSp_{4}}$ and $k\in \bbX$, for any $\bftau\in T_{\GSp_4}$, \[
    (\bfitw k)(\bftau) \coloneq k(\bfitw^{-1}\bftau \bfitw).
\]
We  follow \cite{Faltings-Chai} and define \[
    W^H : = \{ \bfitw \in W_{\GSp_{4}}: \bfitw(\Phi_{\GSp_{4}}^+)\supset \Phi_H^+\} \subset W_{\GSp_{4}}.
\]
Elements in $W^H$ are the so-called \emph{Kostant representatives} of the quotient $W_{\GSp_4}/W_H$. It is well-known that $W^H$ can be described explicitly as \begin{equation}\label{eq: explicit Weyl elements}
    \scalemath{0.9}{ W^H = \left\{ \bfitw_0 = \one_4, \bfitw_1 = \left(\begin{array}{cc|cc}
        1 & & &  \\ & & -1 & \\ \hline & 1 & & \\ &&& 1 
    \end{array}\right), \bfitw_2 = \left(\begin{array}{cc|cc}
         & -1 & &  \\ & & & -1 \\ \hline 1 & & & \\ & & -1& 
    \end{array}\right), \bfitw_3 = \left(\begin{array}{cc|cc}
         & & 1 &  \\ & & & -1 \\ \hline 1 & & & \\ & -1 &&
    \end{array}\right) \right\} }.
\end{equation}
The indices of the elements correspond to the lengths of the elements, \emph{i.e.}, $l(\bfitw_i) = i$.

\begin{Remark}
    In the rest of the paper, we often look at the Weyl element $\bfitw_3^{-1}\bfitw_i$ for any $\bfitw_i\in W^H$. Explicit computation shows that \[
        \bfitw_3^{-1}\bfitw_i = \bfitw_{3-i}\in W_{\GSp_4}
    \]
    as Weyl elements (but not as matrices given in \eqref{eq: explicit Weyl elements}). 
\end{Remark}

Finally, we analyse the Lie algebra $\mathfrak{gsp}_4$ of $\GSp_4$. By the root decomposition, we have \[
    \mathfrak{gsp}_{4} = \frakt_{\GSp_{4}} \oplus \frakn_{\GSp_{4}} \oplus \frakn_{\GSp_{4}}^{\opp} = \frakt_{\GSp_{4}} \oplus \big(\oplus_{\alpha\in \Phi_{\GSp_{4}}} \frakn_{\alpha}\big),
\]where \begin{enumerate}
    \item[$\bullet$] $\frakt_{\GSp_{4}}$, $\frakn_{\GSp_{4}}$, and $\frakn_{\GSp_{4}}^{\opp}$ are the Lie algebras of $T_{\GSp_{4}}$, $N_{\GSp_{4}}$, and $N_{\GSp_{4}}^{\opp}$ respectively; 
    \item[$\bullet$] $\frakn_{\GSp_{4}} = \oplus_{\alpha\in \Phi_{\GSp_{4}}^+} \frakn_{\alpha}$ and $\frakn_{\GSp_{4}}^{\opp} = \oplus_{\alpha\in \Phi_{\GSp_{4}}^-} \frakn_{\alpha}$.
\end{enumerate} 
For each $\alpha\in \Phi_{\GSp_{4}}^+$ (resp., $\Phi_{\GSp_{4}}^-$), let $N_{\alpha}$ be the subgroup of $N_{\GSp_{4}}$ (resp., $N_{\GSp_{4}}^{\opp}$) whose Lie algebra is $\frakn_{\alpha}$. In fact, we have \[
    N_{\alpha} \cong \frakn_{\alpha} \cong \bbA^1
\] as schemes over $\Z$. 

The following explicit coordinate systems will be used throughout the article.
\[
   \scalemath{0.9}{ \begin{array}{llll}
        N_{x_1-x_2}  = \left\{\begin{pmatrix} 1 & a^+ & & \\ & 1 & & & \\ & & 1 & -a^+ \\ & & & 1\end{pmatrix}: a^+\in \bbA^1\right\}, & N_{x_1+x_2-x_0}  = \left\{\begin{pmatrix} 1 &  & z_{22}^+ & \\ & 1 & & z_{22}^+ \\ & & 1 & \\ & & & 1\end{pmatrix}: z_{22}^+\in \bbA^1\right\},\\ \\
        N_{2x_1-x_0} = \left\{\begin{pmatrix} 1 & & & z_{12}^+ \\ & 1 & & & \\ & & 1 & \\ & & & 1\end{pmatrix}: z_{12}^+\in \bbA^1\right\}, & N_{2x_2-x_0}  = \left\{\begin{pmatrix} 1 &  & & \\ & 1 & z_{21}^+ & \\ & & 1 & \\ & & & 1\end{pmatrix}:z_{21}^+\in \bbA^1\right\}
    \end{array}}
\]
and 
\[
   \scalemath{0.9}{  \begin{array}{llll}
        N_{-x_1+x_2}  = \left\{\begin{pmatrix} 1 &  & & \\ a^- & 1 & & & \\ & & 1 &  \\ & & -a^- & 1\end{pmatrix}: a^-\in \bbA^1\right\}, & N_{-x_1-x_2+x_0}  = \left\{\begin{pmatrix} 1 &  &  & \\ & 1 & &  \\z_{22}^- & & 1 & \\ & z_{22}^- & & 1\end{pmatrix}: z_{22}^-\in \bbA^1\right\}, \\ \\
        N_{-2x_1+x_0}  =  \left\{\begin{pmatrix} 1 & & &  \\ & 1 & & & \\ & z_{12}^- & 1 & \\ & & & 1\end{pmatrix}: z_{12}^-\in \bbA^1\right\}, & N_{-2x_2+x_0}  = \left\{\begin{pmatrix} 1 &  & & \\ & 1 &  & \\ & & 1 & \\ z_{21}^- & & & 1\end{pmatrix}: z_{21}^-\in \bbA^1\right\}.
    \end{array}}
\]
Here, the `$+$' and `$-$' in the superscripts indicate whether the corresponding roots are positive or negative.

\subsection{Intermezzo: A multiplicity-one lemma for algebraic representations}\label{subsection: lemma for algebraic representations}

The aim of this subsection is to prove a `multiplicity-one' lemma in the theory of algebraic representations for $\GSp_4$. 

To this end, let $k\in\bbX$ be a dominant weight. Let $K$ be a field containing $\Q$ and consider the $\GSp_4$-representation $V_k$ of highest weight $k$ over $K$. Let $e_k^{\hst}$ be the highest weight vector in $V_k$. Recall that the highest weight vector enjoys the following properties: \begin{itemize}
    \item $\mathrm{span}_K \GSp_4 e_k^{\hst} = V_k$;
    \item it is the unique (up to scalar multiplication) non-zero vector $v\in V_k$ such that for any $\bftau \in T_{\GSp_4}$, $\bftau v = k(\bftau) v$.
\end{itemize} We shall see in latter sections (\emph{e.g.}, \S \ref{subsection: classical ES}) an explicit construction of $V_k$ and $e_k^{\hst}$. 

On the other hand, observe that for any $\bfitw \in W^H$, $\bfitw k$ is a dominant weight for $H$. Consider the vector $\bfitw e_k^{\hst}\in V_k$. Observe that for any $\bftau\in T_H \cong T_{\GSp_4}$, we have \[
    \bftau (\bfitw e_k^{\hst}) = \bfitw (\bfitw^{-1}\bftau \bfitw) e_k^{\hst} = k(\bfitw^{-1}\bftau \bfitw) \bfitw e_k^{\hst} = (\bfitw k)(\bftau) (\bfitw e_k^{\hst}).
\] Thus, if we write \[
    W_{\bfitw k} := \mathrm{span}_K H \bfitw e_k^{\hst},
\]
then $W_{\bfitw k}$ is the $H$-representation of highest weight $\bfitw k$. Moreover, there is a natural inclusion $W_{\bfitw k} \hookrightarrow V_k$ of $H$-representations. 

\begin{Lemma}\label{Lemma: multiplicity-one lemma for alg rep}
    For any $\bfitw \in W^H$, we have \[
        \dim_K \Hom_{H}(W_{\bfitw k}, V_k) = 1.
    \]
\end{Lemma}
\begin{proof}
    It suffices to show that $\bfitw e_{k}^{\hst}$ is the unique (up to scalar multiplication) non-zero vector $v \in V_k$ such that for any $\bftau \in T_H \cong T_{\GSp_4}$, \[
        \bftau v = \bfitw k(\bftau) v.
    \]
Suppose $v\in V_k \smallsetminus \{0\}$ is such a vector, then $\bfitw^{-1} v$ has the property that for any $\bftau \in T_H \cong T_{\GSp_4}$ \[
        \bftau (\bfitw^{-1}v) = \bfitw^{-1} \bfitw \bftau \bfitw^{-1} v = (\bfitw k)(\bfitw \bftau \bfitw^{-1}) \bfitw^{-1}v = k(\bftau) \bfitw^{-1}v.
    \]
    By the properties of the highest weight vector, we see that there exists $a\in K^\times$ such that  \[
        \bfitw^{-1} v = a e_k^{\hst}
    \]
    and hence \[
        v = a \bfitw e_k^{\hst}
    \] 
    as desired. 
\end{proof}

Immediately from Lemma \ref{Lemma: multiplicity-one lemma for alg rep}, we have the following corollary. 

\begin{Corollary}\label{Corollary: desired projection of H-representations}
    For every $\bfitw\in W^H$, $W_{\bfitw k}$ is a direct summand of $V_k$ as an $H$-subrepresentation. Moreover, there is a unique (up to scalar multiplication) nontrivial morphism of $H$-representations $V_k \rightarrow  W_{\bfitw k}$; namely, the projection onto the direct summand.
\end{Corollary}

\subsection{The flag variety}\label{subsection: flag variety}

Define the Siegel parabolic subgroup $P_{\Si}$ by \[
    P_{\Si} \coloneq \begin{pmatrix} \GL_2 & M_2 \\  & \GL_2\end{pmatrix} \cap \GSp_{4}.\footnote{ We point out that in \cite{DRW}, we considered the \emph{opposite Siegel parabolic} and worked with the opposite Bruhat cells therein.  }
\] The algebraic group $P_{\Si}$ has the following alternative description over $\C$: Consider the cocharacter \[
    \mu_{\Si}: \bbG_m \rightarrow \GSp_{4}, \quad a \mapsto \diag(a\one_2, \one_2).
\] Then, we have \[
    P_{\Si}(\C) = \left\{ \bfgamma \in \GSp_4(\C): \lim_{a\rightarrow 0} \mu_{\Si}(a) \bfgamma \mu_{\Si}(a)^{-1} \text{exists}\right\}.
\]

The flag variety (over $\Z$) that we will be using for the whole paper is \[
    \Fl \coloneq P_{\Si} \backslash \GSp_4.
\]
It is a classical result that $\Fl$ admits the so-called Bruhat decomposition \[
    \Fl = \bigsqcup_{\bfitw\in W^H} P_{\Si} \backslash P_{\Si} \bfitw B_{\GSp_4}.
\]
For each $\bfitw\in W^H$, we denote by $\Fl_{\bfitw}$ the Bruhat cell $P_{\Si} \backslash P_{\Si} \bfitw B_{\GSp_4}$. In what follows, we will also consider the following loci \[
    \Fl_{\leq \bfitw} \coloneq \bigsqcup_{\substack{\bfitw'\in W^H\\ l(\bfitw')\leq l(\bfitw)}} \Fl_{\bfitw'} \quad \text{ and }\quad \Fl_{\geq \bfitw} \coloneq \bigsqcup_{\substack{\bfitw'\in W^H\\ l(\bfitw')\geq l(\bfitw)}} \Fl_{\bfitw'}.
\]

\begin{Lemma}\label{Lemma: coordinates on Bruhat cells}
    For any $\bfitw\in W^H$, we have an isomorphism of schemes \[
        \prod_{\alpha \in \Phi_{\GSp_4}^+ \cap (\bfitw^{-1}\Phi^{-, H})} N_{\alpha} \rightarrow \Fl_{\bfitw}, \quad (\bfepsilon_{\alpha})_{\alpha} \mapsto \bfitw \prod_{\alpha \in \Phi_{\GSp_4}^+ \cap (\bfitw^{-1}\Phi^{-, H})} \bfepsilon_{\alpha}.
    \] In particular, we have the following coordinate systems \[
        \begin{array}{lll}
            \Fl_{\one_4} = \{\one_4\}, & & \Fl_{\bfitw_1} = \left\{ \bfitw_1 \begin{pmatrix} 1 &&& \\ & 1 & z_{21}^+& \\ && 1 & \\ &&& 1\end{pmatrix}\right\}, \\  \\
            \Fl_{\bfitw_{2}} = \left\{ \bfitw_2 \begin{pmatrix} 1 & a^+&  & z_{12}^+\\ & 1 &&  \\ && 1 & -a^+ \\ &&& 1\end{pmatrix}\right\}, & & \Fl_{\bfitw_3} = \left\{ \bfitw_3 \begin{pmatrix} 1 && z_{22}^+ & z_{12}^+\\ & 1 & z_{21}^+ & z_{22}^+ \\ && 1 & \\ &&& 1\end{pmatrix} \right\}.
        \end{array}
    \]
\end{Lemma}
\begin{proof}
    The first assertion is a special case of \cite[Lemma 3.1.3]{BP-HigherColeman}. In what follows, we carry out the computations for the coordinate system for each $\Fl_{\bfitw}$. 

    By definition, we have \[
        \Phi^{-, H} = \{ -x_1-x_2+x_0, -2x_2+x_0, -2x_2+x_0 \}.
    \]

    \paragraph{\textbf{The case $\bfitw = \one_4$.}} In this case, we see that $\Phi^{-, H}\cap \Phi_{\GSp_4}^+ = \emptyset$. Thus, the desired result follows. \\

    \paragraph{\textbf{The case $\bfitw = \bfitw_1$.}} In this case, we have \[
        \bfitw_1^{-1} : \begin{array}{l}
            x_1 \mapsto x_1 \\
            x_2 \mapsto x_0 - x_2\\
            x_0 \mapsto x_0
        \end{array} \quad  \text{ and hence } \quad \bfitw_1^{-1}: \begin{array}{rl}
            -x_1-x_2+x_0 & \mapsto -x_1-x_2 \\
            -2x_1+x_0 & \mapsto -2x_1+x_0\\
            -2x_2 + x_0 & \mapsto 2x_2-x_0
        \end{array}.
    \] Consequently, $\Phi_{\GSp_4}^+ \cap (\bfitw_1^{-1}\Phi^{-, H}) = \{2x_2-x_0\}$. The desired coordinate system follows from \[
        N_{2x_2-x_0} = \left\{\begin{pmatrix}1 &&& \\ & 1 & z_{21}^+ & \\ &&1 & \\ &&& 1\end{pmatrix}\right\}.
    \] 

    \paragraph{\textbf{The case $\bfitw = \bfitw_2$.}} In this case, we have \[
        \bfitw_2^{-1} : \begin{array}{l}
            x_1 \mapsto x_2\\
            x_2\mapsto x_0-x_1\\
            x_0 \mapsto x_0
        \end{array} \quad \text{ and hence } \quad \bfitw_{2}^{-1}: \begin{array}{rl}
            -x_1-x_2+x_0 & \mapsto x_1-x_2 \\
            -2x_1+x_0 & \mapsto -2x_2+x_0\\
            -2x_2 + x_0 & \mapsto 2x_1-x_0
        \end{array}.
    \]
    We obtain $\Phi_{\GSp_4}^+ \cap (\bfitw_2^{-1}\Phi^{-, H}) = \{x_1-x_2, 2x_1-x_0\}$. Recall that \[
        N_{x_1-x_2} = \left\{\begin{pmatrix} 1 & a^+ &&\\ & 1 && \\ && 1 & -a^+\\ &&&1\end{pmatrix}\right\}\quad \text{ and }\quad N_{2x_1-x_0} = \left\{\begin{pmatrix} 1 &&& z_{12}^+ \\ &1 &&\\ &&1&\\ &&&1 \end{pmatrix}\right\}
    \] and so the result follows. \\

    \paragraph{\textbf{The case $\bfitw = \bfitw_3$.}} In this case, we have \[
        \bfitw_3^{-1} : \begin{array}{l}
            x_1 \mapsto x_0-x_2\\
            x_2\mapsto x_0-x_1\\
            x_0 \mapsto x_0
        \end{array} \quad \text{ and hence } \quad \bfitw_{2}^{-1}: \begin{array}{rl}
            -x_1-x_2+x_0 & \mapsto x_1+x_2-x_0 \\
            -2x_1+x_0 & \mapsto 2x_2-x_0\\
            -2x_2 + x_0 & \mapsto 2x_1-x_0
        \end{array}.
    \]
    We see that $\bfitw_3^{-1}\Phi^{-, H} \subset \Phi_{\GSp_4}^+$ and the desired result follows from the explicit formulae for $N_{\alpha}$'s. 
\end{proof}

\begin{Remark}\label{Remark: open affine subschemes containing Fl_w}
    For later use, we shall also consider \[
        \Fl_{\bfitw}^{\natural} \coloneq \Fl_{\bfitw_3} \bfitw_3^{-1}\bfitw
    \]
    for any $\bfitw\in W^H$. This is an affine open subscheme in $\Fl$ that contains $\Fl_{\bfitw}$. An easy computation using Lemma \ref{Lemma: coordinates on Bruhat cells} yields that \[
        \Fl_{\bfitw}^{\natural} = \left\{ \begin{pmatrix} 1 &&& \\ & 1&& \\  z_{22}^+ & -z_{12}^+ & 1 & \\ -z_{21}^+ & z_{22}^+ & & 1\end{pmatrix} \bfitw \right\}.
    \]
This leads to alternative coordinate systems \[
        \begin{array}{lll}
            \Fl_{\one_4} = \{\one_4\}, & & \Fl_{\bfitw_1} = \left\{ \begin{pmatrix} 1 &&& \\ & 1&& \\   & -z_{12}^+ & 1 & \\  &  & & 1\end{pmatrix} \bfitw_1 \right\}, \\ \\
            \Fl_{\bfitw_2} = \left\{ \begin{pmatrix} 1 &&& \\ & 1&& \\  z_{22}^+ & -z_{12}^+ & 1 & \\  & z_{22}^+ & & 1\end{pmatrix} \bfitw_2 \right\}, & & \Fl_{\bfitw_3} = \left\{ \begin{pmatrix} 1 &&& \\ & 1&& \\  z_{22}^+ & -z_{12}^+ & 1 & \\ -z_{21}^+ & z_{22}^+ & & 1\end{pmatrix} \bfitw_3 \right\}.
        \end{array}
    \]
\end{Remark}

\vspace{2mm}

We now move on to the world of $p$-adic geometry. Let $\adicFL$ be the rigid analyticfication of $\Fl$ over $\Q_p$, viewed as an adic space. Recall the specialisation map (\cite{Berthelot}) \[
    \mathrm{sp} : \adicFL \rightarrow \Fl_{\F_p}. 
\]
This is a continuous map of topological spaces, locally defined by \[
    \Spa(R, R^+) \rightarrow \Spa(R^+, R^+) \rightarrow \Spec R^+/pR^+, \quad |\cdot(x)| \mapsto \frakp_{x} = \{ a\in R^+: |a(x)|<1\}.
\] 
For any $\bfitw\in W^H$, we define subsets $\adicFL_{\bfitw}$, $ \adicFL_{\leq \bfitw}$, and $\adicFL_{\geq \bfitw}$ of $\adicFL$ as the \emph{tubes} of $\Fl_{\F_p, \bfitw}$, $\Fl_{\F_p, \leq \bfitw}$, and $\Fl_{\F_p, \geq \bfitw}$, respectively; namely, we put \footnote{Notice that the difference between the tube $]\Fl_{\F_p, \bfitw}[ $ and $\mathrm{sp}^{-1}(\Fl_{\F_p, \bfitw})$ consists of only higher rank points. }
\begin{equation}\label{eq: w-loci on the adic flag variety}
    \begin{array}{lll}
        \adicFL_{\bfitw}  = \,\,]\Fl_{\F_p, \bfitw}[  \,\,\,\, \coloneq \text{ the interior of }\mathrm{sp}^{-1}(\Fl_{\F_p, \bfitw}),\\
        \adicFL_{\leq \bfitw}  = \,\,]\Fl_{\F_p, \leq \bfitw}[  \,\,\,\, \coloneq \text{ the interior of }\mathrm{sp}^{-1}(\Fl_{\F_p, \leq \bfitw}),\\
        \adicFL_{\geq \bfitw}  =\,\, ]\Fl_{\F_p, \geq \bfitw}[  \,\,\,\, \coloneq \text{ the interior of }\mathrm{sp}^{-1}(\Fl_{\F_p, \geq \bfitw}).
    \end{array}
\end{equation}

Again, we would like to exhibit an explicit coordinate system on each $\adicFL_{\bfitw}$. To this end, for each $\alpha\in \Phi_{\GSp_4}$, let $N_{\alpha, \F_p}$ be the special fibre of $N_{\alpha}$. Let $\calN_{\alpha}$ be the rigid analytic space (viewed as an adic space) associated with the formal completion of $N_{\alpha}$ along $N_{\alpha, \F_p}$, and let $\calN_{\alpha}^{\circ}$ be the interior of $\mathrm{sp}^{-1}(\one_{4})$ in $\calN_{\alpha}$. One sees that $\calN_{\alpha}$ is isomorphic to the closed unit ball over $\Spa(\Q_p, \Z_p)$ while $\calN_{\alpha}^{\circ}$ is isomorphic to the open unit ball over $\Spa(\Q_p, \Z_p)$.

\begin{Lemma}\label{Lemma: coordinates on tublar nbds of Bruhat cells}
    For any $\bfitw\in W^H$, we have an isomorphism of rigid analytic spaces \[
        \prod_{\alpha \in \Phi_{\GSp_4}^+ \cap (\bfitw^{-1}\Phi^{-, H})} \calN_{\alpha} \times \prod_{\alpha \in \Phi_{\GSp_4}^- \cap (\bfitw^{-1}\Phi^{-, H})} \calN_{\alpha}^{\circ} \rightarrow \adicFL_{\bfitw}, \quad (\bfepsilon_{\alpha})_{\alpha} \mapsto \bfitw \prod_{\alpha\in \bfitw^{-1}\Phi^{-, H}} \bfepsilon_{\alpha}.
    \]
    In particular, we have the following coordinate systems \[
     \scalemath{0.9}{    \begin{array}{lll}
            \adicFL_{\one_4} = \left\{ \begin{pmatrix} 1 &&& \\ & 1 && \\ z_{22}^- & z_{12}^-& 1 & \\ z_{21}^- & z_{22}^- & & 1\end{pmatrix} : |z_{ij}^-| <1\right\}, & & \adicFL_{\bfitw_1} = \left\{ \bfitw_1 \begin{pmatrix} 1 &&&  \\  a^- & 1 & z_{21}^+& \\ & z_{12}^- & 1 & \\  &  & -a^- & 1\end{pmatrix} : \begin{array}{l}
                |\bullet^-| <1  \\
                |\bullet^+| \leq 1
            \end{array} \right\}, \\ \\
            \adicFL_{\bfitw_2} = \left\{ \bfitw_2 \begin{pmatrix} 1 & a^+ &  & z_{12}^+\\ & 1 &&  \\  & & 1 & -a^+ \\ z_{21}^- &  & & 1\end{pmatrix} : \begin{array}{l}
                |\bullet^-| <1  \\
                |\bullet^+| \leq 1 
            \end{array} \right\}, & & \adicFL_{\bfitw_3} = \left\{ \bfitw_3 \begin{pmatrix} 1 && z_{22}^+ & z_{12}^+\\ & 1 & z_{21}^+ & z_{22}^+ \\  & & 1 & \\ &  & & 1\end{pmatrix} :  |z_{ij}^+| \leq 1\right\}
        \end{array}}
    \]
\end{Lemma}
\begin{proof}
    This follows from \cite[Corollary 3.3.5]{BP-HigherColeman} and Lemma \ref{Lemma: coordinates on Bruhat cells}.
\end{proof}

\begin{Remark}\label{Remark: alternative coordinate systems for analytic flag variety}
    For any $\bfitw\in W^H$, recall $\Fl_{\bfitw}^{\natural}$ from Remark \ref{Remark: open affine subschemes containing Fl_w}. Let $\Fl_{\bfitw}^{\natural, \an}$ be the rigid analytification of $\Fl_{\bfitw}^{\natural}$ over $\Spa(\Q_p, \Z_p)$. Then, we may consider $\adicFL_{\bfitw}$ as a subspace of $\Fl_{\bfitw}^{\natural, \an}$. As a consequence, we have the following alternative coordinate systems \[
        \begin{array}{lll}
            \adicFL_{\one_4} = \left\{ \begin{pmatrix} 1 &&& \\ & 1 && \\ z_{22}^+ & -z_{12}^+& 1 & \\ -z_{21}^+ & z_{22}^+ & & 1\end{pmatrix} : |z_{ij}^+| <1\right\}, \\ \\
            \adicFL_{\bfitw_1} = \left\{  \begin{pmatrix} 1 &&&  \\ & 1 && \\ z_{22}^+ & -z_{12}^+ & 1 & \\ -z_{21}^+ & z_{22}^+ & & 1\end{pmatrix}\bfitw_1 : \begin{array}{l}
                |z_{ij}^+| <1 \text{ for }(i,j)\neq (1,2) \\
                |z_{12}^+| \leq 1
            \end{array} \right\}, \\ \\
            \adicFL_{\bfitw_2} = \left\{  \begin{pmatrix} 1 &&  & \\ & 1 &&  \\  z_{22}^+ & -z_{12}^+ & 1 & \\ -z_{21}^+ & z_{22}^+ & & 1\end{pmatrix}\bfitw_2 : \begin{array}{l}
                |z_{21}^+| <1  \\
                |z_{ij}^+| \leq 1 \text{ for }(i,j)\neq (2,1) 
            \end{array} \right\}, \\
            \\
            \adicFL_{\bfitw_3} = \left\{  \begin{pmatrix} 1 &&  & \\ & 1 &  &  \\ z_{22}^+ & -z_{12}^+ & 1 & \\ -z_{21}^+ &  z_{22}^+ &  & 1\end{pmatrix}\bfitw_3 :  |z_{ij}^+| \leq 1\right\}.
        \end{array}
    \]
\end{Remark}

\begin{Remark}\label{Remark: maps between loci on the flag variety}
    For any $\bfitw\in W^H$, consider the automorphism \[
        \iota^{\bfitw}_{\bfitw_3}: \adicFL \rightarrow \adicFL
    \] given by multiplying $\bfitw^{-1}\bfitw_3$ on the right. It follows from Remark \ref{Remark: alternative coordinate systems for analytic flag variety} that $\iota_{\bfitw}^{\bfitw_3}$ restricts to \[
        \iota^{\bfitw}_{\bfitw_3}: \adicFL_{\bfitw} \hookrightarrow\adicFL_{\bfitw_3}.
    \]  
\end{Remark}

For any $\alpha\in \Phi_{\GSp_{4}}$, recall that $ \calN_{\alpha}$ (resp., $ \calN_{\alpha}^{\circ}$) can be naturally identified with a closed unit ball (resp., open unit ball) over $\Spa(\Q_p, \Z_p)$ with coordinate $\bfepsilon_{\alpha}$. Notice that the coordinate $\bfepsilon_{\alpha}$ is well-defined up to a unit. For every $m\in \Q_{\geq 0}$, we further consider the closed and open balls \[
    \calN_{\alpha, m}  \coloneq \{|\bfepsilon_{\alpha}| \leq |p^m|\} \quad \text{ and }\quad 
    \calN_{\alpha, m}^{\circ}  \coloneq \bigcup_{m'>m} \calN_{\alpha, m'}.
\]
Inspired by \cite{BP-HigherColeman}, for any $m, n\in \Q_{\geq 0}$, we consider the following open subsets of $\adicFL_{\bfitw}$. \begin{align*}
    \adicFL_{\bfitw, (m, n)} & \coloneq \image\left( \prod_{\alpha \in \Phi_{\GSp_4}^+ \cap (\bfitw^{-1}\Phi^{-, H})} \calN_{\alpha, m} \times \prod_{\alpha \in \Phi_{\GSp_4}^- \cap (\bfitw^{-1}\Phi^{-, H})} \calN_{\alpha, n}^{\circ} \rightarrow \adicFL_{\bfitw} \right) \\
    \adicFL_{\bfitw, (\overline{m}, n)} & \coloneq \image\left( \prod_{\alpha \in \Phi_{\GSp_4}^+ \cap (\bfitw^{-1}\Phi^{-, H})} \overline{\calN_{\alpha, m}} \times \prod_{\alpha \in \Phi_{\GSp_4}^- \cap (\bfitw^{-1}\Phi^{-, H})} \calN_{\alpha, n}^{\circ} \rightarrow \adicFL_{\bfitw} \right) \\
    \adicFL_{\bfitw, (m, \overline{n})} & \coloneq \image\left( \prod_{\alpha \in \Phi_{\GSp_4}^+ \cap (\bfitw^{-1}\Phi^{-, H})} \calN_{\alpha, m} \times \prod_{\alpha \in \Phi_{\GSp_4}^- \cap (\bfitw^{-1}\Phi^{-, H})} \overline{\calN_{\alpha, n}^{\circ}} \rightarrow \adicFL_{\bfitw} \right) \\
    \adicFL_{\bfitw, (\overline{m}, \overline{n})} & \coloneq \image\left( \prod_{\alpha \in \Phi_{\GSp_4}^+ \cap (\bfitw^{-1}\Phi^{-, H})} \overline{\calN_{\alpha, m}} \times \prod_{\alpha \in \Phi_{\GSp_4}^- \cap (\bfitw^{-1}\Phi^{-, H})} \overline{\calN_{\alpha, n}^{\circ}} \rightarrow \adicFL_{\bfitw} \right) .
\end{align*}
Here, the closures are taken with respect to the analytic topology. In general, these subsets are not necessarily adic spaces (see \cite[Example 3.3.7]{BP-HigherColeman}) but merely topological spaces.

\subsection{Vector bundles and torsors}\label{subsection: vector bundles and torsors over the flag variety}
As a moduli problem, $\Fl$ parametrises maximal Lagrangian subspaces of $V$ with respect to the pairing \eqref{eq: symplectic pairing on standard rep}. As a consequence, there is a universal short exact sequence \begin{equation}\label{eq: universal short exact sequence over the flag variety}
    0 \rightarrow \scrW^{\vee}_{\Fl} \rightarrow \scrO_{\Fl}^4 \rightarrow \scrW_{\Fl} \rightarrow 0,
\end{equation}
where both $\scrW^{\vee}_{\Fl}$ and $\scrW_{\Fl}$ are vector bundles of rank $2$ over $\Fl$. Here, since $\scrO_{\Fl}^4$ is self-dual with respect to the pairing induced by \eqref{eq: symplectic pairing on standard rep}, the kernel of the universal map $\scrO_{\Fl}^4 \rightarrow \scrW_{\Fl}$ can be identified with the dual of $\scrW_{\Fl}$. 

The total space of $\scrW_{\Fl}$ can be identified as $P_{\Si}\backslash (\bbA^2 \times \GSp_4$), where 
 \begin{itemize}
    \item $P_{\Si}$ acts on $\GSp_4$ via left-multiplication; 
    \item by viewing elements in $\bbA^2$ as row vectors, $P_{\Si}$ acts on $\bbA^2$ via \[
        \begin{pmatrix}\bfgamma_a & \bfgamma_b\\ &  \bfgamma_d\end{pmatrix} * \vec{v}   =  \vec{v} \trans\bfgamma_d^{-1}.
    \]
\end{itemize}
Under this identification, for any $\bfgamma\in P_{\Si}$ and $(\vec{v}, \bfalpha)\in \bbA^2 \times \GSp_{4}$, we have $(\vec{v}, \bfgamma \bfalpha) = (\bfgamma^{-1}*\vec{v} , \bfalpha)$ in $P_{\Si}\backslash(\bbA^2 \times \GSp_4)$. Consequently, global sections of $\scrW_{\Fl}$ are identified as \[
    \{\text{algebraic functions }\phi: \GSp_4 \rightarrow \bbA^2: \phi(\bfgamma\bfalpha) = \bfgamma^{-1}*\phi(\bfalpha),\,\,\,\,\forall (\bfgamma, \bfalpha)\in P_{\Si}\times \GSp_4 \}.
\]

For $i=1,2$, consider the algebraic functions \[
    \bfits_i: \GSp_4 \rightarrow \bbA^2, \quad \begin{pmatrix}\bfalpha_a & \bfalpha_b\\ \bfalpha_c & \bfalpha_d\end{pmatrix} \mapsto \begin{pmatrix}
        \bfalpha_{d,1, 3-i} &  \bfalpha_{d,2, 3-i}
    \end{pmatrix}.
\] 
Then, one sees that, for any $\bfgamma = \begin{pmatrix}\bfgamma_a & \bfgamma_b\\ & \bfgamma_d\end{pmatrix} \in P_{\Si}$, \[
    \bfits_i(\bfgamma \bfalpha) = \begin{pmatrix}(\bfgamma_d\bfalpha_d)_{1, 3-i} &  (\bfgamma_d\bfalpha_d)_{2, 3-i}\end{pmatrix} =  \begin{pmatrix}
        \bfalpha_{d,1, 3-i} &  \bfalpha_{d,2, 3-i}
    \end{pmatrix} \trans\bfgamma_d = \bfgamma^{-1} *\bfits_i(\bfalpha) .
\]
In other words, $\bfits_1$ and $\bfits_2$ are global sections of $\scrW_{\Fl}$. 

For any $\bfitw\in W^H$, we define the global section $\bfits_i^{\bfitw}$ by \[
    \bfits_i^{\bfitw}(\bfalpha) = \bfits_i(\bfalpha \bfitw^{-1}).
\] Then, we claim that $\bfits_1^{\bfitw}$ and $\bfits_2^{\bfitw}$ are non-vanishing on $\Fl_{\bfitw}^{\natural}$.  Indeed, it suffices to observe that \[
    \begin{pmatrix} \bfits_2^{\bfitw} \\ \bfits_1^{\bfitw}\end{pmatrix} \left( \begin{pmatrix} 1 &&& \\ & 1&& \\  z_{22}^+ & -z_{12}^+ & 1 & \\ -z_{21}^+ & z_{22}^+ & & 1\end{pmatrix}\bfitw \right) = \begin{pmatrix} \bfits_2 \\ \bfits_1\end{pmatrix} \left(  \begin{pmatrix} 1 &&& \\ & 1&& \\  z_{22}^+ & -z_{12}^+ & 1 & \\ -z_{21}^+ & z_{22}^+ & & 1\end{pmatrix}\right) = \one_2,
\]
using the coordinate systems in Remark \ref{Remark: open affine subschemes containing Fl_w}.

For $\bfitw\in W^H$, we consider global sections $\bfits_1^{\bfitw, \vee}$ and $\bfits_2^{\bfitw, \vee}$ of $\scrW^{\vee}_{\Fl}$ defined by \[
    \bla \bfits_i^{\bfitw, \vee}, \bfits_j^{\bfitw}\bra = \left\{ \begin{array}{ll}
        -1, & i=j \\
        0, & \text{else}
    \end{array}\right.
\]
where $\bla\cdot, \cdot\bra$ is the pairing induced by \eqref{eq: symplectic pairing on standard rep}. 

Moreover, consider an $H$-torsor $H_{\HT}$ over $\Fl$ defined by \[
    H_{\HT} \coloneq \underline{\Isom}^{\mathrm{symp}}( \scrO_{\Fl}^{4}, \scrW^{\vee}_{\Fl} \oplus \scrW_{\Fl} ).
\] Here, `$\Isom^{\mathrm{symp}}$' stands for \emph{isomorphisms that respect both the symplectic pairing and the direct sum decomposition up to units}. In other words, $H_{\HT}$ parametrises splittings of $\scrO_{\Fl}^4\rightarrow\scrW_{\Fl}$ that respect the symplectic pairing induced by \eqref{eq: symplectic pairing on standard rep} up to units.

\begin{Lemma}\label{Lemma: two descriptions of the H-torsor over the flag variety}
    The $H$-torsor $H_{\HT}$ can be identified as \[
        H_{\HT} \cong N_{\Si} \backslash \GSp_4,
    \]
    where $N_{\Si}$ is the unipotent radical given by the Levi decomposition $P_{\Si} = H \ltimes N_{\Si}$.
\end{Lemma}
\begin{proof}
    Note that $N_{\Si} \backslash \GSp_4$ parametrises the following data \begin{itemize}
        \item short exact sequence $0 \rightarrow W^{\vee} \rightarrow V \rightarrow W \rightarrow 0$ that respects the pairing  \eqref{eq: symplectic pairing on standard rep} up to units; 
        \item a basis $\{w_1^{\vee}, w_2^{\vee}\}$ for $W^{\vee}$ and a (dual) basis $\{w_2, w_1\}$ for $W$. 
    \end{itemize}
    One sees that, for a given pair of basis $(\{w_1^{\vee}, w_2^{\vee}\}, \{w_2, w_1\})$, it defines a symplectic isomorphism $\scrO_{\Fl}^4 \rightarrow \scrW^{\vee} \oplus \scrW$. Hence, one obtains a morphism \[
        N_{\Si} \backslash \GSp_{4} \rightarrow H_{\HT}. 
    \]
    One easily checks that this morphism is $H$-equivariant and so the identification follows. 
\end{proof}

Let's now move to the world of $p$-adic geometry. Let $\calH$ (resp., $\calH^{\an}$) be the rigid analytic space associated with the formal completion (resp., rigid analytification, view as an adic space) of $H$. For any $n\in \Z_{>0}$, define $\adicIw^+_{H, n}$ to be the affinoid subgroup of $\calH$ consisting of elements that reduce to $T_{H, \Z/p^n\Z}$ modulo $p^n$. Define \[
    \Iw_{H, n}^+ = \left\{ \bfgamma \in H(\Z_p): (\bfgamma \textrm{ mod } p^n) \in T_{H}(\Z/p^n\Z)\right\}.
\]
Note that $\Iw_{H, 1}^+$ is a subgroup of the (usual) Iwahori subgroup $\Iw_H$ of $H$ at $p$, which is defined as the subgroup of matrices in $H(\Z_p)$ that are congruent to upper triangular matrices modulo $p$. Hence, $\Iw_H^+$ admits a Iwahori decomposition \[
    \Iw_{H, n}^+ = N_{H, n}^{\opp} T_H(\Z_p) N_{H, n},
\] 
where \[
    N_{H, n} = \left\{ \begin{pmatrix} 1 & b && \\ & 1 && \\ && 1 & b' \\ &&& 1\end{pmatrix} \in H(\Z_p): p^n|b, b' \right\}
\] and $N_{H, n}^{\opp}$ is defined similarly but using the upper triangular matrices in place of the lower triangular ones.

Similarly, for any $n\in \Z_{>0}$, we define \[
    \Iw_{\GSp_4, n}^+ = \left\{ \bfgamma\in \GSp_{2g}(\Z_p): (\bfgamma \textrm{ mod } p^n) \in T_{\GSp_{2g}}(\Z/p^n\Z)\right\}.
\] 
This is also a subgroup of the (usual) Iwahori subgroup of $\GSp_4$ at $p$. Hence, it also admits a Iwahori decomposition \[
    \Iw_{\GSp_4, n}^+ = N_{\GSp_4, n}^{\opp} T_{\GSp_4}(\Z_p) N_{\GSp_4, n},
\]
where \[
    N_{\GSp_4, n} \coloneq \left\{ \bfgamma\in N_{\GSp_4}(\Z_p) : \bfgamma \equiv \one_4 \textrm{ mod } p^n \right\}
\] and similar for $N_{\GSp_4, n}^{\opp}$. 

Denote by $\scrW_{\adicFL}^{\vee}$ and $\scrW_{\adicFL}$ the rigid analytifications of $\scrW_{\Fl}^{\vee}$ and $\scrW_{\Fl}$ over $\adicFL$. Then, the global sections $\bfits_i^{\bfitw, \vee}$, $\bfits_i^{\bfitw}$ define global sections on $\scrW_{\adicFL}^{\vee}$ and $\scrW_{\adicFL}$ respectively; we shall abuse the notations and still denote them by $\bfits_i^{\bfitw, \vee}$ and $\bfits_i^{\bfitw}$. 

In what follows, we use the coordinate system in Remark \ref{Remark: alternative coordinate systems for analytic flag variety} for $\adicFL_{\bfitw}$ and abbreviate it as $\begin{pmatrix}\one_2 & \\ \bfitz & \one_2\end{pmatrix} \bfitw$. In particular, $\bfitz$ is viewed as a $2\times 2$ matrix whose entries are functions on $\adicFL_{\bfitw}$.

\begin{Lemma}\label{Lemma: global sections and automorphy factors}
    Let $\bfitw\in W^H$ and $n\in \Z_{>0}$. \begin{enumerate}
        \item[(i)] The locus $\adicFL_{\bfitw}$ is stable under the action of $\Iw_{\GSp_{4}, n}^+$. 
        \item[(ii)] For $\bfalpha = \begin{pmatrix}\bfalpha_a & \bfalpha_b\\ \bfalpha_c & \bfalpha_d\end{pmatrix}\in \Iw_{\GSp_{4}, n}^+$, we write \[
            \bfitw \bfalpha \bfitw^{-1} = \begin{pmatrix} \bfalpha_a^{\bfitw} & \bfalpha_b^{\bfitw}\\ \bfalpha_c^{\bfitw} & \bfalpha_d^{\bfitw} \end{pmatrix}. 
        \] Then we have \[
            \bfalpha^*\bfits_i^{\bfitw} =  \bfits_i^{\bfitw} \trans (\bfalpha_d^{\bfitw} + \bfitz\bfalpha_b^{\bfitw}).
        \]
        \item[(iii)] Keep the notation in (ii). We have \[
            \bfalpha^*\bfits_i^{\bfitw, \vee} = \bfits_i^{\bfitw, \vee} \trans\left( \varsigma(\bfalpha)\oneanti_2\trans (\bfalpha_d^{\bfitw} + \bfitz\bfalpha_b^{\bfitw})^{-1} \oneanti_2 \right).
        \]
    \end{enumerate} 
\end{Lemma}
\begin{proof}
    The first assertion is a special case of \cite[Corollary 3.3.14]{BP-HigherColeman} and the third assertion follows from the second one. It remains to show (ii). 
    
    We start with remarking that $\bfitw \bfalpha \bfitw^{-1}\in \Iw_{\GSp_4, n}^+$ as $\bfitw$ normalises $T_{\GSp_4}$. In particular, entries of $\bfalpha_b^{\bfitw}$ and $\bfalpha_c^{\bfitw}$ are divisible by $p$. Thus, the matrix $\bfalpha_d^{\bfitw} + \bfitz\bfalpha_b^{\bfitw}$ in the statement is invertible. By definition, \begin{align*}
        \bfalpha^*\begin{pmatrix} \bfits_2^{\bfitw} \\ \bfits_1^{\bfitw}\end{pmatrix}\left( \begin{pmatrix}\one_2 & \\ \bfitz & \one_2\end{pmatrix} \bfitw \right) & = \begin{pmatrix} \bfits_2^{\bfitw} \\ \bfits_1^{\bfitw}\end{pmatrix}\left( \begin{pmatrix}\one_2 & \\ \bfitz & \one_2\end{pmatrix} \bfitw \bfalpha \right) \\
        & = \begin{pmatrix} \bfits_2 \\ \bfits_1\end{pmatrix}\left( \begin{pmatrix}\one_2 & \\ \bfitz & \one_2\end{pmatrix} \bfitw \bfalpha \bfitw^{-1} \right).
    \end{align*}
The desired identity then follows from \[
        \begin{pmatrix}\one_2 & \\ \bfitz & \one_2\end{pmatrix} \bfitw \bfalpha \bfitw^{-1} = \begin{pmatrix} \varsigma(\bfalpha)\oneanti_2 \trans(\bfalpha_d^{\bfitw}+\bfitz\bfalpha_b^{\bfitw})^{-1}\oneanti_2 & \bfalpha_b^{\bfitw}\\ & \bfalpha_d^{\bfitw}+\bfitz\bfalpha_b^{\bfitw}\end{pmatrix} \begin{pmatrix} \one_2 & \\ (\bfalpha_d^{\bfitw}+\bfitz\bfalpha_b^{\bfitw})^{-1}(\bfalpha_c^{\bfitw}+\bfitz\bfalpha_a^{\bfitw}) & \one_2\end{pmatrix}.
    \]
\end{proof}

\begin{Remark}\label{Remark: s_i^w are pullbacks from the w_3-locus}
    Recall the injection \[
        \iota^{\bfitw}_{\bfitw_3} \colon \adicFL_{\bfitw} \hookrightarrow \adicFL_{\bfitw_3}
    \] from Remark \ref{Remark: maps between loci on the flag variety}. From the construction, one sees that $\bfits_i^{\bfitw}$ is nothing but the pullback of $\bfits_i^{\bfitw_3}$ via $\iota^{\bfitw}_{\bfitw_3}$.
\end{Remark}

Let $\frakH_{\HT}$ be the formal completion of $H_{\HT}$ along its special fibre over $\F_p$ and put \begin{align*}
    \calH_{\HT} & \coloneq \text{ the rigid analytic space over $\Spa(\Q_p, \Z_p)$ associated with $\frakH_{\HT}$}\\
    \calH_{\HT}^{\an} & \coloneq \text{ the rigid analytification of $H_{\HT}$ over $\Spa(\Q_p, \Z_p)$}.
\end{align*}
One sees that $\calH_{\HT}$ (resp., $\calH_{\HT}^{\an}$) is an $\calH$-torsor (resp., $\calH^{\an}$) over $\adicFL$. 

For later use, we construct an $\adicIw_{H, n}^+$-torsor $\adicIW_{H, n, \adicFL_{\bfitw}}^+$ over $\adicFL_{\bfitw}$ (for each $\bfitw\in W^H$ and every $n\in \Z_{>0}$) together with a commutative diagram \[
    \begin{tikzcd}[column sep = tiny]
        \adicIW_{H, n, \adicFL_{\bfitw}}^+ \arrow[rr]\arrow[rd, "\pr_{\adicFL_{\bfitw}, \Iw_{H, n}^+}"'] & & \calH_{\HT}|_{\adicFL_{\bfitw}}\arrow[ld]\\ & \adicFL_{\bfitw}
    \end{tikzcd}
\]
that is $\adicIw_{H, n}^+$-equivariant. This torsor will be used in the construction of the overconvergent automorphic sheaves in \S \ref{subsection: overconvergent Siegel modular forms: torsors}.

Let $\scrW_{\adicFL_{\bfitw}}$ (resp., $\scrW_{\adicFL_{\bfitw}}^{\vee}$) be the restriction of $\scrW_{\adicFL}$ (resp., $\scrW_{\adicFL}^{\vee}$) on $\adicFL_{\bfitw}$. For any $n\in \Z_{>0}$, we say that a basis $\{\bfita_1^{\vee}, \bfita_2^{\vee}, \bfita_2, \bfita_1\}$ for $\scrW_{\adicFL_{\bfitw}}^{\vee}\oplus \scrW_{\adicFL}$ is \emph{\textit{$n$-compatible}} (with respect to $\{\bfits_1^{\bfitw, \vee}, \bfits_2^{\bfitw, \vee}, \bfits_2^{\bfitw}, \bfits_1^{\bfitw}\}$) if \[
    \mathrm{span}\{\bfita_i^{\vee}\} \equiv \mathrm{span}\{\bfits_i^{\bfitw, \vee}\} \,(\text{mod } p^n)\quad \text{ and } \quad\mathrm{span}\{\bfita_i\} \equiv \mathrm{span}\{\bfits_i^{\bfitw}\}\,(\text{mod } p^n)
\]
for $i=1,2$. We define $\adicIW_{H, n, \adicFL_{\bfitw}}^+$ on $\adicFL_{\bfitw}$ as a moduli problem \[
    \adicIW_{H, n, \adicFL_{\bfitw}}^+(R, R^+) = \left\{ \psi \colon (R^{+})^{4} \xrightarrow{\cong} \scrW_{\adicFL}^{\vee}(R, R^+)\oplus \scrW_{\adicFL}(R, R^+): \{\psi(v_1), ..., \psi(v_4)\} \text{ is $n$-compatible} \right\},
\]
where $\{v_1, ..., v_4\}$ is the standard basis for $(R^{+})^{4}$. Note that there is a canonical element $\psi_{\bfitw}^{\std} \in \adicIW_{H, n, \adicFL_{\bfitw}}^+(R, R^+)$ given by \begin{equation}\label{eq: psi_w^std}
    \psi_{\bfitw}^{\std} \colon \begin{array}{l}
        v_1 \mapsto \bfits_1^{\bfitw, \vee}\\
        v_2 \mapsto \bfits_2^{\bfitw, \vee}\\
        v_3 \mapsto \bfits_2^{\bfitw}\\
        v_4 \mapsto \bfits_1^{\bfitw}
    \end{array}.
\end{equation}
Following a similar argument as in \cite[\S 4.5]{AIP-2015}, one shows that $\adicIW_{H, n, \adicFL_{\bfitw}}^+$ is representable. Moreover, immediately from the moduli description, we have a natural forgetful map \[
    \adicIW_{H, n, \adicFL_{\bfitw}}^+ \rightarrow \calH_{\HT}|_{\adicFL_{\bfitw}}.
\]

\subsection{The \texorpdfstring{$p$}{p}-adic weight space and analytic representations}\label{subsection: weight space and analytic representations}

In this section, we introduce the $p$-adic weight space as well as certain analytic representations, later of which play a central role in the construction of \emph{pseudoautomorphic sheaves} in \S \ref{subsection: pseudoautomorphic sheaves}.  The $p$-adic analysis in this section is well-known to experts. We refer the readers to \cite[\S 3.1]{LW-BianchiAdjoint} for more details.

Let $\Alg_{(\Z_p, \Z_p)}$ be the category of sheafy $(\Z_p, \Z_p)$-algebras. It is well-known that the functor \[
    \Alg_{(\Z_p, \Z_p)} \rightarrow \Sets, \quad (R, R^+) \mapsto \Hom_{\Groups}^{\cts}(T_{\GL_2}(\Z_p), R^\times)
\]
is represented by the Iwasawa algebra $(\Z_p\llbrack T_{\GL_2}(\Z_p)\rrbrack, \Z_p\llbrack T_{\GL_2}(\Z_p)\rrbrack)$. The \emph{$p$-adic weight space} is defined to be \[
    \calW \coloneq \Spa(\Z_p\llbrack T_{\GL_2}(\Z_p)\rrbrack, \Z_p\llbrack T_{\GL_2}(\Z_p)\rrbrack)^{\rig},
\]
where the superscript `$\bullet^{\rig}$' stands for the associated rigid analytic space over $\Spa(\Q_p, \Z_p)$, viewed as an adic space. In other words, $\calW$ is the $\{p\neq 0\}$-part of the adic space $\Spa(\Z_p\llbrack T_{\GL_2}(\Z_p)\rrbrack, \Z_p\llbrack T_{\GL_2}(\Z_p)\rrbrack)$. One sees immediately that $\calW$ is a finite disjoint union of two-dimensional open unit balls. 

\begin{Remark}\label{Remark: `integral' property of weights}
    Given $\kappa \in \Hom_{\Groups}^{\cts}(T_{\GL_2}(\Z_p), R^\times)$, we claim that the image of $\kappa$ lies in $R^{\circ, \times}$. Note that \[
        \kappa(\diag(a_1, a_2)) = \kappa_1(a_1)\kappa_2(a_2),
    \] where each $\kappa_i: \Z_p^\times \rightarrow R^\times$ is a continuous group homomorphism.  Hence, it is enough to show that $\kappa_i(1+ p\Z_p)\subset R^\circ$; namely, to show that if $1+pa\in 1+p\Z_p$, then $\{\kappa_i(1+pa)^n\}_{n\in \Z_{\geq 0}}$ is bounded. This is exactly \cite[Lemma 3.2.1]{LW-BianchiAdjoint}. 
\end{Remark}

\begin{Remark}\label{Remark: viewing a p-adic weight as a character for TGSp}
    In what follows, we always view $\kappa\in \Hom_{\Groups}^{\cts}(T_{\GL_2}(\Z_p), R^\times)$ as a continuous group homomorphism $T_{\GSp_4}(\Z_p) \rightarrow R^\times$ via \[
        \kappa \colon T_{\GSp_4}(\Z_p) \rightarrow R^\times, \quad \diag(\tau_1, \tau_2, \tau_0\tau_2^{-1}, \tau_0\tau_1^{-1}) \mapsto \prod_{i=1}^2\kappa_i(\tau_i).
    \]
    For classical weights, this is the same as the embedding $\Z^2 \xhookrightarrow{(k_1, k_2)\mapsto (k_1, k_2; 0)} \Z^3 \cong \bbX$, where the last isomorphism is \eqref{eq: identification of the character group}. 
\end{Remark}

For the purpose of $p$-adic interpolation, we consider two types of \emph{$p$-adic families of weights} following the convention in \cite{CHJ-2017}.

\begin{Definition}\label{Definition: weights}
    \begin{enumerate}
        \item[(i)] A $\Z_p$-algebra $R$ is \emph{small} if it is $p$-torsion free, reduced, and is finite over $\Z_p\llbrack T_1, ..., T_d\rrbrack$ for some $d\in \Z_{\geq 0}$. In particular, $R$ is equipped with a canonical adic profinite topology and is complete with respect to the $p$-adic topology. 
        \item[(ii)] A \emph{small weight} is a pair $(R_{\calU}, \kappa_{\calU})$, where $R_{\calU}$ is a small $\Z_p$-algebra and $\kappa_{\calU}: T_{\GL_2}(\Z_p) \rightarrow R_{\calU}^\times$ is a continuous group homomorphism such that $\kappa_{\calU}(\diag(1+p, 1+p))-1$ is a topological nilpotent in $R_{\calU}$ with respect to the $p$-adic topology. 
        \item[(iii)] An \emph{affinoid weight} is a pair $(R_{\calU}, \kappa_{\calU})$, where $R_{\calU}$ is a reduced affinoid algebra, topologically of finite type over $\Q_p$, and $\kappa_{\calU} \colon T_{\GL_2}(\Z_p) \rightarrow R_{\calU}^\times$ is a continuous group homomorphism. 
        \item[(iv)] By a \emph{weight}, we mean either a small weight or an affinoid weight. 
    \end{enumerate}
\end{Definition}

\begin{Remark}
When $R$ is a reduced affinoid algebra, we use $R^{\circ}$ to denote the subring of power bounded elements in $R$ as usual. When $R$ is a small $\Z_p$-algebra, we abuse the notation and write $R^{\circ}=R$. This convention simplifies our exposition in the rest of the section.
\end{Remark}

\begin{Remark}\label{Remark: maps of weights to the weight space}
    Given a small weight (resp., an affinoid weight) $(R_{\calU}, \kappa_{\calU})$, there is natural morphism \[
        \calU = \Spa(R_{\calU}, R_{\calU})^{\rig} \rightarrow \calW \quad (\text{resp., }\calU = \Spa(R_{\calU}, R_{\calU}^{\circ}) \rightarrow \calW)
    \]
    by the universal property of the weight space. Occasionally, by abuse of notation, we call $\calU$ a weight. We will call $(R_{\calU}, \kappa_{\calU})$ (or $\calU$) an \emph{open weight} if this natural morphism is an open embedding. 
\end{Remark}

\begin{Remark}\label{Remark: defining rU}
    \normalfont Given a weight $(R_{\calU}, \kappa_{\calU})$, $R_{\calU}[1/p]$ admits a structure of a uniform $\Q_p$-Banach algebra by letting $R_{\calU}^{\circ}$ be its unit ball and equipping it with the corresponding spectral norm, denoted by $|\cdot|_{\calU}$. Then, we define \[
        r_{\calU} \coloneq \min\left\{ r\in \Z_{\geq 0} : |\kappa_{\calU}(\diag(1+p, 1+p))|_{\calU} < p^{-\frac{1}{p^{r}(p-1)}}\right\}.
    \]
    See \cite[pp. 202]{CHJ-2017}.
\end{Remark}

For any $r\in \Q_{>0}$ and $n\in \Z_{\geq 0}$, denote by $C^r(\Z_p^n, \Z_p)$ the space of $r$-analytic functions from $\Z_p^n$ to $\Z_p$ and define \[
    C^{r^+}(\Z_p^n, \Z_p) \coloneq \varprojlim_{r'>r}C^{r'}(\Z_p^n, \Z_p). 
\] For any $i = (i_1, ..., i_n)\in \Z_{\geq 0}^n$, write \begin{equation}\label{eq: basis for r-analytic functions}
    e_i^{(r)} \colon \Z_p^n \rightarrow \Z_p, \quad (x_1, ..., x_n)\mapsto \prod_{j=1}^n\lfloor p^{-r}i_j \rfloor! \begin{pmatrix} x_j\\ i_j\end{pmatrix}.
\end{equation}
The structure theorems (\cite[Theorem 3.1.2 \& Lemma 3.1.5]{LW-BianchiAdjoint}) for $C^r(\Z_p^n, \Z_p)$ and $C^{r^+}(\Z_p^n, \Z_p)$ yields isomorphisms \begin{equation}\label{eq: structure theorem of r-analytic functions}
    C^r(\Z_p^n, \Z_p) \cong \widehat{\bigoplus}_{i\in \Z_{\geq 0}^n}\Z_pe_i^{(r)} \quad \text{ and }\quad C^{r^+}(\Z_p^n, \Z_p)\cong \prod_{i\in \Z_{\geq 0}^n}\Z_pe_i^{(r)}.
\end{equation}

Let $R$ be either a small $\Z_p$-algebra or a reduced affinoid algebra over $\Q_p$, we consider \[
    \begin{array}{cc}
        A^{r, \circ}(\Z_p^n, R) \coloneq C^r(\Z_p^n, \Z_p)\widehat{\otimes}_{\Z_p}R^{\circ}, & A^{r}(\Z_p^n, R) \coloneq A^{r, \circ}(\Z_p^n, R)\Big[\frac{1}{p}\Big],\\
        A^{r^+, \circ}(\Z_p^n, R) \coloneq C^{r^+}(\Z_p^n, \Z_p)\widehat{\otimes}_{\Z_p} R^{\circ}, & A^{r^+}(\Z_p^n, R) \coloneq A^{r^+, \circ}(\Z_p^n, R)\Big[\frac{1}{p}\Big].
    \end{array}
\]
One may view $A^{r, \circ}(\Z_p^n, R)$ (resp., $A^r(\Z_p^n, R)$) as an $R^{\circ}$-submodule (resp., $R^{\circ}[1/p]$-submodule) of the space of continuous functions from $\Z_p^n$ to $R^{\circ}$ (resp., $R^{\circ}[1/p]$).

Recall the Iwahori decomposition $\Iw_{H, 1}^{+} = N_{H, 1}^{\opp}T_H(\Z_p)N_{H, 1}$ and observe that \begin{equation}\label{eq: unipotent radical for H as a p-adic manifold}
    N_{H, 1}^{\opp} \cong \Z_p
\end{equation} as a $p$-adic manifold. We shall from now on fix such an isomorphism. This allows us to make sense of the modules $A^{r, \circ}(N_{H, 1}^{\opp}, R)$, $A^r(N_{N, 1}^{\opp}, R)$, $A^{r^+, \circ}(N_{H, 1}^{\opp}, R)$, and $A^{r^+}(N_{H, 1}^{\opp}, R)$. 

Now, given a weight $(R_{\calU}, \kappa_{\calU})$ and $r\in \Q_{\geq 0}$ with $ r>1+r_{\calU}$. We define \begin{equation}\label{eq: some analytic representations}
    \begin{array}{rl}
        A^{r, \circ}_{\kappa_{\calU}}(\Iw_{H, 1}^+, R_{\calU}) & \coloneq \left\{f \colon \Iw_{H, 1}^+ \rightarrow R_{\calU}^{\circ} : \begin{array}{l}
            f(\bfgamma \bfbeta) = \kappa_{\calU}(\bfbeta)f(\bfgamma), \,\,\,\, \forall \bfgamma\in \Iw_{H, 1}^+, \bfbeta\in T_{H}(\Z_p)N_{H, 1}  \\
            f|_{N_{H, 1}^{\opp}} \in A^{r, \circ}(N_{H, 1}^{\opp}, R_{\calU}) 
        \end{array}\right\} \\
        A_{\kappa_{\calU}}^r(\Iw_{H, 1}^+, R_{\calU}) & \coloneq A^{r, \circ}_{\kappa_{\calU}}(\Iw_{H, 1}^+, R_{\calU})\Big[\frac{1}{p}\Big]\\
        A^{r^+, \circ}_{\kappa_{\calU}}(\Iw_{H, 1}^+, R_{\calU}) & \coloneq \left\{f \colon \Iw_{H, 1}^+ \rightarrow R_{\calU}^{\circ} : \begin{array}{l}
            f(\bfgamma \bfbeta) = \kappa_{\calU}(\bfbeta)f(\bfgamma),\,\,\,\, \forall \bfgamma\in \Iw_{H, 1}^+, \bfbeta\in T_{H}(\Z_p)N_{H, 1}  \\
            f|_{N_{H, 1}^{\opp}} \in A^{r^+, \circ}(N_{H, 1}^{\opp}, R_{\calU}) 
        \end{array}\right\} \\
        A_{\kappa_{\calU}}^{r^+}(\Iw_{H, 1}^+, R_{\calU}) & \coloneq A^{r^+, \circ}_{\kappa_{\calU}}(\Iw_{H, 1}^+, R_{\calU})\Big[\frac{1}{p}\Big].
    \end{array}
\end{equation}
Here, we extend $\kappa_{\calU}$ to $T_H(\Z_p)N_{H, 1}$ by putting $\kappa_{\calU}(N_{H, 1}) = \{1\}$.

The following corollary is immediate from the definition. 

\begin{Corollary}\label{Corollary: structure theorem of analytic representations}
Let $(R_{\calU}, \kappa_{\calU})$ be a weight. Then we have
        \[A_{\kappa_{\calU}}^{r,\circ}(\Iw_{H, 1}^+, R_{\calU}) \cong \widehat{\bigoplus}_{i\in \Z_{\geq 0}^4}R_{\calU}^{\circ}e_i^{(r)}\]
        and 
        \[A_{\kappa_{\calU}}^{r^+, \circ}(\Iw_{H, 1}^+, R_{\calU})\cong \prod_{i\in \Z_{\geq 0}^4}R_{\calU}^{\circ} e_{i}^{(r)}.\] 
         We obtain similar descriptions for $A_{\kappa_{\calU}}^{r}(\Iw_{H, 1}^+, R_{\calU})$ and $A_{\kappa_{\calU}}^{r^+}(\Iw_{H, 1}^+, R_{\calU})$ after inverting $p$.
\end{Corollary}

\begin{Remark}\label{Remark: analytic representation for other groups}
Consider 
\[
     \Iw_{P_{\Si}, 1}^+ \coloneq \left\{ \bfgamma \in P_{\Si}(\Z_p) : (\bfgamma \mod p)\in T_{\GSp_4}(\F_p)\right\}
\] 
which admits a Iwahori decomposition
\[
\Iw_{P_{\Si}}^+ = N_{H, 1}^{\opp} T_H(\Z_p) N_{\GSp_4, 1}.
\]
We may consider analytic representations $A_{\kappa_{\calU}}^{r, \circ}(?, R_{\calU})$, $A_{\kappa_{\calU}}^{r^+, \circ}(?, R_{\calU})$, $A_{\kappa_{\calU}}^{r}(?, R_{\calU})$ and $A_{\kappa_{\calU}}^{r^+}(?, R_{\calU})$ for $? \in \{\Iw_{\GSp_4, 1}^+, \Iw_{P_{\Si}, 1}^+\}$. More precisely, we define
\begin{equation*}
    \begin{array}{rl}
      \scalemath{0.9}{  A^{r, \circ}_{\kappa_{\calU}}(\Iw_{\GSp_4, 1}^+, R_{\calU})} & \scalemath{0.9}{ \coloneq \left\{f \colon \Iw_{\GSp_4, 1}^+ \rightarrow R_{\calU}^{\circ} : \begin{array}{l}
            f(\bfgamma \bfbeta) = \kappa_{\calU}(\bfbeta)f(\bfgamma), \,\,\,\, \forall \bfgamma\in \Iw_{\GSp_4, 1}^+, \bfbeta\in T_{\GSp_4}(\Z_p)N_{\GSp_4, 1}  \\
            f|_{N_{\GSp_4, 1}^{\opp}} \in A^{r, \circ}(N_{\GSp_4, 1}^{\opp}, R_{\calU}) 
        \end{array}\right\} }\\
        \scalemath{0.9}{ A_{\kappa_{\calU}}^r(\Iw_{\GSp_4, 1}^+, R_{\calU})} & \scalemath{0.9}{ \coloneq A^{r, \circ}_{\kappa_{\calU}}(\Iw_{\GSp_4, 1}^+, R_{\calU})\Big[\frac{1}{p}\Big]}\\
        \scalemath{0.9}{ A^{r^+, \circ}_{\kappa_{\calU}}(\Iw_{\GSp_4, 1}^+, R_{\calU})} & \scalemath{0.9}{ \coloneq \left\{f \colon \Iw_{\GSp_4, 1}^+ \rightarrow R_{\calU}^{\circ} : \begin{array}{l}
            f(\bfgamma \bfbeta) = \kappa_{\calU}(\bfbeta)f(\bfgamma),\,\,\,\, \forall \bfgamma\in \Iw_{\GSp_4, 1}^+, \bfbeta\in T_{\GSp_4}(\Z_p)N_{\GSp_4, 1}  \\
            f|_{N_{\GSp_4, 1}^{\opp}} \in A^{r^+, \circ}(N_{\GSp_4, 1}^{\opp}, R_{\calU}) 
        \end{array}\right\} }\\
        \scalemath{0.9}{ A_{\kappa_{\calU}}^{r^+}(\Iw_{\GSp_4, 1}^+, R_{\calU})} & \scalemath{0.9}{ \coloneq A^{r^+, \circ}_{\kappa_{\calU}}(\Iw_{\GSp_4, 1}^+, R_{\calU})\Big[\frac{1}{p}\Big]}
    \end{array}
\end{equation*}
and
\begin{equation*}
    \begin{array}{rl}
        \scalemath{0.9}{ A^{r, \circ}_{\kappa_{\calU}}(\Iw_{P_{\Si}, 1}^+, R_{\calU})} & \scalemath{0.9}{ \coloneq \left\{f \colon \Iw_{P_{\Si}, 1}^+ \rightarrow R_{\calU}^{\circ} : \begin{array}{l}
            f(\bfgamma \bfbeta) = \kappa_{\calU}(\bfbeta)f(\bfgamma), \,\,\,\, \forall \bfgamma\in \Iw_{P_{\Si}, 1}^+, \bfbeta\in T_{H}(\Z_p)N_{\GSp_4, 1}  \\
            f|_{N_{H, 1}^{\opp}} \in A^{r, \circ}(N_{H, 1}^{\opp}, R_{\calU}) 
        \end{array}\right\} }\\
        \scalemath{0.9}{ A_{\kappa_{\calU}}^r(\Iw_{P_{\Si}, 1}^+, R_{\calU})} & \scalemath{0.9}{ \coloneq A^{r, \circ}_{\kappa_{\calU}}(\Iw_{P_{\Si}, 1}^+, R_{\calU})\Big[\frac{1}{p}\Big]}\\
        \scalemath{0.9}{ A^{r^+, \circ}_{\kappa_{\calU}}(\Iw_{P_{\Si}, 1}^+, R_{\calU})} & \scalemath{0.9}{ \coloneq \left\{f \colon \Iw_{P_{\Si}, 1}^+ \rightarrow R_{\calU}^{\circ} : \begin{array}{l}
            f(\bfgamma \bfbeta) = \kappa_{\calU}(\bfbeta)f(\bfgamma),\,\,\,\, \forall \bfgamma\in \Iw_{P_{\Si}, 1}^+, \bfbeta\in T_{H}(\Z_p)N_{\GSp_4, 1}  \\
            f|_{N_{H, 1}^{\opp}} \in A^{r^+, \circ}(N_{H, 1}^{\opp}, R_{\calU}) 
        \end{array}\right\}} \\
        \scalemath{0.9}{ A_{\kappa_{\calU}}^{r^+}(\Iw_{P_{\Si}, 1}^+, R_{\calU})} & \scalemath{0.9}{ \coloneq A^{r^+, \circ}_{\kappa_{\calU}}(\Iw_{P_{\Si}, 1}^+, R_{\calU})\Big[\frac{1}{p}\Big].}
    \end{array}
\end{equation*}
\end{Remark}

\begin{Lemma}\label{Lemma: analytic representation for H is the same as analytic representation for P_Si}
    Let $(R_{\calU}, \kappa_{\calU})$ be a weight and $r\in \Q_{\geq 0}$ such that $r > 1+r_{\calU}$. 
    The natural inclusion $\Iw_{H, 1}^+ \hookrightarrow \Iw_{P_{\Si}, 1}^+$ induces a canonical isomorphism of $R_{\calU}$-modules \[
        A_{\kappa_{\calU}}^r(\Iw_{P_{\Si}, 1}^+, R_{\calU}) \cong A_{\kappa_{\calU}}^r(\Iw_{H, 1}^+, R_{\calU}). 
    \]
    Similar statements hold for $A_{\kappa_{\calU}}^{r, \circ}$, $A_{\kappa_{\calU}}^{r^+, \circ}$, and $A_{\kappa_{\calU}}^{r^+}$. 
\end{Lemma}
\begin{proof}
    Recall the Iwahori decomposition \[
        \Iw_{H, 1}^+ = N_{H, 1}^{\opp} T_H(\Z_p) N_{H, 1} \quad \text{ and }\quad \Iw_{P_{\Si}}^+ = N_{H, 1}^{\opp} T_H(\Z_p) N_{\GSp_4, 1}.
    \]
    Unwinding the definition, one sees that \[
        A_{\kappa_{\calU}}^r(\Iw_{P_{\Si}, 1}^+, R_{\calU}) \cong A^r(N_{H, 1}^{\opp}, R_{\calU}) \cong A_{\kappa_{\calU}}^r(\Iw_{H, 1}^+, R_{\calU}).
    \]
    The other cases are similar. 
\end{proof}

We equip $A_{\kappa_{\calU}}^r(\Iw_{P_{\Si}, 1}^+, R_{\calU})$ with a left $\Iw_{P_{\Si}, 1}^+$-action given by \[
    \Iw_{P_{\Si}, 1}^+ \times A_{\kappa_{\calU}}^r(\Iw_{P_{\Si}, 1}^+, R_{\calU}) \rightarrow A_{\kappa_{\calU}}^r(\Iw_{P_{\Si}, 1}^+, R_{\calU}), \quad (\bfgamma, f) \mapsto \left( \bfalpha \mapsto f(\bfitw_3^{-1} \trans \bfgamma \bfitw_3 \bfalpha) \right).\footnote{Here, note that given $\bfgamma \in \Iw_{P_{\Si}, 1}^+$, $\bfitw_3^{-1} \trans \bfgamma \bfitw_3 \in \Iw_{P_{\Si}, 1}^+$.}
\]
This induces a natural group homomorphism \[
    \rho_{\kappa_{\calU}}^r \colon \Iw_{P_{\Si}, 1}^+ \rightarrow \Aut(A_{\kappa_{\calU}}^r(\Iw_{P_{\Si}, 1}^+, R_{\calU})).
\]
Thanks to Lemma \ref{Lemma: analytic representation for H is the same as analytic representation for P_Si}, we can then view $A_{\kappa_{\calU}}^r(\Iw_{H, 1}^+, R_{\calU})$ as an $\Iw_{P_{\Si}, 1}^+$-representation. By abuse of notation we still write \[
    \rho_{\kappa_{\calU}}^r \colon \Iw_{P_{\Si}, 1}^+ \rightarrow \Aut(A_{\kappa_{\calU}}^r(\Iw_{H, 1}^+, R_{\calU})).
\]
Similar constructions apply to $A_{\kappa_{\calU}}^{r, \circ}$, $A_{\kappa_{\calU}}^{r^+, \circ}$, and $A_{\kappa_{\calU}}^{r^+}$, yielding representations $\rho_{\kappa_{\calU}}^{r, \circ}$, $\rho_{\kappa_{\calU}}^{r^+, \circ}$, and $\rho_{\kappa_{\calU}}^{r^+}$, respectively.

\begin{Remark}\label{Remark: extension of the analytic representation}
    Given a weight $(R_{\calU}, \kappa_{\calU})$ and $r\in \Q_{\geq 0}$ with $r>1+r_{\calU}$, consider \begin{align*}
        \Iw_{H, 1}^{+, (r)} &\coloneq \left\{ \bfgamma =(\bfgamma_{ij})_{i,j} \in H(\calO_{\C_p}) : |\bfgamma_{ij}-\bfgamma_{ij}'|\leq p^{-r} \text{ for some }\bfgamma' = (\bfgamma_{ij}')_{i,j}\in \Iw_{H, 1}^+ \right\}\\
        \Iw_{P_{\Si}, 1}^{+, (r)} & \coloneq \left\{ \bfgamma =(\bfgamma_{ij})_{i,j} \in P_{\Si}(\calO_{\C_p}) : |\bfgamma_{ij}-\bfgamma_{ij}'|\leq p^{-r} \text{ for some }\bfgamma' = (\bfgamma_{ij}')_{i,j}\in \Iw_{P_{\Si}, 1}^+ \right\}.
    \end{align*}
There are Iwahori decompositions \[
        \Iw_{H, 1}^{+, (r)} = N_{H, 1}^{\opp, (r)} T_{H}^{(r)} N_{H, 1}^{(r)} \quad \text{ and }\quad \Iw_{P_{\Si}, 1}^{+, (r)} = N_{H, 1}^{\opp, (r)} T_{H}^{(r)} N_{\GSp_4, 1}^{(r)},
    \]
    where $N_{H, 1}^{\opp, (r)}$, $T_{H}^{(r)}$, $N_{H,1}^{(r)}$, and $N_{\GSp_{4}, 1}^{(r)}$ are defined similarly. 
    For any $f\in A_{\kappa_{\calU}}^r(\Iw_{H, 1}^+, R_{\calU})$, since $f|_{N_{H, 1}^{\opp}}$ is $r$-analytic, it naturally extends to a function on $\Iw_{H, 1}^{+, (r)}$ by \[
        f(\bfepsilon \bfbeta) = \kappa_{\calU}(\bfbeta)f(\bfepsilon)
    \]
    for any $\bfepsilon \in N_{H, 1}^{\opp,(r)}$ and $\bfbeta\in T_H^{(r)}N_{H, 1}^{(r)}$. Here, we have applied \cite[Proposition 2.6]{CHJ-2017} to extend $\kappa_{\calU}$ to a character on $T_H^{(r)} N_{H, 1}^{(r)}$. Consequently, $\rho_{\kappa_{\calU}}^r$ extends to a representation \[
        \rho_{\kappa_{\calU}}^r \colon \Iw_{P_{\Si}, 1}^{+, (r)} \rightarrow \Aut(A_{\kappa_{\calU}}^r(\Iw_{H, 1}^+, R_{\calU})).
    \]
    Similar constructions apply to $\rho_{\kappa_{\calU}}^{r, \circ}$, $\rho_{\kappa_{\calU}}^{r^+, \circ}$, and $\rho_{\kappa_{\calU}}^{r^+}$.
\end{Remark}

\begin{Example}\label{Example: highest weight vector}
    \normalfont Let $(R_{\calU}, \kappa_{\calU})$ be a weight and $r\in \Q_{\geq 0}$ with $r>1+r_{\calU}$. We introduce the \emph{highest weight vector} $e_{\kappa_{\calU}}^{\hst}$ in $A_{\kappa_{\calU}}^{r, \circ}(\Iw_{H, 1}^+, R_{\calU})$ (and hence in $A_{\kappa_{\calU}}^r(\Iw_{H, 1}^+, R_{\calU})$, $A_{\kappa_{\calU}}^{r^+, \circ}(\Iw_{H, 1}^+, R_{\calU})$, and $A_{\kappa_{\calU}}^{r^+}(\Iw_{H, 1}^+, R_{\calU})$). Recall that $\kappa_{\calU} = (\kappa_{\calU, 1}, \kappa_{\calU, 2})$ where $\kappa_{\calU, i} \colon \Z_p^\times \rightarrow R_{\calU}^\times$ is a continuous group homomorphism such that $\kappa_{\calU, i}(1+p\Z_p)\subset R_{\calU}^{\circ}$. Given $\bfalpha = (\bfalpha_{ij})_{1\leq i,j\leq 4}\in \Iw_{\GSp_4, 1}^+$, define \[
        e_{\kappa_{\calU}}^{\hst}(\bfalpha) = \frac{\kappa_{\calU,1}(\bfalpha_{11})}{\kappa_{\calU, 2}(\bfalpha_{11})} \cdot \kappa_{\calU, 2}\left( \det(\bfalpha_{ij})_{1\leq i,j\leq 2}\right).
    \] For $\bfgamma\in \Iw_{\GSp_{4}, 1}^+$, the functions \[
        f_{\kappa_{\calU}}^{\bfgamma} \colon \bfalpha \mapsto e_{\kappa_{\calU}}^{\hst}(\bfgamma\bfalpha)
    \] are elements in $A_{\kappa_{\calU}}^{r, \circ}$. 
\end{Example}

\begin{Lemma}\label{Lemma: tensor product of analytic representations}
    Let $(R_{\calU}, \kappa_{\calU})$ and $(R_{\calV}, \kappa_{\calV})$ be two weights. Suppose they are either both small weights or both affinoid weights. Let $r\in \Q_{\geq 0}$ with $r > 1+ \max\{r_{\calU}, r_{\calV}\}$. Then, there is a natural morphism of $\Iw_{P_{\Si}, 1}^+$-representations \[
        A_{\kappa_{\calU}}^{r}(\Iw_{H, 1}^+, R_{\calU}) \widehat{\otimes}A_{\kappa_{\calV}}^{r}(\Iw_{H, 1}^+, R_{\calU}) \rightarrow A_{\kappa_{\calU} + \kappa_{\calV}}^r(\Iw_{H, 1}^+, R_{\calU})
    \]
    sending $e_{\kappa_{\calU}}^{\hst} \otimes e_{\kappa_{\calV}}^{\hst}$ to $e_{\kappa_{\calU}+\kappa_{\calV}}^{\hst}$. 
\end{Lemma}
\begin{proof}
    Consider the morphism \[
        A_{\kappa_{\calU}}^{r}(\Iw_{H, 1}^+, R_{\calU}) \widehat{\otimes}A_{\kappa_{\calV}}^{r}(\Iw_{H, 1}^+, R_{\calU}) \rightarrow A_{\kappa_{\calU} + \kappa_{\calV}}^r(\Iw_{H, 1}^+, R_{\calU}), \quad f \otimes f' \mapsto (\bfgamma \mapsto f(\bfgamma)f'(\bfgamma)).
    \]
  It is straightforward to verify the $\Iw_{P_{\Si}, 1}^+$-equivariance. The statement on the highest weight vectors follows from the explicit formulation in Example \ref{Example: highest weight vector}.
\end{proof}

\subsection{Pseudoautomorphic sheaves}\label{subsection: pseudoautomorphic sheaves}
Fix $\bfitw\in W^H$ and let $(R_{\calU}, \kappa_{\calU})$ be a weight. Let $r\in \Q_{\geq 0}$ such that $r>1+r_{\calU}$. Define sheaves $\scrA_{\kappa_{\calU}, \adicFL_{\bfitw}}^r$ and $\scrA_{\kappa_{\calU}, \adicFL_{\bfitw}}^{r, \circ}$ on $\adicFL_{\bfitw, (r, r)}$ by \[
    \scrA_{\kappa_{\calU}, \adicFL_{\bfitw}}^r\coloneq A_{\bfitw_3^{-1}\bfitw\kappa_{\calU}}^r(\Iw_{H, 1}^+, R_{\calU}) \widehat{\otimes} \scrO_{\adicFL_{\bfitw, (r, r)}}\]
    and
    \[
    \scrA_{\kappa_{\calU}, \adicFL_{\bfitw}}^{r, \circ}\coloneq A_{\bfitw_3^{-1}\bfitw\kappa_{\calU}}^{r, \circ}(\Iw_{H, 1}^+, R_{\calU}) \widehat{\otimes} \scrO_{\adicFL_{\bfitw, (r, r)}}^+.
\]

\begin{Remark}
One might wonder why there is a twist by $\bfitw_3$. We refer the readers to Remark \ref{Remark: explanation on the weights} below for a brief explanation. 
\end{Remark}

\begin{Proposition}\label{Prop: comparison theorem for pseudoautomorphic sheaves}
    Given $\bfitw$, $(R_{\calU}, \kappa_{\calU})$, and $r$ as above. Let $\calB_{\GSp_4}$ denote the rigid analytic space associated with the formal completion of $B_{\GSp_4}$. Then there is a natural isomorphism of sheaves over $\adicFL_{\bfitw, (r, r)}$ \[
        \scrA_{\kappa_{\calU}, \adicFL_{\bfitw}}^{r} \cong \left(\pr_{\adicFL_{\bfitw}, \Iw_{H, 1}^+}\right)_{*} \scrO_{\adicIW_{H, 1, \adicFL_{\bfitw}}^+}\widehat{\otimes}R_{\calU}[\bfitw\kappa_{\calU}].
    \]
  where the right-hand side stands for the subsheaf of $\left(\pr_{\adicFL_{\bfitw}, \Iw_{H, 1}^+}\right)_{*} \scrO_{\adicIW_{H, 1, \adicFL_{\bfitw}}^+}\widehat{\otimes}R_{\calU}$ consisting of sections $f(\bfgamma)$ such that \[
        f(\bfgamma\bfbeta) = \bfitw\kappa_{\calU}(\bfbeta)f(\bfgamma)
    \] for all $\bfbeta\in \adicIw_{H, 1}^+ \cap \calB_{\GSp_4}$. 
    \end{Proposition}
    
\begin{proof}
    Given an affinoid $\calV = \Spa(R, R^+)\subset \adicFL_{\bfitw, (r, r)}$, there is an identification \[
        \adicIw_{H, 1}^+(R)  \xrightarrow{\cong} \adicIW_{H, 1, \adicFL_{\bfitw}}^+(\calV), \quad \bfgamma \mapsto \psi_{\bfitw}^{\std}\bfgamma,
    \] where $\psi_{\bfitw}^{\std}$ is as defined in \eqref{eq: psi_w^std}. By definition, we have \begin{align*}
    \scalemath{0.9}{ \left(\left(\pr_{\adicFL_{\bfitw}, \Iw_{H, 1}^+}\right)_{*} \scrO_{\adicIW_{H, 1, \adicFL_{\bfitw}}^+}\widehat{\otimes}R_{\calU}[\bfitw\kappa_{\calU}]\right)(\calV)} & \scalemath{0.9}{ = \left\{ \phi \colon \adicIW_{H, 1, \adicFL_{\bfitw}}^+(\calV) \rightarrow R \widehat{\otimes}R_{\calU}: \begin{array}{l}
            \phi(\bfgamma \bfbeta) = \bfitw\kappa_{\calU}(\bfbeta)f(\bfgamma)\\
            \text{for all }\bfbeta\in \adicIw_{H, 1}^+\cap \calB_{\GSp_4}
        \end{array}\right\},}\\
        \scalemath{0.9}{ \scrA_{\kappa_{\calU}, \adicFL_{\bfitw}}^r(\calV)} & \scalemath{0.9}{ = \left\{ \phi \colon \Iw_{H, 1}^+ \rightarrow R\widehat{\otimes}R_{\calU}: \begin{array}{l}
            \phi(\bfgamma\bfbeta) = \bfitw_3^{-1}\bfitw\kappa_{\calU}(\bfbeta)f(\bfgamma)\\
            \text{for all }(\bfbeta, \bfgamma) \in T_H(\Z_p)N_{H, n}\times \Iw_{H, n}^+\\
            f|_{N_{H, 1}^{\opp}} \text{ is $r$-analytic}
        \end{array}\right\}.}
    \end{align*}
    Hence one can define a natural map \begin{equation}\label{eq: natural map comparing two constructions of pseudoautomorphic sheaves}
        \left(\left(\pr_{\adicFL_{\bfitw}, \Iw_{H, 1}^+}\right)_{*} \scrO_{\adicIW_{H, 1, \adicFL_{\bfitw}}^+}\widehat{\otimes}R_{\calU}[\bfitw\kappa_{\calU}]\right)(\calV) \rightarrow \scrA_{\kappa_{\calU}, \adicFL_{\bfitw}}^r(\calV), \quad f \mapsto (\bfgamma \mapsto f(\bfitw_{3}\bfgamma \bfitw_{3}^{-1})).
    \end{equation} Here, note that $\bfitw_3 \Iw_{H, 1}^+ \bfitw_3^{-1} = \Iw_{H, 1}^+$. On the other hand, due to the $r$-analyticity condition on $\scrA_{\kappa_{\calU}, \adicFL_{\bfitw}}^r(\calV)$, every function $\phi$ in $\scrA_{\kappa_{\calU}, \adicFL_{\bfitw}}^r(\calV)$ extends to a function $\phi$ on $\adicIW_{H, 1, \adicFL_{\bfitw}}^+(\calV)$. This means \eqref{eq: natural map comparing two constructions of pseudoautomorphic sheaves} is an isomorphism. 
    
    Finally, one observes that \eqref{eq: natural map comparing two constructions of pseudoautomorphic sheaves} is functorial in $\calV = \Spa(R, R^+)$, meaning that given $\calV' = \Spa(R', {R'}^+)$ the restriction of \eqref{eq: natural map comparing two constructions of pseudoautomorphic sheaves} from $\calV$ to $\calV \cap \calV'$ is the same as the one of \eqref{eq: natural map comparing two constructions of pseudoautomorphic sheaves} from $\calV'$. Therefore, one obtains the desired isomorphism of sheaves by glueing. 
\end{proof}

\begin{Remark}\label{Remark: pseudoautomorphic sheaves}
    \begin{enumerate}
        \item[(i)] A similar statement of Proposition \ref{Prop: comparison theorem for pseudoautomorphic sheaves} holds for $\scrA_{\kappa_{\calU}, \adicFL_{\bfitw}}^{r, \circ}$ while we replace $\scrO_{\adicIW_{H, 1, \adicFL_{\bfitw}}^+}$ with $\scrO_{\adicIW_{H, 1, \adicFL_{\bfitw}}^+}^+$ and $R_{\calU}$ with $R_{\calU}^{\circ}$. 
        \item[(ii)] Applying a similar construction, we may define sheaves $\scrA_{\kappa_{\calU}, \adicFL_{\bfitw}}^{r^+, \circ}$ (resp.,  $\scrA_{\kappa_{\calU}, \adicFL_{\bfitw}}^{r^+}$) by replacing analytic representation $A_{\bfitw_3^{-1}\bfitw\kappa_{\calU}}^{r, \circ}(\Iw_{H, 1}^+, R_{\calU})$ (resp., $A_{\bfitw_3^{-1}\bfitw\kappa_{\calU}}^{r}(\Iw_{H, 1}^+, R_{\calU})$) with $A_{\bfitw_3^{-1}\bfitw\kappa_{\calU}}^{r^+, \circ}(\Iw_{H, 1}^+, R_{\calU})$ (resp., $A_{\bfitw_3^{-1}\bfitw\kappa_{\calU}}^{r^+}(\Iw_{H, 1}^+, R_{\calU})$). In what follows, we shall refer the sheaves $\scrA_{\kappa_{\calU}, \adicFL_{\bfitw}}^{r, \circ}$, $\scrA_{\kappa_{\calU}, \adicFL_{\bfitw}}^{r}$, $\scrA_{\kappa_{\calU}, \adicFL_{\bfitw}}^{r^+, \circ}$, and $\scrA_{\kappa_{\calU}, \adicFL_{\bfitw}}^{r^+}$ as \emph{pseudoautomorphic sheaves}. 
    \end{enumerate}
\end{Remark}
\section{Overconvergent automorphic sheaves for \texorpdfstring{$\GSp_4$}{GSp4}}\label{section: automorphic}
In this section, we study classical and overconvergent Siegel modular forms, viewed as sections of various automorphic sheaves. We start with the definition of Siegel threefolds in \S \ref{subsection: Siegel threefolds} and the definition of classical Siegel modular forms in \S \ref{subsection: classical Siegel modular forms}. Then we provide two constructions of overconvergent Siegel automorphic sheaves in \S \ref{subsection: overconvergent Siegel modular forms: perfectoid} and \S \ref{subsection: overconvergent Siegel modular forms: torsors},  via perfectoid method and analytic torsors, respectively. Finally, we construct the Hecke operators in \S \ref{subsection: Hecke operators}.

\subsection{Siegel threefolds}\label{subsection: Siegel threefolds}

Let $\A_{\Q}$ be the ring of ad\`eles of $\Q$. 
We denote by $\A_{\Q}^{\infty, p}$ the finite ad\`eles away from $p$. Choose a neat compact open subgroup $\Gamma = \prod_{\ell\neq p}\Gamma_{\ell} \subset \GSp_4(\A^{\infty, p}_{\Q})$ such that $\Gamma_{\ell} = \GSp_4(\Z_{\ell})$ for almost all $\ell$. We then define $N = \prod_{\Gamma_{\ell}\neq \GSp_4(\Z_{\ell})} \ell$.

For each $n\in \Z_{>0}$, recall the subgroup $\Iw_{\GSp_4, n}^+$, consisting of those matrices in $\GSp_4(\Z_p)$ that are congruent with diagonal matrices modulo $p^n$. To simplify the notation, we denote by \[
    \Gamma_n \coloneq \Gamma\Iw_{\GSp_4, n}^+,
\]
which is a compact open subgroup of $\GSp_4(\A_{\Q}^{\infty})$. We further denote by $\Gamma_0 = \Gamma\GSp_4(\Z_p)\subset \GSp_4(\A_{\Q}^{\infty})$. 

Consider \begin{align*}
    \bbH_2^{\pm} & = \text{ the Siegel upper-half/lower-half space}\\
    & = \left\{ \bfalpha \in M_2(\C): \begin{array}{l}
         \bfalpha \text{ is symmetric w.r.t the anti-diagonal}  \\
          \mathrm{Im}(\bfalpha) \text{ is positive/negative definite}
    \end{array}\right\}
\end{align*} and denote by $\bbH_2 = \bbH_2^+ \sqcup \bbH_2^-$. The group $\GSp_{4}(\R)$ acts on $\bbH_2^{\pm}$ via the formula \[
    \begin{pmatrix} \bfgamma_a & \bfgamma_b\\ \bfgamma_c & \bfgamma_d\end{pmatrix} \cdot \bfalpha = (\bfgamma_a \bfalpha + \bfgamma_b)^{-1}(\bfgamma_c\bfalpha + \bfgamma_d).
\]
Then for any $n\in \Z_{\geq 0}$, the \emph{complex} Siegel threefold of level $\Gamma_n$ is the locally symmetric space \[
    X_{n}(\C) = \GSp_{4}(\Q) \backslash \GSp_{4}(\A_{\Q}^{\infty}) \times \bbH_2 / \Gamma_n.
\]
To simplify the notation, we write $X = X_{0}$.

In what follows, besides $\Iw_{\GSp_4, n}^+$, we also encounter other level structures at $p$. For instance, we will consider \begin{align*}
    \Iw_{\GSp_4, n} &\coloneq \left\{  \bfgamma\in \GSp_{4}(\Z_p): (\bfgamma \mod p)\in B_{\GSp_4}(\Z/p^n)\right\},\\
    \Gamma(p^n) &\coloneq \left\{ \bfgamma \in \GSp_4(\Z_p) : \bfgamma \equiv \one \mod p^n\right\}.
\end{align*}
The Siegel threefolds of these levels at $p$  will be denoted by $X_{\Iw_{\GSp_4, n}}(\C)$ or $X_{\Gamma(p^n)}(\C)$.

It is well-known that $X_{n}(\C)$ admits a structure of an algebraic variety $X_{n}$ over $\Q$, which can be interpreted as a moduli space of tuples $(A, \lambda, \psi, \{C_{n, i}: i=1, \ldots, 4\})$, where \begin{itemize}
    \item $A$ is a principally polarised abelian surface and $\lambda$ is a principal polarisation of $A$; 
    \item $\psi$ is a $\Gamma$-level structure (cf. \cite{Lan-PhD}{Definition 1.4.1.4}) 
    \item $\{C_{n, i}: i=1, \ldots, 4\}$ is a collection of subgroups of order $p^n$ in $A$ such that \[
        C_{n, i} \cap C_{n, j} = 0\text{ if }i\neq j\quad \text{ and }\quad \langle C_{n, 1}, \ldots, C_{n, 4}\rangle = A[p^n].
    \]
\end{itemize}
Similarily, $X_{\Iw_{\GSp_4, n}}$ and $X_{\Gamma(p^n)}$ can be interpreted as moduli problems in a similar fashion. 

By choosing an auxiliary cone decomposition $\Sigma$, the variety $X$ admits a \emph{toroidal compactification} $X^{\tor}$ (depending on $\Sigma$) that admits the following properties (\cite[Chapter IV, Theorem 6.7]{Faltings-Chai}): \begin{itemize}
    \item There is an injective morphism of schemes $X \hookrightarrow X^{\tor}$ with Zariski dense image.
    \item The boundary $\partial X^{\tor} \coloneq X^{\tor} \smallsetminus X$ is a normal crossing divisor. Endowing $X^{\tor}$ with the log structure defined by $\partial X^{\tor}$, we may then view $X^{\tor}$ as an fs log scheme. 
    \item There is a tautological semiabelian variety $G^{\univ} \rightarrow X^{\tor}$, extending the universal abelian variety $A^{\univ} \rightarrow X$. We denote by $e$ the identity section.
\end{itemize}

It turns out that, by applying a theorem of Fujiwara--Kato (\cite[Theorem 7.6]{Illusie}), the varieties $X_n$, $X_{\Iw_{\GSp_4}, n}$, $X_{\Gamma(p^n)}$ admit toroidal compactifications $X^{\tor}_n$, $X^{\tor}_{\Iw_{\GSp_4, n}}$, $X^{\tor}_{\Gamma(p^n)}$ respectively that sit into a commutative diagram \[
    \begin{tikzcd}
        && X^{\tor}_{\Gamma(p^n)}\arrow[ld]\arrow[ddd]\\
        & X^{\tor}_n\arrow[ld]\arrow[ddd]\\
        X^{\tor}_{\Iw_{\GSp_4, n}}\arrow[ddd]\\
        && X^{\tor}_{\Gamma(p)}\arrow[ld]\\
        & X^{\tor}_1\arrow[ld]\\
        X^{\tor}_{\Iw_{\GSp_4, 1}}\arrow[d]\\
        X^{\tor}
    \end{tikzcd}.
\]
All morphisms in this diagram are finite Kummer \'etale.

We now move to the world of $p$-adic geometry. Let $\calX$ be the rigid analytification of $X$ over $\Spa(\Q_p, \Z_p)$. We adopt similar notations for the other aforementioned varieties (\emph{e.g.}, $\calX_n$, $\calX^{\tor}_n$, etc.). By a slight abuse of notations, we still use $\calX$, $\calX_n$, $\calX^{\tor}_n$, etc. to denote their base change to $\Spa(\C_p, \calO_{\C_p})$. By \cite[Corollaire 4.14]{Pilloni-Stroh}, building on work of Scholze, there is a perfectoid space $\calX^{\tor}_{\Gamma(p^{\infty})}$ such that \[
    \calX^{\tor}_{\Gamma(p^{\infty})} \sim \varprojlim_n \calX^{\tor}_{\Gamma(p^n)},
\]
where the relation `$\sim$' is as defined in \cite[Definition 2.4.1]{Scholze-Weinstein}. The perfectoid space $\calX^{\tor}_{\Gamma(p^{\infty})}$ is the perfectoid space associated with a pro-Kummer \'etale Galois cover of $\calX^{\tor}_n$ (resp., $\calX^{\tor}_{\Iw_{\GSp_4, n}}$; resp., $\calX^{\tor}_{\Gamma(p^n)}$) of Galois group $\Iw_{\GSp_4, n}^+$ (resp., $\Iw_{\GSp_4, n}$; resp., $\Gamma(p^n)$).

One of the important features of the perfectoid space $\calX^{\tor}_{\Gamma(p^{\infty})}$ is that it admits the \emph{Hodge--Tate period map} (\cite{Pilloni-Stroh}) \[
    \pi_{\HT} : \calX^{\tor}_{\Gamma(p^{\infty})} \rightarrow \adicFL
\] whose construction we now briefly recall.

We follow the discussion in \cite[\S 4.4.10]{BP-HigherColeman}. Let $\pi: \calA_n^{\univ} \rightarrow \calX_n$ be the rigid analytification of the universal abelian variety with identity section $e$. Consider the universal Tate module $T_p\calA_{n}^{\univ} \coloneq (R^1\pi_* \Z_p)^{\vee}$, viewed as an \'etale $\Z_p$-local system. Let $\underline{\omega}_{\calA_n^{\univ}}\coloneq e^* \Omega_{\calA_n^{\univ}/\calX_n}^1$ whose dual can be identified with $\Lie \calA_{n}^{\univ}$. Then, the (relative) Hodge--Tate filtration gives rise to a short exact sequence \[
    0 \rightarrow \Lie \calA_n^{\univ}\otimes \widehat{\scrO}_{\calX_n}(1) \rightarrow T_p\calA_n^{\univ}\otimes \widehat{\scrO}_{\calX_{n}} \rightarrow \underline{\omega}_{\calA_n^{\univ}}\otimes\widehat{\scrO}_{\calX_n} \rightarrow 0
\] 
of sheaves of $\widehat{\scrO}_{\calX_n}$-modules on the pro-\'etale site. It turns out this short exact sequence extends to $\calX_n^{\tor}$. More precisely, let $\calG_n^{\univ}$ be the rigid analytification of $G_n^{\univ}$ and let  $\underline{\omega} \coloneq e^*\Omega_{\calG_n^{\univ}/\calX^{\tor}_n}^1$ whose dual can be identified with $\Lie\calG_{n}^{\univ}$. Then there exists a Kummer \'etale $\Z_p$-local system $V_{\Z_p}$ on $\calX_n^{\tor}$, locally of rank $4$, extending $T_p\calA_n^{\univ}$ such that we have a short exact sequence 
\begin{equation}\label{eq: relative HT filtration}
    0 \rightarrow \Lie \calG_n^{\univ}\otimes \widehat{\scrO}_{\calX^{\tor}_n}(1) \rightarrow V_{\Z_p}\otimes \widehat{\scrO}_{\calX^{\tor}_{n}} \rightarrow \underline{\omega}\otimes\widehat{\scrO}_{\calX^{\tor}_n} \rightarrow 0
\end{equation} of sheaves on $\calX^{\tor}_{n, \proket}$. 

Denote by $\calG^{\an}$ (resp., $\calP^{\an}$) the rigid analytification of $\GSp_4$ (resp., $P_{\Si}$). They naturally extends to pro-Kummer \'etale sheaves $\calG^{\an}_{\proket}$ and $\calP^{\an}_{\proket}$ on $\calX^{\tor}_{n, \proket}$; namely, for any $\calU\in \calX^{\tor}_{n, \proket}$, we put
\[\calG^{\an}_{\proket}(\calU):= \calG^{\an}\big(\widehat{\scrO}_{\calX^{\tor}_{n}}(\calU), \widehat{\scrO}_{\calX^{\tor}_{n}}^+(\calU)\big)\]
and
\[\calP^{\an}_{\proket}(\calU):= \calP^{\an}\big(\widehat{\scrO}_{\calX^{\tor}_{n}}(\calU), \widehat{\scrO}_{\calX^{\tor}_{n}}^+(\calU)\big).\]
Moreover, let $\calG^{\an}_{\HT}$ (resp., $\calP^{\an}_{\HT}$) be the pro-Kummer \'etale sheaf on $\calX^{\tor}_{n, \proket}$ parameterising trivialisations of $V_{\Z_p}$ (resp., trivialisations of the short exact sequence \ref{eq: relative HT filtration}). More precisely, suppose $\calU$ is an affinoid perfectoid object in $\calX^{\tor}_{n, \proket}$ with associated affinoid perfectoid space $\Spa(R, R^+)$, we put
\begin{align*}
    \calG_{\HT}^{\an}(\calU) & = \Isom^{\mathrm{symp}}(R^4, V_{\Z_p}\otimes R),\\
    \calP_{\HT}^{\an}(\calU) & = \Isom^{\mathrm{symp}}\left( 0 \rightarrow R^2 \rightarrow R^4\rightarrow R^2 \rightarrow 0,\,\,\, 0 \rightarrow \Lie\calG_n^{\univ}\otimes R \rightarrow V_{\Z_p}\otimes R \rightarrow \underline{\omega}\otimes R \rightarrow 0 \right).
\end{align*}
Note that $\calG^{\an}_{\HT}$ (resp., $\calP^{\an}_{\HT}$) is a $\calG^{\an}_{\proket}$-torsor (resp., $\calP^{\an}_{\proket}$-torsor).

Now, let $\Spa(R, R^+)$ be an affinoid perfectoid subspace of the perfectoid space $\calX^{\tor}_{\Gamma(p^{\infty})}$ which corresponds to an affinoid perfectoid object $\calU$ in the pro-Kummer \'etale site $\calX^{\tor}_{n, \proket}$. Since the torsor $\calG^{\an}_{\HT}$ becomes trivial after pulling back to $\calX^{\tor}_{\Gamma(p^{\infty})}$, \footnote{Here we abuse the terminology and view $\calX^{\tor}_{\Gamma(p^{\infty})}$ as an element in the pro-Kummer \'etale site $\calX^{\tor}_{n, \proket}$.} we obtain an identification
\[\calG^{\an}_{\HT}(\calU)\cong \calG^{\an}_{\proket}(\calU)\cong \GSp_4(R).\]
As $\calP^{\an}_{\HT}$ is a $\calP^{\an}_{\proket}$-torsor, there exists $\bfgamma\in P_{\Si}(R)\backslash \GSp_4(R)$ such that $\calP^{\an}_{\HT}(\calU)=\bfgamma P_{\Si}(R)$. Then we put \[\pi_{\HT}(\calU)=\bfgamma^{-1}\in \adicFL(R, R^+).\] This description of the Hodge--Tate period map $\pi_{\HT}$ coincides with the definition in \cite{Pilloni-Stroh}. One also checks that $\pi_{\HT}$ is equivariant with respect to the natural right $\GSp_4(\Q_p)$-actions on both sides.

\begin{Convention}\label{Convention: abuse of notations}
By abuse of notation, we often identify $\calX^{\tor}_{\Gamma(p^{\infty})}$ as an object in the pro-Kummer \'etale site $\calX^{\tor}_{n, \proket}$. Then it makes sense to consider the localized site $\calX^{\tor}_{n, \proket}/\calX^{\tor}_{\Gamma(p^{\infty})}$, which we denote by $\calX^{\tor}_{\Gamma(p^{\infty}), \proket}$ by further abuse of notations. This convention also applies to affinoid perfectoid subspaces of $\calX^{\tor}_{\Gamma(p^{\infty})}$.
\end{Convention}

\begin{Remark}\label{Remark: normalisation of the flag variety}
We present an alternative description of $\calP_{\HT}^{\an}$ via pullback along the Hodge--Tate map. Consider the restriction $\calP_{\HT}^{\an}|_{\calX^{\tor}_{\Gamma(p^{\infty})}}$ of $\calP_{\HT}^{\an}$ to $\calX^{\tor}_{\Gamma(p^{\infty})}$, viewed as a sheaf on $\calX^{\tor}_{\Gamma(p^{\infty}), \proket}$ in the sense of Convention \ref{Convention: abuse of notations}. On the other hand, view $\calG^{\an}$ as a right $\calP^{\an}$-torsor over $\adicFL$ via \[
        \calG^{\an} \rightarrow\adicFL,\,\,\,\,\bfgamma \mapsto \bfgamma^{-1}.
    \]
Notice that the pullback along $\pi_{\HT}$ induces a map
 \[\pi_{\HT}^*: \mathrm{Sh}(\Mod_{\scrO_{\adicFL_{\an}}})\rightarrow \mathrm{Sh}(\Mod_{\widehat{\scrO}_{\calX^{\tor}_{\Gamma(p^{\infty}), \proket}}})\]
(see \cite[Theorem 4.2.1]{Rodriguez-LocallyAnalytic}). Then there is an isomorphism
\[\calP_{\HT}^{\an}|_{\calX^{\tor}_{\Gamma(p^{\infty})}}\cong \pi_{\HT}^*\calG^{\an}\times^{\bbG_m, \mu_{\Si}} \bbG_m(-1)\]
of $\calP^{\an}_{\proket}|_{\calX^{\tor}_{\Gamma(p^{\infty})}}$-torsors, where 
\[\bbG_m(-1) = \Isom_{\widehat{\scrO}_{\calX_{n, \proket}^{\tor}}}(\widehat{\scrO}_{\calX_{n, \proket}^{\tor}}, \widehat{\scrO}_{\calX_{n, \proket}^{\tor}}(-1))\]
is the $(-1)$-Hodge--Tate twist of $\bbG_m$. See \cite[Theorem 4.2.1]{Rodriguez-LocallyAnalytic} for more details.
\end{Remark}

\subsection{Classical automorphic sheaves}\label{subsection: classical Siegel modular forms}

Let $\Gamma_p$ be any aforementioned level structure at $p$. To recall the definition of classical algebraic Siegel modular forms (of genus $2$), we first construct an auxiliary $H$-torsor $H_{\dR}$ over $X^{\tor}_{\Gamma_p}$. Consider the tautological semiabelian variety $\pi: G^{\univ}_{\Gamma_p} \rightarrow X^{\tor}_{\Gamma_p}$ with identity section $e$. Let $\underline{\omega} \coloneq e^* \Omega_{G^{\univ}_{\Gamma_p}/X^{\tor}_{\Gamma_p}}^1$ and which is identified with the dual of $\Lie G^{\univ}_{\Gamma_p}$. Note that both $\underline{\omega}$ and $\Lie G^{\univ}_{\Gamma_p}$ are vector bundles of rank $2$. Consdier
\begin{align*}
  \scalemath{0.9}{   H_{\dR} :}&\scalemath{0.9}{ = \underline{\Isom}^{\mathrm{symp}}(\scrO_{X_{\Gamma_p}^{\tor}}^4, \Lie G^{\univ}_{\Gamma_p} \oplus \underline{\omega}) }\\
    & \scalemath{0.9}{ = \left\{ \psi_1 \oplus \psi_2: \scrO_{X_{\Gamma_p}^{\tor}}^2 \oplus \scrO_{X_{\Gamma_p}^{\tor}}^2 \rightarrow \Lie G^{\univ}_{\Gamma_p} \oplus \underline{\omega}: \begin{array}{l}
        \psi_1 = \varsigma \trans\psi_2^{-1} \text{ (for some unit $\varsigma$) via the isomorphism }\\
        \text{$\Lie G_{\Gamma_p}^{\univ, \vee} \cong \underline{\omega}$ given by the principal polarisation}
    \end{array}   \right\}.}
\end{align*}
which is an $H$-torsor over $X_{\Gamma_p}^{\tor}$. Let $\pr_{\dR}: H_{\dR} \rightarrow X_{\Gamma_p}^{\tor}$ denote the natural projection.

For an integral weight $k= (k_1, k_2; k_0)\in \Z^3$ with $k_1\geq k_2$, we have $\bfitw_3 k = (-k_2, -k_1; k_0+k_1+k_2)$.  The \emph{classical automorphic sheaf of weight $k$} is defined to be \[
    \underline{\omega}^{k} \coloneq \pr_{\dR, *} \scrO_{H_{\dR}}[\bfitw_3 k]. 
\]
In other words, $\underline{\omega}^k$ is the subsheaf of $\pr_{\dR, *} \scrO_{H_{\dR}}$, consisting of those sections $f$ such that \[
    f(\bfgamma \bfbeta) = \bfitw_3k(\bfbeta) f(\bfgamma)
\] for all $\bfgamma\in H_{\dR}$ and $\bfbeta\in B_H \coloneq B_{\GSp_4} \cap H$. \footnote{Here, as before, we extend $k$ to a character of $B_H$ by putting $k(N_H) = \{1\}$.}
Moreover, let $D_{\Gamma_p} \coloneq X_{\Gamma_p}^{\tor} \smallsetminus X_{\Gamma_p}$ be the boundary divisor. The \emph{classical cuspidal automorphic sheaf} is defined to be \[
    \underline{\omega}^k_{\cusp} \coloneq \underline{\omega}^k(-D_{\Gamma_p}). 
\]
The global sections of $\underline{\omega}^k$ (resp., $\underline{\omega}^k_{\cusp}$) are precisely the classical Siegel modular forms (resp., classical cuspidal Siegel modular forms). See also Remark \ref{Remark: explanation on the weights}.

\begin{Remark}\label{Remark: classical forms in the rigid analytic setting}
Let $\calH^{\an}$, $\calH_{\dR}^{\an}$, and $\calH_{\HT}^{\an}$ be the rigid analytifications of $H$, $H_{\dR}$, and $H_{\HT}$, respectively. In a similar fashion as in Remark \ref{Remark: normalisation of the flag variety}, we provide an alternative description of the $\calH^{\an}$-torsor $\calH_{\dR}^{\an}$. Recall the Hodge--Tate period map $\pi_{\HT}: \calX_{\Gamma(p^{\infty})}^{\tor}\rightarrow \adicFL$ and the natural projection $ h_{\Gamma_p}: \calX_{\Gamma(p^{\infty})}^{\tor}\rightarrow \calX_{n}^{\tor}$. We have an isomorphism
\[h_{\Gamma_p}^* \calH_{\dR}^{\an} = \pi_{\HT}^* \calH_{\HT}^{\an}\times^{\bbG_m, \mu_{\Si}} \bbG_m(-1)\]
of $\calH^{\an}$-torsors on (the analytic site of) $\calX_{\Gamma(p^{\infty})}^{\tor}$. This can be upgraded to an isomorphism of pro-Kummer \'etale sheaves. Indeed, we can naturally extend $\calH^{\an}$ and $\calH_{\dR}^{\an}$ to pro-Kummer \'etale sheaves on $\calX_{n, \proket}^{\tor}$, denoted by $\calH^{an}_{\proket}$ and $\calH_{\dR}^{\an}$ respectively. They are defined in a similar way as in the constructions of $\calG^{\an}_{\proket}$ and $\calG^{an}_{\HT}$, respectively. Then there is an isomorphism
\[\calH_{\dR}^{\an}|_{\calX^{\tor}_{\Gamma(p^{\infty})}}\cong \pi_{\HT}^*\calH^{\an}_{\HT}\times^{\bbG_m, \mu_{\Si}} \bbG_m(-1)\]
in the sense of Remark \ref{Remark: normalisation of the flag variety}. This is basically \cite[Remark 4.4.11]{BP-HigherColeman}, except for the Hodge-Tate twist.
\end{Remark}

\begin{Remark}\label{Remark: explanation on the weights}
    Let us briefly explain our convention, especially the appearance of $\bfitw_3$. Given $k=(k_1, k_2)\in \Z^2$ with $k_1 \geq k_2$, the usual classical automorphic sheaf of weight $k$ in the literature is \[
        \underline{\omega}^k_{\mathrm{trad}} \coloneq \Sym^{k_1-k_2} \underline{\omega} \otimes \left(\det\underline{\omega}\right)^{\otimes k_2}. 
    \]
   It is, in fact, canonically isomorphic to our automorphic sheaf. Indeed, after trivialising $\underline{\omega}$ over an affine $\mathrm{Spec}(R)$, we may view $\underline{\omega}^k_{\mathrm{trad}}(\mathrm{Spec}(R))$ as a $\GL_2$-representation; in fact, it is the $\GL_2$-representation of highest weight $k = (k_1, k_2)$. We may then view it as an $H$-representation via the projection \[
        H \twoheadrightarrow \GL_2, \quad \bfgamma = \begin{pmatrix} \bfgamma_a & \\ & \bfgamma_d \end{pmatrix} \mapsto \bfgamma_d.\footnote{ We use this convention because in the definition of $H_{\dR}$, $\underline{\omega}$ appears in the `second position' in the trivialisation.}
    \]
    Consequently, following a similar argument as in \cite[Remarque 4.1]{Pilloni-HidaGSp}, as an $H$-representation, $\underline{\omega}_{\mathrm{trad}}^k$ has highest weight $(-k_2, -k_1; k_1+k_2) = \bfitw_3 k$.
\end{Remark}

\begin{Remark}\label{Remark: integral classical sheaf}
    The automorphic bundles $\underline{\omega}^k$ admits a integral version. Indeed, we define \[
        \underline{\omega}^{k, +} := \pr_{\dR, *}\scrO_{H_{\dR}}^+[\bfitw_3 k].
    \]
    By \cite[Corollary 4.6.7]{BP-HigherColeman}, the sheaf $\underline{\omega}^{k, +}$ is an integral structure of $\underline{\omega}^k$ (in the sense of \cite[Definition 2.6.1]{BP-HigherColeman}).
\end{Remark}

Next, we discuss Hecke operators on the cohomology of $\underline{\omega}^k$. 

Let $\ell \neq p$ be a prime number. Given $\bfdelta\in \GSp_4(\Q_{\ell})$, we may find cone decompositions $\Sigma$, $\Sigma'$, $\Sigma''$ such that the corresponding toroidal compactifications fit into the diagram (\cite[\S 6.7.4]{FP-Hecke}) \begin{equation}\label{eq: correspondence away from p}
    \begin{tikzcd}
        & X_{\Gamma\Gamma_p \cap \bfdelta \Gamma\Gamma_p \bfdelta^{-1}}^{\Sigma'', \tor} \arrow[ld, "\pr_2"'] & X_{\bfdelta^{-1}\Gamma\Gamma_p \bfdelta \cap \Gamma\Gamma_p}^{\Sigma'' , \tor}\arrow[rd, "\pr_1"] \arrow[l, "\bfdelta"']\\
        X_{\Gamma_p}^{\Sigma', \tor} &&& X_{\Gamma_p}^{\Sigma, \tor}
    \end{tikzcd},
\end{equation} where the top arrow is an isomorphism. We claim that there is a \emph{trace map} \begin{equation}\label{eq: trace map}
     R\pr_{1, *}\pr_1^* \underline{\omega}^k \rightarrow \underline{\omega}^k.
\end{equation} Indeed, by \cite[\S 2.3]{FP-Hecke}, there is a trace map \[
    \mathrm{tr}: R\pr_{1, *}\pr_1^* \scrO_{X_{\Gamma_p}^{\Sigma, \tor}} = R \pr_{1, *}\scrO_{X_{\bfdelta^{-1}\Gamma\Gamma_p \bfdelta \cap \Gamma\Gamma_p}^{\Sigma'' , \tor}} \rightarrow \scrO_{X_{\Gamma_p}^{\Sigma, \tor}};
\] then, \eqref{eq: trace map} is obtained by taking the composition \begin{equation*}
    \begin{tikzcd}[column sep = tiny]
        R\pr_{1, *}\pr_1^* \underline{\omega}^k \arrow[r, equal] \arrow[rrd, bend right = 10, "\eqref{eq: trace map}"'] & R\pr_{1, *}\left(\scrO_{X_{\bfdelta^{-1}\Gamma\Gamma_p \bfdelta \cap \Gamma\Gamma_p}^{\Sigma'' , \tor}}\otimes\pr_1^* \underline{\omega}^k\right) \arrow[r, equal] & \left(R\pr_{1, *}\scrO_{X_{\bfdelta^{-1}\Gamma\Gamma_p \bfdelta \cap \Gamma\Gamma_p}^{\Sigma'' , \tor}}\right) \otimes \underline{\omega}^k \arrow[d]\\
        && \scrO_{X_{\Gamma_p}^{\Sigma, \tor}} \otimes \underline{\omega}^k = \underline{\omega}^k,
    \end{tikzcd},
\end{equation*}
where the second equality follows from the projection formula. The Hecke operator $T_{\bfdelta}$ is then defined to be the composition \[
    \begin{tikzcd}
        R\Gamma(X_{\Gamma_p}^{\Sigma', \tor}, \underline{\omega}^k) \arrow[r, "\pr_2^*"]\arrow[rddd, bend right = 20, "T_{\bfdelta}"'] & R\Gamma(X_{\Gamma\Gamma_p \cap \bfdelta \Gamma\Gamma_p \bfdelta^{-1}}^{\Sigma'', \tor}, \pr_2^* \underline{\omega}^k) \arrow[d, "\bfdelta^*"]\\
        & R\Gamma(X_{\bfdelta^{-1}\Gamma\Gamma_p \bfdelta \cap \Gamma\Gamma_p}^{\Sigma'' , \tor}, \pr_1^* \underline{\omega}^k)\arrow[d, "\cong"]\\
        & R\Gamma(X_{\Gamma_p}^{\Sigma, \tor}, R\pr_{1, *}\pr_1^* \underline{\omega}^k)\arrow[d]\\
        & R\Gamma(X_{\Gamma_p}^{\Sigma, \tor}, \underline{\omega}^k)
    \end{tikzcd},
\]
where the last map is given by the trace map (see \cite[\S 4.2.1]{BP-HigherColeman}). Note that the cohomologies of $R\Gamma(X_{\Gamma_p}^{\Sigma, \tor}, \underline{\omega}^k)$ do not depend on $\Sigma$ (see \cite[Theorem 4.1.8]{BP-HigherColeman}). So it is safe to simplify the notation and write \[
    T_{\bfdelta}: R\Gamma(X_{\Gamma_p}^{\tor}, \underline{\omega}^k) \rightarrow R\Gamma(X_{\Gamma_p}^{\tor}, \underline{\omega}^k). 
\]

For Hecke operators at $p$, we look at the following matrices \[
    \bfitu_{p, 0} \coloneq \begin{pmatrix}1 &&& \\ & 1 && \\ && p & \\ &&& p\end{pmatrix}, \quad \bfitu_{p, 1} \coloneq \begin{pmatrix}1 &&& \\ & p && \\ && p & \\ &&& p^2\end{pmatrix}, \quad \text{ and }\quad \bfitu_p \coloneq \bfitu_{p,0}\bfitu_{p, 1} = \begin{pmatrix} 1 &&& \\ & p && \\ && p^2 & \\ &&& p^3\end{pmatrix}.
\]
Following a similar construction as above, one obtains the Hecke operators $U_{p, 0}^{\mathrm{naive}}$, $U_{p,1}^{\mathrm{naive}}$, and $U_p^{\mathrm{naive}}$, which correspond to $\bfitu_{p,0}$, $\bfitu_{p, 1}$, and $\bfitu_p$ respectively. 

\begin{Definition}\label{Definition: finite slope part of classical cohomology}
    For any $\Gamma_p \in \{\Iw_{\GSp_4, n}, \Iw_{\GSp_4, n}^+\}$ and $k$ as above, the \emph{finite-slope part} of $R\Gamma(X_{\Gamma_p}^{\tor}, \underline{\omega}^k)$ is defined to be \[
        R\Gamma(X_{\Gamma_p}^{\tor}, \underline{\omega}^k)^{\fs} \coloneq R\Gamma(X_{\Gamma_p}^{\tor}, \underline{\omega}) \otimes_{\Z}^L \Z[U_{p, 0}^{\mathrm{naive},\pm 1}, U_{p,1}^{\mathrm{naive},\pm 1}].
    \]
\end{Definition}

\begin{Remark}\label{Remark: comparing fs part with Boxer--Pilloni}
    Compared with the convention in \cite{BP-HigherColeman}, our $R\Gamma(X_{\Gamma_p}^{\tor}, \underline{\omega}^k)^{\fs}$ is the \emph{minus-finite-slope part} therein.
\end{Remark}

\begin{Proposition}\label{Proposition: finite-slope part of classical cohomology is independent to the level}
    For any $n\in \Z_{>0}$, we have natural quasi-isomorphisms \[
        R\Gamma(X_{\Iw_{\GSp_4, 1}}^{\tor}, \underline{\omega}^k)^{\fs} \cong R\Gamma(X_{\Iw_{\GSp_4, n}}^{\tor}, \underline{\omega}^k)^{\fs} \cong R\Gamma(X_{n}^{\tor}, \underline{\omega}^k)^{\fs}.
    \]
\end{Proposition}
\begin{proof}
    The first quasi-isomorphism follows from \cite[Corollary 4.2.16]{BP-HigherColeman}. The proof of the second quasi-isomorphism is similar. Recall the Iwahori decompositions \[
        \Iw_{\GSp_4, n} = N_{\GSp_4, n}^{\opp} T_{\GSp_4}(\Z_p)N_{\GSp_4}(\Z_p) \quad \text{ and }\quad \Iw_{\GSp_4, n}^+ = N_{\GSp_4, n}^{\opp} T_{\GSp_4}(\Z_p) N_{\GSp_4, n}.
    \] 
    We apply \cite[Lemma 4.2.13]{BP-HigherColeman}\footnote{Note that there is a typo therein: $t_3$ should be $t_2^{-1}$.} and follow the notations therein. For $\bfitu\in \{\bfitu_{p,0}, \bfitu_{p,1}, \bfitu_p\}$, we have the following computations. 
    \begin{itemize}
        \item Take $K_1 = K_3 = \Iw_{\GSp_4, n}$, $K_2 = \Iw_{\GSp_4, n}^+$, $t_1 = \one_4$, $t_2 = \bfitu$, we have \[
            \begin{array}{c}
                N_{\GSp_4, n}^{\opp}  \cap \bfitu N_{\GSp_4, n}^{\opp}\bfitu^{-1}  \subset N_{\GSp_4, n}^{\opp} \\
                N_{\GSp_4}(\Z_p) \cap \bfitu N_{\GSp_4}(\Z_p) \bfitu^{-1} N_{\GSp_4, n} \subset N_{\GSp_4}(\Z_p) \subset \bfitu N_{\GSp_4}(\Z_p)\bfitu^{-1}.
            \end{array}
        \] 
        This implies a decomposition of double cosets \[
            [\Iw_{\GSp_4, n} \bfitu \Iw_{\GSp_4, n}] = [\Iw_{\GSp_4,n} \one_4 \Iw_{\GSp_4, n}^+][\Iw_{\GSp_4, n}^+ \bfitu \Iw_{\GSp_4,n}].
        \] 
        \item Take $K_1 = K_3 = \Iw_{\GSp_4, n}^+$, $K_2 = \Iw_{\GSp_4, n}$, $t_1 = \bfitu$, $t_2 = \one_4$, we have  \[
            \begin{array}{c}
                N_{\GSp_4, n}^{\opp}  \cap \bfitu^{-1}N_{\GSp_4, n}^{\opp}\bfitu  \subset N_{\GSp_4, n}^{\opp} \subset \bfitu^{-1}N_{\GSp_4, n}^{\opp}\bfitu\\
                \bfitu^{-1} N_{\GSp_4, n} \bfitu \cap N_{\GSp_4, n} N_{\GSp_4}(\Z_p) \subset N_{\GSp_4, n}.
            \end{array}
        \] 
        
        We then get a decomposition \[
            [\Iw_{\GSp_4, n}^+ \bfitu \Iw_{\GSp_4, n}^+] = [\Iw_{\GSp_4,n}^+ \bfitu \Iw_{\GSp_4, n}][\Iw_{\GSp_4, n} \one_4 \Iw_{\GSp_4,n}^+].
        \]
    \end{itemize}
    Consequently, we have a commutative diagram \[
        \begin{tikzcd}
            R\Gamma(X_n^{\tor}, \underline{\omega}^k) \arrow[r, "U"]\arrow[d, "\mathrm{tr}"] & R\Gamma(X_n^{\tor}, \underline{\omega}^k) \arrow[d, "\mathrm{tr}"]\\
            R\Gamma(X_{\Iw_{\GSp_4, n}}^{\tor}, \underline{\omega}^k) \arrow[r, "U"]\arrow[ru] & R\Gamma(X_{\Iw_{\GSp_4, n}}^{\tor}, \underline{\omega}^k) 
        \end{tikzcd}
    \] where $U$ is the operator associated with $\bfitu$ and the diagonal map is given by $[\Iw_{\GSp_4}^+ \bfitu \Iw_{\GSp_4, n}]$. By definition, the horizontal arrows given by $U$ are quasi-isomorphisms. This then implies that every morphism in the diagram is a quasi-isomorphism. 
\end{proof}

\subsection{Overconvergent automorphic sheaves via perfectoid methods}\label{subsection: overconvergent Siegel modular forms: perfectoid}
We shall follow \cite[Convention 2.2]{CHJ-2017} and use the symbol `$\widehat{\otimes}$' to denote either the \emph{complete tensor product} or the \emph{mixed complete tensor product}. We refer the readers to [\emph{op. cit.}, Definition 6.3] for its definition. See also \cite[Definition 3.1.3]{DRW}.\footnote{We remark that `$\widehat{\otimes}$' agrees with the \emph{solid tensor product} in the sense of \cite{CondensedMath}. J.-F.W. would like to thank Dustin Clausen for helpful discussion regarding this perspective. }

\begin{Definition}\label{Definition: w-loci of Siegel modular varieties}
    Let $\bfitw\in W^H$ and $m, n\in \Q_{\geq 0}$. \begin{enumerate}
        \item[(i)] The \emph{$(\bfitw, m, n)$-locus} on $\calX^{\tor}_{\Gamma(p^{\infty})}$ is defined to be \[
            \calX^{\tor}_{\Gamma(p^{\infty}), \bfitw, (m, n)} \coloneq \pi_{\HT}^{-1}(\adicFL_{\bfitw, (m, n)}).
        \] 
        Recall the coordinate $\begin{pmatrix} \one_2 \\ \bfitz & \one_2\end{pmatrix} \bfitw$ on $\adicFL_{\bfitw, (m, n)}$. We denote by \[
        \begin{pmatrix} \one_2 \\ \frakz & \one_2\end{pmatrix} \bfitw \coloneq \pi_{\HT}^* \left( \begin{pmatrix} \one_2 \\ \bfitz & \one_2\end{pmatrix} \bfitw \right)
        \]
        the corresponding coordinate on $\calX^{\tor}_{\Gamma(p^{\infty}), \bfitw, (m, n)}$. 
        \item[(ii)] Given any level structure $\Gamma_p$ at $p$, let $h_{\Gamma_p} : \calX^{\tor}_{\Gamma(p^{\infty})} \rightarrow \calX^{\tor}_{\Gamma_p}$ be the natural projection. The \emph{$(\bfitw, m, n)$-locus} on $\calX^{\tor}_{\Gamma_p}$ is defined to be \[
            \calX^{\tor}_{\Gamma_p, \bfitw,  (m, n)} \coloneq h_{\Gamma_p}(\calX^{\tor}_{\Gamma(p^{\infty}), \bfitw, (m, n)}).
        \] 
        \item[(iii)] Similarly, we define the \emph{$(\bfitw, \overline{m}, n)$-, $(\bfitw, m, \overline{n})$-, $(\bfitw, \overline{m}, \overline{n})$-loci} on $\calX^{\tor}_{\Gamma(p^{\infty})}$ and $\calX^{\tor}_{\Gamma}$. 
    \end{enumerate}
\end{Definition}

Fix $\bfitw \in W^H$ and let $(R_{\calU}, \kappa_{\calU})$ be a weight. Let $r\in \Q_{\geq 0}$ and $n\in \Z_{\geq 0}$ such that $n \geq r >1+r_{\calU}$. We define the sheaf $\scrA_{\bfitw, \kappa_{\calU}}^r$ (resp., $\scrA_{\bfitw, \kappa_{\calU}}^{r, \circ}$) on $\calX^{\tor}_{\Gamma(p^{\infty}), \bfitw, (r, r)}$ by \[
   \scalemath{0.9}{  \scrA_{\bfitw, \kappa_{\calU}}^r \coloneq A_{\bfitw_3^{-1}\bfitw\kappa_{\calU}}^{r}(\Iw_{H, 1}^+, R_{\calU}) \widehat{\otimes}\scrO_{\calX^{\tor}_{\Gamma(p^{\infty}), \bfitw, (r,r)}} \quad \left(\text{resp., } \scrA_{\bfitw, \kappa_{\calU}}^{r, \circ} \coloneq A_{\bfitw_3^{-1}\bfitw\kappa_{\calU}}^{r, \circ}(\Iw_{H, 1}^+, R_{\calU}) \widehat{\otimes}\scrO^+_{\calX^{\tor}_{\Gamma(p^{\infty}), \bfitw, (r,r)}}\right).}
\]
It is precisely the pullback of the pseudo-automorphic sheaf $\scrA_{\kappa_{\calU}, \adicFL_{\bfitw}}^r$ (resp., $\scrA_{\kappa_{\calU}, \adicFL_{\bfitw}}^{r, \circ}$) defined in \S \ref{subsection: pseudoautomorphic sheaves} via the Hodge--Tate period map
\[
\pi_{\HT}: \calX^{\tor}_{\Gamma(p^{\infty}), \bfitw, (r, r)}\rightarrow \adicFL_{\bfitw, (r,r)}.
\]
On $\calX^{\tor}_{\Gamma(p^{\infty}), \bfitw, (r,r)}$, given any $\bfalpha = \begin{pmatrix}\bfalpha_a & \bfalpha_b\\ \bfalpha_c & \bfalpha_d\end{pmatrix} \in \Iw_{\GSp_4, n}^+$, define \begin{equation}\label{eq: automorphy factor}
    \bfitj_{\bfitw}(\bfalpha, \frakz) \coloneq \begin{pmatrix} \varsigma(\bfalpha) \oneanti_2 \trans(\bfalpha_d^{\bfitw} + \frakz \bfalpha_b^{\bfitw})^{-1}\oneanti_2 & \bfalpha_b^{\bfitw}\\ & \bfalpha_d^{\bfitw} + \frakz\bfalpha_b^{\bfitw} \end{pmatrix} \in \Iw_{P_{\Si}, 1}^{+, (r)}.
\end{equation}
Then, for any $\calU \in \calX^{\tor}_{n, \bfitw, (r,r)}$, we define a left $\Iw_{\GSp_4, n}^+$-action on $\scrA_{\bfitw, \kappa_{\calU}}^r(h_{n}^{-1}(\calU))$ (resp., $\scrA_{\bfitw, \kappa_{\calU}}^{r, \circ}(h_{n}^{-1}(\calU))$) by \begin{equation}\label{eq: defining automorphy action}
    \bfalpha *_{\bfitw, \kappa_{\calU}} f \coloneq \rho^r_{\bfitw_3^{-1}\bfitw\kappa_{\calU}}(\bfitj_{\bfitw}(\bfalpha, \frakz))\bfalpha^*f \quad (\text{resp., } \bfalpha *_{\bfitw, \kappa_{\calU}} f \coloneq \rho^{r, \circ}_{\bfitw_3^{-1}\bfitw\kappa_{\calU}}(\bfitj_{\bfitw}(\bfalpha, \frakz))\bfalpha^*f)
\end{equation}
for any $\bfalpha\in \Iw_{\GSp_4, n}^+$ and $f\in \scrA_{\bfitw, \kappa_{\calU}}^r(h_{n}^{-1}(\calU))$ (resp., $f\in \scrA_{\bfitw, \kappa_{\calU}}^{r, \circ}(h_{n}^{-1}(\calU))$).

\begin{Definition}\label{Definition: overconvergnet automorphic sheaf via perfectoid method}
    Let $(R_{\calU}, \kappa_{\calU})$ be a weight, $r\in \Q_{\geq 0}$ and $n\in \Z_{\geq 0}$ such that $n\geq r > 1+r_{\calU}$. 
    \begin{enumerate}
        \item[(i)] The \emph{$(\bfitw, r)$-overconvergent automorphic sheaf of weight $\bfitw_3^{-1}\bfitw\kappa_{\calU}$} is the subsheaf $\underline{\omega}_{n, r}^{\bfitw_3^{-1}\bfitw\kappa_{\calU}}$ of $h_{n, *}\scrA_{\bfitw, \kappa_{\calU}}^r$ on $\calX^{\tor}_{n, \bfitw, (r, r)}$, consisting of sections $f$ such that \[
            \bfalpha * _{\bfitw, \kappa_{\calU}} f = f
        \] for any $\bfalpha \in \Iw_{\GSp_4, n}^+$. 
        \item[(ii)] The \emph{integral $(\bfitw, r)$-overconvergent automorphic sheaf of weight $\bfitw_3^{-1}\bfitw\kappa_{\calU}$} is the subsheaf $\underline{\omega}_{n, r}^{\bfitw_3^{-1}\bfitw\kappa_{\calU}, \circ}$ of $h_{n, *}\scrA_{\bfitw, \kappa_{\calU}}^{r, \circ}$ on $\calX^{\tor}_{n, \bfitw, (r, r)}$, consisting of sections $f$ such that \[
            \bfalpha * _{\bfitw, \kappa_{\calU}} f = f
        \] for any $\bfalpha \in \Iw_{\GSp_4, n}^+$.
        \item[(iii)] Let $\calD_{n, \bfitw, (r, r)} \coloneq (\calX_{n}^{\tor} \smallsetminus \calX_n) \cap \calX^{\tor}_{n, \bfitw, (r,r)}$ be the boundary divisor of $\calX^{\tor}_{n, \bfitw, (n,n)}$. The \emph{cuspidal $(\bfitw, r)$-overconvergent automorphic sheaf of weight $\bfitw_3^{-1}\bfitw\kappa_{\calU}$} is defined to be \[
            \underline{\omega}_{\cusp, n, r}^{\bfitw_3^{-1}\bfitw\kappa_{\calU}} \coloneq \underline{\omega}_{n, r}^{\bfitw_3^{-1}\bfitw\kappa_{\calU}}(-\calD_{n, \bfitw, (r,r)}).
        \]
        In other words, $\underline{\omega}_{\cusp, n, r}^{\bfitw_3^{-1}\bfitw\kappa_{\calU}}$ is the subsheaf of $\underline{\omega}_{n, r}^{\bfitw_3^{-1}\bfitw\kappa_{\calU}}$, consisting of those sections that vanish at the boundary divisor. 
        \item[(iv)] Similarly, the \emph{cuspidal integral $(\bfitw, r)$-overconvergent automorphic sheaf of weight $\bfitw_3^{-1}\bfitw\kappa_{\calU}$} is defined to be $\underline{\omega}_{\cusp, n, r}^{\bfitw_3^{-1}\bfitw\kappa_{\calU}, \circ} \coloneq \underline{\omega}_{n, r}^{\bfitw_3^{-1}\bfitw\kappa_{\calU}, +}(-\calD_{n, \bfitw, (r,r)})$.
    \end{enumerate}
\end{Definition}

\begin{Remark}\label{Remark: other versions of overconvergent automorphic sheaves via perfectoid methods}
    Similar constructions apply to the situation when we replace $A_{\bfitw_3^{-1}\bfitw\kappa_{\calU}}^r(\Iw_{H, 1}^+, R_{\calU})$ (resp., $A_{\bfitw_3^{-1}\bfitw\kappa_{\calU}}^{r, \circ}(\Iw_{H, 1}^+, R_{\calU})$) with $A_{\bfitw_3^{-1}\bfitw\kappa_{\calU}}^{r^+}(\Iw_{H, 1}^+, R_{\calU})$ (resp., $A_{\bfitw_3^{-1}\bfitw\kappa_{\calU}}^{r^+, \circ}(\Iw_{H, 1}^+, R_{\calU})$). In particular, we have sheaves $\underline{\omega}_{n, r^+}^{\bfitw_3^{-1}\bfitw\kappa_{\calU}}$ and $\underline{\omega}_{\cusp, n, r^+}^{\bfitw_3^{-1}\bfitw\kappa_{\calU}}$ (resp., $\underline{\omega}_{n, r^+}^{\bfitw_3^{-1}\bfitw\kappa_{\calU}, \circ}$ and $\underline{\omega}_{\cusp, n, r^+}^{\bfitw_3^{-1}\bfitw \kappa_{\calU}, \circ}$). From the construction, we see that \[
        \underline{\omega}_{n, r^+}^{\bfitw_3^{-1}\bfitw\kappa_{\calU}} = \varprojlim_{r'>r} \underline{\omega}_{n, r'}^{\bfitw_3^{-1}\bfitw\kappa_{\calU}} \quad \text{ and }\quad \underline{\omega}_{n, r^+}^{\bfitw_3^{-1}\bfitw\kappa_{\calU}, \circ} = \varprojlim_{r'>r} \underline{\omega}_{n, r'}^{\bfitw_3^{-1}\bfitw\kappa_{\calU}, \circ}.
    \]
    Similar statements hold for the cuspidal versions. 
\end{Remark}

\begin{Remark}\label{Remark: classical automorphic sheaf via perfectoid methods}
    Let $k = (k_1, k_2)\in \Z^2$ with $k_1 \geq k_2$, consider \[
        P_{\bfitw_3^{-1}\bfitw k} \coloneq \left\{ f: H \rightarrow \bbA^1: f(\bfgamma \bfbeta) = \bfitw_3^{-1}\bfitw k(\bfbeta)f(\bfgamma) \text{ for all } (\bfgamma, \bfbeta)\in H \times B_H\right\}.
    \] Similar as in Lemma \ref{Lemma: analytic representation for H is the same as analytic representation for P_Si}, we equip it with the left $P_{\Si}$-action by \[
        (\bfgamma * f)(\bfalpha) = f(\bfitw_3^{-1}\trans\bfgamma \bfitw_3 \bfalpha).  
    \] 
    The resulting $P_{\Si}$-representation is denoted by $\rho_{\bfitw_3^{-1}\bfitw k}^{\alg}$. 

    For later use, we define a sheaf $\underline{\omega}_{n, r, \alg}^{\bfitw_3^{-1}\bfitw k}$ as the subsheaf of $h_{n, *} \left(P_{\bfitw_3^{-1}\bfitw k} \otimes \scrO_{\calX_{\Gamma(p^{\infty}), \bfitw, (r, r)}^{\tor}}\right)$ consisting of sections $f$ such that \[
        f = \rho_{\bfitw_3^{-1}\bfitw k}^{\alg}(\bfitj_{\bfitw}(\bfalpha, \frakz)) \bfalpha^* f
    \]
    for any $\bfalpha\in \Iw_{\GSp_4, n}^+$. We shall see later (Remark \ref{Remark: comparison with classical automorphic sheaf}) that $\underline{\omega}_{n, r, \alg}^{\bfitw_3^{-1}\bfitw k}$ can be identified with the restriction of the classical automorphic sheaf $\underline{\omega}^{\bfitw_3^{-1}\bfitw k}$ on $\calX_{\Gamma(p^{\infty}), \bfitw, (r, r)}^{\tor}$.
\end{Remark}

\subsection{Overconvergent automorphic sheaves via analytic torsors}\label{subsection: overconvergent Siegel modular forms: torsors}
Fix $\bfitw\in W^H$ and $r\in \Q_{\geq 0}$. Recall the $\adicIw_{H, n}^+$-torsor $\adicIW_{H, n, \adicFL_{\bfitw}}^+$ over $\adicFL_{\bfitw}$ defined in \S \ref{subsection: vector bundles and torsors over the flag variety}. We define an analytic $\adicIw_{H, n}^+$-torsor $\adicIW_{H, n}^+$ over $\calX_{\Gamma(p^{\infty}), \bfitw, (r, r)}^{\tor}$ via the pullback \[
    \begin{tikzcd}
        \adicIW_{H, n}^+ \arrow[r]\arrow[d, "\pr_{\Iw_{H, n}^+}"'] & \adicIW_{H, n, \adicFL_{\bfitw}}^+|_{\adicFL_{\bfitw, (r, r)}} \arrow[d, "\pr_{\Iw_{H, n}^+, \adicFL_{\bfitw}}"]\\
        \calX_{\Gamma(p^{\infty}), \bfitw, (r, r)}^{\tor} \arrow[r, "\pi_{\HT}"] & \adicFL_{\bfitw, (r, r)}
    \end{tikzcd}.
\]
Note that the pullback exists in the category of analytic adic spaces. 

\begin{Remark}\label{Remark: pro-Kummer \'etale Iw_{H, n}^+-torsors}
At the moment, $\adicIW_{H, n}^+$ is defined as an analytic sheaf on $\calX_{\Gamma(p^{\infty}), \bfitw, (r,r)}^{\tor}$. Once again, we may upgrade everything to the pro-Kummer \'etale site. For later use, we spell out the details here. For any affinoid perfectoid object $\calU$ in the pro-Kummer \'etale site $\calX_{\Gamma(p^{\infty}), \bfitw, (r, r), \proket}^{\tor}$ \footnote{Here, we have abused the notations in the sense of Convention \ref{Convention: abuse of notations}. Namely, $\calX_{\Gamma(p^{\infty}), \bfitw, (r, r), \proket}^{\tor}$ stands for the localized site ${\calX_{n, \proket}^{\tor}}_{/\calX^{\tor}_{\Gamma(p^{\infty}), \bfitw, (r, r)}}$ where $\calX^{\tor}_{\Gamma(p^{\infty}), \bfitw, (r, r)}$ is identified with an affinoid perfectoid object in the pro-Kummer \'etale site $\calX_{n, \proket}^{\tor}$.} with associated affinoid perfectoid space $\Spa(R, R^+)$, we put 
\[
\scalemath{0.9}{ \adicIW_{H, n}^+(\calU) = \left\{ \psi: R^{+, 4} \xrightarrow{\cong} \Lie G_{\Gamma(p^{\infty})}^{\univ} \oplus \underline{\omega}_{\Gamma(p^{\infty})}: \{\psi(v_1), \ldots, \psi(v_4)\} \textrm{ is $n$-compatible w.r.t }\{\fraks_1^{\bfitw, \vee}, \fraks_2^{\bfitw, \vee}, \fraks_2^{\bfitw}, \fraks_1^{\bfitw}\} \right\},}
\]
where 
\begin{itemize}
\item $\Lie G_{\Gamma(p^{\infty})}^{\univ}$ (resp., $\underline{\omega}_{\Gamma(p^{\infty})}$) is the pullback of $\Lie G_{\Gamma_p}^{\univ}$ (resp., $\underline{\omega}$) from $\calX_{\Gamma_p}^{\tor}$ (for any of the aforementioned level structures $\Gamma_p$), and 
\item $\fraks_i^{\bfitw, \vee} = \pi_{\HT}^* \bfits_i^{\bfitw, \vee}$ and $\fraks_i^{\bfitw} = \pi_{\HT}^* \bfits_i^{\bfitw}$ for $i=1, 2$. 
\end{itemize}
We extend $\adicIw_{H, n}^+$ to a pro-Kummer \'etale sheaf $\adicIw_{H, n, \proket}^+$ on $\calX_{\Gamma(p^{\infty}), \bfitw, (r, r), \proket}^{\tor}$ in the same way as we extend $\calG^{\an}$ to $\calG^{\an}_{\proket}$ in \S \ref{subsection: Siegel threefolds}. That is, for every $\calU$ in $\calX_{\Gamma(p^{\infty}), \bfitw, (r, r), \proket}^{\tor}$, we put 
\[
    \adicIw_{H, n, \proket}^+(\calU):=\adicIw_{H, n}^+\left(\widehat{\scrO}_{\calX_{\Gamma(p^{\infty}), \bfitw, (r, r), \proket}^{\tor}}(\calU), \widehat{\scrO}^+_{\calX_{\Gamma(p^{\infty}), \bfitw, (r, r), \proket}^{\tor}}(\calU)\right)
\]
Then there is an isomorphism
\[
 \adicIW_{H,n}^+ = \pi_{\HT}^* \adicIW_{H, n, \adicFL_{\bfitw}}^+ \times^{\bbG_m^+, \mu_{\Si}}\bbG_m^+(-1)
\]
of $\adicIw_{H, n, \proket}^+$ torsors, where $\bbG_m^+:=\widehat{\scrO}_{\calX_{n, \proket}^{\tor}}^{+, \times}$ and
\[\bbG_m^+(-1) :=\Isom_{\widehat{\scrO}_{\calX_{n, \proket}^{\tor}}^+}\left(\widehat{\scrO}_{\calX_{n, \proket}^{\tor}}^+, \widehat{\scrO}_{\calX_{n, \proket}^{\tor}}^+(-1)\right)\]
is obtained by taking a Hodge--Tate twist.
\end{Remark}

\begin{Definition}\label{Definition: auxiliary automorphic sheaves via analytic torsors}
    Fix $\bfitw\in W^H$. Given a weight $(R_{\calU}, \kappa_{\calU})$ and $r\in \Q_{\geq 0}$, $n\in \Z_{\geq 0}$ with $n \geq r>1+r_{\calU}$. \begin{enumerate}
        \item[(i)] The \emph{auxiliary $(\bfitw, r)$-overconvergent automorphic sheaf of weight $\bfitw_3^{-1}\bfitw \kappa_{\calU}$} over $\calX_{n, \bfitw, (r, r)}^{\tor}$ is defined to be \[
            \widetilde{\underline{\omega}}_{n, r}^{\bfitw_3^{-1}\bfitw\kappa_{\calU}} \coloneq \left( h_{n, *}\left(\left(\pr_{\Iw_{H, 1}^+, *}\scrO_{\adicIW_{H, 1}^+}\widehat{\otimes}R_{\calU}\right)[\bfitw \kappa_{\calU}] \right)\right)^{\Iw_{\GSp_4, n}^+}. 
        \]
        \item[(ii)] The \emph{auxiliary integral $(\bfitw, r)$-overconvergent automorphic sheaf of weight $\bfitw_3^{-1}\bfitw \kappa_{\calU}$} over $\calX_{n, \bfitw, (r, r)}^{\tor}$ is defined to be \[
            \widetilde{\underline{\omega}}_{n, r}^{\bfitw_3^{-1}\bfitw\kappa_{\calU}, \circ} \coloneq \left( h_{n, *}\left(\left(\pr_{\Iw_{H, 1}^+, *}\scrO^+_{\adicIW_{H, 1}^+}\widehat{\otimes}R_{\calU}^{\circ}\right)[\bfitw \kappa_{\calU}]\right) \right)^{\Iw_{\GSp_4, n}^+}. 
        \]
    \end{enumerate}
\end{Definition}

\begin{Theorem}\label{Theorem: comparison theorem for automorphic sheaves}
    For any $\bfitw$, $(R_{\calU}, \kappa_{\calU})$, $r$, and $n$ given as above, we have a natural isomorphism of sheaves \[
        \underline{\omega}_{n, r}^{\bfitw_3^{-1}\bfitw\kappa_{\calU}} \cong \underline{\widetilde{\omega}}_{n, r}^{\bfitw_3^{-1}\bfitw\kappa_{\calU}}
    \] over $\calX_{n, \bfitw, (r, r)}^{\tor}$. 
\end{Theorem}
\begin{proof}
    From Proposition \ref{Prop: comparison theorem for pseudoautomorphic sheaves}, we know that \[
        \left(\pr_{\Iw_{H, 1}^+, *}\scrO_{\adicIW_{H, 1}^+}\widehat{\otimes}R_{\calU}\right)[\bfitw \kappa_{\calU}] \cong \scrA_{\bfitw, \kappa_{\calU}}^r. 
    \] We only need to show the compatibility of the $\Iw_{\GSp_4, n}^+$-action.

    By the proof of Proposition \ref{Prop: comparison theorem for pseudoautomorphic sheaves}, we know that the aforementioned isomorphism is given by \[
        f \mapsto \left( \bfgamma \mapsto f(\psi_{\bfitw}^{\std} \bfitw_3 \bfgamma \bfitw_3^{-1}) \right).
    \]
    Then, for any $\bfalpha\in \Iw_{\GSp_4, n}^+$, we know by Lemma \ref{Lemma: global sections and automorphy factors} that \[
        \bfalpha^* \psi_{\bfitw}^{\std} = \psi_{\bfitw}^{\std} \trans\begin{pmatrix}
            \varsigma(\bfalpha) \oneanti_2 \trans(\frakz\bfalpha_b^{\bfitw} + \bfalpha_d^{\bfitw})^{-1}\oneanti_2 & \\ & \frakz\bfalpha_b^{\bfitw} + \bfalpha_d^{\bfitw}
        \end{pmatrix}.
    \]
    Hence, for any $\bfgamma \in \Iw_{H, 1}^+$, \[
        \bfalpha^* \psi_{\bfitw}^{\std} \bfitw_3 \bfgamma \bfitw_3^{-1} = \psi_{\bfitw}^{\std}\bfitw_3 \left(\bfitw_3^{-1} \trans\begin{pmatrix}
            \varsigma(\bfalpha) \oneanti_2 \trans(\frakz\bfalpha_b^{\bfitw} + \bfalpha_d^{\bfitw})^{-1}\oneanti_2 & \\ & \frakz\bfalpha_b^{\bfitw} + \bfalpha_d^{\bfitw}
        \end{pmatrix}\bfitw_3\right) \bfgamma \bfitw_3^{-1}
    \]
    Moreover, note that $\begin{pmatrix}
            \varsigma(\bfalpha) \oneanti_2 \trans(\frakz\bfalpha_b^{\bfitw} + \bfalpha_d^{\bfitw})^{-1}\oneanti_2 & \\ & \frakz\bfalpha_b^{\bfitw} + \bfalpha_d^{\bfitw}
        \end{pmatrix}$ and $\bfitj_{\bfitw}(\bfalpha, \frakz)$ induce the same action on $A_{\bfitw_3^{-1}\bfitw\kappa_{\calU}}^r$ (as the former is the `Levi-part' of the latter). The desired statement follows.
\end{proof}

\begin{Remark}\label{Remark: comparison with classical automorphic sheaf}
    Recall the sheaf $\underline{\omega}_{n, r, \alg}^{\bfitw_3^{-1}\bfitw k}$ from Remark \ref{Remark: classical automorphic sheaf via perfectoid methods}. A similar proof as in Theorem \ref{Theorem: comparison theorem for automorphic sheaves} implies that \[
        \underline{\omega}_{n, r, \alg}^{\bfitw_3^{-1}\bfitw k} \cong \underline{\omega}^{\bfitw_3^{-1}\bfitw k}|_{\calX_{n, \bfitw, (r, r)}^{\tor}}.
    \] See also \cite[\S 3.4]{DRW}. 
\end{Remark}

\begin{Corollary}\label{Corollary: structure of an admissible Banach sheaf}
    For any $\bfitw$, $(R_{\calU}, \kappa_{\calU})$, $r$, and $n$ given as above, $\underline{\omega}_{n, r}^{\bfitw_3^{-1}\bfitw \kappa_{\calU}}$ is an admissible Banach sheaf (in the sense of \cite[Definition A.3.9]{DRW}) with integral model $\underline{\omega}_{n, r}^{\bfitw_3^{-1}\bfitw\kappa_{\calU}, \circ}$.
\end{Corollary}
\begin{proof}
    By Theorem \ref{Theorem: comparison theorem for automorphic sheaves}, we have to show that $\underline{\widetilde{\omega}}_{n, r}^{\bfitw_3^{-1}\bfitw\kappa_{\calU}}$ is an admissible Banach sheaf. The proof is exactly the same as \cite[Lemma 3.3.8 \& Lemma 3.3.10]{DRW}. 
\end{proof}

\begin{Remark}\label{Remark: pro-Kummer \'etale automorphic sheaf}
Thanks to Corollary \ref{Corollary: structure of an admissible Banach sheaf}, we can consider the $p$-adically completed pullback of the automorphic sheaf $\underline{\omega}_{n, r}^{\bfitw_3^{-1}\bfitw \kappa_{\calU}}$ to the pro-Kummer \'etale site $\calX_{n, \bfitw, (r, r), \proket}^{\tor}$; namely, we consider \[
        \widehat{\underline{\omega}}_{n, r}^{\bfitw_3^{-1}\bfitw \kappa_{\calU}} \coloneq \left(\varprojlim_{n} \underline{\omega}_{n, r}^{\bfitw_3^{-1}\bfitw \kappa_{\calU}, \circ} \otimes_{\scrO_{\calX_{n, \bfitw, (r, r)}^{\tor}}} \scrO_{\calX_{n, \bfitw, (r, r)}^{\tor}, \proket}^+/p^n\right)\Big[\frac{1}{p}\Big].
    \]
    This pro-Kummer \'etale incarnation of the automorphic sheaves will play a crucial role in the construction of the overconvergent Eichler--Shimura morphisms in \S \ref{section: OES}. 
\end{Remark}

\subsection{Hecke operators}\label{subsection: Hecke operators}

In this subsection, we discuss the Hecke operators acting on the cohomology of the overconvergent automorphic sheaves constructed in \S \ref{subsection: overconvergent Siegel modular forms: perfectoid} and \S \ref{subsection: overconvergent Siegel modular forms: torsors}. We start with explicit descriptions of the $U_p$-operators.

Recall the matrices \[
    \bfitu_{p, 0} \coloneq \begin{pmatrix}1 &&& \\ & 1 && \\ && p & \\ &&& p\end{pmatrix}, \quad \bfitu_{p, 1} \coloneq \begin{pmatrix}1 &&& \\ & p && \\ && p & \\ &&& p^2\end{pmatrix}, \quad \text{ and }\quad \bfitu_p \coloneq \bfitu_{p,0}\bfitu_{p, 1} = \begin{pmatrix} 1 &&& \\ & p && \\ && p^2 & \\ &&& p^3\end{pmatrix}.
\] These matrices act on $\calX_{\Gamma(p^{\infty})}^{\tor}$ via the $\GSp_{2g}(\Q_p)$-action on $\calX_{\Gamma(p^{\infty})}^{\tor}$. These actions can be described explicitly via the coordinates.

\begin{Lemma}\label{Lemma: dynamics of U_p-operators}
    Given $\bfitw\in W^H$ and $m, n\in \Q_{\geq 0}$, consider $\calX_{\Gamma(p^{\infty}), \bfitw, (m, n)}^{\tor}$ and its coordinate $\begin{pmatrix}\one_2 & \\ \frakz & \one_2\end{pmatrix} \bfitw$. For any $\bfitu \in \{\bfitu_{p, 0}, \bfitu_{p,1}, \bfitu_p\}$, let $\bfitu^{\bfitw, *} \frakz$ denote the coordinate after applying the $\bfitu$-action to the coordinate $\frakz$. \begin{itemize}
        \item When $\bfitw = \bfitw_3$, we have \[
            \bfitu_{p,0}^{\bfitw_3, *}\frakz = p\frakz \quad \text{ and }\quad \bfitu_{p,1}^{\bfitw_3, *}\frakz = \begin{pmatrix}p\frakz_{22}^+ & -p^2\frakz_{12}^+\\ -\frakz_{21}^+ & p\frakz_{22}^+\end{pmatrix}.
        \]
        Thus, $(\calX_{\Gamma(p^{\infty}), \bfitw_3, (m, n)}^{\tor})\bfitu_p \subset \calX_{\Gamma(p^{\infty}), \bfitw_3, (m+1, n+1)}^{\tor}$. 
        \item When $\bfitw = \bfitw_2$, we have \[
            \bfitu_{p,0}^{\bfitw_2, *}\frakz = \begin{pmatrix}\frakz_{22}^+ & -p\frakz_{12}^+\\ -p^{-1}\frakz_{21}^+ & \frakz_{22}^+\end{pmatrix} \quad \text{ and }\quad \bfitu_{p,1}^{\bfitw_2, *}\frakz = \begin{pmatrix}p\frakz_{22}^+ & -p^2\frakz_{12}^+\\ -\frakz_{21}^+ & p\frakz_{22}^+\end{pmatrix}.
        \]Thus, $(\calX_{\Gamma(p^{\infty}), \bfitw_2, (m, n)}^{\tor})\bfitu_p \subset \calX_{\Gamma(p^{\infty}), \bfitw_2, (m+1, n-1)}^{\tor}$.
        \item When $\bfitw = \bfitw_1$, we have \[
            \bfitu_{p,0}^{\bfitw_1, *}\frakz = \begin{pmatrix}\frakz_{22}^+ & -p\frakz_{12}^+\\ -p^{-1}\frakz_{21}^+ & \frakz_{22}^+\end{pmatrix} \quad \text{ and }\quad \bfitu_{p,1}^{\bfitw_1, *}\frakz = \begin{pmatrix}p^{-1}\frakz_{22}^+ & -\frakz_{12}^+\\ -p^{-2}\frakz_{21}^+ & p^{-1}\frakz_{22}^+\end{pmatrix} .
        \] Thus, $(\calX_{\Gamma(p^{\infty}), \bfitw_1, (m, n)}^{\tor})\bfitu_p \subset \calX_{\Gamma(p^{\infty}), \bfitw_1, (m+1, n-3)}^{\tor}$.
        \item When $\bfitw = \one_4$, we have \[
            \bfitu_{p,0}^{*}\frakz = p^{-1}\frakz \quad \text{ and }\quad \bfitu_{p,1}^{*}\frakz = \begin{pmatrix}p^{-1}\frakz_{22}^+ & -\frakz_{12}^+\\ -p^{-2}\frakz_{21}^+ & p^{-1}\frakz_{22}^+\end{pmatrix} .
        \] Thus, $(\calX_{\Gamma(p^{\infty}), \one_4, (m, n)}^{\tor})\bfitu_p \subset \calX_{\Gamma(p^{\infty}), \one_4, (m-3, n-3)}^{\tor}$.
    \end{itemize}
\end{Lemma}
\begin{proof}
    The statements follow from direct computations. 
\end{proof}

Given $\bfitw\in W^H$, a weight $(R_{\calU}, \kappa_{\calU})$, $r\in \Q_{\geq 0}$, and $n\in \Z_{>0}$ such that $n\geq r >1+r_{\calU}$, consider the loci \begin{equation}\label{eq: some technical loci}
    \begin{array}{rl}
        \calX_{n, \leq \bfitw}^{\tor} & \coloneq h_n\left( \pi_{\HT}^{-1}(\adicFL_{\leq \bfitw}) \right),\\
        \calX_{n, \geq \bfitw}^{\tor} & \coloneq h_n\left( \pi_{\HT}^{-1}(\adicFL_{\geq \bfitw}) \right),\\
        \calZ_{n, \bfitw} & \coloneq (\overline{\calX_{n, \leq \bfitw}^{\tor}})\bfitu_p^{-n-1} \cap (\calX_{n, \geq \bfitw}^{\tor})\bfitu_p^{n+1},\\
        \calX^{\tor, \bfitu_p}_{n, \bfitw} & \coloneq (\calX_{n, \geq \bfitw}^{\tor}) \bfitu_p^{n+1}.
    \end{array}
\end{equation}
By the discussion in \cite[\S 6.4.1]{BP-HigherColeman}, we know that $\calZ_{n,\bfitw} \subset \calX_{n, \bfitw, (r, r)}^{\tor}$. In particular, the automorphic sheaf $\underline{\omega}_{n, r}^{\bfitw_3^{-1}\bfitw \kappa_{\calU}}$ is defined in an open neighbourhood of $\calZ_{n,\bfitw}$. We consider the cohomology with supports $R\Gamma_{\calZ_{n, \bfitw}} (\calX^{\tor, \bfitu_p}_{n, \bfitw}, \underline{\omega}_{n, r}^{\bfitw_3^{-1}\bfitw \kappa_{\calU}}) \in \mathrm{D}(R_{\calU})$. Here we have abused the notation in the sense of Remark \ref{Remark: change of ambient spaces}; namely, we define
\[R\Gamma_{\calZ_{n, \bfitw}} (\calX^{\tor, \bfitu_p}_{n, \bfitw}, \underline{\omega}_{n, r}^{\bfitw_3^{-1}\bfitw \kappa_{\calU}}): =R\Gamma_{\calZ_{n, \bfitw}} (\calX^{\tor, \bfitu_p}_{n, \bfitw}\cap \calX_{n, \bfitw, (r, r)}^{\tor}, \underline{\omega}_{n, r}^{\bfitw_3^{-1}\bfitw \kappa_{\calU}}).\]
By (\ref{eq: change of ambient spaces}), there is a natural identification
\begin{equation}\label{eq: aut; change of ambient space}
    R\Gamma_{\calZ_{n, \bfitw}} (\calX^{\tor, \bfitu_p}_{n, \bfitw}, \underline{\omega}_{n, r}^{\bfitw_3^{-1}\bfitw \kappa_{\calU}})\cong R\Gamma_{\calZ_{n, \bfitw}} (\calX_{n, \bfitw, (r, r)}^{\tor}, \underline{\omega}_{n, r}^{\bfitw_3^{-1}\bfitw \kappa_{\calU}}).
\end{equation}
If $(R_{\calU}, \kappa_{\calU})$ is an affinoid weight, we know from \cite[Theorem 6.4.3]{BP-HigherColeman} that $
    R\Gamma_{\calZ_{n, \bfitw}} (\calX^{\tor, \bfitu_p}_{n, \bfitw}, \underline{\omega}_{n, r}^{\bfitw_3^{-1}\bfitw \kappa_{\calU}})$ is represented by an object in $\mathrm{Pro}_{\Z_{\geq 0}}(\mathrm{K}^{\proj}(\Ban(R_{\calU})))$.

\begin{Lemma}\label{Lemma: cohomology independent of toroidal compactification}
    Given $\bfitw$, $(R_{\calU}, \kappa_{\calU})$, $r$, $n$ as above, the complex $R\Gamma_{\calZ_{n, \bfitw}} (\calX^{\tor, \bfitu_p}_{n, \bfitw}, \underline{\omega}_{n, r}^{\bfitw_3^{-1}\bfitw \kappa_{\calU}})$ is independent of the choice of $\Sigma$ in the toroidal compactification. 
\end{Lemma}
\begin{proof}
It suffices to prove the statement for $R\Gamma_{\calZ_{n, \bfitw}} (\calX_{n, \bfitw, (r, r)}^{\tor}, \underline{\omega}_{n, r}^{\bfitw_3^{-1}\bfitw \kappa_{\calU}})$. Suppose $\Sigma$ and $\Sigma'$ are admissible cone decompositions such that $\Sigma$ is a refinement of $\Sigma'$. There is a natural morphism \[
        \pi^{\Sigma}_{\Sigma'}: \calX_n^{\Sigma, \tor} \rightarrow \calX_n^{\Sigma', \tor}
    \] which induces, at the infinite level, a commutative diagram \[
        \begin{tikzcd}
            \calX_{\Gamma(p^{\infty})}^{\Sigma, \tor}\arrow[r, "\pi_{\HT}^{\Sigma}"] \arrow[d, "\pi^{\Sigma}_{\Sigma'}"'] & \adicFL\\
            \calX_{\Gamma(p^{\infty})}^{\Sigma', \tor} \arrow[ru, "\pi_{\HT}^{\Sigma'}"']
        \end{tikzcd}.
    \]
 We consider loci $\calX_{n, \bfitw, (m, n)}^{\Sigma', \tor}$ and $\calZ_{n,\bfitw}^{\Sigma'}$ in a similar way as above. There is an isomorphism \[
        R\Gamma_{\calZ_{n,\bfitw}^{\Sigma}} (\calX_{n, \bfitw, (m, n)}^{\Sigma, \tor}, \underline{\omega}_{n, r}^{\bfitw_3^{-1}\bfitw \kappa_{\calU}}) \cong R\Gamma_{\calZ_{n,\bfitw}^{\Sigma'}} (\calX_{n, \bfitw, (m, n)}^{\Sigma', \tor}, R\pi^{\Sigma}_{\Sigma', *}\underline{\omega}_{n, r}^{\bfitw_3^{-1}\bfitw \kappa_{\calU}}).
    \] We claim that \[
        R^i\pi^{\Sigma}_{\Sigma', *}\underline{\omega}_{n, r}^{\bfitw_3^{-1}\bfitw \kappa_{\calU}} = 0 
    \]
    for all $i>0$. By Corollary \ref{Corollary: structure of an admissible Banach sheaf}, after restricting to an affinoid open $\Spa(R, R^+)$,  we may assume there is a trivialisation \[
        \underline{\omega}_{n, r}^{\bfitw_3^{-1}\bfitw\kappa_{\calU}}|_{\Spa(R, R)} \cong \widehat{\bigoplus}\scrO_{\calX_{n, \bfitw, (r, r)}^{\Sigma, \tor}}|_{\Spa(R, R)} \widehat{\otimes} R_{\calU}.
    \] Since the assertion is local, it reduces to show that \[
        R^i \pi^{\Sigma}_{\Sigma', *} \left(\widehat{\bigoplus}\scrO_{\calX_{n, \bfitw, (r, r)}^{\Sigma, \tor}}|_{\Spa(R, R)} \widehat{\otimes} R_{\calU} \right)= 0 \text{ for }i>0.
    \] 
    Note that \begin{align*}
        R^i \pi^{\Sigma}_{\Sigma', *} \left(\widehat{\bigoplus}\scrO_{\calX_{n, \bfitw, (r, r)}^{\Sigma, \tor}} \widehat{\otimes} R_{\calU} \right)& = R^i \pi^{\Sigma}_{\Sigma', *} \left(\varprojlim_n \left(\bigoplus \scrO^+_{\calX_{n, \bfitw, (r, r)}^{\Sigma, \tor}} \widehat{\otimes} R_{\calU}^{\circ}\right)/p^n\right)\Big[\frac{1}{p}\Big]\\
        & = \left(R^i \pi^{\Sigma}_{\Sigma', *} \varprojlim_n \left(\bigoplus \scrO^+_{\calX_{n, \bfitw, (r, r)}^{\Sigma, \tor}} \widehat{\otimes} R_{\calU}^{\circ}\right)/p^n\right)\Big[\frac{1}{p}\Big] \\
        & = \left(\varprojlim_n R^i \pi^{\Sigma}_{\Sigma', *} \left(\left(\bigoplus \scrO^+_{\calX_{n, \bfitw, (r, r)}^{\Sigma, \tor}} \widehat{\otimes} R_{\calU}^{\circ}\right)/p^n\right)\right)\Big[\frac{1}{p}\Big]\\
        & = \left(\varprojlim_n \bigoplus R^i \pi^{\Sigma}_{\Sigma', *} \left(\left(\scrO^+_{\calX_{n, \bfitw, (r, r)}^{\Sigma, \tor}} \widehat{\otimes} R_{\calU}^{\circ}\right)/p^n\right)\right)\Big[\frac{1}{p}\Big],
    \end{align*}
    where the second equation follows from the fact that localisation commutes with cohomology, the third equation follows from the fact that $\left\{\left(\bigoplus \scrO^+_{\calX_{n, \bfitw, (r, r)}^{\Sigma, \tor}} \widehat{\otimes} R_{\calU}^{\circ}\right)/p^n\right\}_{n\in \Z_{> 0}}$ is Mittag--Leffler, and the fourth equation follows from the fact that cohomology commutes with direct sum. Hence, if one shows that $R^i \pi^{\Sigma}_{\Sigma', *} \left(\left(\scrO^+_{\calX_{n, \bfitw, (r, r)}^{\Sigma, \tor}} \widehat{\otimes} R_{\calU}^{\circ}\right)/p^n\right)=0$ for $i>0$, then we are done.

    Consider the short exact sequence \[
        0 \rightarrow \scrO^+_{\calX_{n, \bfitw, (r, r)}^{\Sigma, \tor}} \widehat{\otimes} R_{\calU}^{\circ} \xrightarrow{\times p^n} \scrO^+_{\calX_{n, \bfitw, (r, r)}^{\Sigma, \tor}} \widehat{\otimes} R_{\calU}^{\circ} \rightarrow \left(\scrO^+_{\calX_{n, \bfitw, (r, r)}^{\Sigma, \tor}} \widehat{\otimes} R_{\calU}^{\circ}\right)/p^n \rightarrow 0
    \] By applying $R \pi^{\Sigma}_{\Sigma', *}$, we obtain an exact sequence \[
        R^i \pi^{\Sigma}_{\Sigma', *} \scrO^+_{\calX_{n, \bfitw, (r, r)}^{\Sigma, \tor}} \widehat{\otimes} R_{\calU}^{\circ} \rightarrow R^i \pi^{\Sigma}_{\Sigma', *} \left(\left(\scrO^+_{\calX_{n, \bfitw, (r, r)}^{\Sigma, \tor}} \widehat{\otimes} R_{\calU}^{\circ}\right)/p^n \right)\rightarrow R^{i+1} \pi^{\Sigma}_{\Sigma', *} \scrO^+_{\calX_{n, \bfitw, (r, r)}^{\Sigma, \tor}} \widehat{\otimes} R_{\calU}^{\circ}.
    \]
    However, we have \[R^i \pi^{\Sigma}_{\Sigma', *} \scrO^+_{\calX_{n, \bfitw, (r, r)}^{\Sigma, \tor}} \widehat{\otimes} R_{\calU}^{\circ} = \left(R^i \pi^{\Sigma}_{\Sigma', *} \scrO^+_{\calX_{n, \bfitw, (r, r)}^{\Sigma, \tor}} \right)\widehat{\otimes} R_{\calU}^{\circ}.\] Indeed, if $(R_{\calU}, \kappa_{\calU})$ is an affinoid weight, this follows from that $R_{\calU}^{\circ}$ is flat over $\Z_p$; if $(R_{\calU}, \kappa_{\calU})$ is a small weight, this follows from \cite[Corollary 6.5]{CHJ-2017}. By \cite[Proposition 7.5]{LanIntegralModels} (see also \cite[Proposition 2.4]{Harris-AutomorphicFormsAndBoundaryCohomology}), $R^i \pi^{\Sigma}_{\Sigma', *} \scrO^+_{\calX_{n, \bfitw, (r, r)}^{\Sigma, \tor}}$ vanishes for $i>0$, we thus conclude the result. 
\end{proof}

Let's now define Hecke operators away from $p$. Let $\ell \neq p$ be a prime number. Given $\bfdelta \in \GSp_4(\Q_{\ell})$, recall the correspondence \eqref{eq: correspondence away from p}, which gives rise to the correspondence \[
    \begin{tikzcd}
        & \calX_{\Gamma\Iw_{\GSp_4, n}^+ \cap \bfdelta \Gamma\Iw_{\GSp_4, n}^+ \bfdelta^{-1}}^{\Sigma'', \tor} \arrow[ld, "\pr_2"'] & \calX_{\bfdelta^{-1}\Gamma\Iw_{\GSp_4, n}^+ \bfdelta \cap \Gamma\Iw_{\GSp_4, n}^+}^{\Sigma'' , \tor}\arrow[rd, "\pr_1"] \arrow[l, "\bfdelta"']\\
        \calX_{n}^{\Sigma', \tor} &&& \calX_{n}^{\Sigma, \tor}
    \end{tikzcd}.
\] We define the loci \[
    \begin{array}{cccc}
        \calZ_{n, \bfitw}^{\Sigma}, & \calZ_{n, \bfitw}^{\Sigma'}, & \calZ_{\bfitw, \Gamma\Iw_{\GSp_4, n}^+ \cap \bfdelta \Gamma \Iw_{\GSp_4, n}^+ \bfdelta^{-1}}^{\Sigma''}, & \calZ_{ \bfdelta^{-1}\Gamma\Iw_{\GSp_4, n}^+ \bfdelta \cap  \Gamma \Iw_{\GSp_4, n}^+, \bfitw}^{\Sigma''}\\
        \calX_{n, \bfitw}^{\Sigma, \tor, \bfitu_p}, & \calX_{n, \bfitw}^{\Sigma', \tor, \bfitu_p}, & \calX_{ \Gamma\Iw_{\GSp_4, n}^+ \cap \bfdelta \Gamma \Iw_{\GSp_4, n}^+ \bfdelta^{-1}, \bfitw}^{\Sigma'', , \tor, \bfitu_p}, & \calX_{ \bfdelta^{-1}\Gamma\Iw_{\GSp_4, n}^+ \bfdelta \cap  \Gamma \Iw_{\GSp_4, n}^+, \bfitw}^{\Sigma'', , \tor, \bfitu_p}
    \end{array}
\] in a similar way as before. 

\begin{Lemma}\label{Lemma: supports conditions are compatible with Hecke correspondences away from p}
    We have the following identifications of loci:\begin{enumerate}
        \item[(i)] $\begin{array}{rl}
             \pr_2^{-1} (\calZ_{n, \bfitw}^{\Sigma'}) & = \calZ_{ \Gamma\Iw_{\GSp_4, n}^+ \cap \bfdelta \Gamma \Iw_{\GSp_4, n}^+ \bfdelta^{-1}, \bfitw}^{\Sigma''}  \\
             \pr_1^{-1}(\calZ_{n, \bfitw}^{\Sigma}) & = \calZ_{ \bfdelta^{-1}\Gamma\Iw_{\GSp_4, n}^+ \bfdelta \cap  \Gamma \Iw_{\GSp_4, n}^+, \bfitw}^{\Sigma''} \\
             \bfdelta^{-1}(\calZ_{ \Gamma\Iw_{\GSp_4, n}^+ \cap \bfdelta \Gamma \Iw_{\GSp_4, n}^+ \bfdelta^{-1}, \bfitw}^{\Sigma''}) & = \calZ_{ \bfdelta^{-1}\Gamma\Iw_{\GSp_4, n}^+ \bfdelta \cap  \Gamma \Iw_{\GSp_4, n}^+, \bfitw}^{\Sigma''}
        \end{array}$;
        \item[(ii)] $\begin{array}{rl}
             \pr_2^{-1} (\calX_{n, \bfitw}^{\Sigma', \tor, \bfitu_p}) & = \calX_{ \Gamma\Iw_{\GSp_4, n}^+ \cap \bfdelta \Gamma \Iw_{\GSp_4, n}^+ \bfdelta^{-1}, \bfitw}^{\Sigma'', \tor, \bfitu_p}  \\
             \pr_1^{-1}(\calX_{n, \bfitw}^{\Sigma, , \tor, \bfitu_p}) & = \calX_{ \bfdelta^{-1}\Gamma\Iw_{\GSp_4, n}^+ \bfdelta \cap  \Gamma \Iw_{\GSp_4, n}^+, \bfitw}^{\Sigma'', , \tor, \bfitu_p} \\
             \bfdelta^{-1}(\calX_{ \Gamma\Iw_{\GSp_4, n}^+ \cap \bfdelta \Gamma \Iw_{\GSp_4, n}^+ \bfdelta^{-1}, \bfitw}^{\Sigma'', , \tor, \bfitu_p}) & = \calX_{ \bfdelta^{-1}\Gamma\Iw_{\GSp_4, n}^+ \bfdelta \cap  \Gamma \Iw_{\GSp_4, n}^+, \bfitw}^{\Sigma'', \tor, \bfitu_p}
        \end{array}.$
    \end{enumerate}
\end{Lemma}
\begin{proof}
    By varying the level at $p$, we obtain the following commutative diagram \begin{equation}\label{eq: Hecke correspondence away from p at infinite level}
        \begin{tikzcd}[column sep = tiny]
            && \adicFL\\
            & \calX_{\Gamma\Gamma(p^{\infty}) \cap \bfdelta \Gamma\Gamma(p^{\infty}) \bfdelta^{-1}}^{\Sigma'', \tor} \arrow[ld, "\pr_2"']\arrow[dd, "h_n^{\Sigma''}"]\arrow[ru, "\pi_{\HT}^{\Sigma''}"] && \calX_{\bfdelta^{-1}\Gamma\Gamma(p^{\infty}) \bfdelta \cap \Gamma\Gamma(p^{\infty})}^{\Sigma'' , \tor}\arrow[rd, "\pr_1"] \arrow[ll, "\bfdelta"']\arrow[dd, "h_n^{\Sigma''}"]\arrow[lu, "\pi_{\HT}^{\Sigma''}"']\\
            \calX_{\Gamma(p^{\infty})}^{\Sigma', \tor}\arrow[dd, "h_n^{\Sigma'}"]\arrow[rruu, bend left = 30, "\pi_{\HT}^{\Sigma'}"] &&&& \calX_{\Gamma(p^{\infty})}^{\Sigma, \tor}\arrow[dd, "h_n^{\Sigma}"]\arrow[lluu, bend right = 30, "\pi_{\HT}^{\Sigma}"']\\
            & \calX_{\Gamma\Iw_{\GSp_4, n}^+ \cap \bfdelta \Gamma\Iw_{\GSp_4, n}^+ \bfdelta^{-1}}^{\Sigma'', \tor} \arrow[ld, "\pr_2"'] && \calX_{\bfdelta^{-1}\Gamma\Iw_{\GSp_4, n}^+ \bfdelta \cap \Gamma\Iw_{\GSp_4, n}^+}^{\Sigma'' , \tor}\arrow[rd, "\pr_1"] \arrow[ll, "\bfdelta"']\\
            \calX_{n}^{\Sigma', \tor} &&&& \calX_{n}^{\Sigma, \tor}
        \end{tikzcd}.
    \end{equation}
    Note that the bottom quadrilaterals are cartesian. The assertions then follow. 
\end{proof}

\begin{Lemma}\label{Lemma: isomorphism of overconvergent automorphic sheaves in Hecke correspondences away from p}
    Consider the overconvergent automorphic sheaf $\underline{\omega}_{n, r}^{\bfitw_3^{-1}\bfitw \kappa_{\calU}}$. We have an isomorphism of sheaves \[
        \bfdelta^* \pr_2^* \underline{\omega}_{n, r}^{\bfitw_3^{-1}\bfitw \kappa_{\calU}} \cong \pr_1^* \underline{\omega}_{n, r}^{\bfitw_3^{-1}\bfitw \kappa_{\calU}}. 
    \]
\end{Lemma}
\begin{proof}
    Due to the commutativity and the $\GSp_4(\Q_p)$-equivariance of the upper triangles in \eqref{eq: Hecke correspondence away from p at infinite level}, the pullbacks of $\begin{pmatrix}\one_2 & \\ \bfitz & \one_2\end{pmatrix} \bfitw$ via the Hodge--Tate period maps are compatible. This implies the desired result. 
\end{proof}

\begin{Lemma}\label{Lemma: projection formula for Hecke operators}
    The natural morphism \[
        \left(R\pr_{1, *} \scrO_{\calX_{ \bfdelta^{-1}\Gamma\Iw_{\GSp_4, n}^+ \bfdelta \cap  \Gamma \Iw_{\GSp_4, n}^+, \bfitw}^{\Sigma'', \tor, \bfitu_p}}\right) \widehat{\otimes} \underline{\omega}_{n, r}^{\bfitw_3^{-1}\bfitw\kappa_{\calU}} \rightarrow R\pr_{1, *}\pr_1^* \underline{\omega}_{n, r}^{\bfitw_3^{-1}\bfitw\kappa_{\calU}}
    \] is an isomorphism. 
\end{Lemma}
\begin{proof}
    Throughout this proof, to ease the notation, we simply write $\scrO$ and $\scrO^+$ for the structure sheaves.   
      
    It suffices to check the isomorphism locally. By Corollary \ref{Corollary: structure of an admissible Banach sheaf}, we know that $\underline{\omega}_{n, r}^{\bfitw_3^{-1}\bfitw\kappa_{\calU}}$ is admissible. That is, locally we can describe $\underline{\omega}_{n, r}^{\bfitw_3^{-1}\bfitw\kappa_{\calU}}$ as \[
        \underline{\omega}_{n, r}^{\bfitw_3^{-1}\bfitw\kappa_{\calU}} = \left( \varprojlim_m \underline{\omega}_{n, r, m}^{\bfitw_3^{-1}\bfitw\kappa_{\calU}}\right)\Big[\frac{1}{p}\Big]= \left(\varprojlim_m \varinjlim_d \underline{\omega}_{n, r, m, d}^{\bfitw_3^{-1}\bfitw\kappa_{\calU}}\right)\Big[\frac{1}{p}\Big],
    \]
    where $\underline{\omega}_{n, r, m}^{\bfitw_3^{-1}\bfitw\kappa_{\calU}} = \underline{\omega}_{n, r}^{\bfitw_3^{-1}\bfitw\kappa_{\calU}, \circ}/\fraka^m$\footnote{ Here, $\fraka$ is a fixed ideal of definition of $R_{\calU}^{\circ}$ containing $p$.} and each $\underline{\omega}_{n, r, m, d}^{\bfitw_3^{-1}\bfitw\kappa_{\calU}}$ is a coherent $\scrO^+\widehat{\otimes}R_{\calU}^{\circ}/\fraka^m$-module, locally free of finite rank. Hence, locally, we have \begin{align*}
        \left(R\pr_{1, *} \scrO\right) \widehat{\otimes} \underline{\omega}_{n, r}^{\bfitw_3^{-1}\bfitw\kappa_{\calU}} & = \left(\varprojlim_m \left(R\pr_{1, *} \scrO^+\right) \otimes \underline{\omega}_{n, r, m}^{\bfitw_3^{-1}\bfitw\kappa_{\calU}}\right)\Big[\frac{1}{p}\Big]\\
        & = \left(\varprojlim_m \left(R\pr_{1, *} \scrO^+\right) \otimes \varinjlim_d\underline{\omega}_{n, r, m,d}^{\bfitw_3^{-1}\bfitw\kappa_{\calU}}\right)\Big[\frac{1}{p}\Big]\\
        & = \left(\varprojlim_m \varinjlim_d\left(R\pr_{1, *} \scrO^+\right) \otimes \underline{\omega}_{n, r, m,d}^{\bfitw_3^{-1}\bfitw\kappa_{\calU}}\right)\Big[\frac{1}{p}\Big]\\
        & \cong \left( \varprojlim_m \varinjlim_d R\pr_{1,*}\pr_1^* \underline{\omega}_{n, r, m, d}^{\bfitw_3^{-1}\bfitw\kappa_{\calU}} \right)\Big[\frac{1}{p}\Big]\\
        & = \left( \varprojlim_m R\pr_{1, *}\pr_1^* \underline{\omega}_{n, r, m}^{\bfitw_3^{-1}\bfitw\kappa_{\calU}} \right)\Big[\frac{1}{p}\Big]\\
        & = \left(R\pr_{1,*}\pr_1^* \underline{\omega}_{n, r}^{\bfitw_3^{-1}\bfitw\kappa_{\calU}, \circ}\right)\Big[\frac{1}{p}\Big]\\
        & = R\pr_{1, *}\pr_1^* \underline{\omega}_{n, r}^{\bfitw_3^{-1}\bfitw\kappa_{\calU}},
    \end{align*}
    where the first and the last equation follows from that localisation is exact, the third and the ante-penultimate equation follows from that we are working locally on an affinoid and cohomology commutes with filtered colimits in such a situation, the penultimate equation is implied by the fact that $\{\underline{\omega}_{n, r, m}^{\bfitw_3^{-1}\bfitw\kappa_{\calU}, \circ}\}_m$ is a Mittag-Leffler system, and the isomorphism follows from the projection formulae applied to the coherent $\scrO^+\widehat{\otimes}R_{\calU}^{\circ}/\fraka^m$-modules, that are locally free of finite rank. This completes the proof. 
\end{proof}

Given lemmas above, we define the operator $T_{\bfdelta}$ as a composition \begin{equation}\label{eq: definition of Hecke operator away from p}
\scalemath{0.9}{    \begin{tikzcd}
        R\Gamma_{\calZ_{n,\bfitw}^{\Sigma'}} (\calX_{n,\bfitw}^{\Sigma', \tor, \bfitu_p}, \underline{\omega}_{n, r}^{\bfitw_3^{-1}\bfitw \kappa_{\calU}}) \arrow[d, "\pr_2^*"] \arrow[dddd, bend right =90, looseness=3, "T_{\bfdelta}"']\\
        R\Gamma_{\calZ_{ \Gamma\Iw_{\GSp_4, n}^+ \cap \bfdelta \Gamma \Iw_{\GSp_4, n}^+ \bfdelta^{-1}, \bfitw}^{\Sigma''}}(\calX_{ \Gamma\Iw_{\GSp_4, n}^+ \cap \bfdelta \Gamma \Iw_{\GSp_4, n}^+ \bfdelta^{-1}, \bfitw}^{\Sigma'', , \tor, \bfitu_p}, \pr_2^*\underline{\omega}_{n, r}^{\bfitw_3^{-1}\bfitw \kappa_{\calU}}) \arrow[d, "\bfdelta^*"]\\
        R\Gamma_{\calZ_{ \bfdelta^{-1}\Gamma\Iw_{\GSp_4, n}^+ \bfdelta \cap  \Gamma \Iw_{\GSp_4, n}^+, \bfitw}^{\Sigma''}}(\calX_{ \bfdelta^{-1}\Gamma\Iw_{\GSp_4, n}^+ \bfdelta \cap  \Gamma \Iw_{\GSp_4, n}^+, \bfitw}^{\Sigma'', , \tor, \bfitu_p}, \pr_1^*\underline{\omega}_{n, r}^{\bfitw_3^{-1}\bfitw \kappa_{\calU}}) \arrow[d, "\cong"]\\
        R\Gamma_{\calZ_{n, \bfitw}^{\Sigma}}(\calX_{n, \bfitw}^{\Sigma, , \tor, \bfitu_p}, R\pr_{1, *}\pr_1^* \underline{\omega}_{n, r}^{\bfitw_3^{-1}\bfitw \kappa_{\calU}})\arrow[d]\\
        R\Gamma_{\calZ_{n, \bfitw}^{\Sigma}}(\calX_{n, \bfitw}^{\Sigma,\tor, \bfitu_p}, \underline{\omega}_{n, r}^{\bfitw_3^{-1}\bfitw \kappa_{\calU}})
    \end{tikzcd},}
\end{equation}
where the last vertical arrow is obtained similarly as \eqref{eq: trace map}. Note here that one needs to replace the use of the projection formula with the one in Lemma \ref{Lemma: projection formula for Hecke operators}. Thanks to Lemma \ref{Lemma: cohomology independent of toroidal compactification}, the diagram induces an operator \[
    T_{\bfdelta}: R\Gamma_{\calZ_{n, \bfitw}}(\calX^{\tor, \bfitu_p}_{n, \bfitw}, \underline{\omega}_{n, r}^{\bfitw_3^{-1}\bfitw\kappa_{\calU}}) \rightarrow R\Gamma_{\calZ_{n, \bfitw}}(\calX^{\tor, \bfitu_p}_{n,\bfitw}, \underline{\omega}_{n, r}^{\bfitw_3^{-1}\bfitw\kappa_{\calU}}). 
\]

Now we look at Hecke operators at $p$. For computational convenience, we define them through explicit formulae. We remark that one can give an equivalent definition through correspondences. For such an approach, we refer the readers to \cite[\S 6.3.9]{BP-HigherColeman}. 

Given $\bfitw\in W^H$ and $\bfitu \in \{\bfitu_{p, 0}, \bfitu_{p,1}, \bfitu_p\}$, we define the $\bfitu$-action on $A_{\bfitw_3^{-1}\bfitw\kappa_{\calU}}^r(\Iw_{H, 1}^+, R_{\calU})$  as follows: for any $f\in A_{\bfitw_3^{-1}\bfitw\kappa_{\calU}}^r(\Iw_{H, 1}^+, R_{\calU})$ and any $\bfgamma = \bfepsilon \bfbeta \in \Iw_{H, 1}^+$ with $\bfepsilon \in N_{H, 1}^{\opp}$ and $\bfbeta \in \Iw_{H, 1}^+ \cap B_H(\Z_p)$, we put \[
    (\bfitu *_{\bfitw} f)(\bfgamma) = f\left( \bfitw_3^{-1}\bfitw \bfitu \bfitw^{-1} \bfitw_3 \bfepsilon (\bfitw_3^{-1}\bfitw \bfitu \bfitw^{-1} \bfitw_3)^{-1}\bfbeta \right).
\]
Here, we use the fact that \[
    \bfitw_3^{-1}\bfitw \bfitu \bfitw^{-1} \bfitw_3 N_{H, 1}^{\opp} (\bfitw_3^{-1}\bfitw \bfitu \bfitw^{-1} \bfitw_3)^{-1} \subset N_{H, 1}^{\opp}. 
\] Together with the $\bfitu$-action on the loci described in Lemma \ref{Lemma: dynamics of U_p-operators}, one obtains a $\bfitu$-action on the sheaf $\scrA_{\bfitw, \kappa_{\calU}}^r$. By abuse of notation we also denote this action by $\bfitu *_{\bfitw} -$.

Consider the double coset decomposition \[
    \Iw_{\GSp_4, n}^+ \bfitu_{p,i} \Iw_{\GSp_4, n}^+ = \bigsqcup_{j} \bfdelta_{ij} \bfitu_{p,i} \Iw_{\GSp_4,n}^+ \quad \text{ and }\quad \Iw_{\GSp_4, n}^+ \bfitu_{p} \Iw_{\GSp_4, n}^+ = \bigsqcup_{j} \bfdelta_{j} \bfitu_{p} \Iw_{\GSp_4,n}^+
\] with $\bfdelta_{ij}, \bfdelta_j\in \Iw_{\GSp_4, n}^+$. Then, for any section $f$ of $\underline{\omega}_{n, r}^{\bfitw_3^{-1}\bfitw \kappa_{\calU}}$, viewed as a section of $\scrA_{\bfitw, \kappa_{\calU}}^r$ invariant under the action \eqref{eq: defining automorphy action}, we define the \emph{na\"{i}ve} Hecke operators\[
    U_{p,i}^{\mathrm{naive}} : f \mapsto  \sum_{j} \bfdelta_{ij} *_{\bfitw, \kappa_{\calU}} \left( \bfitu_{p,i} *_{\bfitw} f \right) \quad \text{ and }\quad 
    U_p^{\mathrm{naive}}  : f \mapsto  \sum_{j} \bfdelta_{j} *_{\bfitw, \kappa_{\calU}} \left( \bfitu_{p} *_{\bfitw} f \right) 
\] as morphisms of complexes \[
    \scalemath{0.9}{ R\Gamma_{(\overline{\calX_{n, \leq \bfitw}^{\tor}})\bfitu^{-n} \cap (\calX_{n, \geq \bfitw}^{\tor})\bfitu^{n+1}}((\calX_{n, \geq \bfitw}^{\tor})\bfitu^{n+1}, \underline{\omega}_{n, r}^{\bfitw_3^{-1}\bfitw \kappa_{\calU}}) \xrightarrow{U^{\mathrm{naive}}} R\Gamma_{(\overline{\calX_{n, \leq \bfitw}^{\tor}})\bfitu^{-n-1} \cap (\calX_{n, \geq \bfitw}^{\tor})\bfitu^{n}}((\calX_{n, \geq \bfitw}^{\tor})\bfitu^{n}, \underline{\omega}_{n, r}^{\bfitw_3^{-1}\bfitw \kappa_{\calU}}).}
\]
From the construction, one sees that $U_p^{\mathrm{naive}} = U_{p,0}^{\mathrm{naive}}U_{p,1}^{\mathrm{naive}}$.

For $\bfitu = \bfitu_p$, we have a diagram \[
    \begin{tikzcd}[column sep = tiny]
        & \scalemath{0.6}{ R\Gamma_{(\overline{\calX_{n, \leq \bfitw}^{\tor}})\bfitu_p^{-n} \cap (\calX_{n, \geq \bfitw}^{\tor})\bfitu_p^{n}} ((\calX_{n, \geq \bfitw}^{\tor})\bfitu_p^{n}, \underline{\omega}_{n, r}^{\bfitw_3^{-1}\bfitw \kappa_{\calU}}) } \arrow[ld, "\Res"']\\
        \scalemath{0.6}{ R\Gamma_{(\overline{\calX_{n, \leq \bfitw}^{\tor}})\bfitu_p^{-n} \cap (\calX_{n, \geq \bfitw}^{\tor})\bfitu_p^{n+1}} ((\calX_{n, \geq \bfitw}^{\tor})\bfitu_p^{n+1}, \underline{\omega}_{n, r}^{\bfitw_3^{-1}\bfitw \kappa_{\calU}}) } \arrow[rr, "U_p^{\mathrm{naive}}"] && \scalemath{0.6}{ R\Gamma_{(\overline{\calX_{n, \leq \bfitw}^{\tor}})\bfitu_p^{-n-1} \cap (\calX_{n, \geq \bfitw}^{\tor})\bfitu_p^{n}} ((\calX_{n, \geq \bfitw}^{\tor})\bfitu_p^{n}, \underline{\omega}_{n, r}^{\bfitw_3^{-1}\bfitw \kappa_{\calU}}) } \arrow[ul, "\mathrm{Cores}"'] \arrow[ld, "\Res"]\\
        & \scalemath{0.6}{ R\Gamma_{\calZ_{n, \bfitw}}(\calX^{\tor, \bfitu_p}_{n,\bfitw}, \underline{\omega}_{n, r}^{\bfitw_3^{-1}\bfitw \kappa_{\calU}}) }\arrow[lu, "\mathrm{Cores}"]
    \end{tikzcd},
\]
where the $\Res$'s (resp., $\mathrm{Cores}$'s) in the diagram are restrictions (resp.,  corestrictions) and the composition on the top coincides with the composition at the bottom. Again by abuse by notation, we denote by $U_p^{\mathrm{naive}}$ the composition $\Res \circ U_p^{\mathrm{naive}} \circ \mathrm{Cores}$ on $R\Gamma_{\calZ_{n,\bfitw}}(\calX_{n,\bfitw}^{\tor, \bfitu_p}, \underline{\omega}_{n, r}^{\bfitw_3^{-1}\bfitw \kappa_{\calU}}) $. By slightly changing the support condition, one can similarly define the operator $U_{p,i}^{\mathrm{naive}}$ on $R\Gamma_{\calZ_{n, \bfitw}}(\calX^{\tor, \bfitu_p}_{n,\bfitw}, \underline{\omega}_{n, r}^{\bfitw_3^{-1}\bfitw \kappa_{\calU}})$.

For $\bfitu\in \{\bfitu_{p,0}, \bfitu_{p,1}, \bfitu_p\}$, we shall renormalise the corresponding operator $U^{\mathrm{naive}}\in \{U^{\mathrm{naive}}_{p,0}, U^{\mathrm{naive}}_{p,1}, U^{\mathrm{naive}}_p\}$. To this end, for $i=0, 1, 2, 3$, we write 
\[
        k_{\bfitw_i} = \left\{ \begin{array}{ll}
            (0,0), & \text{ if }i=0  \\
            (2,0), & \text{ if }i=1 \\ 
            (3,1), & \text{ if }i=2 \\
            (3,3), & \text{ if }i=3
        \end{array}\right. .
\]
Note that, by Kodaira--Spencer isomorphism (\cite[Theorem 1.41 (4)]{LanKS}), we have \[
    \underline{\omega}^{k_{\bfitw_i}} \cong \Omega_{\calX_n^{\tor}}^{\log, i}.
\]
On $R\Gamma_{\calZ_{n, \bfitw_i}}(\calX^{\tor, \bfitu_p}_{n,\bfitw_i}, \underline{\omega}_{n, r}^{\bfitw_3^{-1}\bfitw_i \kappa_{\calU}})$, we then define \[
    U := p^{-v_p(\bfitw_i^{-1}\bfitw_3k_{\bfitw_i}(\bfitu))} U^{\mathrm{naive}}.
\] 
where $U$ stands for $U_{p,0}$, $U_{p,1}$, or $U_p$. It follows that $U_p = U_{p,0} U_{p,1}$. The following table summarises the values of $v_p(\bfitw_i^{-1}\bfitw_3 k_{\bfitw_i}(\bfitu))$: 
\begin{center}
    \begin{tabular}{||c|c|c|c|c||}
         \hline
         & $i=0$ & $i=1$ & $i=2$ & $i=3$ \\
         \hline\hline
         $\bfitu = \bfitu_{p,0}$ & 0 & 0 & 1 & 0 \\
         \hline
         $\bfitu = \bfitu_{p,1}$ & 0 & 2 & 5 & 3 \\
         \hline
    \end{tabular}
\end{center}

\begin{Remark}\label{Remark: purpose of renormalisation}
    The purpose of such renormalisation is due to the fact that the Kodaira--Spencer isomorphism is not Hecke-equivariant (see \cite[pp. 257 -- 258]{Faltings-Chai}). Later in the paper, we shall use the Kodaira--Spencer isomorphism to obtain a morphism \[
        R\Gamma_{\calZ_{n, \bfitw_i}}(\calX^{\tor, \bfitu_p}_{n,\bfitw_i}, \underline{\omega}_{n, r}^{\bfitw_3^{-1}\bfitw_i \kappa_{\calU}}\otimes \Omega_{\calX_n^{\tor}}^{\log, i}) \rightarrow R\Gamma_{\calZ_{n, \bfitw_i}}(\calX^{\tor, \bfitu_p}_{n,\bfitw_i}, \underline{\omega}_{n, r}^{\bfitw_3^{-1}\bfitw_i \kappa_{\calU}+k_{\bfitw_i}}).
    \]
    By considering the na\"{i}ve Hecke operators on the source and the normalised Hecke operator on the target, this morphism is then Hecke-equivariant. 
\end{Remark}

\begin{Remark}
We only discuss the normalisation for Hecke operators at $p$. Technically, there should also be normalisations for those Hecke operators away from $pN$, due to same defect caused by the Kodaira--Spencer isomorphism. However, since these normalisations are given by $p$-adic units, they do not contribute in the $p$-adic valuation. Therefore, we do not spell out the explicit formula and leave them to the interested readers.
\end{Remark}

We shall see that the $U_p$-operator is \emph{potent compact}. 
For reader's convenience, we recall the definition of (potent) compact operators from \cite[\S 2.4]{BP-HigherColeman}.

\begin{Definition}\label{Definition: potent compact operators}
Let $(R, R^+)$ be a complete Tate algebra of finite type over $(\Q_p, \Z_p)$.
\begin{enumerate}
    \item[(i)]  An operator $T\colon M \rightarrow N$ of Banach $R$-modules is \emph{compact} if it is a limit of operators of finite rank.
    \item[(ii)] An operator $T\colon M^\bullet \rightarrow N^\bullet$ in $\mathrm{C}(\Ban(R))$ is \emph{compact} if it is compact in every degree. 
    \item[(iii)] An operator $T\colon M^\bullet \rightarrow N^\bullet$ in $\mathrm{K}^{\proj}(\Ban(R))$ is \emph{compact} if it has a representative in $\mathrm{C}^{\proj}(\Ban(R))$ that is compact.
    \item[(iv)] Let $T\colon \lim_i M_i^{\bullet}\rightarrow \lim_i N_i^{\bullet}$ be a morphism in $\mathrm{Pro}_{\Z_{\geq 0}}(\mathrm{K}^{\proj}(\Ban(R)))$. We say that $T$ is \emph{compact} if there exists a compact operator $T' : M^{\bullet}\rightarrow N^{\bullet}$ in $\mathrm{K}^{\proj}(\Ban(R))$ and a commutative diagram
    \[
    \begin{tikzcd}
        M^{\bullet} \arrow[r, "{T'}"] & N^{\bullet} \arrow[d]\\
        \lim_i M_i^{\bullet} \arrow[u]\arrow[r, "T"] & \lim_i N_i^{\bullet}
    \end{tikzcd}.
\]
   \item[(v)] Recall the natural functor $\mathrm{Pro}_{\Z_{\geq 0}}(\mathrm{K}^{\proj}(\Ban(R)))\rightarrow \mathrm{D}(R)$. Let $T\colon M^{\bullet}\rightarrow N^{\bullet}$ be a map in $\mathrm{D}(R)$ such that both $M^{\bullet}$ and $N^{\bullet}$ are represented by objects in $\mathrm{Pro}_{\Z_{\geq 0}}(\mathrm{K}^{\proj}(\Ban(R)))$. We say $T$ is \emph{compact} if it is represented by a compact morphism in $\mathrm{Pro}_{\Z_{\geq 0}}(\mathrm{K}^{\proj}(\Ban(R)))$.
    \item[(vi)] Let $M^{\bullet}\in \mathrm{D}(R)$ such that $M^{\bullet}$ is represented by an object in $\mathrm{Pro}_{\Z_{\geq 0}}(\mathrm{K}^{\proj}(\Ban(R)))$. Let $T\colon M^{\bullet}\rightarrow M^{\bullet}$ be an endomorphism of $M^{\bullet}$ in $\mathrm{D}(R)$. We say $T$ is \emph{potent compact} if $T^n$ is compact in the sense of (v) for some $n\geq 0$. 
\end{enumerate}
\end{Definition}

For a (potent) compact operator $T$ on $M^{\bullet}\in \mathrm{D}(R)$ as above, there is a way to make sense of the \emph{finite slope part} of $M^{\bullet}$ and $H^i(M^{\bullet})$ following \cite[\S 6.1]{BP-HigherColeman}. We briefly recall the constructions.

\begin{Proposition-Definition}\label{Prop-Defn: finite slope part}
Let $(R, R^+)$ be a complete Tate algebra of finite type over $(\Q_p, \Z_p)$ and let $\calS=\Spa(R, R^+)$. Let $M^{\bullet}\in \mathrm{K}^{\proj}(\Ban(R))$ and let $T:M^{\bullet}\rightarrow M^{\bullet}$ be a compact operator. Let $\scrM^{\bullet}$ be the associated complex of Banach sheaves on $\calS$ and let $H^k(\scrM^{\bullet})$ be the $k$-th cohomology sheaf. Then
\begin{enumerate}
\item[(i)] For each $k$, $H^k(\scrM^{\bullet})$ admits slope decomposition with respect to $T$ in the sense of \cite[Definition 6.1.5]{BP-HigherColeman}. In particular, one can define the \emph{finite slope part} $H^k(\scrM^{\bullet})^{\fs}$ together with a natural projection $H^k(\scrM^{\bullet})\rightarrow H^k(\scrM^{\bullet})^{\fs}$.
\item[(ii)] There exists an object $\scrM^{\bullet, \fs}\in \mathrm{D}(\Mod_{\scrO_{\calS}})$ and a morphism $\scrM^{\bullet}\rightarrow \scrM^{\bullet, \fs}$ (unique up to non-unique quasi-isomorphism) such that $H^k(\scrM^{\bullet, \fs})=H^k(\scrM^{\bullet})^{\fs}$ for all $k$.
\end{enumerate}
Taking global sections, we obtain the finite slope part $H^k(M^{\bullet})^{\fs}$ of $H^k(M^{\bullet})$ (resp., the finite slope part $M^{\bullet, \fs}$ of $M^{\bullet}$) such that $H^k(M^{\bullet, \fs})=H^k(M^{\bullet})^{\fs}$.
\end{Proposition-Definition}

\begin{Proposition-Definition}\label{Prop-Defn: finite slope part pro}
Let $(R, R^+)$ be a complete Tate algebra of finite type over $(\Q_p, \Z_p)$ and let $\calS=\Spa(R, R^+)$. Let $M^{\bullet}\in \mathrm{D}(R)$ such that $M^{\bullet}$ is represented by an object $\lim_i M_i^{\bullet}$ in $\mathrm{Pro}_{\Z_{\geq 0}}(\mathrm{K}^{\proj}(\Ban(R)))$. Let $T:M^{\bullet}\rightarrow M^{\bullet}$ be a compact operator, which induces a compact operator $T_i:M_i^{\bullet}\rightarrow M_i^{\bullet}$ for all $i$ sufficiently large. Let $\scrM_i^{\bullet}$ be the complex of Banach sheaves over $\calS$ corresponding to $M_i^{\bullet}$. Proposition-Definition \ref{Prop-Defn: finite slope part} yields morphisms $H^k(\scrM_i^{\bullet})\rightarrow H^k(\scrM_i^{\bullet})^{\fs}$ and $\scrM_i^{\bullet}\rightarrow \scrM_i^{\bullet, \fs}$ such that 
\begin{enumerate}
\item[(i)] For all $k$, we have $H^k(\scrM_i^{\bullet, \fs})=H^k(\scrM_i^{\bullet})^{\fs}$.
\item[(ii)] For all $k$, $H^k(\scrM_i^{\bullet})^{\fs}\rightarrow H^k(\scrM_{i-1}^{\bullet})^{\fs}$ are isomorphisms.
\end{enumerate}
Taking global sections, we obtain $M_i^{\bullet, \fs}$ and $H^k(M_i^{\bullet})^{\fs}$ such that $H^k(M_i^{\bullet, \fs})=H^k(M_i^{\bullet})^{\fs}$.

Finally, we put $H^k(\scrM^{\bullet})^{\fs}:= H^k(\scrM_i^{\bullet})^{\fs}$ and let $\scrM^{\bullet, \fs}$ be the image of $\scrM_i^{\bullet, \fs}$ in $\mathrm{D}(R)$, for some $i$ sufficiently large. Taking global sections, we obtain $M^{\bullet, \fs}$ and $H^k(M^{\bullet})^{\fs}$. We remark that $\scrM^{\bullet, \fs}$ and $M^{\bullet, \fs}$ depends on the choice of $i$ while $H^k(\scrM^{\bullet})^{\fs}$ and $H^k(M^{\bullet})^{\fs}$ does not. For our purpose, such ambiguity does not harm as we will eventually pass to cohomology.
\end{Proposition-Definition}

Back to our discussion on the $U_p$-operator.

\begin{Proposition}\label{Proposition: U_p potent compact}
The endomorphism $U_p$ is a potent compact operator on $R\Gamma_{\calZ_{n,\bfitw}}(\calX^{\tor, \bfitu_p}_{n, \bfitw}, \underline{\omega}_{n, r}^{\bfitw_3^{-1}\bfitw \kappa_{\calU}})$.
\end{Proposition}

\begin{proof}
This follows from \cite[Theorem 6.4.3]{BP-HigherColeman}.
\end{proof}

\begin{Definition}\label{Defn: finite slope part of complex}
Since $U_p$ is potent compact, say $U_p^n$ is compact for some integer $n$, we can define the \emph{finite slope part} $R\Gamma_{\calZ_{n,\bfitw}}(\calX^{\tor, \bfitu_p}_{n, \bfitw}, \underline{\omega}_{n, r}^{\bfitw_3^{-1}\bfitw \kappa_{\calU}})^{\fs}$ and $H^i_{\calZ_{n,\bfitw}}(\calX^{\tor, \bfitu_p}_{n, \bfitw}, \underline{\omega}_{n, r}^{\bfitw_3^{-1}\bfitw \kappa_{\calU}})^{\fs}$ with respect to $U_p^n$.
\end{Definition}


For later use, we would also like to consider the \emph{small-slope parts}. We first introduce certain numbers $h^{\mathrm{oc}}_{i,j,k}$, $h^{\mathrm{sh}}_{i,k}$, and $h_k$ which will play the role of ``small-slope bounds''.

\begin{Definition}\label{Definition: small-slope bounds}
Let $k = (k_1, k_2)\in \Z^2$ be an integral weight such that $k_1 \geq k_2\geq 0$.
\begin{enumerate}
\item[(i)] For $i=0,1,2,3$ and $j=0,1$, we define 
\[h^{\mathrm{oc}}_{i,j,k}\coloneq \inf_{\bfitw\neq \bfitw_i}\left\{v_p(\bfitw^{-1}\bfitw_i k(\bfitu_{p,j}))\right\}.\]
\item[(ii)] For $i=0,1,2,3$, we define
\[h^{\mathrm{sh}}_{i,k} \coloneq \left((\bfitw_i k)_1 - (\bfitw_i k)_2 +1\right)\cdot v_p\left(-(1, -1)(\bfitw_3^{-1}\bfitw_i \bfitu_{p,1}\bfitw_i^{-1}\bfitw_3)\right) .\]
\item[(iii)] We define 
\[h_k\coloneq \inf_{\bfitw \in W_{\GSp_4}\smallsetminus \{\one_4\}} \left\{ v_p(\bfitw \cdot k(\bfitu_p)) - v_p(k(\bfitu_p)) \right\}.\]
\end{enumerate}
\end{Definition}

\begin{Remark}
These numbers can be computed explicitly.
\begin{enumerate}
        \item[(i)] The following table computes $h^{\mathrm{oc}}_{i,j,k}$.
        \begin{center}
            \begin{tabular}{||c|c|c|c|c||}
                \hline
                 & $i=0$ & $i=1$ & $i=2$ & $i=3$  \\
                 \hline\hline
                 $j=0$ & $k_2$ & $0$ & $0$ & $k_2$ \\
                 \hline
                 $j=1$ & $k_2$ & $k_2$ & $k_2$ & $k_2$ \\
                 \hline
            \end{tabular}
        \end{center} 
        \item[(ii)] The following table computes $h^{\mathrm{sh}}_{i,k}$.
         \begin{center}
            \begin{tabular}{||c|c|c|c||}
                \hline
                 $i=0$ & $i=1$ & $i=2$ & $i=3$  \\
                 \hline\hline
                 $k_1-k_2+1$ & $k_1+k_2+1$ & $k_1+k_2+1$ & $k_1-k_2+1$ \\
                 \hline
            \end{tabular}
        \end{center} 
        \item[(iii)] We have $h_k = \inf\{k_1-k_2+1, k_2+1\}$. (See, for example, \cite[Example 4.5]{BSW-ParEigen}.)
    \end{enumerate}
\end{Remark}

\begin{Definition}\label{Definition: small slope part}
    Let $k = (k_1, k_2)\in \Z^2$ be an integral weight such that $k_1 \geq k_2$. For $\bfitw\in W^H$ (say $\bfitw=\bfitw_i$ for some $i=0,1,2,3$), consider the complex $R\Gamma_{\calZ_{n,\bfitw}}(\calX^{\tor, \bfitu_p}_{n,\bfitw}, \underline{\omega}_{n, r}^{\bfitw_3^{-1}\bfitw k + k_{\bfitw}})$. The \emph{small-slope part} $R\Gamma_{\calZ_{n,\bfitw}}(\calX^{\tor, \bfitu_p}_{n,\bfitw}, \underline{\omega}_{n, r}^{\bfitw_3^{-1}\bfitw k + k_{\bfitw}})^{\sms}$ is defined to be the direct summand of $R\Gamma_{\calZ_{n,\bfitw}}(\calX^{\tor, \bfitu_p}_{n,\bfitw}, \underline{\omega}_{n, r}^{\bfitw_3^{-1}\bfitw k + k_{\bfitw}})$ on which 
    \begin{enumerate}
    \item[(i)] The $p$-adic valuations of the $U_{p,j}$-eigenvalues are smaller than $h^{\mathrm{oc}}_{i,j,k}$, for both $j=0,1$;
    \item[(ii)] The $p$-adic valuations of the $U_{p,1}$-eigenvalues are smaller than $h^{\mathrm{sh}}_{i,k}$;
    \item[(iii)] The $p$-adic valuations of the $U_p$-eigenvalues are smaller than $h_k$.
    \end{enumerate}
The small-slope part $R\Gamma(\calX_{n}^{\tor}, \underline{\omega}^{\bfitw_3^{-1}\bfitw k + k_{\bfitw}})^{\sms}$ of $R\Gamma(\calX_{n}^{\tor}, \underline{\omega}^{\bfitw_3^{-1}\bfitw k + k_{\bfitw}})$ is defined in the same way. Moreover, for the cohomology groups, the small-slope parts $H^i_{\calZ_{n,\bfitw}}(\calX^{\tor, \bfitu_p}_{n,\bfitw}, \underline{\omega}_{n, r}^{\bfitw_3^{-1}\bfitw k + k_{\bfitw}})^{\sms}$ and  $H^i(\calX_{n}^{\tor}, \underline{\omega}^{\bfitw_3^{-1}\bfitw k + k_{\bfitw}})^{\sms}$ are also defined in the same way.
\end{Definition}

\begin{Remark}\label{Remark: small-slope part well-defined}
    Since $R\Gamma_{\calZ_{n,\bfitw}}(\calX^{\tor, \bfitu_p}_{n,\bfitw}, \underline{\omega}_{n, r}^{\bfitw_3^{-1}\bfitw k + k_{\bfitw}})$ is represented by an object in $\mathrm{Pro}_{\Z_{\geq 0}}(\mathrm{K}^{\proj}(\Ban(\Q_p)))$ and $U_p$ is potent compact, \cite[Proposition 5.1.4]{BP-HigherColeman} guarantees the existence of a slope-$\leq h$ decomposition for every $h\in \Q_{\geq 0}$. In particular, the small-slope part of $R\Gamma_{\calZ_{n,\bfitw}}(\calX^{\tor, \bfitu_p}_{n,\bfitw}, \underline{\omega}_{n, r}^{\bfitw_3^{-1}\bfitw k + k_{\bfitw}})$ is well-defined. Moreover, we have
\[
H^i\left(R\Gamma_{\calZ_{n,\bfitw}}(\calX^{\tor, \bfitu_p}_{n,\bfitw}, \underline{\omega}_{n, r}^{\bfitw_3^{-1}\bfitw k + k_{\bfitw}})^{\sms} \right)=H^i_{\calZ_{n,\bfitw}}(\calX^{\tor, \bfitu_p}_{n,\bfitw}, \underline{\omega}_{n, r}^{\bfitw_3^{-1}\bfitw k + k_{\bfitw}})^{\sms}.
\]
\end{Remark}

\begin{Remark}\label{Remark: Betti-small-slope condition in the definition}
In the proof of Theorem \ref{Theorem: classicality theorem in higher Coleman theory} below, it will become clear to the readers that only the conditions (i) and (ii) in Definition \ref{Definition: small slope part} are necessary for the classicality theorem to hold. We include the condition (iii) because we shall compare coherent cohomology groups with Betti cohomology groups later in the paper.  We also remark that, in this paper, we do not pursue the \emph{optimal} slope bound as in \cite[Theorem 1.4.10]{BP-HigherHidaSigel}.
\end{Remark}
        
We have the following classicality theorem for cohomology groups of the overconvergent automorphic sheaves.
    
 \begin{Theorem}[Classicality]\label{Theorem: classicality theorem in higher Coleman theory}   
 There is a natural quasi-isomorphism 
 \[
        R\Gamma(\calX_{n}^{\tor}, \underline{\omega}^{\bfitw_3^{-1}\bfitw k + k_{\bfitw}})^{\sms} \cong  R\Gamma_{\calZ_{\bfitw, n}}(\calX^{\tor, \bfitu_p}_{n,\bfitw}, \underline{\omega}_{n, r}^{\bfitw_3^{-1}\bfitw k + k_{\bfitw}})^{\sms}
\]
which induces an isomorphism
 \[
        H^i(\calX_{n}^{\tor}, \underline{\omega}^{\bfitw_3^{-1}\bfitw k + k_{\bfitw}})^{\sms} \cong  H^i_{\calZ_{\bfitw, n}}(\calX^{\tor, \bfitu_p}_{n,\bfitw}, \underline{\omega}_{n, r}^{\bfitw_3^{-1}\bfitw k + k_{\bfitw}})^{\sms}
\]
for every $i$.
\end{Theorem}
\begin{proof}
This is \cite[Theorem 5.12.3 \& Corollary 6.8.4]{BP-HigherColeman}. Here we sketch the proof for reader's convenience.

The first step is to establish a control theorem at the level of sheaves. By \cite[Proposition 7.2.1]{AIP-2015} (see also \cite[Lemma 6.2.13]{BP-HigherColeman}), for any $i=0,1,2,3$, there is a short exact sequence \[
     0 \rightarrow \underline{\omega}^{\bfitw_3^{-1}\bfitw_i k+k_{\bfitw_i}} \rightarrow \underline{\omega}_{n, r}^{\bfitw_3^{-1}\bfitw_i k+k_{\bfitw_i}} \xrightarrow{\Theta} \underline{\omega}_{n, r}^{s_{(1, -1)}\cdot (\bfitw_3^{-1}\bfitw_i k+k_{\bfitw_i})}
\]
of sheaves over $\calX_{n, \bfitw_i, (r, r)}^{\tor}$,
where $s_{(1,-1)} $ is the reflection associated with the (only) positive simple root $(1, -1)$ for $H$. The map $\Theta$ has the property that \[
    \Theta \bfitu_{p,1} = \left(-(1, -1)(\bfitw_3^{-1}\bfitw_i \bfitu_{p,1} \bfitw_i^{-1}\bfitw_3)\right)^{(\bfitw_i k)_1-(\bfitw_i k)_2 +1} \bfitu_{p,1} \Theta.
\]
Hence, by \eqref{eq: aut; change of ambient space} and using the condition (ii) of Definition \ref{Definition: small slope part}, there are quasi-isomorphisms \[
    \scalemath{0.9}{ R\Gamma_{\calZ_{n, \bfitw_i}}(\calX^{\tor, \bfitu_p}_{n,\bfitw_i}, \underline{\omega}^{\bfitw_3^{-1}\bfitw_i k+k_{\bfitw_i}})^{\sms} \cong R\Gamma_{\calZ_{n, \bfitw_i}}(\calX_{n, \bfitw_i, (r,r)}^{\tor}, \underline{\omega}^{\bfitw_3^{-1}\bfitw_i k+k_{\bfitw_i}})^{\sms} \cong  R\Gamma_{\calZ_{n, \bfitw_i}}(\calX_{n, \bfitw_i, (r,r)}^{\tor}, \underline{\omega}_{n, r}^{\bfitw_3^{-1}\bfitw_i k+k_{\bfitw_i}})^{\sms} }.
\]

Next, we consider the stratification \[
    \calX_n^{\tor} = \calX_{n, \leq \bfitw_3}^{\tor} \supset \overline{\calX_{n, \leq \bfitw_2}^{\tor}} \supset \overline{\calX_{n, \leq \bfitw_1}^{\tor}} \supset \overline{\calX_{n, \leq \one_4}^{\tor}} = \overline{\calX_{n, \one_4}^{\tor}} \supset \emptyset.
\]
By \cite[Theorem 5.4.12]{BP-HigherColeman}, we have a quasi-isomorphism 
\begin{equation}\label{eq: qi}
    R\Gamma_{\overline{\calX_{n, \leq \bfitw_i}^{\tor}}\smallsetminus \overline{\calX_{n, \leq \bfitw_{i-1}}^{\tor}}}(\calX_{n}^{\tor}\smallsetminus \overline{\calX_{n, \leq \bfitw_{i-1}}^{\tor}}, \underline{\omega}^{\bfitw_3^{-1}\bfitw_i k +k_{\bfitw_i}})^{\fs} \cong R\Gamma_{\calZ_{\bfitw_i, n}}(\calX^{\tor, \bfitu_p}_{n,\bfitw_i}, \underline{\omega}^{\bfitw_3^{-1}\bfitw_i k +k_{\bfitw_i}})^{\fs}.
\end{equation}
Hence, it remains to show that the small-slope part of the left-hand side of (\ref{eq: qi}) is quasi-isomorphic to the small-slope part of the classical complex $R\Gamma(\calX_{n}^{\tor}, \underline{\omega}^{\bfitw_3^{-1}\bfitw_i k +k_{\bfitw_i}})$.

The theory of cohomology with supports (\S \ref{section: cohomology with supports}) yields a diagram \begin{equation}\label{eq: diagram for coh. with supports for classical automorphic sheaf}
    \begin{tikzcd}[column sep = tiny]
        \scalemath{0.8}{ R\Gamma_{\overline{\calX_{n, \leq \bfitw_2}^{\tor}}}(\calX_{n}^{\tor}, \underline{\omega}^{\bfitw_3^{-1}\bfitw_i k +k_{\bfitw_i}})^{\fs} } \arrow[r] & \scalemath{0.8}{ R\Gamma(\calX_{n}^{\tor}, \underline{\omega}^{\bfitw_3^{-1}\bfitw_i k +k_{\bfitw_i}})^{\fs} }\arrow[r] & \scalemath{0.8}{ R\Gamma(\calX_{n}^{\tor} \smallsetminus \overline{\calX_{n, \leq \bfitw_2}^{\tor}}, \underline{\omega}^{\bfitw_3^{-1}\bfitw_i k +k_{\bfitw_i}})^{\fs} }\\
        \scalemath{0.8}{ R\Gamma_{\overline{\calX_{n, \leq \bfitw_1}^{\tor}}}(\calX_{n}^{\tor}, \underline{\omega}^{\bfitw_3^{-1}\bfitw_i k +k_{\bfitw_i}})^{\fs} } \arrow[r] & \scalemath{0.8}{ R\Gamma_{\overline{\calX_{n, \leq \bfitw_2}^{\tor}}}(\calX_{n}^{\tor}, \underline{\omega}^{\bfitw_3^{-1}\bfitw_i k +k_{\bfitw_i}})^{\fs} } \arrow[r] \arrow[ul, equal] & \scalemath{0.8}{ R\Gamma_{\overline{\calX_{n, \leq \bfitw_2}^{\tor}}\smallsetminus \overline{\calX_{n, \leq \bfitw_1}^{\tor}}}(\calX_{n}^{\tor} \smallsetminus \overline{\calX_{n, \leq \bfitw_1}^{\tor}}, \underline{\omega}^{\bfitw_3^{-1}\bfitw_i k +k_{\bfitw_i}})^{\fs} } \\
        \scalemath{0.8}{ R\Gamma_{\overline{\calX_{n, \leq \one_4}^{\tor}}}(\calX_{n}^{\tor}, \underline{\omega}^{\bfitw_3^{-1}\bfitw_i k +k_{\bfitw_i}})^{\fs} } \arrow[r] & \scalemath{0.8}{ R\Gamma_{\overline{\calX_{n, \leq \bfitw_1}^{\tor}}}(\calX_{n}^{\tor}, \underline{\omega}^{\bfitw_3^{-1}\bfitw_i k +k_{\bfitw_i}})^{\fs} } \arrow[r] \arrow[ul, equal] & \scalemath{0.8}{ R\Gamma_{\overline{\calX_{n, \leq \bfitw_1}^{\tor}}\smallsetminus \overline{\calX_{n, \leq \one_4}^{\tor}}}(\calX_{n}^{\tor} \smallsetminus \overline{\calX_{n, \leq \one_4}^{\tor}}, \underline{\omega}^{\bfitw_3^{-1}\bfitw_i k +k_{\bfitw_i}})^{\fs} }
    \end{tikzcd},
\end{equation}
where each row is a distinguished triangle.
We aim to show that, after taking the small-slope part, \begin{equation}\label{eq: vanishing of coherent cohomology at wrong locus}
    R\Gamma_{\overline{\calX_{n, \leq \bfitw_j}^{\tor}}\smallsetminus \overline{\calX_{n, \leq \bfitw_{j-1}}^{\tor}}}(\calX_{n}^{\tor} \smallsetminus \overline{\calX_{n, \leq \bfitw_{j-1}}^{\tor}}, \underline{\omega}^{\bfitw_3^{-1}\bfitw_i k +k_{\bfitw_i}})^{\sms} = 0
\end{equation}
for all $j\neq i$.

Note that, for $\bfitu\in \{\bfitu_{p,0}, \bfitu_{p,1}, \bfitu_p\}$, the associated na\"{i}ve operator $U^{\mathrm{naive}}$ can be defined in a similar way as in \eqref{eq: definition of Hecke operator away from p}, using the correspondence \[
    \begin{tikzcd}
        & \calX_{\Gamma\Iw_{\GSp_4, n}^+ \cap \bfitu \Gamma\Iw_{\GSp_4,n}^+ \bfitu^{-1}}^{\Sigma'', \tor}\arrow[ld, "\pr_2"']\arrow[rd, "\pr_1"]\\
        \calX_n^{\Sigma', \tor} && \calX_n^{\Sigma, \tor}
    \end{tikzcd}
\] for some admissible cone decomposition $\Sigma$, $\Sigma'$ and $\Sigma''$, together with an isomorphism \begin{equation}\label{eq: isomorphism of sheaves to define Up}
    \pr_2^*\underline{\omega}^{\bfitw_3^{-1}\bfitw_i k +k_{\bfitw_i}} \xrightarrow[]{\cong} \pr_1^* \underline{\omega}^{\bfitw_3^{-1}\bfitw_i k +k_{\bfitw_i}}
\end{equation}
established by Boxer--Pilloni. Recall the integral sheaves $\underline{\omega}^{\bfitw_3^{-1}\bfitw_i k +k_{\bfitw_i}, +}$ (Remark \ref{Remark: integral classical sheaf}). According to \cite[Lemma 5.9.9]{BP-HigherColeman}, the isomophism \eqref{eq: isomorphism of sheaves to define Up} induces a map \[
    \pr_2^* \underline{\omega}^{\bfitw_3^{-1}\bfitw_i k +k_{\bfitw_i}, +} \rightarrow p^{v_p(\bfitw_j^{-1}\bfitw_3(\bfitw_3^{-1}\bfitw_i k + k_{\bfitw_i})(\bfitu))} \pr_1^* \underline{\omega}^{\bfitw_3^{-1}\bfitw_i k +k_{\bfitw_i}, +}
\]
on $\pr_2^{-1}\calX_{n, \bfitw_j}^{\Sigma\, \tor} \cap \pr_1^{-1} \calX_{n, \bfitw_j}^{\Sigma, \tor}$. 

On the other hand, \cite[Lemma 5.9.10]{BP-HigherColeman} implies that there exists a quasicompact open $\calU \subset \calX_n^{\tor} \smallsetminus \overline{\calX_{n, \leq \bfitw_{j-1}}^{\tor}}$ and a closed $\calZ \subset \overline{\calX_{n, \leq \bfitw_j}^{\tor}} \smallsetminus \overline{\calX_{n, \leq \bfitw_{j-1}}^{\tor}}$ such that the image of 
\[H^s_{\calU \cap \calZ}(\calU, \underline{\omega}^{\bfitw_3^{-1}\bfitw_i k +k_{\bfitw_i}, +})\rightarrow H^s_{\overline{\calX_{n, \leq \bfitw_j}^{\tor}} \smallsetminus \overline{\calX_{n, \leq \bfitw_{j-1}}^{\tor}}}( \calX_n^{\tor} \smallsetminus \overline{\calX_{n, \leq \bfitw_{j-1}}^{\tor}}, \underline{\omega}^{\bfitw_3^{-1}\bfitw_i k +k_{\bfitw_i}})^{\fs}\] is an open bounded submodule in the target. Hence, $p^{-v_p(\bfitw_j^{-1}\bfitw_3(\bfitw_3^{-1}\bfitw_i k + k_{\bfitw_i})(\bfitu))} U^{\mathrm{naive}}$ preserves an open bounded submodule of $H^s_{\overline{\calX_{n, \leq \bfitw_j}^{\tor}} \smallsetminus \overline{\calX_{n, \leq \bfitw_{j-1}}^{\tor}}}( \calX_n^{\tor} \smallsetminus \overline{\calX_{n, \leq \bfitw_{j-1}}^{\tor}}, \underline{\omega}^{\bfitw_3^{-1}\bfitw_i k +k_{\bfitw_i}})^{\fs}$. Therefore, the slopes of $U^{\mathrm{naive}}$ occurring in $R\Gamma_{\overline{\calX_{n, \leq \bfitw_j}^{\tor}}\smallsetminus \overline{\calX_{n, \leq \bfitw_{j-1}}^{\tor}}}(\calX_{n}^{\tor} \smallsetminus \overline{\calX_{n, \leq \bfitw_{j-1}}^{\tor}}, \underline{\omega}^{\bfitw_3^{-1}\bfitw_i k +k_{\bfitw_i}})^{\fs}$ are larger than or equal to $ v_p\left(\bfitw_j^{-1}\bfitw_3(\bfitw_3^{-1}\bfitw_i k + k_{\bfitw_i})(\bfitu)\right)$. It then follows from the definition of the small-slope part that \[
    R\Gamma_{\overline{\calX_{n, \leq \bfitw_j}^{\tor}}\smallsetminus \overline{\calX_{n, \leq \bfitw_{j-1}}^{\tor}}}(\calX_{n}^{\tor} \smallsetminus \overline{\calX_{n, \leq \bfitw_{j-1}}^{\tor}}, \underline{\omega}^{\bfitw_3^{-1}\bfitw_i k +k_{\bfitw_i}})^{\sms} = 0
\]
for $j\neq i$ as desired in \eqref{eq: vanishing of coherent cohomology at wrong locus}.

Finally, together with \eqref{eq: diagram for coh. with supports for classical automorphic sheaf}, we see that the natural maps \[
    \scalemath{0.8}{ R\Gamma(\calX_{n}^{\tor}, \underline{\omega}^{\bfitw_3^{-1}\bfitw_i k +k_{\bfitw_i}})^{\sms} \leftarrow R\Gamma_{\overline{\calX_{n, \leq \bfitw_i}^{\tor}}}(\calX_{n}^{\tor}, \underline{\omega}^{\bfitw_3^{-1}\bfitw_i k +k_{\bfitw_i}})^{\sms} \rightarrow R\Gamma_{\overline{\calX_{n, \leq \bfitw_i}^{\tor}}\smallsetminus \overline{\calX_{n, \leq \bfitw_{i-1}}^{\tor}}}(\calX_{n}^{\tor} \smallsetminus \overline{\calX_{n, \leq \bfitw_{i-1}}^{\tor}}, \underline{\omega}^{\bfitw_3^{-1}\bfitw_i k +k_{\bfitw_i}})^{\sms} }
\]
are quasi-isomorphisms.
\end{proof}

\subsection{Pro-Kummer \'etale cohomology groups of classical automorphic sheaves}\label{subsection: proket coh with support for classical aut sheaves}

In later sections, we shall encounter certain pro-Kummer \'etale variants of the cohomology groups studied in \S \ref{subsection: Hecke operators}. These pro-Kummer \'etale cohomology groups (with or without supports) play a central role in the construction of overconvergent Eichler--Shimura morphisms. The main purpose of \S \ref{subsection: proket coh with support for classical aut sheaves} is to analyse the pro-Kummer \'etale groups of (completed) classical automorphic sheaves. In particular, we prove an analogue of the classicality theorem (Theorem \ref{Theorem: classicality theorem in higher Coleman theory}) for such cohomology groups.

Let $k=(k_1 ,k_2)\in \Z^2$ be an integral weight with $k_1\geq k_2$. Recall the integral subsheaf $\underline{\omega}^{k, +}\subset \underline{\omega}^{k}$ defined in Remark \ref{Remark: integral classical sheaf}, and consider the completed pullbacks of $\underline{\omega}^{k}$ and $\underline{\omega}^{k,+}$ to the pro-Kummer \'etale site $\calX_{n, \proket}^{\tor}$; namely, consider
\[
    \widehat{\underline{\omega}}^{k} := \upsilon^{-1}\underline{\omega}^{k}\otimes_{\upsilon^{-1}\scrO_{\calX_{n}^{\tor}}} \widehat{\scrO}_{\calX_{n, \proket}^{\tor}}
\]
and
\[  
    \widehat{\underline{\omega}}^{k,+} := \upsilon^{-1}\underline{\omega}^{k,+}\otimes_{\upsilon^{-1}\scrO_{\calX_{n}^{\tor}}^+} \widehat{\scrO}_{\calX_{n, \proket}^{\tor}}^+
\]
where $\upsilon: \calX_{n, \proket}^{\tor} \rightarrow \calX_{n, \an}^{\tor}$ is the natural morphism of sites. 

We consider the pro-Kummer \'etale cohomology (with or without supports) of these completed automorphic sheaves; for example, $R\Gamma_{\proket}(\calX_n^{\tor}, \widehat{\underline{\omega}}^{\bfitw_3^{-1}\bfitw_i k})$, $R\Gamma_{\overline{\calX_{n, \leq \bfitw_i}^{\tor}},\proket}(\calX_n^{\tor}, \widehat{\underline{\omega}}^{\bfitw_3^{-1}\bfitw_i k})$, and $R\Gamma_{\overline{\calX_{n, \leq \bfitw_i}^{\tor}}\smallsetminus \overline{\calX_{n, \leq \bfitw_{i-1}}^{\tor}},\proket}(\calX_n^{\tor} \smallsetminus \overline{\calX_{n, \leq \bfitw_{i-1}}^{\tor}}, \widehat{\underline{\omega}}^{\bfitw_3^{-1}\bfitw_i k})$, for $i=0,1,2,3$, as well as their cohomology groups. Just for technical purpose, we define the `small-slope part' of these complexes/cohomology groups in the same way as in Definition \ref{Definition: small slope part}, except that instead of using $U\in \{U_{p,0}, U_{p, 1}, U_p\}$, we use $U^{\mathrm{naive}} \in \{U_{p,0}^{\mathrm{naive}}, U_{p,1}^{\mathrm{naive}}, U_{p}^{\mathrm{naive}}\}$. These small-slope parts are denoted by $R\Gamma_{\proket}(\calX_n^{\tor}, \widehat{\underline{\omega}}^{\bfitw_3^{-1}\bfitw_i k})^{\sms}$, $R\Gamma_{\overline{\calX_{n, \leq \bfitw_i}^{\tor}},\proket}(\calX_n^{\tor}, \widehat{\underline{\omega}}^{\bfitw_3^{-1}\bfitw_i k})^{\sms}$, and $R\Gamma_{\overline{\calX_{n, \leq \bfitw_i}^{\tor}}\smallsetminus \overline{\calX_{n, \leq \bfitw_{i-1}}^{\tor}},\proket}(\calX_n^{\tor} \smallsetminus \overline{\calX_{n, \leq \bfitw_{i-1}}^{\tor}}, \widehat{\underline{\omega}}^{\bfitw_3^{-1}\bfitw_i k})^{\sms}$. 

The following result is a pro-Kummer \'etale analogue of Theorem \ref{Theorem: classicality theorem in higher Coleman theory}.

\begin{Proposition}\label{Proposition: quasi-isomorphisms of pro-Kummer \'etale cohomology for classical automorphic sheaves}
   Let $k=(k_1 ,k_2)\in \Z^2$ be an integral weight with $k_1\geq k_2$. For $i=0,1,2,3$, the natural morphisms \[
        \scalemath{0.8}{ R\Gamma_{\proket}(\calX_n^{\tor}, \widehat{\underline{\omega}}^{\bfitw_3^{-1}\bfitw_i k})^{\sms} \leftarrow R\Gamma_{\overline{\calX_{n, \leq \bfitw_i}^{\tor}},\proket}(\calX_n^{\tor}, \widehat{\underline{\omega}}^{\bfitw_3^{-1}\bfitw_i k})^{\sms} \rightarrow R\Gamma_{\overline{\calX_{n, \leq \bfitw_i}^{\tor}}\smallsetminus \overline{\calX_{n, \leq \bfitw_{i-1}}^{\tor}},\proket}(\calX_n^{\tor} \smallsetminus \overline{\calX_{n, \leq \bfitw_{i-1}}^{\tor}}, \widehat{\underline{\omega}}^{\bfitw_3^{-1}\bfitw_i k})^{\sms} }
    \]
    are quasi-isomorphisms. 
\end{Proposition}
\begin{proof}
    The proof is similar to the one of Theorem \ref{Theorem: classicality theorem in higher Coleman theory}.

    First of all, for $\bfitu\in \{\bfitu_{p,0}, \bfitu_{p,1}, \bfitu\}$, recall the correspondence 
    \begin{equation}\label{eq: diagram correspeondence}
    \begin{tikzcd}
        & \calX_{\Gamma\Iw_{\GSp_4, n}^+ \cap \bfitu \Gamma\Iw_{\GSp_4,n}^+ \bfitu^{-1}}^{\Sigma'', \tor}\arrow[ld, "\pr_2"']\arrow[rd, "\pr_1"]\\
        \calX_n^{\Sigma', \tor} && \calX_n^{\Sigma, \tor}
    \end{tikzcd}
    \end{equation}
    that defines the operator $U\in \{U_{p,0}, U_{p,1}, U_p\}$. By \cite[Lemma 5.9.9]{BP-HigherColeman}, the isomorphism \[
        \pr_2^* \widehat{\underline{\omega}}^{\bfitw_3^{-1}\bfitw_i k} \xrightarrow[]{\cong} \pr_1^* \widehat{\underline{\omega}}^{\bfitw_3^{-1}\bfitw_i k}
    \]
    induces a map \[
        \pr_2^* \widehat{\underline{\omega}}^{\bfitw_3^{-1}\bfitw_i k, +} \rightarrow p^{v_p(\bfitw_j^{-1}\bfitw_i k(\bfitu_p))}\pr_1^* \widehat{\underline{\omega}}^{\bfitw_3^{-1}\bfitw_i k, +}
    \]
    over  $\left(\pr_2^{-1}\calX_{n, \bfitw_j}^{\Sigma\, \tor} \cap  \pr_1^{-1} \calX_{n, \bfitw_j}^{\Sigma, \tor}\right)_{\proket}$, for any $j\neq i$.

Now, we claim that the endomorphism $p^{-v_p(\bfitw_j^{-1}\bfitw_i k(\bfitu))} U^{\mathrm{naive}}$ on the pro-Kummer \'etale cohomology group $H_{\overline{\calX_{n, \leq \bfitw_j}^{\tor}}\smallsetminus \overline{\calX_{n, \leq \bfitw_{j-1}}^{\tor}}, \proket}^{t}(\calX_n^{\tor} \smallsetminus \overline{\calX_{n, \leq \bfitw_{j-1}}^{\tor}}, \widehat{\underline{\omega}}^{\bfitw_3^{-1}\bfitw_i k})^{\fs}$ preserves an open and bounded submodule, for every $t$.

To this end, we first observe that $\calX_n^{\tor}$ is the base extension of the (toroidally compactified) Siegel modular variety $\calX_{n,K}^{\tor}$ defined over some finite extension $K$ of $\Q_p$. Similarly, the loci $\calX_{n, \leq \bfitw}^{\tor}$ (resp., $\calX_{n, \leq \bfitw, (m, n)}^{\tor}$) is the base change of the loci $\calX_{n, K, \leq \bfitw}^{\tor}$ (resp., $\calX_{n, K, \leq \bfitw, (m, n)}^{\tor}$) defined over $K$. The sheaves $\underline{\omega}^{\bfitw_3^{-1}\bfitw k}$ also descends to $K$. Choose a quasi-compact open $\calU_K \subset \calX_{n,K}^{\tor} \smallsetminus \overline{\calX_{n, K,\leq \bfitw_{j-1}}^{\tor}}$ and a closed subset $\calZ_K \subset \overline{\calX_{n,K, \leq \bfitw_j}^{\tor}}\smallsetminus \overline{\calX_{n,K, \leq \bfitw_{j-1}}^{\tor}}$ (as in \cite[Lemma 5.9.10]{BP-HigherColeman}) such that there exists $m, n, s\in \Z_{\geq 0}$ and \begin{align*}
    \calX_{n,K, \bfitw_j, (0,\overline{m+s})}^{\tor} \cap \calX_{n,K, \bfitw_j, (n, \overline{0})}^{\tor} & \subset \calZ_K \cap \calU_K \subset \calX_{n, K,\bfitw_j, (0, \overline{3m})}^{\tor}, \\
    \calX_{n,K, \bfitw_j, (0, \overline{m})}^{\tor} \cap \calX_{n,K, \bfitw_j, (n+s, \overline{0})}^{\tor} & \subset \calU_K \subset \calX_{n, K,\bfitw_j, (3n, -1)}^{\tor}.
\end{align*}
Let $\calU$ and $\calZ$ denote the base change of $\calU_K$ and $\calZ_K$ to $\C_p$, respectively.

Choose a finite open covering $\frakU_{1,K}=\{\calU_s\}_{s\in I}$ of $\calU_K$ by affinoid open subsets $\calU_s$ on which the vector bundle $\underline{\omega}^{\bfitw_3^{-1}\bfitw_i k}$ is trivialized. For each $\calU_s$, let \[\calU_{s,\infty}:=\calU_s\times_{\calX_{n,K}^{\tor}}\calX_{\Gamma(p^{\infty})}^{\tor}.\] 
Then each $\calU_{s,\infty}$ is a log affinoid perfectoid object over $\calX_n^{\tor}$ and $\frakU_{1}:=\{\calU_{s,\infty}\}_{s\in I}$ is a pro-Kummer \'etale covering of $\calU$. By construction, the covering $\frakU_1$ of $\calU$ is, in fact, a pro-Kummer \'etale atlas of $\underline{\widehat{\omega}}^{\bfitw_3^{-1}\bfitw_i k}$ in the sense of Definition \ref{Definition: proket Banach sheaves}. Consequently, the {\v C}ech complex $\check{C}^{\bullet}(\frakU_1, \underline{\widehat{\omega}}^{\bfitw_3^{-1}\bfitw_i k})$ computes $R\Gamma_{\proket}(\calU, \underline{\widehat{\omega}}^{\bfitw_3^{-1}\bfitw_i k})$ by Lemma \ref{Lemma: proket coh can be computed by Cech cohomology}. Choose another finite open covering $\frakU_{2,K}$ of $\calU_K \smallsetminus (\calU_K \cap \calZ_K)$, refining $\frakU_{1,K}$. This induces a pro-Kummer \'etale covering $\frakU_2$ of $\calU \smallsetminus (\calU \cap \calZ)$ by the same recipe. Likewise, the {\v C}ech complex $\check{C}^{\bullet}(\frakU_2, \underline{\widehat{\omega}}^{\bfitw_3^{-1}\bfitw_i k})$ computes $R\Gamma_{\proket}(\calU \smallsetminus (\calU \cap \calZ), \underline{\widehat{\omega}}^{\bfitw_3^{-1}\bfitw_i k})$. 

It follows from the construction that $R\Gamma_{\calU\cap \calZ, \proket}(\calU, \underline{\widehat{\omega}}^{\bfitw_3^{-1}\bfitw_i k})$ is represented by the mapping cone \[
    C^{\bullet}  \coloneq \mathrm{Cone}\left( \check{C}^{\bullet}(\frakU_1, \underline{\widehat{\omega}}^{\bfitw_3^{-1}\bfitw_i k}) \rightarrow \check{C}^{\bullet}(\frakU_2, \underline{\widehat{\omega}}^{\bfitw_3^{-1}\bfitw_i k})  \right)[-1].
\]
By Proposition \ref{Proposition: proket cohomology with supp for Banach sheaves}, we know that \[
    C^{\bullet} \in \mathrm{Pro}_{\Z_{\geq 0}}(\mathrm{K}^{\proj}(\C_p)).
\]
Moreover, the discussions in \cite[\S 5.3]{BP-HigherColeman} (see also \cite[Corollary 5.3.8]{BP-HigherColeman}) implies that $U_p$ is well-defined on $C^{\bullet}$ and is potent compact. In particular, we can consider $C^{\bullet, \fs}$ and $C^{\bullet, \leq h}$. The same proof as in \cite[Theorem 5.4.12]{BP-HigherColeman} yields a quasi-isomorphism \[
    C^{\bullet, \fs} \cong R\Gamma_{\overline{\calX_{n, \leq \bfitw_j}^{\tor}}\smallsetminus \overline{\calX_{n, \leq \bfitw_{j-1}}^{\tor}}, \proket}(\calX_n^{\tor} \smallsetminus \overline{\calX_{n, \leq \bfitw_{j-1}}^{\tor}}, \widehat{\underline{\omega}}^{\bfitw_3^{-1}\bfitw_i k})^{\fs}.
\]
To prove the claim, we consider an integral version of $C^{\bullet}$ given by
 \[
    C^{\bullet, +}  \coloneq \mathrm{Cone}\left( \check{C}^{\bullet}(\frakU_1, \underline{\widehat{\omega}}^{\bfitw_3^{-1}\bfitw_i k,+}) \rightarrow \check{C}^{\bullet}(\frakU_2, \underline{\widehat{\omega}}^{\bfitw_3^{-1}\bfitw_i k,+})  \right)[-1]
\]
which is a subcomplex of open and bounded submodules of $C^{\bullet}$, and consider
\[
    C^{\bullet, +, \leq h} := \image\left( C^{\bullet, +} \rightarrow C^{\bullet} \rightarrow C^{\bullet, \leq h} \right).
\]
Notice that the complexes $C^{\bullet, \leq h}$ is the base changes of a perfect complexes over $K$, and that $C^{\bullet, +, \leq h}\subset C^{\bullet, \leq h}$ is the base change of an open and bounded subcomplex of $\calO_K$-submodules. It follows that $C^{\bullet, +, \leq h}$ is a perfect complex over $\calO_{\C_p}$. Therefore, \[
    H^t(C^{\bullet, +})^{\leq h} \coloneq \image\left( H^t(C^{\bullet, +}) \rightarrow H^t(C^{\bullet})^{\leq h} \right)
\]
is an open bounded submodule. Passing to the limit over $h$, we see that \[
    H^t(C^{\bullet, +})^{\fs} 
    \coloneq \image\left(H^t(C^{\bullet, +}) \rightarrow H^t(C^{\bullet})^{\fs}\right)
\]
is an open bounded submodule. 

To prove the claim, it suffices to show the map \[
    H^t(C^{\bullet, +}) \rightarrow H_{\calU \cap \calZ, \proket}^t(\calU, \widehat{\underline{\omega}}^{\bfitw_3^{-1}\bfitw_i k,+}),
\]
induced by the natural map \[
    C^{\bullet, +} \rightarrow R\Gamma_{\calU \cap \calZ, \proket}(\calU, \widehat{\underline{\omega}}^{\bfitw_3^{-1}\bfitw_i k,+}),
\]
has kernels and cokernels of bounded torsion. However, by using the {\v Cech}-to-cohomology spectral sequence (\cite[\href{https://stacks.math.columbia.edu/tag/03OW}{Tag 03OW}]{stacks-project}), this follows from Lemma \ref{Lemma: higher cohomology of integral structure has bounded torsion}.

Consequently, the slope of $U^{\mathrm{naive}}$ occurring in $R\Gamma_{\overline{\calX_{n, \leq \bfitw_j}^{\tor}}\smallsetminus \overline{\calX_{n, \leq \bfitw_{j-1}}^{\tor}}, \proket}(\calX_n^{\tor} \smallsetminus \overline{\calX_{n, \leq \bfitw_{j-1}}^{\tor}}, \widehat{\underline{\omega}}^{\bfitw_3^{-1}\bfitw_i k})^{\fs}$ should be larger than or equal to $v_p(\bfitw_j^{-1}\bfitw_i k(\bfitu))$. The theory of cohomology with supports (\S \ref{section: cohomology with supports}) yields a diagram \[
        \begin{tikzcd}[column sep = tiny]
        \scalemath{0.8}{ R\Gamma_{\overline{\calX_{n, \leq \bfitw_2}^{\tor}}, \proket}(\calX_{n}^{\tor}, \widehat{\underline{\omega}}^{\bfitw_3^{-1}\bfitw_i k})^{\fs} } \arrow[r] & \scalemath{0.8}{ R\Gamma_{\proket}(\calX_{n}^{\tor}, \widehat{\underline{\omega}}^{\bfitw_3^{-1}\bfitw_i k})^{\fs} }\arrow[r] & \scalemath{0.8}{ R\Gamma_{\proket}(\calX_{n}^{\tor} \smallsetminus \overline{\calX_{n, \leq \bfitw_2}^{\tor}}, \widehat{\underline{\omega}}^{\bfitw_3^{-1}\bfitw_i k})^{\fs} }\\
        \scalemath{0.8}{ R\Gamma_{\overline{\calX_{n, \leq \bfitw_1}^{\tor}}, \proket}(\calX_{n}^{\tor}, \widehat{\underline{\omega}}^{\bfitw_3^{-1}\bfitw_i k})^{\fs} } \arrow[r] & \scalemath{0.8}{ R\Gamma_{\overline{\calX_{n, \leq \bfitw_2}^{\tor}}, \proket}(\calX_{n}^{\tor}, \widehat{\underline{\omega}}^{\bfitw_3^{-1}\bfitw_i k})^{\fs} } \arrow[r] \arrow[ul, equal] & \scalemath{0.8}{ R\Gamma_{\overline{\calX_{n, \leq \bfitw_2}^{\tor}}\smallsetminus \overline{\calX_{n, \leq \bfitw_1}^{\tor}}, \proket}(\calX_{n}^{\tor} \smallsetminus \overline{\calX_{n, \leq \bfitw_1}^{\tor}}, \widehat{\underline{\omega}}^{\bfitw_3^{-1}\bfitw_i k})^{\fs} } \\
        \scalemath{0.8}{ R\Gamma_{\overline{\calX_{n, \leq \one_4}^{\tor}}, \proket}(\calX_{n}^{\tor}, \widehat{\underline{\omega}}^{\bfitw_3^{-1}\bfitw_i k})^{\fs} } \arrow[r] & \scalemath{0.8}{ R\Gamma_{\overline{\calX_{n, \leq \bfitw_1}^{\tor}}, \proket}(\calX_{n}^{\tor}, \widehat{\underline{\omega}}^{\bfitw_3^{-1}\bfitw_i k})^{\fs} } \arrow[r] \arrow[ul, equal] & \scalemath{0.8}{ R\Gamma_{\overline{\calX_{n, \leq \bfitw_1}^{\tor}}\smallsetminus \overline{\calX_{n, \leq \one_4}^{\tor}}, \proket}(\calX_{n}^{\tor} \smallsetminus \overline{\calX_{n, \leq \one_4}^{\tor}}, \widehat{\underline{\omega}}^{\bfitw_3^{-1}\bfitw_i k})^{\fs} }
    \end{tikzcd},
    \]
    where each row is a distinguished triangle. Passing to the small-slope part (with respect to the na\"{i}ve Hecke operators), one sees that \[
        R\Gamma_{\overline{\calX_{n, \leq \bfitw_j}^{\tor}}\smallsetminus \overline{\calX_{n, \leq \bfitw_{j-1}}^{\tor}}, \proket}(\calX_{n}^{\tor} \smallsetminus \overline{\calX_{n, \leq \bfitw_1}^{\tor}}, \widehat{\underline{\omega}}^{\bfitw_3^{-1}\bfitw_i k})^{\sms} = 0        
    \]
    whenever $j\neq i$. The desired result follows. 
\end{proof}
\section{Overconvergent cohomology groups for \texorpdfstring{$\GSp_4$}{GSp4}}\label{section: OC}

In this section, we introduce the so-called \emph{overconvergent cohomology groups} which are designed to $p$-adically interpolate the \'etale cohomology groups in the Eichler--Shimura decompostion (cf. Theorem \ref{thm: Faltings-Chai}). In \S \ref{subsection: Betti overconvergent cohomology}, we recall the original definitions of Hansen following \cite{Hansen-PhD}. For our purpose, we re-interpret these notions in terms of Kummer and pro-Kummer \'etale cohomology groups of sheaves $\scrO\!\!\scrD_{\kappa_{\calU}}^{r}$, as we will explain in \S \ref{subsection: Kummer and pro-Kummer \'etale overconvergent cohomology}. Here we follow the ideas from \cite{Hansen-Iwasawa} and \cite{CHJ-2017}. Readers are also encouraged to consult \cite{DRW}. In \S \ref{subsection: alternative construction on flag variety}, we present an alternative construction of the sheaves $\scrO\!\!\scrD_{\kappa_{\calU}}^{r}$ on the flag variety. This will be used in the construction of Eichler--Shimura morphisms in \S \ref{subsection: OES}. Finally, in \S \ref{subsection: pro-Kummer \'etale overconvergent cohomology with supports}, we introduce certain variants of such overconvergent cohomology groups, in terms of pro-Kummer \'etale cohomology with supports. Such variants are indispensible if one wants to $p$-adically interpolate the \emph{entire} Eichler--Shimura decomposition; in fact, they already appear in the statement of Theorem \ref{thm: main thm intro}.

\subsection{Betti cohomology groups}\label{subsection: Betti overconvergent cohomology}
Let $(R_{\calU}, \kappa_{\calU})$ be a weight and let $r\in \Q_{\geq 0}$ with $r>1+r_{\calU}$. Recall the spaces of analytic functions (cf. Remark \ref{Remark: analytic representation for other groups}) \[
    A_{\kappa_{\calU}}^{r, \circ}(\Iw_{\GSp_4, 1}^+, R_{\calU}), \quad A_{\kappa_{\calU}}^{r^+, \circ}(\Iw_{\GSp_4, 1}^+, R_{\calU}), \quad A_{\kappa_{\calU}}^r(\Iw_{\GSp_{4}, 1}^+, R_{\calU}), \quad \text{ and }\quad A_{\kappa_{\calU}}^{r^+}(\Iw_{\GSp_4, 1}^+, R_{\calU}). 
\]
To simplify the notations, we drop the `$(\Iw_{\GSp_4, 1}^+, R_{\calU})$' in the notations when everything is clear in the context. 

We equip with these spaces the following two $\Iw_{\GSp_{4}, 1}^+$-actions: \begin{enumerate}
    \item[(i)] The right $\Iw_{\GSp_{4}, 1}^+$-action by the left translation, \emph{i.e.}, \[
        (f\cdot \bfgamma)(\bfalpha) = f(\bfgamma \bfalpha) 
    \] for $\bfgamma, \bfalpha\in \Iw_{\GSp_4, 1}^+$. 
    \item[(ii)] The left $\Iw_{\GSp_4, 1}^+$-action by left translation of the transpose, \emph{i.e.}, \[
        (\bfgamma \cdot f)(\bfalpha) = f(\trans\bfgamma \bfalpha)
    \] for $\bfgamma, \bfalpha\in \Iw_{\GSp_4, 1}^+$.
\end{enumerate}
Taking duals, we obtain the corresponding spaces of distributions:\[
    \begin{array}{rlrl}
        D_{\kappa_{\calU}}^{r, \circ} &\coloneq \Hom_{R_{\calU}^{\circ}}^{\cts}(A_{\kappa_{\calU}}^{r, \circ}, R_{\calU}^{\circ}), & D_{\kappa_{\calU}}^{r^+, \circ} &\coloneq \Hom_{R_{\calU}^{\circ}}^{\cts}(A_{\kappa_{\calU}}^{r^+, \circ}, R_{\calU}^{\circ}),  \\
        D_{\kappa_{\calU}}^{r} & \coloneq  D_{\kappa_{\calU}}^{r, \circ}\Big[\frac{1}{p}\Big], &  D_{\kappa_{\calU}}^{r^+} & \coloneq  D_{\kappa_{\calU}}^{r^+, \circ}\Big[\frac{1}{p}\Big].
    \end{array}
\]
The right $\Iw_{\GSp_4, 1}^+$-actions on $A_{\kappa_{\calU}}^{r, \circ}$, $A_{\kappa_{\calU}}^{r}$, $A_{\kappa_{\calU}}^{r^+, \circ}$, and $A_{\kappa_{\calU}}^{r^+}$ induce left $\Iw_{\GSp_4, 1}^+$-actions on $D_{\kappa_{\calU}}^{r, \circ}$, $D_{\kappa_{\calU}}^{r}$, $D_{\kappa_{\calU}}^{r^+, \circ}$, and $D_{\kappa_{\calU}}^{r^+}$, respectively. 

Before we proceed, we fix an isomorphism \[
        N_{\GSp_4, 1}^{\opp} \cong \Z_p^4.
    \] 
of $p$-adic manifolds which is compatible with \eqref{eq: unipotent radical for H as a p-adic manifold}. Also recall the vectors
\[
    e_i^{(r)} \colon \Z_p^4 \rightarrow \Z_p, \quad (x_1, ..., x_n)\mapsto \prod_{j=1}^n\lfloor p^{-r}j \rfloor! \begin{pmatrix} x_j\\ i_j\end{pmatrix}.
\]
in $C^r(\Z_p^4, \Z_p)$ introduced in \eqref{eq: basis for r-analytic functions}, where $i\in \Z_{\geq 0}^4$. Let $e_i^{(r), \vee}$ denote the dual vectors. The following result is straightforward.

\begin{Proposition}\label{Proposition: structure theorem for distributions}
    Let $(R_{\calU}, \kappa_{\calU})$ be a weight. Then we have
        \[A_{\kappa_{\calU}}^{r,\circ} \cong \widehat{\bigoplus}_{i\in \Z_{\geq 0}^4}R_{\calU}^{\circ}e_i^{(r)}\]
        and hence
        \[D_{\kappa_{\calU}}^{r, \circ} \cong \prod_{i\in \Z_{\geq 0}^4}R_{\calU}^{\circ} e_{i}^{(r), \vee}.\] 
        Similarly, 
        \[A_{\kappa_{\calU}}^{r^+, \circ} \cong \prod_{i\in \Z_{\geq 0}^4}R_{\calU}^{\circ} e_{i}^{(r)}\] 
        and hence
         \[D_{\kappa_{\calU}}^{r^+,\circ} \cong \widehat{\bigoplus}_{i\in \Z_{\geq 0}^4}R_{\calU}^{\circ}e_i^{(r), \vee}.\]
         We obtain similar descriptions for $A_{\kappa_{\calU}}^{r}$, $D_{\kappa_{\calU}}^{r}$, $A_{\kappa_{\calU}}^{r^+}$, and $D_{\kappa_{\calU}}^{r^+}$ after inverting $p$.
\end{Proposition}
\begin{proof}
    This follows immediately from 
    \[
    C^r(\Z_p^n, \Z_p) \cong \widehat{\bigoplus}_{i\in \Z_{\geq 0}^n}\Z_pe_i^{(r)} \quad \text{ and }\quad C^{r^+}(\Z_p^n, \Z_p)\cong \prod_{i\in \Z_{\geq 0}^n}\Z_pe_i^{(r)}.
    \]
\end{proof}

Let $M \in \{A_{\kappa_{\calU}}^{r, \circ},A_{\kappa_{\calU}}^{r^+, \circ}, A_{\kappa_{\calU}}^{r}, A_{\kappa_{\calU}}^{r^+}, D_{\kappa_{\calU}}^{r, \circ}, D_{\kappa_{\calU}}^{r^+, \circ}, D_{\kappa_{\calU}}^{r}, D_{\kappa_{\calU}}^{r^+} \}$. Since $M$ admits a left $\Iw_{\GSp_4, 1}^+$-action (and so a left $\Iw_{\GSp_4, n}^+$-action for any $n\in \Z_{>0}$), it defines a local system on $X_{n}(\C)$ (see, for example, \cite[\S 2.2]{Ash-Stevens}). Consequently, one can consider cohomology groups $H^i(X_n(\C), M)$ of $X_n(\C)$ with coefficients in $M$. By the discussion in \cite[\S 2.2]{Hansen-PhD}, we know that these cohomology groups can be computed via the \emph{augmented Borel--Serre cochain complex} $C^{\bullet}(\Iw_{\GSp_4, n}^+, M)$. For the reader's convenience, we briefly recall the definition. Let $X_{n}^{\BS}(\C)$ be the Borel--Serre compactification of $X_n(\C)$ and fix a finite triangulation on $X_n^{\BS}(\C)$. Then the augmented Borel--Serre cochain complex $C^{\bullet}(\Iw_{\GSp_4, n}^+, M)$ is defined to be the cochain complex associated with this simplicial decomposition with coefficients in $M$. 

\begin{Remark}\label{Remark: ON-able}
    Suppose $(R_{\calU}, \kappa_{\calU})$ is an affinoid weight and $M\in \{A_{\kappa_{\calU}}^r, D_{\kappa_{\calU}}^{r^+}\}$. By Proposition \ref{Proposition: structure theorem for distributions}, we have an identification \[
        M \cong \widehat{\bigoplus}_{i\in \Z_{\geq 0}^4} R_{\calU}.
    \] 
    Since $C^{\bullet}(\Iw_{\GSp_4, n}^+, M)$ is a finite cochain complex, the total space \[
        C_{\kappa_{\calU}}^{\mathrm{tot}}(\Iw_{\GSp_4, n}^+, M) \coloneq \bigoplus_{j} C^j(\Iw_{\GSp_4, n}^+, M)
    \] is a potentially ON-able Banach module over $R_{\calU}$ (\cite[pp. 70]{Buzzard_2007}).
\end{Remark}

For $M \in \{A_{\kappa_{\calU}}^{r, \circ},A_{\kappa_{\calU}}^{r^+, \circ}, A_{\kappa_{\calU}}^{r}, A_{\kappa_{\calU}}^{r^+}, D_{\kappa_{\calU}}^{r, \circ}, D_{\kappa_{\calU}}^{r^+, \circ}, D_{\kappa_{\calU}}^{r}, D_{\kappa_{\calU}}^{r^+} \}$, we now define the Hecke operators on $H^i(X_n(\C), M)$. Similar to \S \ref{subsection: Hecke operators}, we treat the two cases separately: the Hecke operators away from $p$ and the Hecke operators at $p$. 

Let $\ell \neq p$ be a prime number. 
For any $\bfdelta\in \GSp_4(\Q_{\ell})$, consider the diagram \[
    \begin{tikzcd}
        & X_{\Gamma\Iw_{\GSp_4, n}^+ \cap \bfdelta \Gamma\Iw_{\GSp_4, n}^+ \bfdelta^{-1}} \arrow[ld, "\pr_2"'] & X_{\bfdelta^{-1}\Gamma\Iw_{\GSp_4, n}^+\bfdelta \cap \Gamma\Iw_{\GSp_4, n}^+}\arrow[rd, "\pr_1"]\arrow[l, "\bfdelta"']\\
        X_n &&& X_n
    \end{tikzcd},
\] where the top arrow is an isomorphism. By applying \cite[\S A.2]{LW-BianchiAdjoint}, we obtain the Hecke-operator $T_{\bfdelta}$ as the composition \[
    \begin{tikzcd}
        H^i(X_n(\C), M) \arrow[r, "\pr_2^{-1}"] \arrow[rddd, bend right = 20, "T_{\bfdelta}"'] & H^i(X_{\Gamma\Iw_{\GSp_4, n}^+ \cap \bfdelta \Gamma\Iw_{\GSp_4, n}^+ \bfdelta^{-1}}(\C), M) \arrow[d, "\bfdelta^{-1}"]\\ & H^i(X_{\bfdelta^{-1}\Gamma\Iw_{\GSp_4, n}^+\bfdelta \cap \Gamma\Iw_{\GSp_4, n}^+}(\C), M) \arrow[d, "\cong"]\\
        & H^i(X_n(\C), \pr_{1, *}\pr_1^{-1}M)\arrow[d, "\mathrm{tr}"]\\
        & H^i(X_n(\C), M)
    \end{tikzcd}.
\]

We now discuss the Hecke operators at $p$. Recall the matrices \[
    \bfitu_{p, 0} \coloneq \begin{pmatrix}1 &&& \\ & 1 && \\ && p & \\ &&& p\end{pmatrix}, \quad \bfitu_{p, 1} \coloneq \begin{pmatrix}1 &&& \\ & p && \\ && p & \\ &&& p^2\end{pmatrix}, \quad \text{ and }\quad \bfitu_p \coloneq \bfitu_{p,0}\bfitu_{p, 1} = \begin{pmatrix} 1 &&& \\ & p && \\ && p^2 & \\ &&& p^3\end{pmatrix}.
\]
Although one may also define the action of Hecke operators at $p$ on $H^i(X_n(\C), M)$ via correspondences, it would be more convenient for us to define them via explicit formulae. For $\bfitu\in \{\bfitu_{p, 0}, \bfitu_{p, 1}, \bfitu_p\}$, observe that \[
    \bfitu N_{\GSp_4, n}^{\opp} \bfitu^{-1} \subset N_{\GSp_4, n}^{\opp}. 
\] We then define the operator $\bfitu$ on $A_{\kappa_{\calU}}^{r, \circ}$ via \[
    (\bfitu \cdot f)(\bfnu \bftau \bfepsilon) = f(\bfitu \bfnu \bfitu^{-1}\bftau \bfepsilon)\] for all $f\in A_{\kappa_{\calU}}^{r, \circ}$, $\bfnu\in N_{\GSp_4, n}^{\opp}$, $\bftau\in T_{\GSp_4}(\Z_p)$, and $\bfepsilon\in N_{\GSp_4, n}$. This induces an operator on $M$.

\begin{Remark}
    When $\bfitu = \bfitu_p$ and $(R_{\calU}, \kappa_{\calU})$ is an affinoid weight, the operator $\bfitu$ defines a compact operator on $M \in \{A_{\kappa_{\calU}}^r, D_{\kappa_{\calU}}^{r^+}\}$. See \cite[\S 2.2]{Hansen-PhD} for the case $M = A_{\kappa_{\calU}}^r$ and \cite[Corollary 3.3.10]{Johansson-Newton} for the case $M = D_{\kappa_{\calU}}^{r^+}$. 
\end{Remark}

Recall the double coset decompositions \[
    \Iw_{\GSp_4, n}^+ \bfitu_{p,i} \Iw_{\GSp_4, n}^+ = \bigsqcup_{j} \bfdelta_{ij} \bfitu_{p,i} \Iw_{\GSp_4,n}^+ \quad \text{ and }\quad \Iw_{\GSp_4, n}^+ \bfitu_{p} \Iw_{\GSp_4, n}^+ = \bigsqcup_{j} \bfdelta_{j} \bfitu_{p} \Iw_{\GSp_4,n}^+
\] with $\bfdelta_{ij}, \bfdelta_j\in \Iw_{\GSp_4, n}^+$. We define the Hecke operators \[
    \begin{array}{llcl}
        U_{p,i}\colon & H^t(X_{n}(\C), M) & \xrightarrow{[\mu]\mapsto \sum_{j}\bfdelta_{ij}\cdot (\bfitu_{p,i}\cdot [\mu])} & H^t(X_n(\C), M),\\
        U_p\colon & H^t(X_{n}(\C), M) & \xrightarrow{[\mu]\mapsto \sum_{j}\bfdelta_{j}\cdot (\bfitu_{p}\cdot [\mu])} & H^t(X_n(\C), M).
    \end{array}
\]
These operators are independent of the choices of representatives (see for example the discussion after \cite[Definition 3.2.2]{DRW}).
It follows from the construction that $U_p = U_{p,0}\circ U_{p,1} = U_{p,1}\circ U_{p,0}$. Finally, we point out that we do not renormalise these operators.

\subsection{Kummer \'etale and pro-Kummer \'etale cohomology groups}\label{subsection: Kummer and pro-Kummer \'etale overconvergent cohomology}
The goal in \S \ref{subsection: Kummer and pro-Kummer \'etale overconvergent cohomology} is to re-interpret the Betti cohomology groups in \S \ref{subsection: Betti overconvergent cohomology} in terms of certain Kummer and pro-Kummer \'etale cohomology groups. 

We start with the case of small weights. The case of affinoid weights will be studied in the second half of \S \ref{subsection: Kummer and pro-Kummer \'etale overconvergent cohomology}. For the reason why we treat the two cases separately, see Remark \ref{Remark: small vs affinoid}.

Let $(R_{\calU}, \kappa_{\calU})$ be a small weight and let $r\in \Q_{\geq 0}$ such that $r>1+r_{\calU}$. Let $\fraka_{\calU}$ be an ideal of definition of $R_{\calU}$ and we assume that $p\in \fraka_{\calU}$. As explained in \cite[\S 4.1]{DRW}, building on ideas from \cite{Hansen-Iwasawa} and \cite{CHJ-2017}, there is a decreasing $\Iw_{\GSp_4, n}^+$-stable filtration $\Fil^{\bullet} D_{\kappa_{\calU}}^{r, \circ}$ on $D_{\kappa_{\calU}}^{r, \circ}$ such that \[
    D_{\kappa_{\calU}, j}^{r, \circ} \coloneq D_{\kappa_{\calU}}^{r, \circ}/\Fil^j D_{\kappa_{\calU}}^{r, \circ}
\] is a finite $\Z_p$-module and \[
    D_{\kappa_{\calU}}^{r, \circ} = \varprojlim_{j} D_{\kappa_{\calU}, j}^{r, \circ}
\] is a profinite flat $\Z_p$-module (\cite[Definition 6.1]{CHJ-2017}). 

We can impose a similar filtration on $A_{\kappa_{\calU}}^{r^+, \circ}$. Indeed, applying Proposition \ref{Proposition: structure theorem for distributions}, the natural map $A_{\kappa_{\calU}}^{r^+, \circ} \rightarrow A_{\kappa_{\calU}}^{(r+1)^+, \circ}$ is given by  \[
    A_{\kappa_{\calU}}^{r^+, \circ} \cong \prod_{i\in \Z_{\geq 0}^4} R_{\calU} e_i^{(r)} \rightarrow \prod_{i\in \Z_{\geq 0}^4} R_{\calU} e_i^{(r+1)} \cong A_{\kappa_{\calU}}^{(r+1)^+, \circ}, \quad e_i^{(r)} \mapsto \frac{\prod_{j=1}^4 \lfloor p^{-r} i_j \rfloor ! }{ \prod_{j=1}^4  \lfloor p^{-(r+1)}i_j \rfloor ! } e_i^{(r+1)}.
\]
Let $c_i^{(r)} \coloneq \frac{\prod_{j=1}^4 \lfloor p^{-r} i_j \rfloor ! }{ \prod_{j=1}^4  \lfloor p^{-(r+1)}i_j \rfloor ! }$. By Legendre's formula, we have \[
    v_p(c_i^{(r)}) = \sum_{j=1}^4\sum_{t>0} \big(\left\lfloor \frac{i_j}{p^{r+t}}\right\rfloor - \left\lfloor \frac{i_j}{p^{r+1+t}}\right\rfloor\big) = \sum_{j=1}^4 \left\lfloor \frac{i_j}{p^{r}}\right\rfloor \rightarrow \infty
\]
as $i\rightarrow \infty$. Therefore, the image of the map \[
    A_{\kappa_{\calU}}^{r^+, \circ} \rightarrow A_{\kappa_{\calU}}^{(r+1)^+, \circ}/\fraka_{\calU}^{j} A_{\kappa_{\calU}}^{(r+1)^+, \circ}
\] is finite. Define \[
    \Fil^j A_{\kappa_{\calU}}^{r^+, \circ} \coloneq \ker\left( A_{\kappa_{\calU}}^{r^+, \circ} \rightarrow A_{\kappa_{\calU}}^{(r+1)^+, \circ}/\fraka_{\calU}^{j} A_{\kappa_{\calU}}^{(r+1)^+, \circ}\right)
\]
and \[
    A_{\kappa_{\calU}, j}^{r^+, \circ} \coloneq A_{\kappa_{\calU}}^{r^+, \circ}/\Fil^j A_{\kappa_{\calU}}^{r, \circ}. 
\]
It follows that \[
    A_{\kappa_{\calU}, j}^{r^+, \circ}  \cong \bigoplus_{\substack{ i\in \Z_{\geq 0}^4\\ v_p(c_i^{(r)}) <j }} R_{\calU}/(\fraka_{\calU}^j, p^{i-v_p(c_i^{(r)})})
\] and \[
    A_{\kappa_{\calU}}^{r^+, \circ} = \varprojlim_j A_{\kappa_{\calU}, j}^{r, \circ}
\] as a profinite flat $\Z_p$-module. 

We now explain how to compute the Betti cohomology groups in terms of certain (Kummer) \'etale cohomology groups. Let $M\in \{A_{\kappa_{\calU}}^{r^+}, D_{\kappa_{\calU}}^{r}\}$ and let $M^{\circ} \in \{A_{\kappa_{\calU}}^{r^+, \circ}, D_{\kappa_{\calU}}^{r, \circ}\}$ be the corresponding integral version. Let $\Fil^{\bullet}M^{\circ}$ be the filtration discussed above and $M_j^{\circ}$ be the $j$-th graded piece. Since $\Fil^{\bullet} M^{\circ}$ is $\Iw_{\GSp_4, n}^+$-stable, each $M_j^{\circ}$ defines an \'etale local system $\scrM_j^{\circ}$ \footnote{For simplicity of exposition, we adopt the following notation for the rest of \S \ref{section: OC}.
\begin{itemize}
\item When $M=A_{\kappa_{\calU}}^{r^+}$ and $M^{\circ}=A_{\kappa_{\calU}}^{r^+, \circ}$, the terms $M_j^{\circ}$, $\scrM_j^{\circ}$, $\scrM^{\circ}$, $\scrM$, $\scrO\!\!\scrM^{\circ}$, and $\scrO\!\!\scrM$ stand for $A_{\kappa_{\calU}, j}^{r^+, \circ}$, $\scrA_{\kappa_{\calU}, j}^{r^+, \circ}$, $\scrA_{\kappa_{\calU}}^{r^+, \circ}$, $\scrA_{\kappa_{\calU}}^{r^+}$, $\scrO\!\!\scrA_{\kappa_{\calU}}^{r^+, \circ}$, and $\scrO\!\!\scrA_{\kappa_{\calU}}^{r^+}$, respectively.
\item When $M=D_{\kappa_{\calU}}^{r}$ and $M^{\circ}=D_{\kappa_{\calU}}^{r, \circ}$, the terms $M_j^{\circ}$, $\scrM_j^{\circ}$, $\scrM^{\circ}$, $\scrM$, $\scrO\!\!\scrM^{\circ}$, and $\scrO\!\!\scrM$ stand for $D_{\kappa_{\calU}, j}^{r, \circ}$, $\scrD_{\kappa_{\calU}, j}^{r, \circ}$, $\scrD_{\kappa_{\calU}}^{r, \circ}$, $\scrD_{\kappa_{\calU}}^{r}$, $\scrO\!\!\scrD_{\kappa_{\calU}}^{r, \circ}$, and $\scrO\!\!\scrD_{\kappa_{\calU}}^{r}$, respectively.
\end{itemize}
} on $\calX_n$ via \[
    \pi_1^{\et}(\calX_n) \twoheadrightarrow \Iw_{\GSp_4, n}^+ \rightarrow \Aut(M_j^{\circ}).
\]
This leads to an inverse system of \'etale local systems $\{\scrM_j^{\circ}\}_{j}$ so that we can define \'etale cohomology groups \[
    H_{\et}^i(\calX_n, \scrM^{\circ}) \coloneq \varprojlim_{j} H_{\et}^i(\calX_n, \scrM_j^{\circ})\] and \[H_{\et}^i(\calX_n, \scrM) \coloneq H_{\et}^i(\calX_n, \scrM^{\circ})\Big[\frac{1}{p}\Big].
\]

These \'etale cohomology groups can be also identified with certain Kummer \'etale cohomology groups on the toroidal compactifications of $\calX_{n}$. Consider the natural morphism of sites \[
    \jmath_{\ket}: \calX_{n, \et} \rightarrow \calX_{n, \ket}^{\tor}\] 
and consider the Kummer \'etale cohomology groups
\[H_{\ket}^i(\calX_{n}^{\tor}, \scrM^{\circ})  \coloneq \varprojlim_{j} H_{\ket}^i(\calX_{n}^{\tor}, \jmath_{\ket, *}\scrM_j^{\circ})\]
and
\[H_{\ket}^i(\calX_n^{\tor}, \scrM)  \coloneq H_{\ket}^i(\calX_{n}^{\tor}, \scrM^{\circ})\Big[\frac{1}{p}\Big].\]
Applying \cite[Corollary 4.6.7]{Diao}, we obtain natural isomorphisms\[
        H_{\ket}^i(\calX_{n}^{\tor}, \scrM^{\circ}) \cong H_{\et}^i(\calX_n, \scrM^{\circ}) \quad \text{ and }\quad H_{\ket}^i(\calX_n^{\tor}, \scrM) \cong H_{\et}^i(\calX_n, \scrM).
\]

\begin{Proposition}\label{Proposition: comparison of Betti and Kummer \'etale cohomology}
    Let $(R_{\calU}, \kappa_{\calU})$ be a small weight. Let $r\in \Q_{\geq 0}$ with $r>1+r_{\calU}$. Let $M^{\circ}\in \{A_{\kappa_{\calU}}^{r^+, \circ}, D_{\kappa_{\calU}}^{r, \circ}\}$ and $M = M^{\circ}[1/p]$ as above.
    For every $i$, there are natural isomorphisms 
    \[
        H^i(X_n(\C), M^{\circ}) \cong H_{\et}^i(\calX_n, \scrM^{\circ}) \cong H_{\ket}^i(\calX_n^{\tor}, \scrM^{\circ})
        \] 
        and 
        \[
        H^i(X_n(\C), M) \cong H_{\et}^i(\calX_n, \scrM)\cong H_{\ket}^i(\calX_n^{\tor}, \scrM). 
    \]
\end{Proposition}
\begin{proof}
    The proof goes verbatim as in \cite[Proposition 4.2.2]{DRW}. 
\end{proof}

For our purpose, we would like to further interpret these cohomology groups in terms of pro-Kummer \'etale cohomology groups. To this end, recall the natural projection of sites \[
    \nu: \calX_{n, \proket}^{\tor} \rightarrow \calX_{n, \ket}^{\tor}. 
\]  
Given $M$ and $M^{\circ}$ as above, we define sheaves \[
    \scrO\!\!\!\scrM^{\circ} \coloneq \varprojlim_{j} \left(\nu^{-1}\jmath_{\ket, *}\scrM_{j}^{\circ}\otimes_{\Z_p} \widehat{\scrO}_{\calX_{n, \proket}^{\tor}}^+\right)\]
    and $\scrO\!\!\!\scrM \coloneq \scrO\!\!\!\scrM^{\circ}[1/p]$ on the pro-Kummer \'etale site $\calX_{n, \proket}^{\tor}$.

\begin{Proposition}\label{Proposition: comparing Kummer \'etale and pro-Kummer \'etale cohomology}
 There is a natural $\Gal_{\Q_p}$-equivariant almost isomorphism 
\[
        H_{\ket}^i(\calX_{n}^{\tor}, \scrM^{\circ}) \widehat{\otimes}\calO_{\C_p} \cong^a H_{\proket}^i(\calX_n^{\tor}, \scrO\!\!\!\scrM^{\circ})
    \] and hence a $\Gal_{\Q_p}$-equivariant isomorphism \[
        H_{\ket}^i(\calX_{n}^{\tor}, \scrM) \widehat{\otimes}{\C_p} \cong H_{\proket}^i(\calX_n^{\tor}, \scrO\!\!\!\scrM). 
    \]
\end{Proposition}
\begin{proof}
    The proof follows from a similar argument as in the proof of \cite[Proposition 5.1.2]{DRW}.
\end{proof}

Finally, we also consider cohomology groups with compact support. Recall the localisation functors
\[\jmath_{\ket, !}: \mathrm{Sh}(\calX_{n,\et})\rightarrow \mathrm{Sh}(\calX_{n,\ket}^{\tor})\]
and
\[\jmath_{\proket, !}: \mathrm{Sh}(\calX_{n,\proet})\rightarrow \mathrm{Sh}(\calX_{n,\proket}^{\tor})\]
constructed in \cite[\S 4.5 \& Definition 5.2.1]{Diao}. We define the Kummer \'etale cohomology groups with compact supports
\[H_{\ket, c}^i(\calX_n^{\tor}, \scrM^{\circ}) \coloneq \varprojlim_{j}H_{\ket}^i(\calX_n^{\tor}, \jmath_{\ket, !}\scrM_j^{\circ})\]
and
\[H_{\ket, c}^i(\calX_n^{\tor}, \scrM) \coloneq H_{\ket, c}^i(\calX_n^{\tor}, \scrM^{\circ})\Big[\frac{1}{p}\Big]\]
as well as the pro-Kummer \'etale cohomology groups with compact supports
\[H_{\proket, c}^i(\calX_n^{\tor}, \scrO\!\!\!\scrM^{\circ}) \coloneq H_{\proket}^i(\calX_n^{\tor}, \jmath_{\proket, !}\,\jmath_{\proket}^{-1}\scrO\!\!\!\scrM^{\circ})\]
and
\[H_{\proket, c}^i(\calX_n^{\tor}, \scrO\!\!\!\scrM) \coloneq H_{\proket}^i(\calX_n^{\tor}, \jmath_{\proket, !}\,\jmath_{\proket}^{-1}\scrO\!\!\!\scrM) = H_{\proket, c}^i(\calX_n^{\tor}, \scrO\!\!\!\scrM^{\circ})\Big[\frac{1}{p}\Big].\]
A similar argument as in the proof of Proposition \ref{Proposition: comparing Kummer \'etale and pro-Kummer \'etale cohomology} yields a $\Gal_{\Q_p}$-equivariant isomorphism \[
    H_{\ket, c}^i(\calX_n^{\tor}, \scrM) \widehat{\otimes}\C_p \cong H_{\proket, c}^i(\calX_n^{\tor}, \scrO\!\!\scrM).
\]
The following lemma provides an alternative description of the pro-Kummer \'etale cohomology with compact support. 

\begin{Lemma}\label{Lemma: cuspidal pro-Kummer \'etale sheaves}
   Let $\calD_n$ be the boundary divisor of $\calX_n^{\tor}$. Then there is an isomorphism of pro-Kummer \'etale sheaves 
\[
        \scrO\!\!\scrM(-\calD_{n}) \cong \jmath_{\proket, !}\,\jmath_{\proket}^{-1} \scrO\!\!\scrM
\]
where $\scrO\!\!\scrM(-\calD_{n})$ stands for the subsheaf of $\scrO\!\!\scrM$ of sections vanishing along $\calD_{n}$.
\end{Lemma}
\begin{proof}
    For every $j$, the sheaf $\scrM_{j}^{\circ}$ is an \'etale $\Z/p^{n}\Z$-local system for some $n$. We take $n_j$ to be the minimal such integer. Then   there is an isomorphism
\[
    \scrO\!\!\scrM^{\circ}\cong \varprojlim_{j}\nu^{-1}(\jmath_{\ket, *}\scrM_{j}^{\circ}\otimes_{\Z/p^{n_j}\Z} \scrO_{\calX_{n}^{\tor}, \ket}^+/p^{n_j})
\]
of pro-Kummer \'etale sheaves. We write $\scrO\!\!\scrM_{j, \ket}:= \jmath_{\ket, *}\scrM_{j}^{\circ}\otimes_{\Z/p^{n_j}\Z} \scrO_{\calX_{n}^{\tor}, \ket}^+/p^{n_j}$. 

Let $\imath: \calD_n \hookrightarrow \calX_n^{\tor}$ be the strict closed immersion; in particular, we endow $\calD_n$ with the pullback log structure from $\calX_n^{\tor}$.
We have short exact sequences\[
    0 \rightarrow \jmath_{\ket, !}\jmath_{\ket}^{-1} \scrO\!\!\scrM_{j, \ket} \rightarrow \scrO\!\!\scrM_{j, \ket} \rightarrow \imath_{\ket, *}\imath_{\ket}^{-1} \scrO\!\!\scrM_{j, \ket} \rightarrow 0
\] 
and 
\[
    0 \rightarrow \scrO\!\!\scrM_{j, \ket}(-\calD_{n}) \rightarrow \scrO\!\!\scrM_{j, \ket} \rightarrow \imath_{\ket, *}\imath_{\ket}^{-1} \scrO\!\!\scrM_{j, \ket} \rightarrow 0.
\] 
Indeed, the first exact sequence follows from \cite[Lemma 4.5.3]{Diao} while the second follows from definitions. These short exact sequences further pullback to short exact sequence over $\calX_{n, \proket}^{\tor}$ by \cite[Corollary 5.1.8]{Diao}. Taking limit with respect to $j$ and then inverting $p$, we arrive at short exact sequences \[
    0 \rightarrow \varprojlim_{j} \nu^{-1}\left(\jmath_{\ket, !}\jmath_{\ket}^{-1} \scrO\!\!\scrM_{j, \ket}\right) \rightarrow \scrO\!\!\scrM \rightarrow \varprojlim_{j} \nu^{-1}\left(\imath_{\ket, *}\imath_{\ket}^{-1} \scrO\!\!\scrM_{j, \ket}\right)\rightarrow 0
\]
and 
\[
    0 \rightarrow \scrO\!\!\scrM(-\calD_n)\rightarrow \scrO\!\!\scrM \rightarrow \varprojlim_{j} \nu^{-1}\left(\imath_{\ket, *}\imath_{\ket}^{-1} \scrO\!\!\scrM_{j, \ket}\right) \rightarrow 0.
\]
Note that the corresponding $R^1\lim$'s vanish because the system $\{\scrM_j^{\circ}\}$ is Mittag-Leffler. Consequently, we obtain the desired isomorphism
\[\scrO\!\!\scrM(-\calD_n) \cong \varprojlim_{j} \nu^{-1}\left(\jmath_{\ket, !}\jmath_{\ket}^{-1} \scrO\!\!\scrM_{j, \ket}\right)\cong  \jmath_{\proket, !}\,\jmath_{\proket}^{-1} \scrO\!\!\scrM.\]
\end{proof}

So far, given a small weight $(R_{\calU}, \kappa_{\calU})$ and $r\in \Q_{\geq 0}$ with $r> 1+r_{\calU}$, we have defined sheaves $\scrO\!\!\scrA_{\kappa_{\calU}}^{r^+, \circ}$, $\scrO\!\!\scrA_{\kappa_{\calU}}^{r^+}$, $\scrO\!\!\scrD_{\kappa_{\calU}}^{r, \circ}$, and $\scrO\!\!\scrD_{\kappa_{\calU}}^{r}$ on $\calX_{n, \proket}^{\tor}$. Taking duals, we define \[
    \begin{array}{cc}
        \scrO\!\!\scrA_{\kappa_{\calU}}^{r, \circ}  \coloneq \sheafHom_{R_{\calU}\widehat{\otimes}\widehat{\scrO}_{\calX_{n, \proket}^{\tor}}^+}\left( \scrO\!\!\scrD_{\kappa_{\calU}}^{r, \circ}, R_{\calU}\widehat{\otimes}\widehat{\scrO}_{\calX_{n, \proket}^{\tor}}^+\right), & \scrO\!\!\scrA_{\kappa_{\calU}}^{r} \coloneq \scrO\!\!\scrA_{\kappa_{\calU}}^{r, \circ}\Big[\frac{1}{p}\Big],\\
        \scrO\!\!\scrD_{\kappa_{\calU}}^{r^+, \circ}  \coloneq \sheafHom_{R_{\calU}\widehat{\otimes}\widehat{\scrO}_{\calX_{n, \proket}^{\tor}}^+}\left( \scrO\!\!\scrA_{\kappa_{\calU}}^{r^+, \circ}, R_{\calU}\widehat{\otimes}\widehat{\scrO}_{\calX_{n, \proket}^{\tor}}^+\right), & \scrO\!\!\scrD_{\kappa_{\calU}}^{r^+} \coloneq \scrO\!\!\scrD_{\kappa_{\calU}}^{r^+, \circ}\Big[\frac{1}{p}\Big],
    \end{array}
\]
where the internal Hom is taken in the category of topological $R_{\calU}\widehat{\otimes}\widehat{\scrO}_{\calX_{n, \proket}^{\tor}}^+$-modules.

To wrap up \S \ref{subsection: Kummer and pro-Kummer \'etale overconvergent cohomology}, we extend these constructions to affinoid weights. Consider a small weight $(R_{\calU}, \kappa_{\calU})$ together with an affinoid open $\calV = \Spa(R_{\calV}, R_{\calV}^{\circ})$ in $\calU$. Let $\kappa_{\calV}$ be the induced continuous character through the embedding $\calV \subset \calU$. For $r>1+r_{\calV}$, we define \[
    \scrO\!\!\scrA_{\kappa_{\calV}}^r \coloneq \scrO\!\!\scrA_{\kappa_{\calU}}^r \widehat{\otimes} R_{\calV} \quad \text{ and }\quad \scrO\!\!\scrD_{\kappa_{\calV}}^{r^+} \coloneq \scrO\!\!\scrD_{\kappa_{\calU}}^{r^+} \widehat{\otimes} R_{\calV}. 
\]
We have the following structure theorem. 

\begin{Lemma}\label{Lemma: structure theorem of OD^{r^+} and OA^r}
    Let $(R_{\calU}, \kappa_{\calU})$, $(R_{\calV}, \kappa_{\calV})$, and $r$ be as above. Let $U$ be an affinoid perfectoid object in the pro-Kummer \'etale site $\calX_{n, \proket}^{\tor}$, with associated affinoid perfectoid space $\Spa(R, R^+)$. Then there are identifications \[
        \scrO\!\!\scrA_{\kappa_{\calV}}^{r}(U) \cong \widehat{\bigoplus}_{i\in \Z_{\geq 0}^4} \big(R_{\calV}\widehat{\otimes}R\big) e_i^{(r)} \quad \text{ and }\quad \scrO\!\!\scrD_{\kappa_{\calV}}^{r^+}(U) \cong \widehat{\bigoplus}_{i\in \Z_{\geq 0}^4} \big(R_{\calV}\widehat{\otimes}R\big) e_i^{(r), \vee}.
    \]
\end{Lemma}
\begin{proof}
It follows from Proposition \ref{Proposition: structure theorem for distributions} that
   \[
        \scrO\!\!\scrD_{\kappa_{\calU}}^{r, \circ}(U) \cong \prod_{i\in \Z_{\geq 0}^4} \big(R_{\calU}\widehat{\otimes}R^{\circ}\big) e_i^{(r), \vee}\quad \text{ and }\quad \scrO\!\!\scrA_{\kappa_{\calU}}^{r^+, \circ}(U) \cong \prod_{i\in \Z_{\geq 0}^4} \big(R_{\calU}\widehat{\otimes}R^{\circ}\big) e_i^{(r)} .
    \]
    The desired identifications then follow from taking dual and taking $R_{\calV}\widehat{\otimes}-$. 
\end{proof}

For an affinoid weight $(R_{\calV}, \kappa_{\calV})$ as above, we further define \[
    \begin{array}{c}
        \scrO\!\!\scrD_{\kappa_{\calV}}^r \coloneq \sheafHom_{R_{\calV}\widehat{\otimes}\widehat{\scrO}_{\calX_{n, \proket}^{\tor}}}(\scrO\!\!\scrA_{\kappa_{\calV}}^r, R_{\calV}\widehat{\otimes}\widehat{\scrO}_{\calX_{n, \proket}^{\tor}}),\\
        \scrO\!\!\scrA_{\kappa_{\calV}}^r \coloneq \sheafHom_{R_{\calV}\widehat{\otimes}\widehat{\scrO}_{\calX_{n, \proket}^{\tor}}}(\scrO\!\!\scrD_{\kappa_{\calV}}^r, R_{\calV}\widehat{\otimes}\widehat{\scrO}_{\calX_{n, \proket}^{\tor}}).
    \end{array}
\]
For $s\geq r > 1+r_{\calV}$, there are natural injections and surjections \begin{equation}\label{eq: sandwiches of pro-Kummer \'etale sheaves for OA and OD}
    \scrO\!\!\scrA_{\kappa_{\calV}}^r \hookrightarrow \scrO\!\!\scrA_{\kappa_{\calV}}^{r^+} \hookrightarrow \scrO\!\!\scrA_{\kappa_{\calV}}^s \quad \text{ and }\quad \scrO\!\!\scrD_{\kappa_{\calV}}^{s^+} \twoheadrightarrow \scrO\!\!\scrD_{\kappa_{\calV}}^{s} \twoheadrightarrow \scrO\!\!\scrD_{\kappa_{\calV}}^{r^+}.
\end{equation}

\begin{Remark}\label{Remark: small vs affinoid}
    The readers might wonder why we went through such an indirect construction to define sheaves $\scrO\!\!\scrA_{\kappa_{\calV}}^r$, $ \scrO\!\!\scrA_{\kappa_{\calV}}^{r^+}$, $ \scrO\!\!\scrD_{\kappa_{\calV}}^r$, and $ \scrO\!\!\scrD_{\kappa_{\calV}}^{r^+}$ for affinoid weights. Let us explain briefly in this remark. First of all, one needs a well-behaved integral structure to associate with (Kummer) \'etale local systems. Such an integral structure only exists when we work with small weights (following the idea in \cite{Hansen-Iwasawa}). Secondly, we will need a notion of \emph{finite-slope part} of pro-Kummer \'etale cohomology with supports for affinoid weights. Notice that $\scrO\!\!\scrA_{\kappa_{\calV}}^{r^+}$ and $\scrO\!\!\scrD_{\kappa_{\calV}}^r$ are not sheaves of Banach $\widehat{\scrO}_{\calX_{n, \proket}^{\tor}}\widehat{\otimes}R_{\calV}$-modules (in the sense of Definition \ref{Definition: proket Banach sheaves}), but $\scrO\!\!\scrA_{\kappa_{\calV}}^{r}$ and $\scrO\!\!\scrD_{\kappa_{\calV}}^{r^+}$ are. Therefore, it is necessary to work with $\scrO\!\!\scrA_{\kappa_{\calV}}^{r}$ and $\scrO\!\!\scrD_{\kappa_{\calV}}^{r^+}$ when we define the finite-slope part of pro-Kummer \'etale cohomology with supports with coefficients in $\scrO\!\!\scrA_{\kappa_{\calV}}^{r^+}$ and $\scrO\!\!\scrD_{\kappa_{\calV}}^r$. 
\end{Remark}

\begin{Remark}\label{Remark: pro-Kummer \'etale overconvergent cohomology for affinoid weights}
    Let $(R_{\calU}, \kappa_{\calU})$, $(R_{\calV}, \kappa_{\calV})$, and $r$ be as above.
     \begin{enumerate}
        \item[(i)]  There are isomorphisms \[
           \scalemath{0.9}{ H^i(X_n(\C), D_{\kappa_{\calV}}^r)\widehat{\otimes}\C_p \cong \left(H^i(X_n(\C), D_{\kappa_{\calU}}^r)\widehat{\otimes}\C_p\right)\widehat{\otimes}_{R_{\calU}}R_{\calV} \cong H_{\proket}^i(\calX^{\tor}_n, \scrO\!\!\scrD_{\kappa_{\calU}}^r)\widehat{\otimes}_{R_{\calU}}R_{\calV} \cong H_{\proket}^i(\calX^{\tor}_n, \scrO\!\!\scrD_{\kappa_{\calV}}^r),}
        \] where the middle isomorphism follows from Proposition \ref{Proposition: comparison of Betti and Kummer \'etale cohomology} and Proposition \ref{Proposition: comparing Kummer \'etale and pro-Kummer \'etale cohomology}. Similar results hold for $\scrO\!\!\scrA_{\kappa_{\calV}}^r$ and for compactly supported cohomology groups. 
        \item[(ii)] A similar statement of Lemma \ref{Lemma: cuspidal pro-Kummer \'etale sheaves} for affinoid weights also follows from such a base change. 
    \end{enumerate} 
\end{Remark}

\subsection{An alternative construction on the flag variety}\label{subsection: alternative construction on flag variety}
In \S \ref{subsection: Kummer and pro-Kummer \'etale overconvergent cohomology}, we constructed sheaves 
\begin{itemize}
\item $\scrO\!\!\scrA_{\kappa_{\calU}}^r$, $\scrO\!\!\scrA_{\kappa_{\calU}}^{r+}$, $\scrO\!\!\scrD_{\kappa_{\calU}}^r$, $\scrO\!\!\scrD_{\kappa_{\calU}}^{r+}$ for small weights $(R_{\calU}, \kappa_{\calU})$; and
\item $\scrO\!\!\scrA_{\kappa_{\calV}}^r$, $\scrO\!\!\scrA_{\kappa_{\calV}}^{r+}$, $\scrO\!\!\scrD_{\kappa_{\calV}}^r$, $\scrO\!\!\scrD_{\kappa_{\calV}}^{r+}$ for affinoid weights $(R_{\calV}, \kappa_{\calV})$
\end{itemize} 
on the pro-Kummer \'etale site $\calX_{n, \proket}^{\tor}$. For later use, we need a similar construction of such sheaves on the flag variety; namely, we construct
\begin{itemize}
\item $\scrO\!\!\scrA_{\kappa_{\calU}, \adicFL}^r$, $\scrO\!\!\scrA_{\kappa_{\calU}, \adicFL}^{r+}$, $\scrO\!\!\scrD_{\kappa_{\calU}, \adicFL}^r$, $\scrO\!\!\scrD_{\kappa_{\calU}, \adicFL}^{r+}$ for small weights $(R_{\calU}, \kappa_{\calU})$; and
\item $\scrO\!\!\scrA_{\kappa_{\calV}, \adicFL}^r$, $\scrO\!\!\scrA_{\kappa_{\calV}, \adicFL}^{r+}$, $\scrO\!\!\scrD_{\kappa_{\calV}, \adicFL}^r$, $\scrO\!\!\scrD_{\kappa_{\calV}, \adicFL}^{r+}$ for affinoid weights $(R_{\calV}, \kappa_{\calV})$
\end{itemize} 
on the pro-\'etale site $\adicFL_{\proet}$. As we shall see in Proposition \ref{prop: compare two constructions}, the two constructions are related via the Hodge--Tate period map.

Once again we start with small weights. Let $(R_{\calU}, \kappa_{\calU})$ be a small weight and let $r\geq r_{\calU}+1$. Recall the profinite systems $\{D_{\kappa_{\calU}, j}^{r, \circ}\}_j$ and $\{A_{\kappa_{\calU}, j}^{r^+, \circ}\}_j$. Let $\scrA_{\kappa_{\calU}, j, \adicFL}^{r^+, \circ}$ (resp., $\scrD_{\kappa_{\calU}, j, \adicFL}^{r, \circ}$) be the \'etale constant sheaf on $\adicFL$ associated with $A_{\kappa_{\calU}, j}^{r^+, \circ}$ (resp., $D_{\kappa_{\calU}, j}^{r, \circ}$). We define sheaves \[
    \scrO\!\!\scrA_{\kappa_{\calU}, \adicFL}^{r^+, \circ} := \varprojlim_j \nu^{-1}\scrA_{\kappa_{\calU}, j, \adicFL}^{r^+, \circ}\otimes_{\Z_p}\widehat{\scrO}_{\adicFL, \proet}^+ \quad , \quad \scrO\!\!\scrA_{\kappa_{\calU}, \adicFL}^{r^+} := \scrO\!\!\scrA_{\kappa_{\calU}, \adicFL}^{r^+, \circ}\Big[\frac{1}{p}\Big],
\]
and 
\[
    \scrO\!\!\scrD_{\kappa_{\calU}, \adicFL}^{r, \circ} := \varprojlim_j \nu^{-1}\scrD_{\kappa_{\calU}, j, \adicFL}^{r, \circ}\otimes_{\Z_p}\widehat{\scrO}_{\adicFL, \proet}^+ \quad , \quad \scrO\!\!\scrD_{\kappa_{\calU}, \adicFL}^{r} := \scrO\!\!\scrD_{\kappa_{\calU}, \adicFL}^{r, \circ}\Big[\frac{1}{p}\Big],
\]
where $\nu: \adicFL_{\proet} \rightarrow \adicFL_{\et}$ is the natural projection of sites. Similar to \S\ref{subsection: Kummer and pro-Kummer \'etale overconvergent cohomology}, we then consider sheaves \[
    \begin{array}{cc}
        \scrO\!\!\scrA_{\kappa_{\calU}, \adicFL}^{r, \circ}  \coloneq \sheafHom_{R_{\calU}\widehat{\otimes}\widehat{\scrO}_{\adicFL_{\proet}}}\left( \scrO\!\!\scrD_{\kappa_{\calU}, \adicFL}^{r, \circ}, R_{\calU}\widehat{\otimes}\widehat{\scrO}_{\adicFL_{\proet}}^+\right), & \scrO\!\!\scrA_{\kappa_{\calU}, \adicFL}^{r} \coloneq \scrO\!\!\scrA_{\kappa_{\calU}, \adicFL}^{r, \circ}\Big[\frac{1}{p}\Big],\\
        \scrO\!\!\scrD_{\kappa_{\calU}, \adicFL}^{r^+, \circ}  \coloneq \sheafHom_{R_{\calU}\widehat{\otimes}\widehat{\scrO}_{\adicFL_{\proet}}}\left( \scrO\!\!\scrA_{\kappa_{\calU}, \adicFL}^{r^+, \circ}, R_{\calU}\widehat{\otimes}\widehat{\scrO}_{\adicFL_{\proet}}^+\right), & \scrO\!\!\scrD_{\kappa_{\calU}, \adicFL}^{r^+} \coloneq \scrO\!\!\scrD_{\kappa_{\calU}, \adicFL}^{r^+, \circ}\Big[\frac{1}{p}\Big],
    \end{array}
\] 

Now we treat the case for affinoid weights. Consider a small weight $(R_{\calU}, \kappa_{\calU})$ together with an affinoid open $\calV = \Spa(R_{\calV}, R_{\calV}^{\circ})$ in $\calU$. Let $\kappa_{\calV}$ be the induced continuous character through the embedding $\calV \subset \calU$. For $r>1+r_{\calV}$, we define \[
    \scrO\!\!\scrA_{\kappa_{\calV}, \adicFL}^r \coloneq \scrO\!\!\scrA_{\kappa_{\calU}, \adicFL}^r \widehat{\otimes} R_{\calV} \quad \text{ and }\quad \scrO\!\!\scrD_{\kappa_{\calV}, \adicFL}^{r^+} \coloneq \scrO\!\!\scrD_{\kappa_{\calU}, \adicFL}^{r^+} \widehat{\otimes} R_{\calV}. 
\]
Then we define \[
    \begin{array}{c}
        \scrO\!\!\scrD_{\kappa_{\calV}, \adicFL}^r \coloneq \sheafHom_{R_{\calV}\widehat{\otimes}\widehat{\scrO}_{\calX_{n, \proket}^{\tor}}}(\scrO\!\!\scrA_{\kappa_{\calV}, \adicFL}^r, R_{\calV}\widehat{\otimes}\widehat{\scrO}_{\calX_{n, \proket}^{\tor}}),\\
        \scrO\!\!\scrA_{\kappa_{\calV}, \adicFL}^r \coloneq \sheafHom_{R_{\calV}\widehat{\otimes}\widehat{\scrO}_{\calX_{n, \proket}^{\tor}}}(\scrO\!\!\scrD_{\kappa_{\calV}, \adicFL}^r, R_{\calV}\widehat{\otimes}\widehat{\scrO}_{\calX_{n, \proket}^{\tor}}).
    \end{array}
\]
For $s\geq r > 1+r_{\calV}$, there are natural injections and surjections \begin{equation}\label{eq: sandwiches of pro-Kummer \'etale sheaves for OA and OD on flag variety}
    \scrO\!\!\scrA_{\kappa_{\calV}, \adicFL}^r \hookrightarrow \scrO\!\!\scrA_{\kappa_{\calV}, \adicFL}^{r^+} \hookrightarrow \scrO\!\!\scrA_{\kappa_{\calV}, \adicFL}^s \quad \text{ and }\quad \scrO\!\!\scrD_{\kappa_{\calV}, \adicFL}^{s^+} \twoheadrightarrow \scrO\!\!\scrD_{\kappa_{\calV}, \adicFL}^{s} \twoheadrightarrow \scrO\!\!\scrD_{\kappa_{\calV}, \adicFL}^{r^+}.
\end{equation}

In fact, the two constructions in \S \ref{subsection: Kummer and pro-Kummer \'etale overconvergent cohomology} and \S \ref{subsection: alternative construction on flag variety} are related via the Hodge--Tate period map. More precisely, consider the Hodge--Tate period map $\pi_{\HT}: \calX^{\tor}_{\Gamma(p^{\infty})}\rightarrow \adicFL$ which induces $\pi_{\HT}: \calX^{\tor}_{\Gamma(p^{\infty}), \proket}\rightarrow \adicFL_{\proet}$. We also consider the natural projection $h_n: \calX^{\tor}_{\Gamma(p^{\infty})}\rightarrow \calX^{\tor}_{n}$ which induces $h_n: \calX^{\tor}_{\Gamma(p^{\infty}), \proket}\rightarrow \calX^{\tor}_{n, \proket}$.

\begin{Proposition}\label{prop: compare two constructions}
Let $(R_{\calU}, \kappa_{\calU})$ be a small weight and let $r\geq r_{\calU}+1$. Then there is a natural identification
\[h_n^*\scrO\!\!\scrA_{\kappa_{\calU}}^r\simeq \pi_{\HT}^*\scrO\!\!\scrA_{\kappa_{\calU}, \adicFL}^r.\]
Similar results hold for $\scrO\!\!\scrA_{\kappa_{\calU}}^{r+}$, $\scrO\!\!\scrD_{\kappa_{\calU}}^r$, $\scrO\!\!\scrD_{\kappa_{\calU}}^{r+}$ for small weights $(R_{\calU}, \kappa_{\calU})$, and for $\scrO\!\!\scrA_{\kappa_{\calV}}^r$, $\scrO\!\!\scrA_{\kappa_{\calV}}^{r+}$, $\scrO\!\!\scrD_{\kappa_{\calV}}^r$, $\scrO\!\!\scrD_{\kappa_{\calV}}^{r+}$ for affinoid weights $(R_{\calV}, \kappa_{\calV})$.
\end{Proposition}

\begin{proof}
This follows immediately from the constructions.
\end{proof}

\subsection{Pro-Kummer \'etale cohomology groups with supports}\label{subsection: pro-Kummer \'etale overconvergent cohomology with supports}
Let $(R_{\calV}, \kappa_{\calV})$ be an affinoid weights and let $r>1+r_{\calV}$. This section is dedicated to the study of the pro-Kummer \'etale cohomology groups (with supports) with coefficients in $\scrO\!\!\scrA_{\kappa_{\calV}}^r$, $ \scrO\!\!\scrA_{\kappa_{\calV}}^{r^+}$, $ \scrO\!\!\scrD_{\kappa_{\calV}}^r$, and $ \scrO\!\!\scrD_{\kappa_{\calV}}^{r^+}$. See \S \ref{section: cohomology with supports} for a theory of pro-Kummer \'etale cohomology with supports.

Consider the stratification of the flag variety $\adicFL$ by closed subsets \[
    \adicFL = \adicFL_{\leq \bfitw_3} \supset \overline{\adicFL_{\leq \bfitw_2}} \supset \overline{\adicFL_{\leq \bfitw_1}} \supset \overline{\adicFL_{\leq \one_4}} = \overline{\adicFL_{\one}} \supset \emptyset.
\]
By defining \[
    \calX_{n, \leq \bfitw}^{\tor} \coloneq h_n(\pi_{\HT}^{-1}(\adicFL_{\leq \bfitw})),
\]
we arrive at a stratification on the Siegel modular varieties \[
    \calX_{n}^{\tor} = \calX_{n, \leq \bfitw_3}^{\tor} \supset \overline{\calX_{n, \leq \bfitw_2}^{\tor}} \supset \overline{\calX_{n, \leq \bfitw_1}^{\tor}} \supset \overline{\calX_{n, \leq \one_4}^{\tor}} = \overline{\calX_{n, \one_4}^{\tor}} \supset \emptyset.
\]

Let $\scrO\!\!\scrM$ be any of $\scrO\!\!\scrA_{\kappa_{\calV}}^r$, $ \scrO\!\!\scrA_{\kappa_{\calV}}^{r^+}$, $ \scrO\!\!\scrD_{\kappa_{\calV}}^r$, and $ \scrO\!\!\scrD_{\kappa_{\calV}}^{r^+}$. By Proposition \ref{Proposition: a spectral sequence for stratifications}, there is a diagram \begin{equation}\label{eq: diagram for coh. with supports (1)}
    \begin{tikzcd}[column sep = small]
        \scalemath{0.9}{ R\Gamma_{\overline{\calX_{n, \leq \bfitw_2}^{\tor}}, \proket}(\calX_{n}^{\tor}, \scrO\!\!\scrM) } \arrow[r] & \scalemath{0.9}{ R\Gamma_{\proket}(\calX_{n}^{\tor}, \scrO\!\!\scrM) }\arrow[r] & \scalemath{0.9}{ R\Gamma_{\proket}(\calX_{n}^{\tor} \smallsetminus \overline{\calX_{n, \leq \bfitw_2}^{\tor}}, \scrO\!\!\scrM) }\\
        \scalemath{0.9}{ R\Gamma_{\overline{\calX_{n, \leq \bfitw_1}^{\tor}}, \proket}(\calX_{n}^{\tor}, \scrO\!\!\scrM) } \arrow[r] & \scalemath{0.9}{ R\Gamma_{\overline{\calX_{n, \leq \bfitw_2}^{\tor}}, \proket}(\calX_{n}^{\tor}, \scrO\!\!\scrM) } \arrow[r] \arrow[ul, equal] & \scalemath{0.9}{ R\Gamma_{\overline{\calX_{n, \leq \bfitw_2}^{\tor}}\smallsetminus \overline{\calX_{n, \leq \bfitw_1}^{\tor}}, \proket}(\calX_{n}^{\tor} \smallsetminus \overline{\calX_{n, \leq \bfitw_1}^{\tor}}, \scrO\!\!\scrM) } \\
        \scalemath{0.9}{ R\Gamma_{\overline{\calX_{n, \leq \one_4}^{\tor}}, \proket}(\calX_{n}^{\tor}, \scrO\!\!\scrM) } \arrow[r] & \scalemath{0.9}{ R\Gamma_{\overline{\calX_{n, \leq \bfitw_1}^{\tor}}, \proket}(\calX_{n}^{\tor}, \scrO\!\!\scrM) } \arrow[r] \arrow[ul, equal] & \scalemath{0.9}{ R\Gamma_{\overline{\calX_{n, \leq \bfitw_1}^{\tor}}\smallsetminus \overline{\calX_{n, \leq \one_4}^{\tor}}, \proket}(\calX_{n}^{\tor} \smallsetminus \overline{\calX_{n, \leq \one_4}^{\tor}}, \scrO\!\!\scrM) }
    \end{tikzcd},
\end{equation}
where the rows are all distinguished triangles. This diagram gives rise to an $E_1$-spectral sequence \[
    E_1^{i,j} = H_{\overline{\calX_{n, \leq \bfitw_{3-j}}^{\tor}}\smallsetminus \overline{\calX_{n, \leq \bfitw_{3-j-1}}^{\tor}}, \proket}^{i+j}(\calX_n^{\tor} \smallsetminus \overline{\calX_{n, \leq \bfitw_{3-j-1}}^{\tor}}, \scrO\!\!\scrM) \Rightarrow H_{\proket}^{i+j}(\calX_n^{\tor}, \scrO\!\!\scrM).
\]

We shall now discuss the Hecke actions on the complexes \[R\Gamma_{\overline{\calX_{n, \leq \bfitw_i}^{\tor}}\smallsetminus \overline{\calX_{n, \leq \bfitw_{i-1}}^{\tor}} ,\proket}(\calX_n^{\tor} \smallsetminus \overline{\calX_{n, \leq \bfitw_{i-1}}^{\tor}}, \scrO\!\!\scrM).\] For the Hecke operators away from $pN$, the constructions are similar to the ones in \S \ref{subsection: Hecke operators}. We leave the details to the reader. In what follows, we shall focus on the Hecke operators at $p$. Our construction is highly inspired by \cite[\S 5]{BP-HigherColeman}.

First of all, for any $\bfitw\in W^H$ and $m, r\in \Q_{\geq 0}$, consider \begin{align*}
    \calX_{n, \bfitw, (m, r)}^{\tor} \coloneq h_n\left( \pi_{\HT}^{-1}\adicFL_{\bfitw, (m, r)} \right), \quad & \calX_{n, \bfitw, (\overline{m}, \overline{n})}^{\tor} \coloneq h_n\left( \pi_{\HT}^{-1}\adicFL_{\bfitw, (\overline{m}, \overline{r})} \right),\\
    \calX_{n, \bfitw, (\overline{m}, r)}^{\tor} \coloneq h_n\left( \pi_{\HT}^{-1}\adicFL_{\bfitw, (\overline{m}, r)} \right), \quad & \calX_{n, \bfitw, (m, \overline{r})}^{\tor} \coloneq h_n\left( \pi_{\HT}^{-1}\adicFL_{\bfitw, (m, \overline{r})} \right).
\end{align*}
It follows from \cite[Lemma 3.3.22]{BP-HigherColeman} that \[
    \overline{\calX_{n, \leq \bfitw_i}^{\tor}}\smallsetminus \overline{\calX_{n, \leq \bfitw_{i-1}}^{\tor}}  = \calX_{n, \geq \bfitw_i}^{\tor} \cap \overline{\calX_{n, \leq \bfitw_i}^{\tor}} = \calX_{n, \bfitw_i, (0,\overline{0})}^{\tor}
\]
for all $i$. Consequently, together with \eqref{eq: change of ambient spaces}, we obtain a quasi-isomorphism \[
    R\Gamma_{\overline{\calX_{n, \leq \bfitw_i}^{\tor}}\smallsetminus \overline{\calX_{n, \leq \bfitw_{i-1}}^{\tor}} ,\proket}(\calX_n^{\tor} \smallsetminus \overline{\calX_{n, \leq \bfitw_{i-1}}^{\tor}}, \scrO\!\!\scrM) \cong R\Gamma_{\calX_{n, \bfitw_i, (0,\overline{0})}^{\tor}, \proket}(\calX_{n, \geq \bfitw_i}^{\tor}, \scrO\!\!\scrM).
\]

For $\bfitu \in \{\bfitu_{p, 0}, \bfitu_{p,1}, \bfitu_p\}$, thanks to the description of the $\bfitu$-action on $M$ in \S \ref{subsection: Betti overconvergent cohomology} and Lemma \ref{Lemma: dynamics of U_p-operators}, we obtain a morphism of complexes (see also \cite[\S 5.2]{BP-HigherColeman}) \[
    R\Gamma_{\overline{\calX_{n, \leq \bfitw_i}^{\tor}} \cap (\calX_{n, \geq \bfitw_i}^{\tor})\bfitu, \proket}((\calX_{n, \geq \bfitw_i}^{\tor})\bfitu, \scrO\!\!\scrM) \xrightarrow{U} R\Gamma_{(\overline{\calX_{n, \leq \bfitw_i}^{\tor}})\bfitu^{-1} \cap \calX_{n, \geq \bfitw_i}^{\tor}\bfitu, \proket}(\calX_{n, \geq \bfitw_i}^{\tor}, \scrO\!\!\scrM).
\]
When $\bfitu = \bfitu_p$, we arrive at a diagram \begin{equation}\label{eq: diagram for defining finite-slope part on cohomology with support}
    \begin{tikzcd}
        & \scalemath{0.6}{ R\Gamma_{\calX_{n, \bfitw_i, (0,\overline{0})}^{\tor}, \proket}(\calX_{n, \geq \bfitw_i}^{\tor}, \scrO\!\!\scrM) } \arrow[ld, "\mathrm{Res}"']\\
        \scalemath{0.6}{ R\Gamma_{\overline{\calX_{n, \leq \bfitw_i}^{\tor}} \cap (\calX_{n, \geq \bfitw_i}^{\tor})\bfitu_p, \proket}((\calX_{n, \geq \bfitw_i}^{\tor})\bfitu_p, \scrO\!\!\scrM) } \arrow[rr, "U_p"] && \scalemath{0.6}{ R\Gamma_{(\overline{\calX_{n, \leq \bfitw_i}^{\tor}})\bfitu_p^{-1} \cap \calX_{n, \geq \bfitw_i}^{\tor}\bfitu, \proket}(\calX_{n, \geq \bfitw_i}^{\tor}, \scrO\!\!\scrM) } \arrow[lu, "\mathrm{Cores}"']\arrow[ld, "\mathrm{Res}"]\\
        & \scalemath{0.6}{ R\Gamma_{(\overline{\calX_{n, \leq \bfitw_i}^{\tor}})\bfitu_p^{-1} \cap (\calX_{n, \geq \bfitw_i}^{\tor})\bfitu_p, \proket}((\calX_{n, \geq \bfitw_i}^{\tor})\bfitu_p, \scrO\!\!\scrM) } \arrow[lu, "\mathrm{Cores}"]
    \end{tikzcd},
\end{equation}
where the composition on the top coincides with the composition at the bottom. By abuse of notation we still denote by $U_p$ the operator $\mathrm{Cores} \circ U_p \circ \mathrm{Res}$ acting on $R\Gamma_{\calX_{n, \bfitw_i, (0,\overline{0})}^{\tor}, \proket}(\calX_{n, \geq \bfitw_i}^{\tor}, \scrO\!\!\scrM)$.

\begin{Proposition}\label{Proposition: Up is nice on proket cohomology with support with big coefficients}
    Let $\bfitw_i \in W^H$. Let $(R_{\calV}, \kappa_{\calV})$ be an affinoid weight and let $r\in \Q_{\geq 0}$ such that $r> 1+r_{\calV}$. Suppose $\scrO\!\!\scrM \in \{\scrO\!\!\scrD_{\kappa_{\calV}}^{r^+}, \scrO\!\!\scrA_{\kappa_{\calV}}^r\}$. The following statements hold. 
    \begin{enumerate}
        \item[(i)]  The complex $ R\Gamma_{\overline{\calX_{n, \leq \bfitw_i}^{\tor}} \smallsetminus \overline{\calX_{n, \leq \bfitw_{i-1}}^{\tor}}, \proket}(\calX_{n}^{\tor}\smallsetminus \overline{\calX^{\tor}_{n, \leq \bfitw_{i-1}}}, \scrO\!\!\scrM)$ is represented by an object in $\mathrm{Pro}_{\Z_{\geq 0}}(\mathrm{K}^{\proj}(\Ban(R_{\calV})))$.
        \item[(ii)] $U_p$ is a potent compact operator on $R\Gamma_{\overline{\calX_{n, \leq \bfitw_i}^{\tor}} \smallsetminus \overline{\calX_{n, \leq \bfitw_{i-1}}^{\tor}}, \proket}(\calX_{n}^{\tor}\smallsetminus \overline{\calX^{\tor}_{n, \leq \bfitw_{i-1}}}, \scrO\!\!\scrM)$. 
    \end{enumerate}
\end{Proposition}
\begin{proof}
    For (i), one first notices that $\scrO\!\!\scrM$ is an ON-able sheaf of Banach $\widehat{\scrO}_{\calX_{n, \proket}^{\tor}}\widehat{\otimes}R_{\calU}$-modules (in the sense of Definition \ref{Definition: proket Banach sheaves}) by Lemma \ref{Lemma: structure theorem of OD^{r^+} and OA^r}. We would like to apply Proposition \ref{Proposition: proket cohomology with supp for Banach sheaves}. Notice that the complex $R\Gamma_{\overline{\calX_{n, \leq \bfitw_i}^{\tor}} \smallsetminus \overline{\calX_{n, \leq \bfitw_{i-1}}^{\tor}}, \proket}(\calX_{n}^{\tor}\smallsetminus \overline{\calX^{\tor}_{n, \leq \bfitw_{i-1}}}, \scrO\!\!\scrM)$ is quasi-isomorphic to the complex $R\Gamma_{\calX_{n, \bfitw_i, (0,\overline{0})}^{\tor}, \proket}(\calX_{n, \geq \bfitw_i}^{\tor}, \scrO\!\!\scrM)$. It remains to check $\calX_{n, \geq \bfitw_i}^{\tor}$ and $\overline{\calX_{n, \leq \bfitw_i}^{\tor}}$ satisfy the conditions therein. Indeed, there is a commutative diagram (see \cite[\S 4.4]{BP-HigherColeman}) \[
        \begin{tikzcd}
            & \calX_{\Gamma(p^{\infty})}^{\tor} \arrow[d, "\pi_{\min}^{\tor}"]\arrow[rd, "\pi_{\HT}"]\arrow[ldd, "h_n"']\\
            & \calX_{\Gamma(p^{\infty})}^{\min} \arrow[r, "\pi_{\HT}^{\min}"']\arrow[ldd] & \adicFL\\
            \calX_{n}^{\tor}\arrow[d, "\pi_{\min}^{\tor}"']\\
            \calX_{n}^{\min}
        \end{tikzcd},
    \]
    where $\calX_{n}^{\min}$ is the minimal compactification of $\calX_n$, $\calX_{\Gamma(p^{\infty})}^{\min}$ is the associated (minimally compactified) perfectoid Siegel modular variety constructed in \cite{Scholze-2015}, and $\pi_{\HT}^{\min}$ is the Hodge--Tate period map. Moreover, note that $\pi_{\min}^{\tor}$ is a finite morphism and $\pi_{\HT}^{\min}$ is affine. The desired properties follow. 

    For the second assertion, since $\bfitu_p$ is a compact operator on $M \in \{D_{\kappa_{\calU}}^{r^+}, A_{\kappa_{\calU}}^r\}$, it is enough to show that the `corestriction-restriction' map is compact. However, this is exactly Proposition \ref{Proposition: corestriction-restriction is a compact map}. Note that, to check the subspaces satisfy the conditions therein, one applies the same argument as in \cite[Theorem 5.4.3]{BP-HigherColeman}.
\end{proof}

Thanks to Proposition \ref{Proposition: Up is nice on proket cohomology with support with big coefficients}, when $(R_{\calU}, \kappa_{\calU})$ is an affinoid weight, we may consider the finite-slope part $R\Gamma_{\overline{\calX_{n, \leq \bfitw_i}^{\tor}} \smallsetminus \overline{\calX_{n, \leq \bfitw_{i-1}}^{\tor}}, \proket}(\calX_{n}^{\tor}\smallsetminus \overline{\calX^{\tor}_{n, \leq \bfitw_{i-1}}}, \scrO\!\!\scrD_{\kappa_{\calU}}^{r^+})^{\fs}$ with respect to $U_p$ as in Proposition-Definition \ref{Prop-Defn: finite slope part pro}. Since the slope-$\leq h$ decomposition on $D_{\kappa_{\calU}}^{r^+}$ is independent of $r$ \footnote{ This follows from similar arguments as in \cite[\S 3.1]{Hansen-PhD}.}, for $s\geq r >1+r_{\calU}$, the natural map $\scrO\!\!\scrD_{\kappa_{\calU}}^{s^+} \rightarrow \scrO\!\!\scrD_{\kappa_{\calU}}^{r^+}$ gives rise to a quasi-isomorphism \[
   \scalemath{0.9}{ R\Gamma_{\overline{\calX_{n, \leq \bfitw_i}^{\tor}} \smallsetminus \overline{\calX_{n, \leq \bfitw_{i-1}}^{\tor}}, \proket}(\calX_{n}^{\tor}\smallsetminus \overline{\calX^{\tor}_{n, \leq \bfitw_{i-1}}}, \scrO\!\!\scrD_{\kappa_{\calU}}^{s^+})^{\fs} \xrightarrow{\cong} R\Gamma_{\overline{\calX_{n, \leq \bfitw_i}^{\tor}} \smallsetminus \overline{\calX_{n, \leq \bfitw_{i-1}}^{\tor}}, \proket}(\calX_{n}^{\tor}\smallsetminus \overline{\calX^{\tor}_{n, \leq \bfitw_{i-1}}}, \scrO\!\!\scrD_{\kappa_{\calU}}^{r^+})^{\fs}. }
\]
Hence, \eqref{eq: sandwiches of pro-Kummer \'etale sheaves for OA and OD} yields a commutative diagram \[
    \begin{tikzcd}[column sep = tiny]
        \scalemath{0.6}{ R\Gamma_{\overline{\calX_{n, \leq \bfitw_i}^{\tor}} \smallsetminus \overline{\calX_{n, \leq \bfitw_{i-1}}^{\tor}}, \proket}(\calX_{n}^{\tor}\smallsetminus \overline{\calX^{\tor}_{n, \leq \bfitw_{i-1}}}, \scrO\!\!\scrD_{\kappa_{\calU}}^{s^+}) } \arrow[r]\arrow[d] & \scalemath{0.6}{ R\Gamma_{\overline{\calX_{n, \leq \bfitw_i}^{\tor}} \smallsetminus \overline{\calX_{n, \leq \bfitw_{i-1}}^{\tor}}, \proket}(\calX_{n}^{\tor}\smallsetminus \overline{\calX^{\tor}_{n, \leq \bfitw_{i-1}}}, \scrO\!\!\scrD_{\kappa_{\calU}}^{r}) } \arrow[r] & \scalemath{0.6}{ R\Gamma_{\overline{\calX_{n, \leq \bfitw_i}^{\tor}} \smallsetminus \overline{\calX_{n, \leq \bfitw_{i-1}}^{\tor}}, \proket}(\calX_{n}^{\tor}\smallsetminus \overline{\calX^{\tor}_{n, \leq \bfitw_{i-1}}}, \scrO\!\!\scrD_{\kappa_{\calU}}^{r^+}) } \arrow[d] \\
        \scalemath{0.6}{ R\Gamma_{\overline{\calX_{n, \leq \bfitw_i}^{\tor}} \smallsetminus \overline{\calX_{n, \leq \bfitw_{i-1}}^{\tor}}, \proket}(\calX_{n}^{\tor}\smallsetminus \overline{\calX^{\tor}_{n, \leq \bfitw_{i-1}}}, \scrO\!\!\scrD_{\kappa_{\calU}}^{s^+})^{\fs} }  \arrow[rr, "\cong"] && \scalemath{0.6}{ R\Gamma_{\overline{\calX_{n, \leq \bfitw_i}^{\tor}} \smallsetminus \overline{\calX_{n, \leq \bfitw_{i-1}}^{\tor}}, \proket}(\calX_{n}^{\tor}\smallsetminus \overline{\calX^{\tor}_{n, \leq \bfitw_{i-1}}}, \scrO\!\!\scrD_{\kappa_{\calU}}^{r^+})^{\fs} }
    \end{tikzcd}.
\]
We then define \[
    \scalemath{0.9}{R\Gamma_{\overline{\calX_{n, \leq \bfitw_i}^{\tor}} \smallsetminus \overline{\calX_{n, \leq \bfitw_{i-1}}^{\tor}}, \proket}(\calX_{n}^{\tor}\smallsetminus \overline{\calX^{\tor}_{n, \leq \bfitw_{i-1}}}, \scrO\!\!\scrD_{\kappa_{\calU}}^{r})^{\fs} := R\Gamma_{\overline{\calX_{n, \leq \bfitw_i}^{\tor}} \smallsetminus \overline{\calX_{n, \leq \bfitw_{i-1}}^{\tor}}, \proket}(\calX_{n}^{\tor}\smallsetminus \overline{\calX^{\tor}_{n, \leq \bfitw_{i-1}}}, \scrO\!\!\scrD_{\kappa_{\calU}}^{r^+})^{\fs} }
\]
and note that this definition is independent to $r$.

\begin{Corollary}\label{Corollary: finite-slope part for cohomology with support is independent to the support condition}
    Every morphism in the diagram \eqref{eq: diagram for defining finite-slope part on cohomology with support} induces a quasi-isomorphism on the finite-slope parts. 
\end{Corollary}
\begin{proof}
    The proof is the same as in \cite[Corollary 5.3.2]{BP-HigherColeman}.
\end{proof}

\begin{Remark}\label{Remark: finite slope part for proket cohomology with support with big coeff}
    When $(R_{\calU}, \kappa_{\calU})$ is an affinoid weight, we can similarly define the \emph{finite-slope part} of the following complexes.\begin{itemize}
        \item $R\Gamma_{\overline{\calX_{n, \leq \bfitw_i}^{\tor}} \smallsetminus \overline{\calX_{n, \leq \bfitw_{i-1}}^{\tor}}, \proket}(\calX_{n}^{\tor}\smallsetminus \overline{\calX^{\tor}_{n, \leq \bfitw_{i-1}}}, \scrO\!\!\scrM)$ for $\scrO\!\!\scrM\in \{ \scrO\!\!\scrA_{\kappa_{\calU}}^{r^+}, \scrO\!\!\scrA_{\kappa_{\calU}}^r \}$; 
        \item $R\Gamma_{\calZ_{n, \bfitw_i}, \proket}(\calX^{\tor, \bfitu_p}_{n,\bfitw}, \scrO\!\!\scrM)$ for $\scrO\!\!\scrM\in \{ \scrO\!\!\scrA_{\kappa_{\calU}}^{r^+}, \scrO\!\!\scrA_{\kappa_{\calU}}^r, \scrO\!\!\scrD_{\kappa_{\calU}}^{r+}, \scrO\!\!\scrD_{\kappa_{\calU}}^r \}$.
    \end{itemize}
\end{Remark}

\begin{Theorem}\label{Theorem: change support condition for overconvergent cohomology}
    Let $\bfitw_i\in W^H$ with $i=0, ..., 3$. Let $(R_{\calU}, \kappa_{\calU})$ be an affinoid weight. Let $r\in \Q_{\geq 0}$ and $n\in \Z_{>0}$ such that $n\geq r >1+r_{\calU}$. There is a natural quasi-isomorphism \[
        R\Gamma_{\overline{\calX_{n, \leq \bfitw_i}^{\tor}} \smallsetminus \overline{\calX_{n, \leq \bfitw_{i-1}}^{\tor}}, \proket}(\calX_{n}^{\tor}\smallsetminus \overline{\calX^{\tor}_{n, \leq \bfitw_{i-1}}}, \scrO\!\!\scrD_{\kappa_{\calU}}^r)^{\fs} \cong R\Gamma_{\calZ_{n, \bfitw_i}, \proket}(\calX^{\tor, \bfitu_p}_{n,\bfitw}, \scrO\!\!\scrD_{\kappa_{\calU}}^r)^{\fs}.
    \]
A similar statement holds by replacing $\scrO\!\!\scrD_{\kappa_{\calU}}^r$ with $\scrO\!\!\scrA_{\kappa_{\calU}}^r$.
\end{Theorem}
\begin{proof}
    Given Proposition \ref{Proposition: Up is nice on proket cohomology with support with big coefficients}, one argues in a similar way as in \cite[Theorem 5.4.12]{BP-HigherColeman}. We leave the details to the reader. 
\end{proof}
\section{The overconvergent Eichler--Shimura morphisms}\label{section: OES}

In this section, we construct the overconvergent Eichler--Shimura morphisms which relate the overconvergent cohomology groups constructed in \S \ref{section: OC} to the cohomology of automorphic sheaves constructed in \S \ref{section: automorphic}. As mentioned in \S \ref{subsection: main results}, these morphisms are induced from Hecke- and Galois-equivariant morphisms
\[
    \mathrm{ES}_{\kappa_{\calU}}^{\bfitw, r} \colon \scrO\!\!\scrD_{\kappa_{\calU}}^r \rightarrow \widehat{\underline{\omega}}_{n, r}^{\bfitw_3^{-1}\bfitw\kappa_{\calU}}(\bfitw\kappa_{\calU}^{\cyc})
\]
of sheaves on the pro-Kummer \'etale site $\calX_{n, \bfitw, (r, r), \proket}^{\tor}$.

We start in \S \ref{subsection: classical ES} with a quick review of the classical Eichler--Shimura decomposition of Faltings--Chai, followed by a reinterpretation of their decomposition in our setup. These observations inspire our main constructions and will be useful when we study the decompositions around a nice-enough point on the eigenvariety. In \S \ref{subsection: OES}, we construct the morphisms $\mathrm{ES}_{\kappa_{\calU}}^{\bfitw, r}$ and the overconvergent Eichler--Shimura morphisms. They serve as $p$-adic interpolations of the classical picture. In \S \ref{subsection: OES at classical weights}, we study the behaviour of these morphisms when specialising at classical weights. Finally, in \S \ref{subsection: eigenvarieties} and \S \ref{subsection: ES filtration on eigenvariety}, we study the equidimensional eigenvariety and prove decomposition results around a nice-enough point on the eigenvariety. As an application, we propose a new way to construct big Galois representations and read of their Hodge--Tate--Sen weights via the overconvergent Eichler--Shimura morphisms.

\subsection{The classical Eichler--Shimura morphisms}\label{subsection: classical ES}

For $\bfitw\in W^H$, recall that \[
    k_{\bfitw} = \left\{ \begin{array}{ll}
        (0,0), & \text{ if }\bfitw = \bfitw_0 = \one_4  \\
        (2,0), & \text{ if }\bfitw = \bfitw_1 \\ 
        (3,1), & \text{ if }\bfitw = \bfitw_2 \\
        (3,3), & \text{ if }\bfitw = \bfitw_3
    \end{array}\right. 
\] 
For a weight $\kappa_{\calU}=(\kappa_{\calU, 1}, \kappa_{\calU, 2})$, recall the `cyclotomic twist' of $\kappa_{\calU}$ defined by
\[
    \bfitw\kappa_{\calU}^{\cyc} = \left\{ \begin{array}{ll}
        0, & \text{if }\bfitw = \bfitw_3 \\
        \kappa_{\calU, 2}, & \text{if }\bfitw = \bfitw_2\\
        \kappa_{\calU, 1}, & \text{if }\bfitw = \bfitw_1\\
        \kappa_{\calU,1}+\kappa_{\calU, 2} & \text{if }\bfitw = \bfitw_0 = \one_4
    \end{array} \right. 
\] 
There is a similar notion for integral weights $k=(k_1, k_2)$. For integral weights $k=(k_1, k_2; k_0)$ \footnote{We remind the reader that $(k_1, k_2)\in\Z^2$ is identified with $(k_1, k_2; 0)\in \Z^3$ as in \S \ref{subsection: weight space and analytic representations}.}, we also recall the classical automorphic sheaves $\underline{\omega}^k$ constructed in \S \ref{subsection: classical Siegel modular forms}.

We have the following theorem by Falting--Chai (\cite[Chapter VI, Theorem 6.2]{Faltings-Chai}). 

\begin{Theorem}[$p$-adic Eichler--Shimura decomposition for $\GSp_4$]\label{Theorem: Faltings--Chai's ES decomposition}
    Let $k = (k_1, k_2)\in \Z^{2}$ such that $k_1 \geq k_2>0$.  Let $V_k$ be the $\GSp_4$-representation of highest weight $k$; i.e., \[
        V_k := \left\{f: \GSp_4 \rightarrow \bbA^1: f(\bfgamma\bfbeta) = k(\bfbeta)f(\bfgamma) \text{ for all }(\bfgamma, \bfbeta)\in \GSp_4\times B_{\GSp_4}\right\}.
    \] Let $V_k^{\vee}$ be the dual of $V_k$.\footnote{The left $\GSp_4$-action on $V_k^{\vee}$ is given by the left-translation of $\GSp_4$ on $V_k$. } Then there exists a Hecke- and Galois-stable $4$-step decreasing filtration $\Fil^{\bullet}_{\mathrm{ES}}$ on $H_{\et}^3 (X_{n, \C_p}, V_k^{\vee})\otimes_{\Q_p}\C_p$ which induces isomorphisms
    \[
        \Gr_{\mathrm{ES}}^{3-i} \cong H^{3-i}(X_{n, \C_p}^{\tor}, \underline{\omega}^{\bfitw_3^{-1}\bfitw_i k + k_{\bfitw_i}})(\bfitw_ik^{\cyc}-i),
     \] 
   $i=0, 1, 2, 3$, on the graded pieces. This induces a Hecke- and Galois-equivariant decomposition \[
        H_{\et}^3 (X_{n, \C_p}, V_k^{\vee})\otimes_{\Q_p}\C_p
        \cong \bigoplus_{i=0}^3 \Gr_{\mathrm{ES}}^{3-i}.
    \]
\end{Theorem}

Our goal is to interpolate this decomposition in $p$-adic families. Faltings--Chai's proof of Theorem \ref{Theorem: Faltings--Chai's ES decomposition} uses the dual BGG resolution and the comparison theoerm between $p$-adic \'etale cohomology and $p$-adic de Rham cohomology. Below, we propose an alternative way to understand this theorem (after localising at a \emph{nice-enough} automorphic representation) which does not use the dual BGG resolution or any comparison theorems from $p$-adic Hodge theory. It will become clear how such an interpretation inspires our construction of the $p$-adic interpolations.

In what follows, we will often assume the following conditions hold for certain automorphic representations. 

\begin{Assumption}\label{Assumption: Multiplicity One for cuspidal automorphic representations}
    Let $\Pi = (\pi = \otimes_{v}' \pi_v, \varphi_p) $ be a datum consisting of an irreducible cuspidal automorphic representation $\pi$ of $\GSp_4(\A_{\Q})$ and a vector $\varphi_p\in \pi_p$ such that \begin{enumerate}
        \item[(i)] $\pi$ is of cohomological weight $k = (k_1, k_2)\in \Z^2$ with $k_1 \geq k_2>0$; 
        \item[(ii)] $\dim \pi^{\Gamma_{\ell}}_{\ell} =1$ for all $\ell \neq p$, in particular, $\pi$ is spherical outside $pN$; 
        \item[(iii)] $ \varphi_p \in \pi_p^{\Iw_{\GSp_4, n}^+}$ and it has non-zero $U_{p, i}$-eigenvalues.
    \end{enumerate} 
    Such a datum $\Pi = (\pi, \varphi_p)$ is called a \emph{$p$-stabilisation} of $\pi$ (although we do not require $\Pi$ being spherical at $p$). Let $\bbT \coloneq \left(\bigotimes_{\ell\neq p} \Z_p[\Gamma_{\ell} \backslash \GSp_{4}(\Q_{\ell})/\Gamma_{\ell}]\right) \otimes \Z_p[U_{p, 0}, U_{p,1}]$ be the abstract Hecke algebra. Let $\frakm_{\Pi}$ be the maximal ideal of the Hecke algebra defined by $(\pi, \varphi_p)$; that is $(\pi, \varphi_p)$ defines a Hecke eigensystem $\lambda_{\Pi}: \bbT\otimes K \rightarrow K$ (for some field $K\supset \Q_p$ living in $\C_p \cong \C$) and $\frakm_{\Pi}$ is the kernel of $\lambda_{\Pi}$. We assume that for every $\bfitw\in W^H$, \[
            \dim_{\C_p} H^{3-l(\bfitw)}(X_{n, \C_p}^{\tor}, \underline{\omega}^{\bfitw_3^{-1}\bfitw k + k_{\bfitw}})_{\frakm_{\Pi}} = 1.
        \]
\end{Assumption}

\begin{Remark}\label{Remark:1dimensional}
   Assumption \ref{Assumption: Multiplicity One for cuspidal automorphic representations} is a \emph{multiplicity-one} assumption; a similar assumption can also be found in \cite[\S 12]{GT-TWGSp4}. We remark the following: \begin{enumerate}
       \item[(i)] There exist some CAP representations whose corresponding eigensystems appear in $H^3_{\et}$ but they do not appear in all four degrees of coherent cohomology (cf. \cite[Hypothesis A (7)]{Weissauer-4dimGalRep}). We believe our method can be extended to this case. However, due to the length of the paper, we leave it to the interested reader. 
       \item[(ii)] In general, we do not not know about the newform theory for $\GSp_4$. However, in \cite{RS-NewformGSp4}, Robert--Schimidt developed a (local) newform theory for the representations of paramodular level. We point out that paramodular levels are not neat levels while we have chosen $\Gamma$ to be a neat level. We explain in \S \ref{subsection: non-neat level} how one can obtain similar results (such as Theorem \ref{Theorem: big OES diagram}) for non-neat levels.
       \item[(iii)] If the representation $\pi$ is generic, meaning it admits a Whittaker model (see \cite[\S 0.5]{Soudry}), then it is known that $\pi$ satisfies strong multiplicity one \cite[Theorem 1.5]{Soudry}, meaning that if one consider another generic $\pi'$ such the local components $\pi_v$ and $\Pi'_v$ are isomorphic for almost all $v$, then $\pi=\pi'$. Moreover, if the level is paramodular, we know by \cite[Theorem 4.5]{RosnerWissauer} that there is no non-generic automorphic representation isomorphic to $\pi$ almost everywhere. This means that if $\pi$ is paramodular, then it satisfies our \emph{multiplicity-one} assumption. 
       It is a folklore expectation that if $\pi$ is generic and non-endoscopic, the same strong multiplicity one result among all representations (not only generic) should hold.
   \end{enumerate}  
\end{Remark}

\begin{Corollary}\label{Corollary: there exists a unique ES decomposition}
    Suppose $\Pi = (\pi, \varphi_p)$ satisfies Assumption \ref{Assumption: Multiplicity One for cuspidal automorphic representations}. 
    There exists a unique Hecke- and Galois-stable $4$-step decreasing filtration $\Fil^{\bullet}_{\mathrm{ES}, k, \frakm_{\Pi}}$ of $H_{\et}^3 (X_{n, \C_p}, V_k^{\vee})_{\frakm_{\Pi}}\otimes_{\Q_p}\C_p$ which induces a Hecke- and Galois-equivariant isomorphism \[
        \Gr_{\mathrm{ES}, k, \frakm_{\Pi}}^{3-i} \cong H^{3-i}(X_{n, \C_p}^{\tor}, \underline{\omega}^{\bfitw_3^{-1}\bfitw_i k + k_{\bfitw_i}})_{\frakm_{\Pi}}(\bfitw_i k^{\cyc} - i)\] for $i=0, 1, 2, 3$, on the graded pieces. Moreover, the filtration induces a Hecke- and Galois-equivariant decomposition  \[
        H_{\et}^3 (X_{n, \C_p}, V_k^{\vee})_{\frakm_{\Pi}}\otimes_{\Q_p}\C_p  \cong \bigoplus_{i=0}^3 H^{3-i}(X_{n, \C_p}^{\tor}, \underline{\omega}^{\bfitw_3^{-1}\bfitw_i k + k_{\bfitw_i}})_{\frakm_{\Pi}}(\bfitw_i k^{\cyc} - i).
    \]
\end{Corollary}
\begin{proof} 
This is an immediate corollary of Theorem \ref{Theorem: Faltings--Chai's ES decomposition}. The uniqueness follows from the fact that $\frakm_{\Pi}$ has cohomological weight $k_1\geq k_2 >0$ and hence the Hodge--Tate weights $\{\bfitw_i k^{\cyc} - i\colon i=0,1,2,3\}$ are distinct.
\end{proof}

In the rest of \S \ref{subsection: classical ES}, we propose the constructions of a family of maps, by which we name \emph{classical Eichler--Shimura morphisms}. We shall see how these constructions recover the Eichler--Shimura filtration/decomposition in Corollary \ref{Corollary: there exists a unique ES decomposition}.

We split the construction into five steps. Recall that $\calX_n$ and $\calX_{n}^{\tor}$ stand for the rigid analytic space over $\Spa(\C_p, \calO_{\C_p})$ associated with $X_n$ and $X_n^{\tor}$, respectively. \\

\paragraph{\textbf{Construction 1.}} First of all, analogous to our discussion in \S \ref{subsection: Kummer and pro-Kummer \'etale overconvergent cohomology}, we consider the \'etale local system $\scrV_k^{\vee}$ on $\calX_{n, \et}$ associated with $V_k^{\vee}$ and consider the the pro-Kummer \'etale sheaf \[
    \scrO\!\!\scrV_k^{\vee} := \nu^{-1}\jmath_{\ket, *}\scrV_k^{\vee} \otimes_{\Q_p} \widehat{\scrO}_{\calX_{n, \proket}^{\tor}}
\] 
where $\jmath_{\ket}:\calX_{n, \et}\rightarrow \calX_{n, \ket}^{\tor}$ and $\nu:\calX_{n, \proket}^{\tor}\rightarrow \calX_{n, \ket}^{\tor}$ are natural morphism of sites. Similar to Proposition \ref{Proposition: comparing Kummer \'etale and pro-Kummer \'etale cohomology}, there is a natural isomorphism \[
    H_{\et}^3(\calX_n, V_k^{\vee})\otimes_{\Q_p}\C_p \cong H_{\proket}^3(\calX_n^{\tor}, \scrO\!\!\scrV_k^{\vee}).
\]

\paragraph{\textbf{Construction 2.}} On the other hand, we consider the completed pullback of the classical automorphic sheaves to the pro-Kummer \'etale site. For $k' \in \{\bfitw_3^{-1}\bfitw k: \bfitw\in W^H\}$, similar to Remark \ref{Remark: pro-Kummer \'etale automorphic sheaf}, we consider
\[
    \widehat{\underline{\omega}}^{k'} := \upsilon^{-1}\underline{\omega}^{k'} \otimes_{\upsilon^{-1}\scrO_{\calX_n^{\tor}}} \widehat{\scrO}_{\calX_{n, \proket}^{\tor}}
\] 
where $\upsilon: \calX_{n, \proket}^{\tor}\rightarrow \calX_{n, \an}^{\tor}$ is the natural projection of sites. There is a Leray spectral sequence 
\begin{equation}\label{eq: Leray sp sq}
    E_2^{i,j} =H^i(\calX_n^{\tor}, R^j\upsilon_* \widehat{\underline{\omega}}^{k'}) \Rightarrow H^{i+j}_{\proket}(\calX_n^{\tor}, \widehat{\underline{\omega}}^{k'}).
\end{equation}
By the projection formula and \cite[Proposition A.2.3]{DRW}, we have \[
    R^j\upsilon_* \widehat{\underline{\omega}}^{k'} \cong \underline{\omega}^{k'}\otimes R^j \upsilon_* \widehat{\scrO}_{\calX_{n, \proket}^{\tor}} \cong \underline{\omega}^{k'} \otimes \Omega_{\calX_n^{\tor}}^{\log, j}(-j).
\] The spectral sequence becomes \begin{equation}\label{eq: Leray spectral sequence for classical cohomology}
    E_2^{i,j} = H^i(\calX_n^{\tor}, \underline{\omega}^{k'} \otimes \Omega_{\calX_n^{\tor}}^{\log, j})(-j) \Rightarrow H^{i+j}_{\proket}(\calX_n^{\tor}, \widehat{\underline{\omega}}^{k'}).
\end{equation}

\paragraph{\textbf{Construction 3.}} Since $k_2>0$, we can apply \cite[Theorem 4.1]{Lan-Vanishing} (see also [\emph{op. cit.}, Example 4.17]) and obtain \[
    \begin{array}{rl}
        H^i(\calX_n^{\tor}, \underline{\omega}^{(k_1, -k_2; k_2)} \otimes \Omega_{\calX_n^{\tor}}^{\log, j}) = 0 & \text{ for }i=0\\
        H^i(\calX_n^{\tor}, \underline{\omega}^{(k_2, -k_1; k_1)} \otimes \Omega_{\calX_n^{\tor}}^{\log, j}) = 0 & \text{ for }i=0, 1\\
        H^i(\calX_n^{\tor}, \underline{\omega}^{(-k_2, -k_1; k_1+k_2)} \otimes \Omega_{\calX_n^{\tor}}^{\log, j}) = 0 & \text{ for }i=0, 1, 2
    \end{array}.
\]
As a result, the spectral sequences \eqref{eq: Leray spectral sequence for classical cohomology} give rise to edge maps \begin{equation}\label{eq: edge maps for classical cohomology}
    \begin{array}{rl}
        H^3_{\proket}(\calX_n^{\tor}, \widehat{\underline{\omega}}^{(k_1, k_2; 0)}) & \rightarrow H^0(\calX_n^{\tor}, \underline{\omega}^{(k_1, k_2; 0)} \otimes \Omega_{\calX_n^{\tor}}^{\log, 3})(-3)\\
        H^3_{\proket}(\calX_n^{\tor}, \widehat{\underline{\omega}}^{(k_1, -k_2; k_2)}) & \rightarrow H^1(\calX_n^{\tor}, \underline{\omega}^{(k_1, -k_2; k_2)} \otimes \Omega_{\calX_n^{\tor}}^{\log, 2})(-2)\\
        H^3_{\proket}(\calX_n^{\tor}, \widehat{\underline{\omega}}^{(k_2, -k_1; k_1)}) & \rightarrow H^2(\calX_n^{\tor}, \underline{\omega}^{(k_2, -k_1; k_1)}\otimes \Omega_{\calX_n^{\tor}}^{\log, 1})(-1)\\
        H^3_{\proket}(\calX_n^{\tor}, \widehat{\underline{\omega}}^{(-k_2, -k_1; k_1+k_2)}) & \rightarrow H^3(\calX_n^{\tor}, \underline{\omega}^{(-k_2, -k_1; k_1+k_2)}).
    \end{array}
\end{equation}
Namely,
\[H^3_{\proket}(\calX_n^{\tor}, \widehat{\underline{\omega}}^{\bfitw_3^{-1}\bfitw_i k})  \rightarrow H^{3-i}(\calX_n^{\tor}, \underline{\omega}^{\bfitw_3^{-1}\bfitw_i k} \otimes \Omega_{\calX_n^{\tor}}^{\log, i})(-i)\]
for $i=0,1,2,3$.
Note that the targets of these maps further project to $H^{3-i}(\calX_n^{\tor}, \underline{\omega}^{\bfitw_3^{-1}\bfitw_i k + k_{\bfitw_i}})(-i)$ via the Kodaira--Spencer isomorphism (\cite[Theorem 1.41 (4)]{LanKS}).\\

\paragraph{\textbf{Construction 4.}} For $\bfitw\in W^H$, we construct a Hecke- and Galois-equivariant morphism of pro-Kummer \'etale sheaves \begin{equation}\label{eq: algebraic ES on the level of pro-Kummer \'etale sheaves}
    \ES_{k}^{\bfitw, \alg}: \scrO\!\!\scrV_k^{\vee} \rightarrow \widehat{\underline{\omega}}^{\bfitw_3^{-1}\bfitw k}(\bfitw k^{\cyc})
\end{equation}
on $\calX_{n, \proket}^{\tor}$. It serves as a bridge connecting the objects studied in Construction 1 \& 2. In fact, we will make the construction on the flag variety and then pullback along the Hodge--Tate period map.

Consider the pullback diagram \[
    \begin{tikzcd}
        \iota^{\bfitw, *}_{\bfitw_3}\calH_{\HT}^{\an}\arrow[r]\arrow[d, "\iota^{\bfitw, *}_{\bfitw_3}\pr_{\adicFL, \HT}"'] & \calH_{\HT}^{\an}\arrow[d, "\pr_{\adicFL, \HT}"]\\
        \adicFL \arrow[r, "\iota^{\bfitw}_{\bfitw_3}"] & \adicFL
    \end{tikzcd},
\]
where $\iota^{\bfitw}_{\bfitw_3}$ is the antomorphism in Remark \ref{Remark: maps between loci on the flag variety} given by multiplying $\bfitw^{-1}\bfitw_3$ from the right. Since $\calH_{\HT}^{\an} \rightarrow \adicFL$ is an $\calH^{\an}$-torsor, the pullback $\iota^{\bfitw, *}_{\bfitw_3}\calH_{\HT}^{\an} \rightarrow \adicFL$ is also a $\calH^{\an}$-torsor, where $\calH^{\an}$ acts via $\bfitw^{-1}\bfitw_3 \calH^{\an}\bfitw_3^{-1}\bfitw$. Given $k = (k_1, k_2; k_0)\in \Z^3$ with $k_1\geq k_2$, let \[
    \underline{\omega}_{\adicFL}^k := \pr_{\adicFL, \HT, *}\scrO_{\calH_{\HT}^{\an}}[\bfitw_3 k],
\]\emph{i.e.}, the subsheaf of $\pr_{\adicFL, \HT, *}\scrO_{\calH_{\HT}^{\an}}$ consisting of sections on which $\calB_H^{\an}$ acts via $\bfitw_3 k$. There is a natural isomorphism \[
    \iota^{\bfitw, *}_{\bfitw_3} \underline{\omega}_{\adicFL}^k \cong \underline{\omega}^{\bfitw_3^{-1}\bfitw k}_{\adicFL}.
\]

Now, recall the universal short exact sequence \[
    0 \rightarrow \scrW_{\adicFL}^{\vee} \rightarrow \scrO_{\adicFL}^4 \xrightarrow{\HT_{\adicFL}} \scrW_{\adicFL} \rightarrow 0
\] over $\adicFL$. Fix $k = (k_1, k_2)\in \Z^2$ with $k_1\geq k_2>0$. It is well-known that (see also Remark \ref{Remark: explanation on the weights}) \[
    \underline{\omega}_{\adicFL}^k = \Sym^{k_1-k_2}\scrW_{\adicFL} \otimes (\det \scrW_{\adicFL})^{\otimes k_2}.
\] The map $\HT_{\adicFL}$ induces a map \[
    \HT_{\adicFL}^k: \Sym^{k_1-k_2}\scrO_{\adicFL}^4\otimes \Sym^{k_2} \wedge^2\scrO_{\adicFL}^4 \rightarrow \underline{\omega}_{\adicFL}^k.
\] Pulling back $\HT_{\adicFL}^k$ via $\iota^{\bfitw}_{\bfitw_3}$, one obtains \[
    \HT_{\adicFL}^{\bfitw_3^{-1}\bfitw k}: \Sym^{k_1-k_2}\scrO_{\adicFL}^4\otimes \Sym^{k_2} \wedge^2\scrO_{\adicFL}^4 \rightarrow \underline{\omega}_{\adicFL}^{\bfitw_3^{-1}\bfitw k}.
\]
where we have identified $\iota^{\bfitw, *}_{\bfitw_3}\scrO_{\adicFL} \cong \scrO_{\adicFL}$.

Note that the $\GSp_4$-representation $V_k$ is naturally an irreducible subrepresentation (see, for example, \cite[Lecture 17]{Fulton-Harris}) 
\[
    V_k \hookrightarrow \Sym^{k_1-k_2}\Q_p^4 \otimes \Sym^{k_2}\wedge^2 \Q_p^4.
\] 
Composing with the isomorphism $V_k^{\vee} \cong V_k$ induced by the symplectic pairing \eqref{eq: symplectic pairing on standard rep}, we obtain
\[
    V_k^{\vee} \hookrightarrow \Sym^{k_1-k_2}\Q_p^4 \otimes \Sym^{k_2}\wedge^2 \Q_p^4.
\] 
We may view $V_k^{\vee}$ as an \'etale $\Q_p$-local systems over $\adicFL$, and hence a pro-\'etale $\Q_p$-local system. Tensoring with the complete pro-\'etale structure sheaf $\widehat{\scrO}_{\adicFL, \proet}$, we obtain a sheaf $\scrO\!\!\scrV_{k,\adicFL}^{\vee}$ together with an inclusion \begin{equation}\label{eq: natural inclusion of GSp4-reps over the pro-\'etale site of the flag variety}
    \scrO\!\!\scrV_{k, \adicFL}^{\vee} \hookrightarrow \Sym^{k_1-k_2}\widehat{\scrO}_{\adicFL, \proet}^4 \otimes \Sym^{k_2}\wedge^2 \widehat{\scrO}_{\adicFL, \proet}^4.
\end{equation}

On the other hand, we take the completed pullback of $\underline{\omega}_{\adicFL}^k$ to the pro-\'etale site $\adicFL_{\proet}$ and obtain a vector bundle of $\widehat{\scrO}_{\adicFL, \proet}$-modules $\widehat{\underline{\omega}}_{\adicFL}^{k}$. Combining \eqref{eq: natural inclusion of GSp4-reps over the pro-\'etale site of the flag variety} with $\HT_{\adicFL}^{\bfitw_3^{-1}\bfitw k}$, we arrive at  a morphism of pro-\'etale sheaves \begin{equation}\label{eq: algebraic pseudo-ES over the flag variety}
    \mathrm{PES}_k^{\bfitw}: \scrO\!\!\scrV_{k, \adicFL}^{\vee} \rightarrow \widehat{\underline{\omega}}_{\adicFL}^{\bfitw_3^{-1}\bfitw k}.
\end{equation} 
Pulling back $\mathrm{PES}_k^{\bfitw}$ via $\pi_{\HT}$, we obtain a Hecke- and Galois-equivariant morphism of pro-Kummer \'etale sheaves \begin{equation}\label{eq: algebraic ES on the level of pro-Kummer \'etale sheaves at infinity level}
    \ES_{k}^{\bfitw, \alg}: \scrO\!\!\scrV_k^{\vee}|_{\calX_{\Gamma(p^{\infty})}^{\tor}} \rightarrow \widehat{\underline{\omega}}^{\bfitw_3^{-1}\bfitw k}(\bfitw k^{\cyc})|_{\calX_{\Gamma(p^{\infty})}^{\tor}}
\end{equation}
over $\calX_{\Gamma(p^{\infty}), \proket}^{\tor}$ \footnote{Here we abuse the notation and identify the slice category ${\calX_{n, \proket}^{\tor}}_{/\calX_{\Gamma(p^{\infty})}^{\tor}}$ with $\calX_{\Gamma(p^{\infty}), \proket}^{\tor}$.}. For an explanation on the Galois twists, see Remark \ref{Remark: classical forms in the rigid analytic setting}. One check that the morphism is $\Iw_{\GSp_4, n}^+$-equivariant. Therefore, it descends to $\calX_{n}^{\tor}$ and we obtain the desired morphism (\ref{eq: algebraic ES on the level of pro-Kummer \'etale sheaves}).\\

\paragraph{\textbf{Construction 5.}} Applying pro-Kummer \'etale cohomology with supports on the morphism (\ref{eq: algebraic ES on the level of pro-Kummer \'etale sheaves}), we obtain
\begin{equation}\label{eq: ES with support}
    \scalemath{0.9}{ H^3_{\overline{\calX_{n, \leq \bfitw_i}^{\tor}} \smallsetminus \overline{\calX_{n, \leq \bfitw_{i-1}}^{\tor}}, \proket}(\calX_n^{\tor} \smallsetminus \overline{\calX_{n, \leq \bfitw_{i-1}}^{\tor}}, \scrO\!\!\scrV_k^{\vee}) \rightarrow H^3_{\overline{\calX_{n, \leq \bfitw_i}^{\tor}} \smallsetminus \overline{\calX_{n, \leq \bfitw_{i-1}}^{\tor}}, \proket}(\calX_n^{\tor} \smallsetminus \overline{\calX_{n, \leq \bfitw_{i-1}}^{\tor}}, \widehat{\underline{\omega}}^{\bfitw_3^{-1}\bfitw_i k})(\bfitw_i k^{\cyc}) }
\end{equation} 
for $i=0, 1, 2, 3$. We consider an analogue of the Leray spectral sequence (\ref{eq: Leray sp sq}) with support condition 
\[
\scalemath{0.9}{ E_2^{s,t} = H^s_{\overline{\calX_{n, \leq \bfitw_i}^{\tor}} \smallsetminus \overline{\calX_{n, \leq \bfitw_{i-1}}^{\tor}}}(\calX_n^{\tor} \smallsetminus \overline{\calX_{n, \leq \bfitw_{i-1}}^{\tor}}, R^t\upsilon_*\widehat{\underline{\omega}}^{\bfitw_3^{-1}\bfitw_i k}) \Rightarrow H^{s+t}_{\overline{\calX_{n, \leq \bfitw_i}^{\tor}} \smallsetminus \overline{\calX_{n, \leq \bfitw_{i-1}}^{\tor}}, \proket}(\calX_n^{\tor} \smallsetminus \overline{\calX_{n, \leq \bfitw_{i-1}}^{\tor}}, \widehat{\underline{\omega}}^{\bfitw_3^{-1}\bfitw_i k}).}
\]
Taking the finite-slope parts and applying \cite[Theorem 5.7.3]{BP-HigherColeman}, the spectral sequence yields edge maps
\[\scalemath{0.8}{ H^3_{\overline{\calX_{n, \leq \bfitw_i}^{\tor}} \smallsetminus \overline{\calX_{n, \leq \bfitw_{i-1}}^{\tor}}, \proket}(\calX_n^{\tor} \smallsetminus \overline{\calX_{n, \leq \bfitw_{i-1}}^{\tor}}, \widehat{\underline{\omega}}^{\bfitw_3^{-1}\bfitw_i k})^{\fs}  \rightarrow H^{3-i}_{\overline{\calX_{n, \leq \bfitw_i}^{\tor}} \smallsetminus \overline{\calX_{n, \leq \bfitw_{i-1}}^{\tor}}}(\calX_n^{\tor} \smallsetminus \overline{\calX_{n, \leq \bfitw_{i-1}}^{\tor}}, \underline{\omega}^{\bfitw_3^{-1}\bfitw_i k+k_{\bfitw_i}})^{\fs}(\bfitw_i k^{\cyc}-i)}\]
for $i=0,1,2,3$. Combined with (\ref{eq: ES with support}), we arrive at a Hecke- and Galois-equivariant morphism \[
    \scalemath{0.8}{ H^3_{\overline{\calX_{n, \leq \bfitw_i}^{\tor}} \smallsetminus \overline{\calX_{n, \leq \bfitw_{i-1}}^{\tor}}, \proket}(\calX_n^{\tor} \smallsetminus \overline{\calX_{n, \leq \bfitw_{i-1}}^{\tor}}, \scrO\!\!\scrV_k^{\vee})^{\fs} \rightarrow H^{3-i}_{\overline{\calX_{n, \leq \bfitw_i}^{\tor}} \smallsetminus \overline{\calX_{n, \leq \bfitw_{i-1}}^{\tor}}}(\calX_n^{\tor} \smallsetminus \overline{\calX_{n, \leq \bfitw_{i-1}}^{\tor}}, \underline{\omega}^{\bfitw_3^{-1}\bfitw_i k+k_{\bfitw_i}})^{\fs}(\bfitw_i k^{\cyc}-i).}
\]

Now we further pass to the `small-slope parts'. For $H_{\proket}^3(\calX_n^{\tor}, \scrO\!\!\scrV_k^{\vee})$, $H_{\overline{\calX_{n, \leq \bfitw_i}^{\tor}}, \proket}^3(\calX_n^{\tor}, \scrO\!\!\scrV_k^{\vee})$, and $H_{\overline{\calX_{n, \leq \bfitw_2}^{\tor}} \smallsetminus \overline{\calX_{n, \leq \bfitw_1}^{\tor}}, \proket}^3(\calX_n^{\tor} \smallsetminus \overline{\calX_{n, \leq \bfitw_1}^{\tor}}, \scrO\!\!\scrV_k^{\vee})$, we define their \emph{small-slope parts} in the same way as in Definition \ref{Definition: small slope part}. (Notice that the Hecke operators are un-normalised.) It follows from the classicality theorem (Theorem \ref{Theorem: classicality theorem in higher Coleman theory}) that the small-slope part of the second term coincides with the small-slope part of $H^{3-i}(\calX_n^{\tor}, \underline{\omega}^{\bfitw_3^{-1}\bfitw_i k +k_{\bfitw_i}})(\bfitw_i k^{\cyc}-i)$. Consequently, taking small-slope part and applying the classicality theorem, we arrive at a Hecke- and Galois-equivariant diagram 
\begin{equation}\label{eq: Hecke-Galois equiv diagram algebraic}
\begin{tikzcd}[column sep = tiny]
            \scalemath{0.8}{ H_{\proket}^3(\calX_n^{\tor}, \scrO\!\!\scrV_k^{\vee})^{\sms}  } \arrow[r] & \scalemath{0.8}{ H_{\proket}^{3}(\calX_{n, \bfitw_3}^{\tor}, \scrO\!\!\scrV_k^{\vee})^{\sms} } \arrow[r] & \scalemath{0.8}{ H^0(\calX_n^{\tor}, \underline{\omega}^{(k_1+3, k_2+3)})^{\sms}(-3) }\\
            \scalemath{0.8}{ H_{\overline{\calX_{n, \leq \bfitw_2}^{\tor}}, \proket}^3(\calX_n^{\tor}, \scrO\!\!\scrV_k^{\vee})^{\sms} } \arrow[r]\arrow[u] & \scalemath{0.8}{ H_{\overline{\calX_{n, \leq \bfitw_2}^{\tor}} \smallsetminus \overline{\calX_{n, \leq \bfitw_1}^{\tor}}, \proket}^3(\calX_n^{\tor} \smallsetminus \overline{\calX_{n, \leq \bfitw_1}^{\tor}}, \scrO\!\!\scrV_k^{\vee})^{\sms} } \arrow[r] & \scalemath{0.8}{ H^1(\calX_n^{\tor}, \underline{\omega}^{(k_1+3, -k_2+1; k_2)})^{\sms}(k_2-2) }\\
            \scalemath{0.8}{ H_{\overline{\calX_{n, \leq \bfitw_1}^{\tor}}, \proket}^3(\calX_n^{\tor}, \scrO\!\!\scrV_k^{\vee})^{\sms} } \arrow[r]\arrow[u] & \scalemath{0.8}{ H_{\overline{\calX_{n, \leq \bfitw_1}^{\tor}} \smallsetminus \overline{\calX_{n, \leq \one_4}^{\tor}}, \proket}^3(\calX_n^{\tor} \smallsetminus \overline{\calX_{n, \leq \one_4}^{\tor}}, \scrO\!\!\scrV_k^{\vee})^{\sms} } \arrow[r] & \scalemath{0.8}{ H^2(\calX_n^{\tor}, \underline{\omega}^{(k_2+2, -k_1; k_1)})^{\sms}(k_1-1) }\\
            \scalemath{0.8}{ H_{\overline{\calX_{n, \leq \one_4}^{\tor}}, \proket}^3(\calX_n^{\tor}, \scrO\!\!\scrV_k^{\vee})^{\sms}} \arrow[r, equal]\arrow[u] & \scalemath{0.8}{ H_{\overline{\calX_{n,  \one_4}^{\tor}}, \proket}^3(\calX_n^{\tor}, \scrO\!\!\scrV_k^{\vee})^{\sms} }\arrow[r] & \scalemath{0.8}{ H^3(\calX_n^{\tor}, \underline{\omega}^{(-k_2, -k_1; k_1+k_2)})^{\sms}(k_1+k_2)}
        \end{tikzcd}
\end{equation}
The (compositions of) the horizontal maps are referred to as the \emph{classical Eichler--Shimura morphisms}. Note that we consider the un-normalised Hecke operators on the pro-Kummer \'etale cohomology groups, but consider the normalised Hecke operators on the coherent cohomology groups on the right-hand side of the diagram (see Remark \ref{Remark: purpose of renormalisation}).

\begin{Definition}\label{Defn: nice enough}
Let $\Pi = (\pi, \varphi_p)$ be a $p$-stabilisation of an irreducible automorphic representation of weight $k=(k_1, k_2)\in \Z^2$ such that $k_1\geq k_2>0$.
\begin{enumerate}
\item[(i)] We say that $\Pi$ has \emph{small slope} if \[H_{\et}^3 (X_{n, \C_p}, V_k^{\vee})_{\frakm_{\Pi}}^{\sms} = H_{\et}^3 (X_{n, \C_p}, V_k^{\vee})_{\frakm_{\Pi}}.\]
\item[(ii)] We say that $\Pi$ is \emph{nice-enough} if it satisfies Assumption \ref{Assumption: Multiplicity One for cuspidal automorphic representations} and has small slope.
\end{enumerate}
\end{Definition}

Now suppose $\Pi = (\pi, \varphi_p)$ is nice-enough. In particular, we have identifications of 1-dimensional $\C_p$-vector spaces
\[
H^{3-i}(\calX_n^{\tor}, \underline{\omega}^{\bfitw_3^{-1}\bfitw_i k + k_{\bfitw_i}})^{\sms}_{\frakm_{\Pi}}=H^{3-i}(\calX_n^{\tor}, \underline{\omega}^{\bfitw_3^{-1}\bfitw_i k + k_{\bfitw_i}})_{\frakm_{\Pi}}
\]
for all $i$. Localising the entire diagram (\ref{eq: Hecke-Galois equiv diagram algebraic}) at $\frakm_{\Pi}$, we obtain a Hecke- and Galois-equivariant diagram
\begin{equation}\label{eq: Hecke-Galois equiv diagram algebraic localised}
\begin{tikzcd}[column sep = tiny]
            \scalemath{0.8}{ H_{\proket}^3(\calX_n^{\tor}, \scrO\!\!\scrV_k^{\vee})^{\sms}_{\frakm_{\Pi}}  } \arrow[r, "g_3"] & \scalemath{0.8}{ H_{\proket}^{3}(\calX_{n, \bfitw_3}^{\tor}, \scrO\!\!\scrV_k^{\vee})^{\sms}_{\frakm_{\Pi}} } \arrow[r, "h_3"] & \scalemath{0.8}{ H^0(\calX_n^{\tor}, \underline{\omega}^{(k_1+3, k_2+3)})_{\frakm_{\Pi}}(-3) }\\
            \scalemath{0.8}{ H_{\overline{\calX_{n, \leq \bfitw_2}^{\tor}}, \proket}^3(\calX_n^{\tor}, \scrO\!\!\scrV_k^{\vee})^{\sms}_{\frakm_{\Pi}} } \arrow[r, "g_2"]\arrow[u, "f_2"] & \scalemath{0.8}{ H_{\overline{\calX_{n, \leq \bfitw_2}^{\tor}} \smallsetminus \overline{\calX_{n, \leq \bfitw_1}^{\tor}}, \proket}^3(\calX_n^{\tor} \smallsetminus \overline{\calX_{n, \leq \bfitw_1}^{\tor}}, \scrO\!\!\scrV_k^{\vee})^{\sms}_{\frakm_{\Pi}} } \arrow[r, "h_2"] & \scalemath{0.8}{ H^1(\calX_n^{\tor}, \underline{\omega}^{(k_1+3, -k_2+1; k_2)})_{\frakm_{\Pi}}(k_2-2) }\\
            \scalemath{0.8}{ H_{\overline{\calX_{n, \leq \bfitw_1}^{\tor}}, \proket}^3(\calX_n^{\tor}, \scrO\!\!\scrV_k^{\vee})^{\sms}_{\frakm_{\Pi}} } \arrow[r, "g_1"]\arrow[u, "f_1"] & \scalemath{0.8}{ H_{\overline{\calX_{n, \leq \bfitw_1}^{\tor}} \smallsetminus \overline{\calX_{n, \leq \one_4}^{\tor}}, \proket}^3(\calX_n^{\tor} \smallsetminus \overline{\calX_{n, \leq \one_4}^{\tor}}, \scrO\!\!\scrV_k^{\vee})^{\sms}_{\frakm_{\Pi}} } \arrow[r, "h_1"] & \scalemath{0.8}{ H^2(\calX_n^{\tor}, \underline{\omega}^{(k_2+2, -k_1; k_1)})_{\frakm_{\Pi}}(k_1-1) }\\
            \scalemath{0.8}{ H_{\overline{\calX_{n, \leq \one_4}^{\tor}}, \proket}^3(\calX_n^{\tor}, \scrO\!\!\scrV_k^{\vee})^{\sms}_{\frakm_{\Pi}}} \arrow[r, equal, "g_0"]\arrow[u, "f_0"] & \scalemath{0.8}{ H_{\overline{\calX_{n,  \one_4}^{\tor}}, \proket}^3(\calX_n^{\tor}, \scrO\!\!\scrV_k^{\vee})^{\sms}_{\frakm_{\Pi}} }\arrow[r, "h_0"] & \scalemath{0.8}{ H^3(\calX_n^{\tor}, \underline{\omega}^{(-k_2, -k_1; k_1+k_2)})_{\frakm_{\Pi}}(k_1+k_2)}
        \end{tikzcd}
\end{equation}

The left column of this diagram gives rise to an explicit construction of the filtration $\Fil_{\ES, k, \frakm_{\Pi}}^{\bullet}$ in Corollary \ref{Corollary: there exists a unique ES decomposition}. This is summarised in the next proposition.

\begin{Proposition}\label{Prop: recover classical ES}
The following hold. 
\begin{enumerate}
\item[(i)] The 4-dimensional $\C_p$-vector space $H_{\proket}^3(\calX_n^{\tor}, \scrO\!\!\scrV_k^{\vee})^{\sms}_{\frakm_{\Pi}}=H_{\proket}^3(\calX_n^{\tor}, \scrO\!\!\scrV_k^{\vee})_{\frakm_{\Pi}} $ admits Hecke- and Galois-stable a decreasing filtration $\Fil^{\bullet}$ given by $\Fil^0=H_{\proket}^3(\calX_n^{\tor}, \scrO\!\!\scrV_k^{\vee})^{\sms}_{\frakm_{\Pi}}$, 
\[
    \Fil^{3-i} = \image(f_2\circ \cdots \circ f_i: H_{\overline{\calX_{n, \leq \bfitw_i}^{\tor}}, \proket}^3(\calX_n^{\tor}, \scrO\!\!\scrV_k^{\vee})^{\sms}_{\frakm_{\Pi}} \rightarrow H_{\proket}^3(\calX_n^{\tor}, \scrO\!\!\scrV_k^{\vee})^{\sms}_{\frakm_{\Pi}}).
\]
for $i=0, 1, 2$, and $\Fil^4=0$. Moreover, we have $\mathrm{dim}_{\C_p} \Fil^i=4-i$, for $i=0,1,2,3,4$.
\item[(ii)] The arrows $h_0$, $h_1$, $h_2$, $h_3$ are surjective.
\item[(iii)] The compositions $h_i\circ g_i$ are surjective, for $i=0,1,2,3$.
\item[(iv)] The surjections $h_i\circ g_i$ induce natural Hecke- and Galois-equivariant isomorphisms
\[
\Gr^{3-i}\cong H^{3-i}(\calX_n^{\tor}, \underline{\omega}^{\bfitw_3^{-1}\bfitw_i k +k_{\bfitw_i}})_{\frakm_{\Pi}}(\bfitw_i k^{\cyc}-i)
\]
for $i=0,1,2,3$, where $\Gr^i:=\Fil^i/\Fil^{i+1}$. In particular, $\Fil^{\bullet}$ coincides with the filtration $\Fil_{\ES, k, \frakm_{\Pi}}^{\bullet}$ in Corollary \ref{Corollary: there exists a unique ES decomposition}.
\end{enumerate}
\end{Proposition}

\begin{proof}
Firstly, the morphism $\ES_{k}^{\bfitw, \alg}$ in \eqref{eq: algebraic ES on the level of pro-Kummer \'etale sheaves} induces a map on pro-Kummer \'etale cohomology
\begin{equation}\label{eq: ES alg proket}
    H^3_{\proket}(\calX_{n}^{\tor}, \scrO\!\!\scrV_k^{\vee}) \rightarrow H_{\proket}^3(\calX_n^{\tor}, \widehat{\underline{\omega}}^{\bfitw_3^{-1}\bfitw_i k})(\bfitw_ik^{\cyc})
\end{equation}
for $i=0,1,2,3$. Also recall the maps
\begin{equation}\label{eq: ES alg Leray}
H_{\proket}^3(\calX_n^{\tor}, \widehat{\underline{\omega}}^{\bfitw_3^{-1}\bfitw_i k})(\bfitw_ik^{\cyc}) \rightarrow H^{3-i}(\calX_n^{\tor}, \underline{\omega}^{\bfitw_3^{-1}\bfitw_ik+k_{\bfitw_i}})(\bfitw k^{\cyc}-i)
\end{equation}
constructed in Construction 3. Combining (\ref{eq: ES alg proket}) and (\ref{eq: ES alg Leray}) and taking localisation at $\frakm_{\Pi}$, we obtain morphisms
\begin{equation}\label{eq: ES alg proket Leray}
H^3_{\proket}(\calX_{n}^{\tor}, \scrO\!\!\scrV_k^{\vee})_{\frakm_{\Pi}} \rightarrow H^{3-i}(\calX_n^{\tor}, \underline{\omega}^{\bfitw_3^{-1}\bfitw_ik+k_{\bfitw_i}})_{\frakm_{\Pi}}(\bfitw k^{\cyc}-i)
\end{equation}
for $i=0,1,2,3$. (Note that, when $i=3$, the map (\ref{eq: ES alg proket Leray}) is just $h_3\circ g_3$.) By assumption, the target of (\ref{eq: ES alg proket Leray}) is a 1-dimensional $\C_p$-vector space. We claim that (\ref{eq: ES alg proket Leray}) is surjective, and hence non-trivial. Indeed, recall the Leray spectral sequence (\ref{eq: Leray spectral sequence for classical cohomology}). Taking localisation at $\frakm_{\Pi}$, we obtain a spectral sequence
\[
    E_2^{s,t} = H^{s}(\calX_n^{\tor}, \underline{\omega}^{\bfitw_3^{-1}\bfitw_i k} \otimes \Omega_{\calX_n^{\tor}}^{\log, t})_{\frakm_{\Pi}}(-t) \Rightarrow H^{s+t}_{\proket}(\calX_n^{\tor}, \widehat{\underline{\omega}}^{\bfitw_3^{-1}\bfitw_i k})_{\frakm_{\Pi}}.
\]
If $i\neq t$, Assumption \ref{Assumption: Multiplicity One for cuspidal automorphic representations} implies $H^{3-t}(\calX_n^{\tor}, \underline{\omega}^{\bfitw_3^{-1}\bfitw_i k} \otimes \Omega_{\calX_n^{\tor}}^{\log, t})_{\frakm_{\Pi}} =0$ (because they contribute to the wrong cohomological weight). Hence, the edge map \eqref{eq: ES alg Leray} (after localising at $\frakm_{\Pi}$) is a surjection. It remains to show that (\ref{eq: ES alg proket}) (after localising at $\frakm_{\Pi}$) is surjective. Notice that $\underline{\widehat{\omega}}_{\adicFL}^{\bfitw_3^{-1}\bfitw_i k}$ is locally modelled \footnote{Here we adopt the terminology from \cite{BP-HigherColeman}. We say a pro-Kummer \'etale sheaf $\scrV$ is \emph{locally modelled} on $V$ if for every log affinoid perfectoid $\calU$, with corresponding affinoid perfectoid space $\Spa(R, R^+)$, we have $\scrV(\calU)=V\otimes R$.} on the irreducible (algebraic) $H$-representation $W_{\bfitw_i k}$ of highest weight $\bfitw_i k$, and the morphism (\ref{eq: algebraic ES on the level of pro-Kummer \'etale sheaves}) is modelled on a morphism of $H$-representations $\alpha_i: V_k^{\vee}\rightarrow W_{\bfitw_i k}$. One observes that $\alpha_i$ is nontrivial: for $i=3$, the map $\alpha_3:V_k^{\vee}\rightarrow W_{\bfitw_3 k}$ is nontrivial as it is nonzero on the highest weight vector (see \cite[\S 5.3]{DRW}); for general $i$, $\alpha_i$ is a twist of $\alpha_3$ by conjugating with $\bfitw_3^{-1}\bfitw_i$ (see Construction 4) and hence also nontrivial. After identifying $V_k$ with $V_k^{\vee} $ via self-duality, it follows from Corollary \ref{Corollary: desired projection of H-representations} that $\alpha$ must be the projection onto a direct summand of $H$-subrepresentation. Consequently, the morphism (\ref{eq: algebraic ES on the level of pro-Kummer \'etale sheaves}) is the projection onto a direct summand of $\scrO_{\calX_{n, \proket}^{\tor}}$-modules, and hence the map (\ref{eq: ES alg proket}) (after localising at $\frakm_{\Pi}$) is surjective.

For $i=0,1,2$, we obtain a commutative diagram
\begin{equation}\label{eq: a commutative diagram for classical ES at small-slope aut rep}
    \begin{tikzcd}
        H_{\overline{\calX_{n, \leq \bfitw_i}^{\tor}}, \proket}^3(\calX_n^{\tor}, \scrO\!\!\scrV_k^{\vee})^{\sms}_{\frakm_{\Pi}} \arrow[r, "f_2\circ \cdots \circ f_i"]\arrow[d, "g_i"'] & H^3_{\proket}(\calX_n^{\tor}, \scrO\!\!\scrV_k^{\vee})^{\sms}_{\frakm_{\Pi}} \arrow[d, two heads, "\textrm{(\ref{eq: ES alg proket Leray})}"] \\
        H^3_{\overline{\calX_{n, \leq \bfitw_i}^{\tor}} \smallsetminus \overline{\calX_{n, \leq \bfitw_{i-1}}^{\tor}}, \proket}(\calX_n^{\tor} \smallsetminus \overline{\calX_{n, \leq \bfitw_{i-1}}^{\tor}}, \scrO\!\!\scrV_k^{\vee})^{\sms}_{\frakm_{\Pi}} \arrow[r, "h_i"] & H^{3-i}(\calX_n^{\tor} , \underline{\omega}^{\bfitw_3^{-1}\bfitw_i k+k_{\bfitw_i}})^{\sms}_{\frakm_{\Pi}}(\bfitw_i k^{\cyc}-i) 
    \end{tikzcd}.
\end{equation} 
By a similar argument as above, we see that $h_i$ is surjective. Indeed, $h_i$ factors as a composition \[
    \begin{tikzcd}[row sep = small, column sep = small]
        \scalemath{0.8}{ H^3_{\overline{\calX_{n, \leq \bfitw_i}^{\tor}} \smallsetminus \overline{\calX_{n, \leq \bfitw_{i-1}}^{\tor}}, \proket}(\calX_n^{\tor} \smallsetminus \overline{\calX_{n, \leq \bfitw_{i-1}}^{\tor}}, \scrO\!\!\scrV_k^{\vee})^{\sms}_{\frakm_{\Pi}} } \arrow[r] & \scalemath{0.8}{ H^3_{\overline{\calX_{n, \leq \bfitw_i}^{\tor}} \smallsetminus \overline{\calX_{n, \leq \bfitw_{i-1}}^{\tor}}, \proket}(\calX_n^{\tor} \smallsetminus \overline{\calX_{n, \leq \bfitw_{i-1}}^{\tor}}, \widehat{\underline{\omega}}^{\bfitw_3^{-1}\bfitw_i k})^{\sms}_{\frakm_{\Pi}}(\bfitw_i k^{\cyc}) } \arrow[d]\\
        & \scalemath{0.8}{ H^{3-i}(\calX_n^{\tor} , \underline{\omega}^{\bfitw_3^{-1}\bfitw_i k+k_{\bfitw_i}})^{\sms}_{\frakm_{\Pi}}(\bfitw_i k^{\cyc}-i) }
    \end{tikzcd}.
\]
where the first arrow is surjective as $\widehat{\underline{\omega}}^{\bfitw_3^{-1}\bfitw_i k}$ can be identified as a direct summand of $\scrO\!\!\scrV_k^{\vee}$, while the second arrow is an edge map which is surjective due to \cite[Theorem 5.7.3]{BP-HigherColeman}. When $i=0$, the map $g_0$ is an identity, which implies that $f_2\circ f_1\circ f_0$ is non-trivial. In particular, all of $f_2$, $f_2\circ f_1$, and $f_2\circ f_1\circ f_0$ are non-trivial. We claim that $h_i\circ g_i$ are non-trivial, for all $i=0,1,2,3$. This is already known for $i=0$ and $i=3$. For $i=1,2$, observe the cummutative diagram
\[
    \begin{tikzcd}[column sep = tiny]
        \scalemath{0.7}{ H_{\overline{\calX_{n, \leq \bfitw_i}^{\tor}}, \proket}^3(\calX_n^{\tor}, \scrO\!\!\scrV_k^{\vee})^{\sms}_{\frakm_{\Pi}} } \arrow[r, "g_i"]\arrow[d, two heads] & \scalemath{0.7}{H^3_{\overline{\calX_{n, \leq \bfitw_i}^{\tor}} \smallsetminus \overline{\calX_{n, \leq \bfitw_{i-1}}^{\tor}}, \proket}(\calX_n^{\tor} \smallsetminus \overline{\calX_{n, \leq \bfitw_{i-1}}^{\tor}}, \scrO\!\!\scrV_k^{\vee})^{\sms}_{\frakm_{\Pi}}  } \arrow[d, two heads] \arrow[rd, "h_i"] & \\
        \scalemath{0.7}{ H_{\overline{\calX_{n, \leq \bfitw_i}^{\tor}}, \proket}^3(\calX_n^{\tor}, \widehat{\underline{\omega}}^{\bfitw_3^{-1}\bfitw_i k})^{\sms}_{\frakm_{\Pi}}(\bfitw_i k^{\cyc}) } \arrow[r] & \scalemath{0.7}{H^3_{\overline{\calX_{n, \leq \bfitw_i}^{\tor}} \smallsetminus \overline{\calX_{n, \leq \bfitw_{i-1}}^{\tor}}, \proket}(\calX_n^{\tor} \smallsetminus \overline{\calX_{n, \leq \bfitw_{i-1}}^{\tor}}, \widehat{\underline{\omega}}^{\bfitw_3^{-1}\bfitw_i k})^{\sms}_{\frakm_{\Pi}}(\bfitw_i k^{\cyc})  } \arrow[r, two heads] & \scalemath{0.7}{ H^{3-i}(\calX_n^{\tor} , \underline{\omega}^{\bfitw_3^{-1}\bfitw_i k+k_{\bfitw_i}})^{\sms}_{\frakm_{\Pi}}(\bfitw_i k^{\cyc}-i) }
    \end{tikzcd}
\]
By Proposition \ref{Proposition: quasi-isomorphisms of pro-Kummer \'etale cohomology for classical automorphic sheaves}, the bottom-left horizontal map is an isomorphism. We immediately conclude that $h_i \circ g_i$ is a surjection. 

Finally, by dimension counting, it is straightforward to conclude that $\dim_{\C_p}\Fil^i=4-i$, and that the surjection $h_i\circ g_i$ factors through the quotient $\Gr^{3-i}$, for all $i=0,1,2,3$.
\end{proof}

Proposition \ref{Prop: recover classical ES} leads to the following open question. 

\begin{Question}\label{Question: concentration of pro-Kummer et cohomology at a nice-enough point}
    Supose $\Pi = (\pi, \varphi_p)$ is nice-enough. Do the localisations of the pro-Kummer {\'e}tale cohomology groups with supports \[
         H^j_{\overline{\calX_{n, \leq \bfitw_i}^{\tor}}}(\calX_{n}^{\tor}, \scrO\!\!\scrV_k^{\vee})_{\frakm_{\Pi}} \quad \text{ and }\quad H^j_{\overline{\calX_{n, \leq \bfitw_i}^{\tor}} \smallsetminus \overline{\calX_{n, \leq \bfitw_{i-1}}^{\tor}}}(\calX_{n}^{\tor}\smallsetminus \overline{\calX_{n, \leq \bfitw_{i-1}}^{\tor}}, \scrO\!\!\scrV_k^{\vee})_{\frakm_{\Pi}}
    \]
    concentrate in degree 3?
\end{Question}

\paragraph{\textbf{Summary.}} The key ingredient in our construction is the Hecke- and Galois-equivariant morphisms $\ES_{k}^{\bfitw, \alg}$ of pro-Kummer \'etale sheaves. Therefore, the key to $p$-adically interpolate the decomposition of Faltings--Chai is to construct $p$-adic interpolations of  $\ES_{k}^{\bfitw, \alg}$. Indeed, this is achieved in \S \ref{subsection: OES}.

\subsection{Overconvergent Eichler--Shimura morphisms in family}\label{subsection: OES}
We finally construct the \emph{overconvergent Eichler--Shimura morphisms}, as in the title of the paper. These morphisms relate the overconvergent cohomology groups constructed in \S \ref{section: OC} to the cohomology groups of the automorphic sheaves constructed in \S \ref{section: automorphic}, 
and they $p$-adically interpolate the classical Eichler--Shimura morphisms constructed in \S \ref{subsection: classical ES} (Construction 5). 

Inspired by the discussion in \S \ref{subsection: classical ES}, we will construct morphisms at the level of pro-\'etale sheaves on the flag variety. The desired overconvergent Eichler--Shimura morphisms are obtained by pullback along the Hodge--Tate period map, and then taking cohomology.

Given a weight $(R_{\calU}, \kappa_{\calU})$ and $r\in \Q_{\geq 0}$ with $r>1+r_{\calU}$, we first establish a morphism of $R_{\calU}$-modules $D_{\kappa_{\calU}}^r \rightarrow A_{\kappa_{\calU}}^r$. Recall the highest weight vector $e_{\kappa_{\calU}}^{\hst}$ in Example \ref{Example: highest weight vector} and $f_{\kappa_{\calU}}^{\bfgamma}\in A_{\kappa_{\calU}}^r$ for any $\bfgamma \in \Iw_{\GSp_4, 1}^+$. We define \[
    \Phi_{\kappa_{\calU}}^r : D_{\kappa_{\calU}}^r \rightarrow A_{\kappa_{\calU}}^r, \quad \mu \mapsto \left( \bfgamma \mapsto \mu(f_{\kappa_{\calU}}^{\bfgamma}) \right).
\]
This morphism then induces a morphism of pro-\'etale sheaves \[
    \Phi_{\kappa_{\calU}}^r: \scrO\!\!\scrD_{\kappa_{\calU}, \adicFL}^r \rightarrow \scrO\!\!\scrA_{\kappa_{\calU}, \adicFL}^r
\]
where $\scrO\!\!\scrD_{\kappa_{\calU}, \adicFL}^r$ and $\scrO\!\!\scrA_{\kappa_{\calU}, \adicFL}^r$ are the pro-\'etale sheaves on $\adicFL$ constructed in \S \ref{subsection: alternative construction on flag variety}. The morphism further extends to a commutative diagram \[
    \begin{tikzcd}
        \scrO\!\!\scrD_{\kappa_{\calU}}^r \arrow[r, "\Phi_{\kappa_{\calU}}^r"] & \scrO\!\!\scrA_{\kappa_{\calU}}^r\arrow[d]\\
        \scrO\!\!\scrD_{\kappa_{\calU}}^{r^+}\arrow[u]\arrow[r, "\Phi_{\kappa_{\calU}}^{r^+}"] & \scrO\!\!\scrA_{\kappa_{\calU}}^{r^+}
    \end{tikzcd}.
\]

On the other hand, we consider the $p$-adic completed pro-\'etale pullback of the pseudoautomorphic sheaves; namely, for each $\bfitw\in W^H$, consider \[
    \widehat{\scrA}_{\kappa_{\calU}, \adicFL_{\bfitw}}^{r, \circ} := \varprojlim_{j} \left( \scrA_{\kappa_{\calU}, \adicFL_{\bfitw}}^{r, \circ} \otimes_{\scrO_{\adicFL_{\bfitw, (r, r)}}^+} \scrO^+_{\adicFL_{\bfitw, (r, r)}, \proet}/p^j\right) \quad \text{ and }\quad \widehat{\scrA}_{\kappa_{\calU}, \adicFL_{\bfitw}}^{r} := \widehat{\scrA}_{\kappa_{\calU}, \adicFL_{\bfitw}}^{r, \circ}\Big[\frac{1}{p}\Big]
\]
where $\scrA_{\kappa_{\calU}, \adicFL_{\bfitw}}^{r, \circ}$ is the pseudoautomorphic sheaf on $\adicFL_{\bfitw, (r, r)}$ (cf. \S \ref{subsection: pseudoautomorphic sheaves}). For any affinoid perfectoid object $\calV_{\infty}\in \adicFL_{\bfitw, (r,r), \proet}$, consider the map
 \[
    \Psi_{\kappa_{\calU}}^{\bfitw, r}: \scrO\!\!\scrA_{\kappa_{\calU}, \adicFL}^r (\calV_{\infty}) \rightarrow \widehat{\scrA}_{\kappa_{\calU}, \adicFL_{\bfitw}}^{r}(\calV_{\infty}), \quad f\mapsto \left( \bfgamma \mapsto f\left( \trans\left( \bfitw^{-1}\bfitw_3\trans\bfgamma \bfitw_3^{-1}\begin{pmatrix}\one_2 & \\ \bfitz & \one_2\end{pmatrix} \bfitw\right) \right) \right)
\] for any $\bfgamma\in \Iw_{H, 1}^+$. To see that this map is well-defined, we first identify
\[\scrO\!\!\scrA_{\kappa_{\calU}, \adicFL}^r (\calV_{\infty})=A_{\kappa_{\calU}}^r\widehat{\otimes}\widehat{\scrO}_{\adicFL_{\bfitw}, \proet}(\calV_{\infty})\]
and
\[\widehat{\scrA}_{\kappa_{\calU}, \adicFL_{\bfitw}}^{r}(\calV_{\infty})=A_{\kappa_{\calU}}^{r-\an}(\Iw_{H, 1}^+, R_{\calU}\widehat{\otimes}\widehat{\scrO}_{\adicFL_{\bfitw, (r,r), \proet}}(\calV_{\infty})).\] Then notice that the matrix \[
    \bfitw^{-1}\bfitw_3\trans\bfgamma \bfitw_3^{-1}\begin{pmatrix}\one_2 & \\ \bfitz & \one_2\end{pmatrix}
\] is a diagonal matrix after modulo $p$, so it is valid to evaluate $f$ at this matrix. We also notice that for $\bfbeta\in \Iw_{H, 1}^+ \cap B_{\GSp_4}$, we have \[
    \bfitw^{-1} \bfitw_3 \trans\bfbeta \bfitw_3^{-1}\bfitw \in \Iw_{\GSp_4, 1}^+ \cap B_{\GSp_4},
\] hence the map $\Psi_{\kappa_{\calU}}^{\bfitw, r}$ is well-defined. This induces a map of sheaves
\[\Psi_{\kappa_{\calU}}^{\bfitw, r}: \scrO\!\!\scrA_{\kappa_{\calU}, \adicFL}^r \rightarrow \widehat{\scrA}_{\kappa_{\calU}, \adicFL_{\bfitw}}^{r}.\]

Composed with $\Phi_{\kappa_{\calU}}^r: \scrO\!\!\scrD_{\kappa_{\calU}, \adicFL}^r \rightarrow \scrO\!\!\scrA_{\kappa_{\calU}, \adicFL}^r$, we arrive at the morphism
\[
    \mathrm{PES}_{\kappa_{\calU}}^{\bfitw, r}: \scrO\!\!\scrD_{\kappa_{\calU}}^r \xrightarrow{\Phi_{\kappa_{\calU}}^r}\scrO\!\!\scrA_{\kappa_{\calU}}^r \xrightarrow{\Psi_{\kappa_{\calU}}^{\bfitw, r}} \widehat{\scrA}_{\kappa_{\calU}, \adicFL_{\bfitw}}^r.
\]
Unwinding everything, $\mathrm{PES}_{\kappa_{\calU}}^{\bfitw, r}$ is given by the explicit formula 
\begin{equation}\label{eq: explicit formula for pseudo-ES}
    \mathrm{PES}_{\kappa_{\calU}}^{\bfitw, r}(\mu \otimes g)(\bfgamma) = g\left(\int_{\bfalpha\in \Iw_{\GSp_4,1}^+} e_{\kappa_{\calU}}^{\hst}\left( \bfitw^{-1}\bfitw_3 \trans\bfgamma \bfitw_3^{-1}\begin{pmatrix}\one_2 & \\ \bfitz & \one_2\end{pmatrix}\bfitw \bfalpha \right) d\mu\right)
\end{equation} for any section $g$ of $\widehat{\scrO}_{\adicFL_{\bfitw, (r,r)}, \proet}$ and any $\mu\in D_{\kappa_{\calU}}^r$. 

Now we pullback $\mathrm{PES}_{\kappa_{\calU}}^{\bfitw, r}$ via the Hodge--Tate period map
\[
        \pi_{\HT}: \calX^{\tor}_{\Gamma(p^{\infty}), \bfitw, (r,r), \proket} \rightarrow \adicFL_{\bfitw, (r,r), \proet}.
\]  
It is evident from the construction that $\pi_{\HT}^*\widehat{\scrA}_{\kappa_{\calU}, \adicFL_{\bfitw}}^r$ is precisely the restriction of the completed pro-Kummer \'etale automorphic sheaf $\widehat{\underline{\omega}}_{n, r}^{\bfitw_3^{-1}\bfitw\kappa_{\calU}}$ (defined in Remark \ref{Remark: pro-Kummer \'etale automorphic sheaf}) on $\calX_{\Gamma(p^{\infty}), \bfitw, (r,r)}^{\tor}$. Keeping track of the Galois action, the pullback of $\mathrm{PES}_{\kappa_{\calU}}^{\bfitw, r}$ via $\pi_{\HT}$ yields a morphism
\begin{equation}\label{eq: sheaf version OES infinite level}
     \ES_{\kappa_{\calU}}^{\bfitw, r} : \scrO\!\!\scrD_{\kappa_{\calU}}^r|_{\calX_{\Gamma(p^{\infty}), \bfitw, (r,r)}^{\tor}} \rightarrow \widehat{\underline{\omega}}_{n, r}^{\bfitw_3^{-1}\bfitw\kappa_{\calU}}|_{\calX_{\Gamma(p^{\infty}), \bfitw, (r,r)}^{\tor}}(\bfitw\kappa_{\calU}^{\cyc}),
\end{equation}
where \[
    \bfitw\kappa_{\calU}^{\cyc} = \kappa_{\calU}\left( \bfitw^{-1}\mu_{\Si}(\chi_{\cyc})\bfitw \right)
\] and $\chi_{\cyc}: \Gal_{\Q_p} \rightarrow \Z_p^\times$ is the $p$-adic cyclotomic character. 

\begin{Remark}\label{Remark: explicit formulae for the twists}
    The Tate twist $(\bfitw\kappa_{\calU}^{\cyc})$ in (\ref{eq: sheaf version OES infinite level}) can be computed explicitly.
\begin{itemize}
        \item When $\bfitw = \one_4$, \[
            \bfitw \kappa_{\calU}^{\cyc} = \kappa_{\calU}(\mu_{\Si}(\chi_{\cyc})) =\kappa_{\calU}(\diag(\chi_{\cyc}, \chi_{\cyc}, 1, 1)) =  \kappa_{\calU, 1}(\chi_{\cyc})\kappa_{\calU, 2}(\chi_{\cyc}).
        \] 
        \item When $\bfitw = \bfitw_1$, \[
            \bfitw \kappa_{\calU}^{\cyc} = \kappa_{\calU}(\bfitw_1^{-1}\mu_{\Si}(\chi_{\cyc}) \bfitw_1) = \kappa_{\calU}(\diag(\chi_{\cyc}, 1, \chi_{\cyc}, 1)) = \kappa_{\calU, 1}(\chi_{\cyc}).
        \] 
        \item When $\bfitw = \bfitw_2$, \[
            \bfitw \kappa_{\calU}^{\cyc} = \kappa_{\calU}(\bfitw_2^{-1}\mu_{\Si}(\chi_{\cyc}) \bfitw_2) = \kappa_{\calU}(\diag(1, \chi_{\cyc}, \chi_{\cyc}, 1)) = \kappa_{\calU, 2}(\chi_{\cyc}).
        \] 
        \item When $\bfitw = \bfitw_3$, \[
                \bfitw \kappa_{\calU}^{\cyc} = \kappa_{\calU}(\bfitw_3^{-1}\mu_{\Si}(\chi_{\cyc}) \bfitw_3) = \kappa_{\calU}(\diag(1, 1, \chi_{\cyc}, \chi_{\cyc})) = 1.
        \]
    \end{itemize}
\end{Remark}

\begin{Proposition}\label{Proposition: the ES map at the infinity descends}
    Let $\bfitw\in W^H$. Let $(R_{\calU}, \kappa_{\calU})$ be a weight and let $r\in \Q_{\geq 0}$, $n\in \Z_{>0}$ such that $n\geq r > 1+r_{\calU}$. The map $\mathrm{ES}_{\kappa_{\calU}}^{\bfitw, r}$ defined in \eqref{eq: sheaf version OES infinite level} is $\Iw_{\GSp_4, n}^+$-equivariant. Therefore, it descends to a morphism
\begin{equation}\label{eq: sheaf version OES finite level}
        \mathrm{ES}_{\kappa_{\calU}}^{\bfitw, r}: \scrO\!\!\scrD_{\kappa_{\calU}}^r \rightarrow \widehat{\underline{\omega}}_{n, r}^{\bfitw_3^{-1}\bfitw \kappa_{\calU}}(\bfitw\kappa_{\calU}^{\cyc}).
\end{equation}
on $\calX_{n, \bfitw, (r,r), \proket}^{\tor}$. 
\end{Proposition}
\begin{proof}
    For any section $\mu\otimes g$ of $\scrO\!\!\scrD_{\kappa_{\calU}}^r$, $\bfdelta = \begin{pmatrix}\bfdelta_a & \bfdelta_b\\ \bfdelta_c & \bfdelta_d\end{pmatrix}\in \Iw_{\GSp_4, n}^+$, and $\bfgamma\in \Iw_{H, 1}^+$, we have \begin{align*}
       &\scalemath{0.9}{ \mathrm{ES}_{\kappa_{\calU}}^{\bfitw, r} (\bfdelta^*(\mu \otimes g))(\bfgamma)}\\
        & \scalemath{0.9}{= (\bfdelta^* g) \left( \int_{\bfalpha\in \Iw_{\GSp_4, 1}^+} e_{\kappa_{\calU}}^{\hst}\left( \bfitw^{-1}\bfitw_3 \trans \bfgamma \bfitw_3^{-1} \begin{pmatrix}\one_2 & \\ \frakz & \one_2\end{pmatrix} \bfitw  \bfalpha \right) d \bfdelta\mu\right) }\\
        & \scalemath{0.9}{= (\bfdelta^* g) \left( \int_{\bfalpha\in \Iw_{\GSp_4, 1}^+} e_{\kappa_{\calU}}^{\hst}\left( \bfitw^{-1}\bfitw_3 \trans \bfgamma \bfitw_3^{-1} \begin{pmatrix}\one_2 & \\ \frakz & \one_2\end{pmatrix} \bfitw \bfdelta \bfalpha \right) d \mu\right) }\\
        & \scalemath{0.9}{= (\bfdelta^* g) \left( \int_{\bfalpha\in \Iw_{\GSp_4, 1}^+} e_{\kappa_{\calU}}^{\hst} \left( \bfitw^{-1}\bfitw_3 \trans\bfgamma \bfitw_3^{-1} \bfitj_{\bfitw}(\bfdelta, \frakz)\bfitw_3\bfitw_3^{-1}\begin{pmatrix}\one_2 & \\ (\bfdelta_d^{\bfitw} + \frakz\bfdelta_b^{\bfitw})^{-1}(\bfdelta_c^{\bfitw}+\frakz\bfdelta_a^{\bfitw}) & \one_2\end{pmatrix} \bfitw \bfalpha \right) d\mu \right)}\\
        & \scalemath{0.9}{= (\bfdelta^* g)\rho_{\bfitw_3^{-1}\bfitw\kappa_{\calU}}^r(\bfitj_{\bfitw}(\bfdelta, \frakz)) \left(\int_{\bfalpha\in \Iw_{\GSp_4, 1}^+} e_{\kappa_{\calU}}^{\hst}\left(\bfitw^{-1}\bfitw_3 \trans\bfgamma \bfitw_3^{-1}\begin{pmatrix}\one_2 & \\ (\bfdelta_d^{\bfitw} + \frakz\bfdelta_b^{\bfitw})^{-1}(\bfdelta_c^{\bfitw}+\frakz\bfdelta_a^{\bfitw}) & \one_2\end{pmatrix} \bfitw \bfalpha \right) d\mu\right)}\\
        &\scalemath{0.9}{ = \left(\bfdelta *_{\bfitw, \kappa_{\calU}} \mathrm{ES}_{\kappa_{\calU}}^{\bfitw, r}(\mu \otimes g)\right)}(\bfgamma).
    \end{align*}
\end{proof}

\begin{Proposition}\label{Proposition: OES is u_p-equivariant}
   Let $\bfitw\in W^H$. Let $(R_{\calU}, \kappa_{\calU})$ be a weight and let $r\in \Q_{\geq 0}$, $n\in \Z_{>0}$ such that $n\geq r > 1+r_{\calU}$. The map (\ref{eq: sheaf version OES finite level}) is compatible with the actions of $\bfitu_{p, 0}$, $\bfitu_{p, 1}$, and $\bfitu_{p}$. 
\end{Proposition}
\begin{proof}
    Let $\bfitu \in \{\bfitu_{p, 0}, \bfitu_{p,1}, \bfitu_p\}$ and write $\bfitu = \diag(\bfitu_a, \bfitu_d)$. Given a section $\mu \otimes g$ for $\scrO\!\!\scrD_{\kappa_{\calU}}^r$ and any $\bfgamma \in \Iw_{H, 1}^+$ we have  \begin{align*}
     & \scalemath{0.9}{  \mathrm{ES}_{\kappa_{\calU}}^{\bfitw, r}((\bfitu \cdot \mu) \otimes (\bfitu^* g) )(\bfgamma)} \\
        & \scalemath{0.9}{= (\bfitu^* g) \left( \int_{\bfalpha\in \Iw_{\GSp_4, 1}^+} \kappa_{\calU}(\bfbeta_{\bfgamma})\kappa_{\calU}(\bfbeta_{\bfalpha})e_{\kappa_{\calU}}^{\hst} \left( \bfitw^{-1}\bfitw_3 \trans\bfepsilon_{\bfgamma} \bfitw_3^{-1} \begin{pmatrix}\one_2 & \\ \frakz & \one_2\end{pmatrix} \bfitw \bfitu \bfepsilon_{\bfalpha}\bfitu^{-1}\right) d\mu \right)}\\
        & \scalemath{0.9}{= (\bfitu^* g) \left( \int_{\bfalpha\in \Iw_{\GSp_4, 1}^+} \kappa_{\calU}(\bfbeta_{\bfgamma})\kappa_{\calU}(\bfbeta_{\bfalpha}) e_{\kappa_{\calU}}^{\hst}\left(\bfitu^{-1}\bfitw^{-1}\bfitw_3 \trans\bfepsilon_{\bfgamma}\bfitw_3^{-1} \bfitu^{\bfitw} \begin{pmatrix}\one_2 & \\ \bfitu^{\bfitw, *}\frakz & \one_4\end{pmatrix} \bfitw \bfepsilon_{\bfalpha}\right) \right)}\\
        & \scalemath{0.9}{= (\bfitu^* g) \left( \int_{\bfalpha\in \Iw_{\GSp_4, 1}^+} \kappa_{\calU}(\bfbeta_{\bfgamma})\kappa_{\calU}(\bfbeta_{\bfalpha}) e_{\kappa_{\calU}}^{\hst}\left(\bfitw^{-1}\bfitw_3(\bfitu^{\bfitw_3^{-1}\bfitw})^{-1} \trans\bfepsilon_{\bfgamma}\bfitu^{\bfitw_3^{-1}\bfitw} \bfitw_3^{-1}\begin{pmatrix}\one_2 & \\ \bfitu^{\bfitw, *}\frakz & \one_4\end{pmatrix} \bfitw \bfepsilon_{\bfalpha}\right) \right)}\\
        & \scalemath{0.9}{= (\bfitu^* g) \left( \int_{\bfalpha\in \Iw_{\GSp_4, 1}^+} \kappa_{\calU}(\bfbeta_{\bfgamma})\kappa_{\calU}(\bfbeta_{\bfalpha}) e_{\kappa_{\calU}}^{\hst}\left(\bfitw^{-1}\bfitw_3\trans \left(\bfitu^{\bfitw_3^{-1}\bfitw}\bfepsilon_{\bfgamma}(\bfitu^{\bfitw_3^{-1}\bfitw})^{-1} \right) \bfitw_3^{-1}\begin{pmatrix}\one_2 & \\ \bfitu^{\bfitw, *}\frakz & \one_4\end{pmatrix} \bfitw \bfepsilon_{\bfalpha}\right) \right)}\\
        & \scalemath{0.9}{= \bfitu *_{\bfitw} \mathrm{ES}_{\kappa_{\calU}}^{\bfitw, r}(\mu \otimes g)(\bfgamma),}
    \end{align*}
    where \begin{itemize}
        \item in the first equality, we write $\bfalpha = \bfepsilon_{\bfalpha}\bfbeta_{\bfalpha}$ (resp., $\bfgamma = \bfepsilon_{\bfgamma}\bfbeta_{\bfgamma}$) with $\bfepsilon_{\bfalpha}\in N_{\GSp_4, 1}^{\opp}$ and $\bfbeta_{\bfalpha}\in T_{\GSp_4}(\Z_p)N_{\GSp_4, 1}$ (resp., $\bfepsilon_{\bfgamma}\in N_{H, 1}^{\opp}$ and $\bfbeta_{\bfgamma}\in T_H(\Z_p)N_{H, 1}$);
        \item in the second equality, we move the position of $\bfitu^{-1}$ thanks to the property of determinants; 
        \item in the third equality, $\bfitu^{\bfitw_3^{-1}\bfitw}$ stands for the conjugation of $\bfitu$ by $\bfitw_3^{-1}\bfitw$; namely, $\bfitu^{\bfitw_3^{-1}\bfitw}=\bfitw_3^{-1} \bfitw \bfitu \bfitw^{-1}\bfitw_3$; 
        \item in the fourth equality, we use the fact that $\bfitu^{\bfitw_3^{-1}\bfitw}$ is invariant under transposition.
    \end{itemize}
\end{proof}

Finally, we explain how to construct the desired overconvergent Eichler--Shimura morphisms by taking cohomology groups on the map of sheaves (\ref{eq: sheaf version OES finite level}). The readers are referred to \S \ref{section: cohomology with supports} for a theory of pro-Kummer \'etale cohomology with supports.

Given $\bfitw$, $(R_{\calU}, \kappa_{\calU})$, $r$, and $n$ as above, recall the loci \[
    \calZ_{n, \bfitw} = (\overline{\calX_{n, \leq \bfitw}^{\tor}})\bfitu_p^{-n-1} \cap (\calX_{n, \geq \bfitw}^{\tor})\bfitu_p^{n+1}\quad \text{ and }\quad \calX^{\tor, \bfitu_p}_{n,\bfitw} = (\calX_{n, \geq \bfitw}^{\tor})\bfitu_p^{n+1}
\]
defined in (\ref{eq: some technical loci}). The morphism $\mathrm{ES}_{\kappa_{\calU}}^{\bfitw, r}$ gives rise to a morphism in cohomology 
\[
       \mathrm{ES}_{\kappa_{\calU}}^{\bfitw, r}: R\Gamma_{\calZ_{n, \bfitw}, \proket}(\calX^{\tor, \bfitu_p}_{n, \bfitw}, \scrO\!\!\scrD_{\kappa_{\calU}}^r) \rightarrow R\Gamma_{\calZ_{n, \bfitw}, \proket}(\calX^{\tor, \bfitu_p}_{n, \bfitw}, \widehat{\underline{\omega}}_{n, r}^{\bfitw_3^{-1}\bfitw \kappa_{\calU}})(\bfitw\kappa_{\calU}^{\cyc}).
\] 

Thanks to Proposition \ref{Proposition: the ES map at the infinity descends} and Proposition \ref{Proposition: OES is u_p-equivariant}, we know that $\mathrm{ES}_{\kappa_{\calU}}^{\bfitw, r}$ is $U$-equivariant (for $U\in \{U_{p,0}, U_{p,1}, U_p\}$). Moreover, we have seen that the $U_p$-operator acts compactly on both cohomology groups. Therefore, when $(R_{\calU}, \kappa_{\calU})$ is an affinoid weight, we can take the finite-slope part on both sides and arrive at 
\begin{equation}\label{eq: OES for finite-slope complexes}
\mathrm{ES}_{\kappa_{\calU}}^{\bfitw, r}: R\Gamma_{\calZ_{n, \bfitw}, \proket}(\calX^{\tor, \bfitu_p}_{n, \bfitw}, \scrO\!\!\scrD_{\kappa_{\calU}}^r)^{\fs} \rightarrow R\Gamma_{\calZ_{n, \bfitw}, \proket}(\calX^{\tor, \bfitu_p}_{n, \bfitw}, \widehat{\underline{\omega}}_{n, r}^{\bfitw_3^{-1}\bfitw \kappa_{\calU}})^{\fs}(\bfitw\kappa_{\calU}^{\cyc}).
\end{equation}

\begin{Proposition}\label{Proposition: OES at w}
    Let $\bfitw\in W^H$. Let $(R_{\calU}, \kappa_{\calU})$ be an affinoid weight and let $r\in \Q_{\geq 0}$, $n\in \Z_{>0}$ such that $n\geq r > 1+r_{\calU}$. Then $\mathrm{ES}_{\kappa_{\calU}}^{\bfitw, r}$ induces a Hecke- and Galois-equivariant morphism \[
        \scalemath{1}{\mathrm{ES}_{\kappa_{\calU}}^{\bfitw, r}: H^3_{\calZ_{n, \bfitw}, \proket}(\calX^{\tor, \bfitu_p}_{n, \bfitw}, \scrO\!\!\scrD_{\kappa_{\calU}}^r)^{\fs} \rightarrow H^{3-l(\bfitw)}_{\calZ_{n, \bfitw}, \ket}(\calX^{\tor, \bfitu_p}_{n, \bfitw}, \underline{\omega}_{n, r}^{\bfitw_3^{-1}\bfitw \kappa_{\calU}+k_{\bfitw}})^{\fs}(\bfitw\kappa_{\calU}^{\cyc})}.
    \]
\end{Proposition}
\begin{proof}
Consider the Leray spectral sequence
\begin{equation}\label{eq: Leray spectral sequence with support}
E_2^{j,i}=H^j_{\calZ_{n, \bfitw}, \ket}(\calX^{\tor, \bfitu_p}_{n, \bfitw}, R^i\nu_* \widehat{\underline{\omega}}_{n, r}^{\bfitw_3^{-1}\bfitw\kappa_{\calU}})^{\fs}\Rightarrow H_{\calZ_{n, \bfitw}, \proket}^{j+i}(\calX^{\tor, \bfitu_p}_{n, \bfitw}, \widehat{\underline{\omega}}_{n, r}^{\bfitw_3^{-1}\bfitw\kappa_{\calU}})^{\fs}.
\end{equation}
By the generalised projection formula in \cite[Proposition A.3.11]{DRW}, we have 
\[R^i\nu_* \widehat{\underline{\omega}}_{n, r}^{\bfitw_3^{-1}\bfitw\kappa_{\calU}}\cong \underline{\omega}_{n, r}^{\bfitw_3^{-1}\bfitw \kappa_{\calU}} \widehat{\otimes} R^{i}\nu_* \widehat{\scrO}_{\calX_{n, \bfitw, (r, r)}^{\tor}, \proket}.\]
By \cite[Proposition A.2.3]{DRW}, we have
 \begin{equation}\label{eq: computation of pushforward}
        R^{l(\bfitw)}\nu_* \widehat{\scrO}_{\calX_{n, \bfitw, (r, r)}^{\tor}, \proket} \cong \Omega_{\calX_{n, \bfitw, (r, r)}^{\tor}}^{\log, l(\bfitw)} (-l(\bfitw)).
\end{equation}
Moreover,  Kodaira--Spencer isomorphism (\cite[Theorem 1.41 (4)]{LanKS}) implies that 
\begin{equation}\label{eq: Kodaira-Spencer}
        \Omega_{\calX_{n, \bfitw, (r, r)}^{\tor}}^{\log, l(\bfitw)} \cong \underline{\omega}^{k_{\bfitw}}|_{\calX_{n, \bfitw, (r, r)}^{\tor}}.
\end{equation}
Now, applying \cite[Theorem 6.7.3]{BP-HigherColeman}, we know that the finite slope part of the cohomology groups vanish in low degrees in the spectral sequence (\ref{eq: Leray spectral sequence with support}). This yields an edge map 
\[H_{\calZ_{n, \bfitw}, \proket}^3(\calX^{\tor, \bfitu_p}_{n, \bfitw}, \widehat{\underline{\omega}}_{n, r}^{\bfitw_3^{-1}\bfitw\kappa_{\calU}})^{\fs}\rightarrow H^{3-l(\bfitw)}_{\calZ_{n, \bfitw}, \ket}(\calX^{\tor, \bfitu_p}_{n, \bfitw}, R^{l(\bfitw)}\nu_* \widehat{\underline{\omega}}_{n, r}^{\bfitw_3^{-1}\bfitw\kappa_{\calU}})^{\fs}\]
while the target is isomorphic to $H^{3-l(\bfitw)}_{\calZ_{n, \bfitw}, \ket}(\calX^{\tor, \bfitu_p}_{n, \bfitw}, \underline{\omega}_{n, r}^{\bfitw_3^{-1}\bfitw \kappa_{\calU}+k_{\bfitw}})^{\fs}$ using (\ref{eq: computation of pushforward}), (\ref{eq: Kodaira-Spencer}), and Lemma \ref{Lemma: tensor product of analytic representations}.

Finally, composing with $H^3$ of \eqref{eq: OES for finite-slope complexes}, we arrive at the desired map
\[
        \scalemath{1}{\mathrm{ES}_{\kappa_{\calU}}^{\bfitw, r}: H^3_{\calZ_{n, \bfitw}, \proket}(\calX^{\tor, \bfitu_p}_{n, \bfitw}, \scrO\!\!\scrD_{\kappa_{\calU}}^r)^{\fs} \rightarrow H^{3-l(\bfitw)}_{\calZ_{n, \bfitw}, \ket}(\calX^{\tor, \bfitu_p}_{n, \bfitw}, \underline{\omega}_{n, r}^{\bfitw_3^{-1}\bfitw \kappa_{\calU}+k_{\bfitw}})^{\fs}(\bfitw\kappa_{\calU}^{\cyc})}.
\]
The Galois-equiariance follows from the functoriality of our construction. Notice that we have kept track of the Galois twist during the process. The Hecke-operators away from $Np$ are defined via correspondences, it is then straightforward to check the Hecke-equivariance. For Hecke operators at $p$, the Hecke-equivariance follows from Proposition \ref{Proposition: OES is u_p-equivariant} (see also \cite[Proposition 5.2.5]{DRW}).
\end{proof}

\begin{Theorem}\label{Theorem: big OES diagram}
There is a natural Hecke- and Galois-equivariant diagram \[
        \begin{tikzcd}[column sep = tiny]
            \scalemath{0.8}{ H^3_{\proket}(\calX_n^{\tor}, \scrO\!\!\scrD_{\kappa_{\calU}}^r)^{\fs} } \arrow[r] & \scalemath{0.8}{ H^3_{\proket}(\calX_{n, \bfitw_3, (r,r)}^{\tor}, \scrO\!\!\scrD_{\kappa_{\calU}}^r)^{\fs} } \arrow[r] & \scalemath{0.8}{ H^0(\calX_{n, \bfitw_3, (r,r)}^{\tor}, \underline{\omega}_{n, r}^{\kappa_{\calU} + (3,3)})^{\fs}(-3) }\\
            \scalemath{0.8}{ H^3_{\overline{\calX_{n, \leq \bfitw_2}^{\tor}}, \proket}(\calX_n^{\tor}, \scrO\!\!\scrD_{\kappa_{\calU}}^r)^{\fs} } \arrow[r] \arrow[u] & \scalemath{0.8}{ H_{\calZ_{n, \bfitw_2}, \proket}^3(\calX^{\tor, \bfitu_p}_{n, \bfitw_2}, \scrO\!\!\scrD_{\kappa_{\calU}}^r)^{\fs} } \arrow[r] & \scalemath{0.8}{ H_{\calZ_{n, \bfitw_2}}^1(\calX^{\tor, \bfitu_p}_{n, \bfitw_2}, \underline{\omega}_{n, r}^{\bfitw_3^{-1}\bfitw_2\kappa_{\calU} + (3,1)})^{\fs}(\bfitw_2\kappa_{\calU}^{\cyc} - 2) }\\
            \scalemath{0.8}{ H^3_{\overline{\calX_{n, \leq \bfitw_1}^{\tor}}, \proket}(\calX_n^{\tor}, \scrO\!\!\scrD_{\kappa_{\calU}}^r)^{\fs} } \arrow[r] \arrow[u] & \scalemath{0.8}{ H_{\calZ_{n, \bfitw_1}, \proket}^3(\calX^{\tor, \bfitu_p}_{n, \bfitw_1}, \scrO\!\!\scrD_{\kappa_{\calU}}^r)^{\fs} } \arrow[r] & \scalemath{0.8}{ H_{\calZ_{n, \bfitw_1}}^2(\calX^{\tor, \bfitu_p}_{n, \bfitw_1}, \underline{\omega}_{n, r}^{\bfitw_3^{-1}\bfitw_1\kappa_{\calU} + (2,0)})^{\fs}(\bfitw_1\kappa_{\calU}^{\cyc} - 1) }\\
            \scalemath{0.8}{ H_{\overline{\calX_{n, \one_4}^{\tor}}, \proket}^3(\calX_n^{\tor}, \scrO\!\!\scrD_{\kappa_{\calU}}^r)^{\fs} } \arrow[r]\arrow[u] & \scalemath{0.8}{ H_{\calZ_{n, \one_4}, \proket}^3(\calX^{\tor, \bfitu_p}_{n, \one_4}, \scrO\!\!\scrD_{\kappa_{\calU}}^r)^{\fs} } \arrow[r] & \scalemath{0.8}{ H_{\calZ_{n, \one_4}}^3(\calX^{\tor, \bfitu_p}_{n, \one_4}, \underline{\omega}_{n, r}^{\bfitw_3^{-1}\kappa_{\calU}})^{\fs}(\kappa_{\calU}^{\cyc}) }
        \end{tikzcd}
    \]
where the second horizontal map of each row is $\mathrm{ES}_{\kappa_{\calU}}^{\bfitw, r}$ as in Proposition \ref{Proposition: OES at w}. 
\end{Theorem}

\begin{proof}
    This follows immediately from \eqref{eq: diagram for coh. with supports (1)}, Theorem \ref{Theorem: change support condition for overconvergent cohomology}, and Proposition \ref{Proposition: OES at w}.
\end{proof}

The (compositions of) the horizontal maps in Theorem \ref{Theorem: big OES diagram} are the desired \emph{overconvergent Eichler--Shimura morphisms}, as indicated in the title of this article. In fact, the top row coincides with the morphism constructed in \cite{DRW}.

There is an analogue for cuspforms. Indeed, tensoring with the boundary divisor, (\ref{eq: sheaf version OES finite level}) induces 
a morphism of pro-Kummer \'etale sheaves \[
    \mathrm{ES}_{\kappa_{\calU}, \cusp}^{\bfitw, r}: \scrO\!\!\scrD_{\kappa_{\calU}}^r(-\calD_n) \rightarrow \underline{\widehat{\omega}}_{\cusp,n, r}^{\bfitw_3^{-1}\bfitw\kappa_{\calU}}(\bfitw\kappa_{\calU}^{\cyc}),
\] which is again compatible with the action of $\bfitu_{p,0}$, $\bfitu_{p, 1}$, and $\bfitu_p$. A similar construction as in Proposition  \ref{Proposition: OES at w} produces a \emph{cuspidal overconvergent Eichler Shimura morphism}
\[
        \scalemath{1}{\mathrm{ES}_{\kappa_{\calU}, \cusp}^{\bfitw, r}: H^3_{\calZ_{n, \bfitw}, \proket}(\calX^{\tor, \bfitu_p}_{n, \bfitw}, \scrO\!\!\scrD_{\kappa_{\calU}}^r(-\calD_n))^{\fs} \rightarrow H^{3-l(\bfitw)}_{\calZ_{n, \bfitw}, \ket}(\calX^{\tor, \bfitu_p}_{n, \bfitw}, \underline{\omega}_{\cusp, n, r}^{\bfitw_3^{-1}\bfitw \kappa_{\calU}+k_{\bfitw}})^{\fs}(\bfitw\kappa_{\calU}^{\cyc})}
    \]
which fits into a Hecke- and Galois-equivariant commutative diagram
\[
    \begin{tikzcd}
        H^3_{\calZ_{n, \bfitw}, \proket}(\calX^{\tor, \bfitu_p}_{n, \bfitw}, \scrO\!\!\scrD_{\kappa_{\calU}}^r)^{\fs} \arrow[r, "\mathrm{ES}_{\kappa_{\calU}}^{\bfitw, r}"] &  H^{3-l(\bfitw)}_{\calZ_{n, \bfitw}, \ket}(\calX^{\tor, \bfitu_p}_{n, \bfitw}, \underline{\omega}_{n, r}^{\bfitw_3^{-1}\bfitw \kappa_{\calU}+k_{\bfitw}})^{\fs}(\bfitw\kappa_{\calU}^{\cyc})\\
        H^3_{\calZ_{n, \bfitw}, \proket}(\calX^{\tor, \bfitu_p}_{n, \bfitw}, \scrO\!\!\scrD_{\kappa_{\calU}}^r(-\calD_n))^{\fs} \arrow[r, "\mathrm{ES}_{\kappa_{\calU}, \cusp}^{\bfitw, r}"]\arrow[u]& H^{3-l(\bfitw)}_{\calZ_{n, \bfitw}, \ket}(\calX^{\tor, \bfitu_p}_{n, \bfitw}, \underline{\omega}_{\cusp,n, r}^{\bfitw_3^{-1}\bfitw \kappa_{\calU}+k_{\bfitw}})^{\fs}(\bfitw\kappa_{\calU}^{\cyc})\arrow[u]
    \end{tikzcd}.
\] Following the notations in \cite{BP-HigherColeman}, we denote by $\overline{H}_{?}^i$ the image of $H_?^i(\bullet, \bullet(-\calD_n))$ in $H_?^i(\bullet, \bullet)$, usually referred as the \emph{interior cohomology}.  

We have the following analogue of Theorem \ref{Theorem: big OES diagram} for interior cohomology groups.

\begin{Theorem}\label{Theorem: big OES diagram, parabolic version}
There is a natural Hecke- and Galois-equivariant diagram \[
        \begin{tikzcd}[column sep = tiny]
            \scalemath{0.8}{ \overline{H}^3_{\proket}(\calX_n^{\tor}, \scrO\!\!\scrD_{\kappa_{\calU}}^r)^{\fs} } \arrow[r] & \scalemath{0.8}{ \overline{H}^3_{\proket}(\calX_{n, \bfitw_3, (r,r)}^{\tor}, \scrO\!\!\scrD_{\kappa_{\calU}}^r)^{\fs} } \arrow[r] & \scalemath{0.8}{ \overline{H}^0(\calX_{n, \bfitw_3, (r,r)}^{\tor}, \underline{\omega}_{n, r}^{\kappa_{\calU} + (3,3)})^{\fs}(-3) }\\
            \scalemath{0.8}{ \overline{H}^3_{\overline{\calX_{n, \leq \bfitw_2}^{\tor}}, \proket}(\calX_n^{\tor}, \scrO\!\!\scrD_{\kappa_{\calU}}^r)^{\fs} } \arrow[r] \arrow[u] & \scalemath{0.8}{ \overline{H}_{\calZ_{n, \bfitw_2}, \proket}^3(\calX^{\tor, \bfitu_p}_{n, \bfitw_2}, \scrO\!\!\scrD_{\kappa_{\calU}}^r)^{\fs} } \arrow[r] & \scalemath{0.8}{ \overline{H}_{\calZ_{n, \bfitw_2}}^1(\calX^{\tor, \bfitu_p}_{n, \bfitw_2}, \underline{\omega}_{n, r}^{\bfitw_3^{-1}\bfitw_2\kappa_{\calU} + (3,1)})^{\fs}(\bfitw_2\kappa_{\calU}^{\cyc} - 2) }\\
            \scalemath{0.8}{ \overline{H}^3_{\overline{\calX_{n, \leq \bfitw_1}^{\tor}}, \proket}(\calX_n^{\tor}, \scrO\!\!\scrD_{\kappa_{\calU}}^r)^{\fs} } \arrow[r] \arrow[u] & \scalemath{0.8}{ \overline{H}_{\calZ_{n, \bfitw_1}, \proket}^3(\calX^{\tor, \bfitu_p}_{n, \bfitw_1}, \scrO\!\!\scrD_{\kappa_{\calU}}^r)^{\fs} } \arrow[r] & \scalemath{0.8}{ \overline{H}_{\calZ_{n, \bfitw_1}}^2(\calX^{\tor, \bfitu_p}_{n, \bfitw_1}, \underline{\omega}_{n, r}^{\bfitw_3^{-1}\bfitw_1\kappa_{\calU} + (2,0)})^{\fs}(\bfitw_1\kappa_{\calU}^{\cyc} - 1) }\\
            \scalemath{0.8}{ \overline{H}_{\overline{\calX_{n, \one_4}^{\tor}}, \proket}^3(\calX_n^{\tor}, \scrO\!\!\scrD_{\kappa_{\calU}}^r)^{\fs} } \arrow[r]\arrow[u] & \scalemath{0.8}{ \overline{H}_{\calZ_{n, \one_4}, \proket}^3(\calX^{\tor, \bfitu_p}_{n, \one_4}, \scrO\!\!\scrD_{\kappa_{\calU}}^r)^{\fs} } \arrow[r] & \scalemath{0.8}{ \overline{H}_{\calZ_{n, \one_4}}^3(\calX^{\tor, \bfitu_p}_{n, \one_4}, \underline{\omega}_{n, r}^{\bfitw_3^{-1}\kappa_{\calU}})^{\fs}(\kappa_{\calU}^{\cyc}) }
        \end{tikzcd}
    \]
where the last horizontal map of each row is the cuspidal overconvergent Eichler Shimura morphism constructed above. 
\end{Theorem}

\begin{Remark}\label{Remark: parabolic cohomology}
    In the diagram of Theorem \ref{Theorem: big OES diagram, parabolic version}, notice that \[
        \overline{H}^0(\calX_{n, \bfitw_3, (r,r)}^{\tor}, \underline{\omega}_{n, r}^{\kappa_{\calU} + (3,3)})^{\fs}(-3) = H^0(\calX_{n, \bfitw_3, (r,r)}^{\tor}, \underline{\omega}_{\cusp, n, r}^{\kappa_{\calU} + (3,3)})^{\fs}(-3)
    \] and \[
        \overline{H}^3_{\proket}(\calX_n^{\tor}, \scrO\!\!\scrD_{\kappa_{\calU}}^r)^{\fs} \cong H_{\Par}^3(X_n(\C), D_{\kappa_{\calU}}^r)\widehat{\otimes}\C_p
    \] by Lemma \ref{Lemma: cuspidal pro-Kummer \'etale sheaves}. Here, $H_{\Par}^3(X_n(\C), D_{\kappa_{\calU}}^r)$ stands for the image of $H_{c}^3(X_n(\C), D_{\kappa_{\calU}}^r)$ in $H^3(X_n(\C), D_{\kappa_{\calU}}^r)$.
\end{Remark}

\subsection{Overconvergent Eichler--Shimura morphisms at classical weights}\label{subsection: OES at classical weights}

Throughout this subsection, let $k = (k_1, k_2)\in \Z^2$ such that $k_1\geq k_2$. Specialising the diagram in Theorem \ref{Theorem: big OES diagram} to the classical weight $k$, we obtain the following diagram 
\begin{equation}\label{eq: big OES diagram at classical weight}
    \begin{tikzcd}[column sep = tiny]
            \scalemath{0.9}{ H^3_{\proket}(\calX_n^{\tor}, \scrO\!\!\scrD_{k}^r)^{\fs} } \arrow[r] & \scalemath{1}{ H^3_{\proket}(\calX_{n, \bfitw_3, (r,r)}^{\tor}, \scrO\!\!\scrD_{k}^r)^{\fs} } \arrow[r] & \scalemath{0.9}{ H^0(\calX_{n, \bfitw_3, (r,r)}^{\tor}, \underline{\omega}_{n, r}^{(k_1+3, k_2+3;0))})^{\fs}(-3) }\\
            \scalemath{0.9}{ H^3_{\overline{\calX_{n, \leq \bfitw_2}^{\tor}}, \proket}(\calX_n^{\tor}, \scrO\!\!\scrD_{k}^r)^{\fs} } \arrow[r] \arrow[u] & \scalemath{0.9}{ H_{\calZ_{n, \bfitw_2}, \proket}^3(\calX^{\tor, \bfitu_p}_{n, \bfitw_2}, \scrO\!\!\scrD_{k}^r)^{\fs} } \arrow[r] & \scalemath{0.9}{ H_{\calZ_{n, \bfitw_2}}^2(\calX^{\tor, \bfitu_p}_{n, \bfitw_2}, \underline{\omega}_{n, r}^{(k_1+3, -k_2+1; k_2))})^{\fs}(k_2 - 2) }\\
            \scalemath{0.9}{ H^3_{\overline{\calX_{n, \leq \bfitw_1}^{\tor}}, \proket}(\calX_n^{\tor}, \scrO\!\!\scrD_{k}^r)^{\fs} } \arrow[r] \arrow[u] & \scalemath{0.9}{ H_{\calZ_{n, \bfitw_1}, \proket}^3(\calX^{\tor, \bfitu_p}_{n, \bfitw_1}, \scrO\!\!\scrD_{k}^r)^{\fs} } \arrow[r] & \scalemath{0.9}{ H_{\calZ_{n, \bfitw_1}}^1(\calX^{\tor, \bfitu_p}_{n, \bfitw_1}, \underline{\omega}_{n, r}^{(k_2+2, k_1; k_1))})^{\fs}(k_1 - 1) }\\
            \scalemath{0.9}{ H_{\overline{\calX_{n, \one_4}^{\tor}}, \proket}^3(\calX_n^{\tor}, \scrO\!\!\scrD_{k}^r)^{\fs} } \arrow[r]\arrow[u] & \scalemath{0.9}{ H_{\calZ_{n, \one_4}, \proket}^3(\calX^{\tor, \bfitu_p}_{n, \one_4}, \scrO\!\!\scrD_{k}^r)^{\fs} } \arrow[r] & \scalemath{0.9}{ H_{\calZ_{n, \one_4}}^3(\calX^{\tor, \bfitu_p}_{n, \one_4}, \underline{\omega}_{n, r}^{(-k_2, -k_1; k_1+k_2))})^{\fs}(k_1+k_2) }
        \end{tikzcd}.
\end{equation}

We would like to answer the following natural question: how does the diagram (\ref{eq: big OES diagram at classical weight}) compare with the diagram (\ref{eq: Hecke-Galois equiv diagram algebraic localised}) in \S \ref{subsection: classical ES} induced from the classical Eichler--Shimura morphisms?

First of all, recall the sheaf $\underline{\omega}_{n, r, \alg}^{\bfitw_3^{-1}\bfitw k}$ from Remark \ref{Remark: classical automorphic sheaf via perfectoid methods}. Consider its completed pullback to the pro-Kummer \'etale site\[
    \widehat{\underline{\omega}}_{n, r, \alg}^{\bfitw_3^{-1}\bfitw k} := \upsilon^{-1}\underline{\omega}_{n, r, \alg}^{\bfitw_3^{-1}\bfitw k} \otimes_{\upsilon^{-1}\scrO_{\calX_{n, \bfitw, (r, r), \proket}^{\tor}}} \widehat{\scrO}_{\calX_{n, \bfitw, (r, r), \proket}^{\tor}}
\] 
where $\upsilon: \calX_{n, \bfitw, (r, r), \proket}^{\tor}\rightarrow \calX_{n, \bfitw, (r, r), \an}^{\tor}$ is the natural projection of sites. Note that $\widehat{\underline{\omega}}_{n, r, \alg}^{\bfitw_3^{-1}\bfitw k} = \widehat{\underline{\omega}}^{\bfitw_3^{-1}\bfitw k}|_{\calX_{n, \bfitw, (r, r)}^{\tor}}$ by Remark \ref{Remark: comparison with classical automorphic sheaf}. Moreover, recall the pro-Kummer \'etale sheaf $\scrO\!\!\scrV_k^{\vee}$ and $\scrO\!\!\scrV_{k,\adicFL}^{\vee}$ from \S \ref{subsection: classical ES}.  We would like to obtain a Hecke- and Galois-equivariant morphism of pro-Kummer \'etale sheaves \begin{equation}\label{eq: explicit ES at the sheaf level}
    \mathrm{ES}_k^{\bfitw,r, \alg}: \scrO\!\!\scrV_k^{\vee}|_{\calX_{n, \bfitw, (r, r)}^{\tor}} \rightarrow \widehat{\underline{\omega}}_{n, r, \alg}^{\bfitw_3^{-1}\bfitw k}(\bfitw k^{\cyc})
\end{equation} 
and compare it with $\ES_k^{\bfitw, r}$.

The construction is similar to the one of $\ES_{\kappa_{\calU}}^{\bfitw, r}$ in \S \ref{subsection: OES}. To this end, recall
\[
    P_{\bfitw_3^{-1}\bfitw k} = \left\{ f: H \rightarrow \bbA^1: f(\bfgamma \bfbeta) = \bfitw_3^{-1}\bfitw k(\bfbeta)f(\bfgamma) \text{ for all } (\bfgamma, \bfbeta)\in H \times B_H\right\}
\] from Remark \ref{Remark: classical automorphic sheaf via perfectoid methods}. Over $\adicFL_{\bfitw, (r,r)}$, consider the pro-\'etale sheaf \[
    \widehat{\scrP}_{\bfitw_3^{-1}\bfitw k} \coloneq P_{\bfitw_3^{-1}\bfitw k} \otimes_{\Q_p} \widehat{\scrO}_{\adicFL_{\bfitw, (r,r), \proet}}.
\] 
It follows from the construction that \[
    \pi_{\HT}^* \widehat{\scrP}_{\bfitw_3^{-1}\bfitw k} \cong \widehat{\underline{\omega}}_{n, r, \alg}^{\bfitw_3^{-1}\bfitw k}(\bfitw k^{\cyc})|_{\calX_{\Gamma(p^{\infty}), \bfitw, (r,r)}^{\tor}}.
\]
For the Galois twist, see, for example, Remark \ref{Remark: classical forms in the rigid analytic setting}. 

To construct \eqref{eq: explicit ES at the sheaf level}, we first construct a morphism \[
    \mathrm{PES}_{k}^{\bfitw, r,\alg}: \scrO\!\!\scrV_{k, \adicFL}^{\vee}|_{\adicFL_{\bfitw, (r,r)}} \rightarrow \widehat{\scrP}_{\bfitw_3^{-1}\bfitw k}
\] 
on the flag variety. Given any affinoid perfectoid object $\calV_{\infty}\in \adicFL_{\bfitw, (r, r), \proet}$, define \begin{equation}\label{eq: definition for PES_k^w,r,alg}
    \begin{split}
        \scrO\!\!\scrV_k^{\vee}(\calV_{\infty}) & \rightarrow \widehat{\scrP}_{\bfitw_3^{-1}\bfitw k}(\calV_{\infty})\\
        \mu\otimes g & \mapsto \left(\bfgamma \mapsto g\left(\int_{\bfalpha\in \GSp_4(\Q_p)}e_{k}^{\hst}\left(\bfitw^{-1}\bfitw_3\trans\bfgamma\bfitw_3^{-1}\begin{pmatrix}\one_2 & \\ \bfitz & \one_2\end{pmatrix}\bfitw \bfalpha\right)d\mu\right) \right)
    \end{split}
\end{equation}
for any section $g\in \widehat{\scrO}_{\adicFL_{\bfitw, (r,r), \proet}}(\calV_{\infty})$ and any $\mu\in V_k^{\vee}$. One checks that this indeed defines a map of sheaves. Next, pulling back $\mathrm{PES}_k^{\bfitw, r, \alg}$ via the Hodge--Tate period map, we obtain a Galois-equivariant morphism\[
    \ES_k^{\bfitw, r, \alg} : \scrO\!\!\scrV_{k}^{\vee}|_{\calX_{\Gamma(p^{\infty}), \bfitw, (r,r)}^{\tor}} \rightarrow \widehat{\underline{\omega}}_{n, r, \alg}^{\bfitw_3^{-1}\bfitw k}(\bfitw k^{\cyc})|_{\calX_{\Gamma(p^{\infty}), \bfitw, (r,r)}^{\tor}}.
\]
A similar computation as in Proposition \ref{Proposition: the ES map at the infinity descends} shows that $\ES_k^{\bfitw, r, \alg}$ descends to a morphism \[
    \ES_k^{\bfitw, r, \alg}: \scrO\!\!\scrV_{k}^{\vee}|_{\calX_{n, \bfitw, (r,r)}^{\tor}} \rightarrow \widehat{\underline{\omega}}_{n, r, \alg}^{\bfitw_3^{-1}\bfitw k}(\bfitw k^{\cyc})
\] 
on $\calX_{n, \bfitw, (r,r), \proket}^{\tor}$. Moreover, this morphism is also $\bfitu$-equivariant (for $\bfitu\in \{\bfitu_{p,0}, \bfitu_{p,1}, \bfitu_p\}$) by the same computation as in Proposition \ref{Proposition: OES is u_p-equivariant}.

We claim that $\mathrm{ES}_k^{\bfitw,r, \alg}$ agrees with the restriction of $\mathrm{ES}_k^{\bfitw, \alg}$ (see \eqref{eq: algebraic ES on the level of pro-Kummer \'etale sheaves}) on $\calX_{n, \bfitw, (r, r)}^{\tor}$. Indeed, when $\bfitw = \bfitw_3$, this is explained in \cite[Lemma 5.3.2]{DRW}. For other $\bfitw$, the map $\mathrm{ES}_k^{\bfitw, \alg}$ is obtained by twisting $\mathrm{ES}_k^{\bfitw_3, \alg}$ (as explained in \S \ref{subsection: classical ES}). The desired statement then follows from the 
the explicit formula \eqref{eq: definition for PES_k^w,r,alg}. To simplify the notation, we then occasionally drop the superscript `$r$' in $\mathrm{ES}_k^{\bfitw,r, \alg}$. 

From the construction, we observe that, over $\calX_{n, \bfitw, (r,r)}^{\tor}$, the morphism $\mathrm{ES}_k^{\bfitw, r}:  \scrO\!\!\scrD_k^r\rightarrow \widehat{\underline{\omega}}_{n, r}^{\bfitw_3^{-1}\bfitw k}(\bfitw k^{\cyc})$ factors as a Hecke- and Galois-equivariant diagram
\[
    \begin{tikzcd}
        \scrO\!\!\scrD_k^r \arrow[r, "\mathrm{ES}_k^{\bfitw, r}"] \arrow[d] & \widehat{\underline{\omega}}_{n, r}^{\bfitw_3^{-1}\bfitw k}(\bfitw k^{\cyc})\\
        \scrO\!\!\scrV_k^{\vee} \arrow[r, "\mathrm{ES}_k^{\bfitw, \alg}"] & \widehat{\underline{\omega}}_{n, r, \alg}^{\bfitw_3^{-1}\bfitw k}(\bfitw k^{\cyc}) \arrow[u, hook]
    \end{tikzcd}
\]
where the morphism $\scrO\!\!\scrD_k^r \rightarrow \scrO\!\!\scrV_k^{\vee}$ is induced from the natural inclusion $V_k \hookrightarrow A_k^r$ and the morphism $\widehat{\underline{\omega}}_{n, r, \alg}^{\bfitw_3^{-1}\bfitw k} \hookrightarrow \widehat{\underline{\omega}}_{n, r}^{\bfitw_3^{-1}\bfitw k}$ is induced by the natural inclusion $P_{\bfitw_3^{-1}\bfitw k} \hookrightarrow A_{\bfitw_3^{-1}\bfitw k}^r(\Iw_{H, 1}^+, \Q_p)$.

Following a similar construction as in \S \ref{subsection: OES}, the morphism $\mathrm{ES}_k^{\bfitw, \alg}$ induces a morphism 
\[\mathrm{ES}_k^{\bfitw, \alg}:H_{\calZ_{n, \bfitw}, \proket}^3(\calX^{\tor, \bfitu_p}_{n, \bfitw}, \scrO\!\!\scrV_k^{\vee})^{\fs}\rightarrow H_{\calZ_{n, \bfitw}}^{3-l(\bfitw)}(\calX^{\tor, \bfitu_p}_{n, \bfitw}, {\underline{\omega}}_{n, r, \alg}^{\bfitw_3^{-1}\bfitw k + k_{\bfitw}})^{\fs}(\bfitw k^{\cyc}-l(\bfitw))\]
on the cohomology groups. It fits into a Hecke- and Galois-equivariant commutative diagram 
\[
    \begin{tikzcd}
        H_{\calZ_{n, \bfitw}, \proket}^3(\calX^{\tor, \bfitu_p}_{n, \bfitw}, \scrO\!\!\scrD_k^r)^{\fs} \arrow[r, "\mathrm{ES}_{k}^{\bfitw, r}"]\arrow[d] & H_{\calZ_{n, \bfitw}}^{3-l(\bfitw)}(\calX^{\tor, \bfitu_p}_{n, \bfitw}, {\underline{\omega}}_{n, r}^{\bfitw_3^{-1}\bfitw k + k_{\bfitw}})^{\fs}(\bfitw k^{\cyc}-l(\bfitw))\\
        H_{\calZ_{n, \bfitw}, \proket}^3(\calX^{\tor, \bfitu_p}_{n, \bfitw}, \scrO\!\!\scrV_k^{\vee})^{\fs} \arrow[r, "\mathrm{ES}_k^{\bfitw, \alg}"] & H_{\calZ_{n, \bfitw}}^{3-l(\bfitw)}(\calX^{\tor, \bfitu_p}_{n, \bfitw}, {\underline{\omega}}_{n, r, \alg}^{\bfitw_3^{-1}\bfitw k + k_{\bfitw}})^{\fs}(\bfitw k^{\cyc}-l(\bfitw)) \arrow[u]
    \end{tikzcd}.
\]

For the rest of \S \ref{subsection: OES at classical weights}, we compare the diagram (\ref{eq: big OES diagram at classical weight}) with the diagram (\ref{eq: Hecke-Galois equiv diagram algebraic localised}). 

Let $\Pi = (\pi, \varphi_p)$ be a pair satisfying Assumption \ref{Assumption: Multiplicity One for cuspidal automorphic representations} and let $\frakm_{\Pi}$ be the corresponding maximal ideal in the Hecke-algebra. We further assume that $\Pi$ is nice-enough in the sense of Definition \ref{Defn: nice enough}. Localising the diagram (\ref{eq: big OES diagram at classical weight}) at $\frakm_{\Pi}$ and taking the small-slope parts, we obtain a Hecke- and Galois- equivariant diagram 
\begin{equation}\label{eq: big OES diagram at classical weight and nice point}
        \begin{tikzcd}[column sep = tiny]
            \scalemath{0.8}{ H^3_{\proket}(\calX_n^{\tor}, \scrO\!\!\scrD_{k}^r)^{\sms}_{\frakm_{\Pi}} } \arrow[r] & \scalemath{0.8}{ H^3_{\proket}(\calX_{n, \bfitw_3, (r,r)}^{\tor}, \scrO\!\!\scrD_{k}^r)^{\sms}_{\frakm_{\Pi}} } \arrow[r] & \scalemath{0.8}{ H^0(\calX_{n, \bfitw_3, (r,r)}^{\tor}, \underline{\omega}_{n, r}^{(k_1+3, k_2+3;0))})^{\sms}_{\frakm_{\Pi}}(-3) }\\
            \scalemath{0.8}{ H^3_{\overline{\calX_{n, \leq \bfitw_2}^{\tor}}, \proket}(\calX_n^{\tor}, \scrO\!\!\scrD_{k}^r)^{\sms}_{\frakm_{\Pi}} } \arrow[r] \arrow[u] & \scalemath{0.8}{ H_{\calZ_{n, \bfitw_2}, \proket}^3(\calX^{\tor, \bfitu_p}_{n, \bfitw_2}, \scrO\!\!\scrD_{k}^r)^{\sms}_{\frakm_{\Pi}} } \arrow[r] & \scalemath{0.8}{ H_{\calZ_{n, \bfitw_2}}^2(\calX^{\tor, \bfitu_p}_{n, \bfitw_2}, \underline{\omega}_{n, r}^{(k_1+3, -k_2+1; k_2))})^{\sms}_{\frakm_{\Pi}}(k_2 - 2) }\\
            \scalemath{0.8}{ H^3_{\overline{\calX_{n, \leq \bfitw_1}^{\tor}}, \proket}(\calX_n^{\tor}, \scrO\!\!\scrD_{k}^r)^{\sms}_{\frakm_{\Pi}} } \arrow[r] \arrow[u] & \scalemath{0.8}{ H_{\calZ_{n, \bfitw_1}, \proket}^3(\calX^{\tor, \bfitu_p}_{n, \bfitw_1}, \scrO\!\!\scrD_{k}^r)^{\sms}_{\frakm_{\Pi}} } \arrow[r] & \scalemath{0.8}{ H_{\calZ_{n, \bfitw_1}}^1(\calX^{\tor, \bfitu_p}_{n, \bfitw_1}, \underline{\omega}_{n, r}^{(k_2+2, k_1; k_1))})^{\sms}_{\frakm_{\Pi}}(k_1 - 1) }\\
            \scalemath{0.8}{ H_{\overline{\calX_{n, \one_4}^{\tor}}, \proket}^3(\calX_n^{\tor}, \scrO\!\!\scrD_{k}^r)^{\sms}_{\frakm_{\Pi}} } \arrow[r]\arrow[u] & \scalemath{0.8}{ H_{\calZ_{n, \one_4}, \proket}^3(\calX^{\tor, \bfitu_p}_{n, \one_4}, \scrO\!\!\scrD_{k}^r)^{\sms}_{\frakm_{\Pi}} } \arrow[r] & \scalemath{0.8}{ H_{\calZ_{n, \one_4}}^3(\calX^{\tor, \bfitu_p}_{n, \one_4}, \underline{\omega}_{n, r}^{(-k_2, -k_1; k_1+k_2))})^{\sms}_{\frakm_{\Pi}}(k_1+k_2) }
        \end{tikzcd}.
    \end{equation}
    Here the small-slope parts of $H^3_{\proket}(\calX_n^{\tor}, \scrO\!\!\scrD_{k}^r)$, $H^3_{\overline{\calX_{n, \leq \bfitw_i}^{\tor}}, \proket}(\calX_n^{\tor}, \scrO\!\!\scrD_{k}^r)$, and $H_{\calZ_{n, \bfitw_i}, \proket}^3(\calX^{\tor, \bfitu_p}_{n, \bfitw_i}, \scrO\!\!\scrD_{k}^r)$ are defined in the same way as in Definition \ref{Definition: small slope part}. (The Hecke operators are un-normalised.)

\begin{Proposition}\label{Proposition: diagrams coincide}
The digram (\ref{eq: big OES diagram at classical weight and nice point}) coincides with the diagram (\ref{eq: Hecke-Galois equiv diagram algebraic localised}).
\end{Proposition}

\begin{proof}
The desired statement follows from the following observations:
\begin{itemize}
\item The morphism $\ES_k^{\bfitw, r}$ is compatible with $\ES_{k}^{\bfitw, r, \alg}$ and the latter agrees with the restriction of $\ES_{k}^{\bfitw, \alg}$ on $\calX_{n, \bfitw, (r,r)}^{\tor}$.
\item The classicality Theorem (Theorem \ref{Theorem: classicality theorem in higher Coleman theory}) yields \begin{align*}
            H_{\calZ_{n, \bfitw}}^{3-l(\bfitw)}(\calX^{\tor, \bfitu_p}_{n, \bfitw}, {\underline{\omega}}_{n, r}^{\bfitw_3^{-1}\bfitw k + k_{\bfitw}})^{\sms}_{\frakm_{\Pi}} & \cong H_{\calZ_{n, \bfitw}}^{3-l(\bfitw)}(\calX^{\tor, \bfitu_p}_{n, \bfitw}, {\underline{\omega}}_{n, r, \alg}^{\bfitw_3^{-1}\bfitw k + k_{\bfitw}})^{\sms}_{\frakm_{\Pi}}\\
            & \cong H^{3-l(\bfitw)}(\calX_n^{\tor}, \underline{\omega}^{\bfitw_3^{-1}\bfitw k + k_{\bfitw}})^{\sms}_{\frakm_{\Pi}}.
        \end{align*}
\item The control theorem at the level of sheaves (\cite[Corollary 6.2.18]{BP-HigherColeman}) and Ash--Steven's control theorem (\cite[Theorem 3.2.5]{Hansen-PhD} or \cite[Theorem 6.4.1]{Ash-Stevens}) imply an isomorphism 
\[\scalemath{0.9}{H^3_{\overline{\calX_{n, \leq \bfitw_i}^{\tor}} \smallsetminus \overline{\calX_{n, \leq \bfitw_{i-1}}^{\tor}}, \proket}(\calX_{n}^{\tor}\smallsetminus \overline{\calX^{\tor}_{n, \leq \bfitw_{i-1}}}, \scrO\!\!\scrD_{\kappa_{\calU}}^r)^{\sms}_{\frakm_{\Pi}}\cong H^3_{\overline{\calX_{n, \leq \bfitw_i}^{\tor}} \smallsetminus \overline{\calX_{n, \leq \bfitw_{i-1}}^{\tor}}, \proket}(\calX_{n}^{\tor}\smallsetminus \overline{\calX^{\tor}_{n, \leq \bfitw_{i-1}}}, \scrO\!\!\scrV_k^{\vee})^{\sms}_{\frakm_{\Pi}}}\]
while the source of the map is isomorphic to $H^3_{\calZ_{n, \bfitw_i}, \proket}(\calX^{\tor, \bfitu_p}_{n,\bfitw}, \scrO\!\!\scrD_{\kappa_{\calU}}^r)^{\sms}_{\frakm_{\Pi}}$ by Theorem \ref{Theorem: change support condition for overconvergent cohomology}. Similarly, the control theorems also imply 
\[H^3_{\overline{\calX_{n, \leq \bfitw_i}^{\tor}}, \proket}(\calX_n^{\tor}, \scrO\!\!\scrD_{k}^r)^{\sms}_{\frakm_{\Pi}}\cong H^3_{\overline{\calX_{n, \leq \bfitw_i}^{\tor}}, \proket}(\calX_n^{\tor}, \scrO\!\!\scrV_k^{\vee})^{\sms}_{\frakm_{\Pi}}.\]
\end{itemize}
\end{proof}

\subsection{Eigenvarieties}\label{subsection: eigenvarieties}
Recall the weight space \[
    \calW = \Spa(\Z_p\llbrack T_{\GL_2}(\Z_p)\rrbrack, \Z_p\llbrack T_{\GL_2}(\Z_p)\rrbrack)^{\rig}. 
\]
In this subsection, we aim to construct two eigenvarieties $\calE^{\mathrm{oc}}$ and $\calE^{\mathrm{aut}}$ over $\calW$ (coming from $\scrO\!\!\scrD_{\kappa_{\calU}}^r$ and $\underline{\omega}_{n, r}^{\bfitw_3^{-1}\bfitw\kappa_{\calU}}$ respectively) and then compare them. We begin with the construction of the spectral varieties following \cite[\S 6]{BP-HigherColeman}. 

Let $\scrN$ be either of the following $\scrO_{\calW}$-modules: \begin{itemize}
    \item[(OC)] $\scrN(\calU) = H_{\proket}^3(\calX_n^{\tor}, \scrO\!\!\scrD_{\kappa_{\calU}}^r)$
    \item[(Aut)] $\scrN(\calU) = \bigoplus_{i=0}^{3} H_{\calZ_{n, \bfitw_i}}^{3-i}(\calX^{\tor, \bfitu_p}_{n, \bfitw_i}, \underline{\omega}_{n, r}^{\bfitw_3^{-1}\bfitw_i \kappa_{\calU}+k_{\bfitw_i}})(\bfitw_i \kappa_{\calU}^{\cyc} -i)$
\end{itemize} for any open affinoid weight $(R_{\calU}, \kappa_{\calU})$. By \cite[Proposition 3.1.5]{Hansen-PhD} and \cite[Proposition 6.1.11 \& Lemma 6.1.17]{BP-HigherColeman}, there exists $h\in \Q_{\geq 0}$ such that $\scrN(\calU)$ has \emph{slope-$\leq h$ decomposition} \[
    \scrN(\calU) = \scrN(\calU)^{\leq h} \oplus \scrN(\calU)^{>h}
\] with respect to the $U_p$-operator. Moreover, the slope decomposition is independent of the choice of $r$.\footnote{Even though, in \cite{Hansen-PhD}, the construction for the middle degree eigenvariety is not spelt out in detail, the constructions indeed apply to a single degree of cohomology; see \cite[\S 5]{BSW-ParEigen} and the references therein.}  Since $U_p$ is invertible on $\scrN(\calU)^{\leq h}$, we may consider the map \[
    R_{\calU}\widehat{\otimes}\C_p[X] \rightarrow \End_{R_{\calU}\widehat{\otimes}\C_p}(\scrN(\calU)^{\leq h}), \quad X \mapsto U_p^{-1}.
\] Let $I$ denote the kernel of this map, and consider \[
    \calZ_{\calU, h} := \Spa(R_{\calU}\widehat{\otimes}\C_p[X]/I, (R_{\calU}\widehat{\otimes}\C_p[X]/I)^+),
\] where $(R_{\calU}\widehat{\otimes}\C_p[X]/I)^+$ is the integral closure of $R_{\calU}^{\circ}\widehat{\otimes}\calO_{\C_p}$ in $R_{\calU}\widehat{\otimes}\C_p[X]/I$. The \emph{\textit{spectral variety}} is then defined to be \[
    \calZ := \left(\bigsqcup_{\calU, h} \calZ_{\calU, h}\right)/\sim,
\] where the relation $\sim$ is given by $\calZ_{\calU, h} \hookrightarrow \calZ_{\calU, h'}$ for $h\geq h'$ and $\calZ_{\calU', h}\hookrightarrow \calZ_{\calU, h}$ for $\calU' \hookrightarrow \calU$. We shall use the notation $\calZ^{\mathrm{oc}}$ (resp., $\calZ^{\mathrm{aut}}$) if $\scrN$ is of (OC) (resp., (Aut)). 

Next, we construct the eigenvarieties. Let $\bbT_{\calU}^{\leq h}$ be the equidimensional reduced $R_{\calU}\widehat{\otimes}\C_p$-algebra generated by the spherical Hecke operators (i.e., those at $\ell$ such that $\Gamma_{\ell} = \GSp_4(\Z_{\ell})$) and $U_{p, i}$'s in $\End_{R_{\calU}\widehat{\otimes}\C_p}(\scrN(\calU)^{\leq h})$. Consider the sheaf $\scrT$ on $\calZ$ given by \[
    \scrT(\calZ_{\calU, h})  := \bbT_{\calU}^{\leq h}.
\]
Since $\scrN(\calU)^{\leq h}$ is of finite rank over $R_{\calU}\widehat{\otimes}\C_p$, $\bbT_{\calU}^{\leq h}$ is a finite algebra over $R_{\calU}\widehat{\otimes}\C_p$. The \emph{\textit{eigenvariety}} $\calE$ is defined to be the relative adic spectrum \[
    \calE := \Spa_{\calZ}(\scrT, \scrT^+),
\] where $\scrT^+$ is defined as in \cite[Lemma A.3]{Johansson-Newton}. From the construction, there are natural maps \[
    \wt: \calE \rightarrow \calZ \rightarrow \calW_{\C_p}
\] whose composition is called the \emph{\textit{weight map}}. The weight map is locally finite and equidimensional. We shall use the notation $\calE^{\mathrm{oc}}$ (resp., $\calE^{\mathrm{aut}}$) if $\scrN$ is of (OC) (resp., (Aut)). Note that $\calE^{\mathrm{oc}}$ is the middle-degree version of the eigenvariety considered in \cite{Hansen-PhD}.

\begin{Proposition}\label{Proposition: comparison of eigenvarieties}
    There is an isomorphism of eigenvarieties $\calE^{\mathrm{oc}} \cong \calE^{\mathrm{aut}}$. 
\end{Proposition}

\begin{proof}
    The statement follows from applying \cite[Theorem 5.1.2]{Hansen-PhD} twice (once in each direction). To check the condition of this theorem, observe that the relevant very Zariski dense subsets consist of those points corresponding to small-slope classical cuspidal automorphic representations of $\GSp_4$. Then we use the classicality theorems (see Theorem \ref{Theorem: classicality theorem in higher Coleman theory} and \cite[Theorem 3.2.5]{Hansen-PhD}) and the classical Eichler--Shimura decomposition (Theorem \ref{Theorem: Faltings--Chai's ES decomposition}).  
\end{proof}

From now on, we shall identify $\calE^{\mathrm{oc}}$ and $\calE^{\mathrm{aut}}$, and denote them by $\calE$.

\begin{Corollary}\label{Corollary: equidimensionality}
    The eigenvariety $\calE$ is equidimensional of dimension $2$. 
\end{Corollary}
\begin{proof}
    The eigenvariety $\calE$ is equidimensional by construction. By \cite[Lemma 5.1.4]{Hansen-PhD}, it is either of dimension $\dim \calW =2$ or dimension $0$. However, when $\scrN$ is of (Aut), \cite[Proposition 6.9.4]{BP-HigherColeman} implies that $\scrN(\calU)^{\leq h}$ admits a torsion-free submodule over $R_{\calU}\widehat{\otimes}\C_p$. In particular, $\dim \bbT_{\calU}^{\leq h} \geq \dim R_{\calU}$, and hence $\calE$ can only be of dimension $2$. 
\end{proof}

\subsection{Overconvergent Eicher--Shimura decomposition on the eigenvariety}\label{subsection: ES filtration on eigenvariety}

Throughout \S \ref{subsection: ES filtration on eigenvariety}, let $\Pi$ be an irreducible cuspidal automorphic representation of $\GSp_4$ of cohomological weight $k = (k_1, k_2)$ with $k_1 \geq k_2>0$. Let $\Pi = (\pi, \varphi_p)$ be a $p$-stabilisation of $\Pi$ which satisfies Assumption \ref{Assumption: Multiplicity One for cuspidal automorphic representations} and is nice-enough in the sense of Definition \ref{Defn: nice enough}. Let $\frakm_{\Pi}$ be the corresponding maximal ideal in the Hecke algebra. Then $\frakm_{\Pi}$ defines a point $x_{\Pi}$ on the eigenvariety $\calE$.

\begin{Definition}
Let $\Pi = (\pi, \varphi_p)$, $\frakm_{\Pi}$, and $x_{\Pi}$ be as above. A \emph{good finite-slope family} passing through $\Pi$ consists of the following data:
\begin{itemize}
\item An affinoid weight $(R_{\calU}, \kappa_{\calU})$ such that $\calU=\Spa(R_{\calU}, R^{\circ}_{\calU})$ contains $\wt(x_{\Pi})=k$.
\item A connected affinoid neighbourhood $\calV \subset \calE$ containing $x_{\Pi}$ such that $\calV$ is a connected component of $\wt^{-1}(\calU_{\C_p})$.
\end{itemize}
In this case, we also say that the good finite-slope family is of weight $\kappa_{\calU}$. We write $e_{\calV}$ for the idempotent in $\scrO_{\wt^{-1}(\calU_{\C_p})}$ defining $\calV$.  
\end{Definition}

The main goal of \S \ref{subsection: ES filtration on eigenvariety} is to prove the following theorem, which asserts the existence of an overconvergent Eichler--Shimura filtration on the eigenvariety around each nice-enough point.

\begin{Theorem}\label{Theorem: overconvergent Eichler--Shimura decomposition}
Let $\Pi = (\pi, \varphi_p)$, $\frakm_{\Pi}$, and $x_{\Pi}$ be as above. Then there exists a good finite-slope family $\calV$ of weight $\kappa_{\calU}$ passing through $\Pi$ such that \begin{enumerate}
        \item[(i)] There exists $h\in \Q_{\geq 0}$ with $h$ such that $(R_{\calU}, \kappa_{\calU})$ is `slope-$h$-adapted' in the sense that the image of $\calV$ in $\calZ$ is contained in the image of $\calZ_{\calU, h}$;
        \item[(ii)] Define by a decreasing filtration $\Fil_{\mathrm{ES}, \calV}^{\bullet}$ on $e_{\calV}H_{\proket}^3(\calX_n^{\tor}, \scrO\!\!\scrD_{\kappa_{\calU}}^r)^{\leq h}$ by \begin{itemize}
   \item $\Fil^0_{\mathrm{ES}, \calV} \coloneq e_{\calV}H_{\proket}^3(\calX_n^{\tor}, \scrO\!\!\scrD_{\kappa_{\calU}}^r)^{\leq h}$;
   \item $\Fil_{\ES, \calV}^{3-i}  \coloneq e_{\calV} \image\left( H^3_{\overline{\calX_{n, \leq \bfitw_i}^{\tor}}, \proket}(\calX_n^{\tor}, \scrO\!\!\scrD_{\kappa_{\calU}}^r)^{\leq h} \rightarrow H_{\proket}^3(\calX_n^{\tor}, \scrO\!\!\scrD_{\kappa_{\calU}}^r)^{\leq h} \right)$ for $i=0,1,2$;
\item $\Fil^4_{\mathrm{ES}, \calV} \coloneq 0$.
\end{itemize}   Then $\Fil_{\ES, \calV}^{\bullet}$ is a Hecke- and Galois-stable filtration such that the graded pieces $\Gr_{\mathrm{ES}, \calV}^{\bullet}$ admit caonical Hecke- and Galois-equivariant isomorphisms
\[
            \Gr_{\mathrm{ES}, \calV}^{3-i} \cong e_{\calV}H_{\calZ_{n, \bfitw_i}}^{3-i}(\calX^{\tor, \bfitu_p}_{n,\bfitw_i}, \underline{\omega}_{n, r}^{\bfitw_3^{-1}\bfitw_i\kappa_{\calU}+k_{\bfitw_i}})^{\leq h}(\bfitw_i\kappa_{\calU}^{\cyc} - i)
\] 
of $R_{\calU}\widehat{\otimes}\C_p$-modules, for $i=0,1,2,3$.
    \end{enumerate} 
Moreover, by further shrinking $\calV$ if necessary, there is a Hecke- and Galois-equivariant decomposition 
    \[
        e_{\calV}H_{\proket}^3(\calX_n^{\tor}, \scrO\!\!\scrD_{\kappa_{\calU}}^r)^{\leq h} \cong \bigoplus_{i=0}^{3} e_{\calV}H_{\calZ_{n, \bfitw_i}}^{3-i}(\calX^{\tor, \bfitu_p}_{n, \bfitw_i}, \underline{\omega}_{n, r}^{\bfitw_3^{-1}\bfitw_i\kappa_{\calU}+k_{\bfitw_i}})^{\leq h}(\bfitw_i\kappa_{\calU}^{\cyc} - i)
    \] of $R_{\calU}\widehat{\otimes}\C_p$-modules, which specialises to the Eichler--Shimura decomposition in Proposition \ref{Prop: recover classical ES}.
\end{Theorem}
\begin{proof}
We first prove a local version of the theorem, then show that the assertions remain true in a sufficiently small neighhourhood of $\frakm_{\Pi}$. We split the proof in several steps.\\

\paragraph{\textbf{Step 1: The local statements.}}
Since $\Pi$ has small slope, there exists $h$ such that \[
        H_{\proket}^{3}(\calX_n^{\tor}, \scrO\!\!\scrV_{k}^{\vee})^{\leq h}_{\frakm_{\Pi}}\cong H_{\et}^3(X_{n, \C_p}, V_k^{\vee})_{\frakm_{\Pi}}^{\leq h}\otimes\C_p = H_{\et}^{3}(X_{n, \C_p}, V_k^{\vee})_{\frakm_{\Pi}}\otimes\C_p\cong H_{\proket}^{3}(\calX_n^{\tor}, \scrO\!\!\scrV_{k}^{\vee})_{\frakm_{\Pi}}.
    \]
    Let $(R_{\calU}, \kappa_{\calU})$ be an affinoid weight such that $\calU$ contains $k$ and such that $(R_{\calU}, \kappa_{\calU})$ is slope-$h$-adapted. Let $\frakm_k$ denote the maximal ideal of $R_{\calU}$ corresponding to the classical point $k\in \calU$ and let $R_{\calU, \frakm_k}$ denote the localisation. The digram in Theorem \ref{Theorem: big OES diagram} gives rise to the following Hecke- and Galois-equivariant diagram \begin{equation}\label{eq: big ES diagram for family localised at a nice-enough point}
        \begin{tikzcd}[column sep = tiny]
            \scalemath{0.8}{ H^3_{\proket}(\calX_n^{\tor}, \scrO\!\!\scrD_{\kappa_{\calU}}^r)^{\leq h}_{\frakm_{\Pi}} } \arrow[r] & \scalemath{0.8}{ H^3_{\proket}(\calX_{n, \bfitw_3, (r,r)}^{\tor}, \scrO\!\!\scrD_{\kappa_{\calU}}^r)^{\leq h}_{\frakm_{\Pi}} } \arrow[r] & \scalemath{0.8}{ H^0(\calX_{n, \bfitw_3, (r,r)}^{\tor}, \underline{\omega}_{n, r}^{\kappa_{\calU} + (3,3)})^{\leq h}_{\frakm_{\Pi}}(-3) }\\
            \scalemath{0.8}{ H^3_{\overline{\calX_{n, \leq \bfitw_2}^{\tor}}, \proket}(\calX_n^{\tor}, \scrO\!\!\scrD_{\kappa_{\calU}}^r)^{\leq h}_{\frakm_{\Pi}} } \arrow[r] \arrow[u] & \scalemath{0.8}{ H_{\calZ_{n, \bfitw_2}, \proket}^3(\calX^{\tor, \bfitu_p}_{n, \bfitw_2}, \scrO\!\!\scrD_{\kappa_{\calU}}^r)^{\leq h}_{\frakm_{\Pi}} } \arrow[r] & \scalemath{0.8}{ H_{\calZ_{n, \bfitw_2}}^1(\calX^{\tor, \bfitu_p}_{n, \bfitw_2}, \underline{\omega}_{n, r}^{\bfitw_3^{-1}\bfitw_2\kappa_{\calU} + (3,1)})^{\leq h}_{\frakm_{\Pi}}(\bfitw_2\kappa_{\calU}^{\cyc} - 2) }\\
            \scalemath{0.8}{ H^3_{\overline{\calX_{n, \leq \bfitw_1}^{\tor}}, \proket}(\calX_n^{\tor}, \scrO\!\!\scrD_{\kappa_{\calU}}^r)^{\leq h}_{\frakm_{\Pi}} } \arrow[r] \arrow[u] & \scalemath{0.8}{ H_{\calZ_{n, \bfitw_1}, \proket}^3(\calX^{\tor, \bfitu_p}_{n, \bfitw_1}, \scrO\!\!\scrD_{\kappa_{\calU}}^r)^{\leq h}_{\frakm_{\Pi}} } \arrow[r] & \scalemath{0.8}{ H_{\calZ_{n, \bfitw_1}}^2(\calX^{\tor, \bfitu_p}_{n, \bfitw_1}, \underline{\omega}_{n, r}^{\bfitw_3^{-1}\bfitw_1\kappa_{\calU} + (2,0)})^{\leq h}_{\frakm_{\Pi}}(\bfitw_1\kappa_{\calU}^{\cyc} - 1) }\\
            \scalemath{0.8}{ H_{\overline{\calX_{n, \one_4}^{\tor}}, \proket}^3(\calX_n^{\tor}, \scrO\!\!\scrD_{\kappa_{\calU}}^r)^{\leq h}_{\frakm_{\Pi}} } \arrow[r]\arrow[u] & \scalemath{0.8}{ H_{\calZ_{n, \one_4}, \proket}^3(\calX^{\tor, \bfitu_p}_{n, \one_4}, \scrO\!\!\scrD_{\kappa_{\calU}}^r)^{\leq h}_{\frakm_{\Pi}} } \arrow[r] & \scalemath{0.8}{ H_{\calZ_{n, \one_4}}^3(\calX^{\tor, \bfitu_p}_{n, \one_4}, \underline{\omega}_{n, r}^{\bfitw_3^{-1}\kappa_{\calU}})^{\leq h}_{\frakm_{\Pi}}(\kappa_{\calU}^{\cyc}) }
        \end{tikzcd}
    \end{equation}
of $R_{\calU, \frakm_{k}}\widehat{\otimes}\C_p$-modules. We define a decreasing filtration $\{\Fil^j_{\ES, \kappa_{\calU}, \frakm_{\Pi}}\}_{0\leq j\leq 4}$ on $H_{\proket}^3(\calX_n^{\tor}, \scrO\!\!\scrD_{\kappa_{\calU}}^r)_{\frakm_{\Pi}}^{\leq h}$ by \[
    \Fil_{\ES, \kappa_{\calU}, \frakm_{\Pi}}^{3-i} \coloneq \image\left( H^3_{\overline{\calX_{n, \leq \bfitw_i}^{\tor}}, \proket}(\calX_n^{\tor}, \scrO\!\!\scrD_{\kappa_{\calU}}^r)^{\leq h}_{\frakm_{\Pi}} \rightarrow H_{\proket}^3(\calX_n^{\tor}, \scrO\!\!\scrD_{\kappa_{\calU}}^r)^{\leq h}_{\frakm_{\Pi}} \right)
\] 
for $i=0,1,2,3$ and $\Fil^4_{\ES, \kappa_{\calU}, \frakm_{\Pi}}=0$.

We shall prove the following local statements:
\begin{enumerate}
\item[(a)] For each $\bfitw\in W^H$, the $R_{\calU, \frakm_{k}}\widehat{\otimes}\C_p$-module $H_{\calZ_{n, \bfitw}}^{3-l(\bfitw)}(\calX^{\tor, \bfitu_p}_{n, \bfitw}, \underline{\omega}^{\bfitw_3^{-1}\bfitw \kappa_{\calU}+k_{\bfitw}})^{\leq h}_{\frakm_{\Pi}}$ is free of rank 1. The specialisation map induces an isomorphism
\[\scalemath{0.9}{H_{\calZ_{n, \bfitw}}^{3-l(\bfitw)}(\calX^{\tor, \bfitu_p}_{n, \bfitw}, \underline{\omega}^{\bfitw_3^{-1}\bfitw \kappa_{\calU}+k_{\bfitw}})^{\leq h}_{\frakm_{\Pi}}\otimes_{R_{\calU, \frakm_{k}}}\Q_p\cong H_{\calZ_{n, \bfitw}}^{3-l(\bfitw)}(\calX^{\tor, \bfitu_p}_{n, \bfitw}, {\underline{\omega}}_{n, r}^{\bfitw_3^{-1}\bfitw k + k_{\bfitw}})^{\leq h}_{\frakm_{\Pi}}\cong H^{3-l(\bfitw)}(\calX_n^{\tor}, \underline{\omega}^{\bfitw_3^{-1}\bfitw k + k_{\bfitw}})_{\frakm_{\Pi}}}\]
where $R_{\calU, \frakm_{k}}\rightarrow \Q_p$ is the natural map \[R_{\calU, \frakm_{k}}\rightarrow R_{\calU, \frakm_{k}}/\frakm_{k}R_{\calU, \frakm_{k}}\cong \Q_p.\]
\item[(b)] The $R_{\calU, \frakm_{k}}\widehat{\otimes}\C_p$-module $H_{\proket}^{3}(\calX_n^{\tor}, \scrO\!\!\scrD_{\kappa_{\calU}}^r)_{\frakm_{\Pi}}^{\leq h}$ is free of rank 4. The specialisation map induces an isomorphism
\[H_{\proket}^{3}(\calX_n^{\tor}, \scrO\!\!\scrD_{\kappa_{\calU}}^r)_{\frakm_{\Pi}}^{\leq h}\otimes_{R_{\calU, \frakm_{k}}}\Q_p\cong H_{\proket}^{3}(\calX_n^{\tor}, \scrO\!\!\scrD_{k}^r)_{\frakm_{\Pi}}^{\leq h}\cong H_{\proket}^{3}(\calX_n^{\tor}, \scrO\!\!\scrV_{k}^{\vee})_{\frakm_{\Pi}}.\]
\item[(c)] For $i=0,1,2,3$, the $R_{\calU, \frakm_{k}}\widehat{\otimes}\C_p$-module $\Fil^{3-i}_{\ES, \kappa_{\calU}, \frakm_{\Pi}}$ is free of rank $i+1$. The specialisation map induces an isomorphism
\[\Fil^{3-i}_{\ES, \kappa_{\calU}, \frakm_{\Pi}}\otimes_{R_{\calU, \frakm_{k}}}\Q_p\cong \Fil^{3-i}_{\ES, k, \frakm_{\Pi}}\]
where $\Fil^{\bullet}_{\ES, k, \frakm_{\Pi}}$ is the filtration in Corollary \ref{Corollary: there exists a unique ES decomposition} and Proposition \ref{Prop: recover classical ES}.
\item[(d)] For $i=0,1,2,3$, the graded piece $\Gr^{3-i}_{\ES, \kappa_{\calU}, \frakm_{\Pi}}$ is a free $R_{\calU, \frakm_{k}}\widehat{\otimes}\C_p$-module of rank 1 and the specialisation map induces an isomorphism \[
    \Gr_{\ES, \kappa_{\calU}, \frakm_{\Pi}}^{3-i} \otimes_{R_{\calU, \frakm_k}}\Q_p \cong \Gr_{\ES, k, \frakm_{\Pi}}^{3-i}.
\] 
\item[(e)] For $i=0, 1, 2, 3$, there exists a canonical Hecke- and Galois-equivariant isomorphism \[
    \Gr_{\ES, \kappa_{\calU}, \frakm_{\Pi}}^{3-i} \cong H_{\calZ_{n, \bfitw_i}}^{3-i}(\calX^{\tor, \bfitu_p}_{n, \bfitw_i}, \underline{\omega}_{n, r}^{\bfitw_3^{-1}\bfitw_i \kappa_{\calU}+k_{\bfitw_i}})^{\leq h}_{\frakm_{\Pi}}(\bfitw_i \kappa_{\calU}^{\cyc} -i).
\]
\end{enumerate}

\paragraph{\textbf{Step 2: Proof of (a) and (b).}}
By Assumption \ref{Assumption: Multiplicity One for cuspidal automorphic representations}, we know that $H^{3-l(\bfitw)}(\calX_n^{\tor}, \underline{\omega}^{\bfitw_3^{-1}\bfitw k + k_{\bfitw}})^{\leq h}_{\frakm_{\Pi}}$ is a 1-dimensional $\C_p$-vector space. By \cite[Proposition 6.9.4 (2)]{BP-HigherColeman}, the specialisation map 
\[
H_{\calZ_{n, \bfitw}}^{3-l(\bfitw)}(\calX^{\tor, \bfitu_p}_{n, \bfitw}, \underline{\omega}^{\bfitw_3^{-1}\bfitw \kappa_{\calU}+k_{\bfitw}})^{\leq h}_{\frakm_{\Pi}}\otimes_{R_{\calU, \frakm_{k}}}\Q_p\rightarrow H_{\calZ_{n, \bfitw}}^{3-l(\bfitw)}(\calX^{\tor, \bfitu_p}_{n, \bfitw}, {\underline{\omega}}_{n, r}^{\bfitw_3^{-1}\bfitw k + k_{\bfitw}})^{\leq h}_{\frakm_{\Pi}}
\]
is an isomorphism. Hence, by Nakayama's Lemma, $H_{\calZ_{n, \bfitw}}^{3-l(\bfitw)}(\calX^{\tor, \bfitu_p}_{n, \bfitw}, \underline{\omega}^{\bfitw_3^{-1}\bfitw \kappa_{\calU}+k_{\bfitw}})^{\leq h}_{\frakm_{\Pi}}$ is generated by one element over $R_{\calU, \frakm_{k}}\widehat{\otimes}\C_p$. However, since $\Pi$ is cuspidal, the vanishing theorem (\cite[Theorem 4.1]{Lan-Vanishing}) implies that $H^{*}(\calX_n^{\tor}, \underline{\omega}^{\bfitw_3^{-1}\bfitw k + k_{\bfitw}})^{\leq h}_{\frakm_{\Pi}}$ is concentrated in degree $3-l(\bfitw)$.  Hence, by \cite[Lemma 2.9]{BDJ22}, $H_{\calZ_{n, \bfitw}}^{3-l(\bfitw)}(\calX^{\tor, \bfitu_p}_{n, \bfitw}, \underline{\omega}^{\bfitw_3^{-1}\bfitw \kappa_{\calU}+k_{\bfitw}})^{\leq h}_{\frakm_{\Pi}}$ is free of rank $1$ over $R_{\calU, \frakm_k}\widehat{\otimes}\C_p$. This proves (a).

A similar argument applies to (b). Indeed, since $\Pi = (\Pi, \varphi_p)$ has small slope, Stevens's control theorem (\cite[Theorem 3.2.5]{Hansen-PhD}) produces an isomorphism \[
    H_{\proket}^{i}(\calX_n^{\tor}, \scrO\!\!\scrD_{k}^r)_{\frakm_{\Pi}}^{\leq h}\cong H_{\proket}^{i}(\calX_n^{\tor}, \scrO\!\!\scrV_{k}^{\vee})_{\frakm_{\Pi}}
\] 
for every $i$. Since $\Pi = (\pi, \varphi_p)$ is cuspidal, \cite[Theorem 4.10]{Lan-Vanishing} implies that $H_{\proket}^*(\calX_n^{\tor}, \scrO\!\!\scrV_k^{\vee})_{\frakm_{\Pi}}$ is concentrated in degree 3. Hence, $H_{\proket}^{*}(\calX_n^{\tor}, \scrO\!\!\scrD_{k}^r)_{\frakm_{\Pi}}^{\leq h}$ is concentrated in degree 3 and $H_{\proket}^{3}(\calX_n^{\tor}, \scrO\!\!\scrD_{k}^r)_{\frakm_{\Pi}}^{\leq h}$ is 4-dimensional (by Assumption \ref{Assumption: Multiplicity One for cuspidal automorphic representations}). We conclude by applying \cite[Lemma 2.9]{BDJ22} again. \\

\paragraph{\textbf{Step 3: Proof of (c) and (d).}}
Consider the commutative diagram
\[
    \begin{tikzcd}[column sep = small]
        H^3_{\overline{\calX_{n, \leq \bfitw_i}^{\tor}}, \proket}(\calX_n^{\tor}, \scrO\!\!\scrD_{\kappa_{\calU}}^r)^{\leq h}_{\frakm_{\Pi}} \arrow[r]\arrow[d] &  H_{\proket}^3(\calX_n^{\tor}, \scrO\!\!\scrD_{\kappa_{\calU}}^r)^{\leq h}_{\frakm_{\Pi}} \arrow[r, "\Res_{\kappa_{\calU}}"]\arrow[d, two heads] & H_{\proket}^3(\calX_n^{\tor}\smallsetminus \overline{\calX_{n, \leq \bfitw_i}^{\tor}}, \scrO\!\!\scrD_{\kappa_{\calU}}^r)^{\leq h}_{\frakm_{\Pi}}\arrow[d] \\
        H^3_{\overline{\calX_{n, \leq \bfitw_i}^{\tor}}, \proket}(\calX_n^{\tor}, \scrO\!\!\scrD_{k}^r)^{\leq h}_{\frakm_{\Pi}} \arrow[r] &  H_{\proket}^3(\calX_n^{\tor}, \scrO\!\!\scrD_{k}^r)^{\leq h}_{\frakm_{\Pi}} \arrow[r, "\Res_{k}"] & H_{\proket}^3(\calX_n^{\tor}\smallsetminus \overline{\calX_{n, \leq \bfitw_i}^{\tor}}, \scrO\!\!\scrD_{k}^r)^{\leq h}_{\frakm_{\Pi}}
    \end{tikzcd}
\]
where the vertical arrows are induced by the specialisation maps. This then induces a commutative diagram \begin{equation}\label{eq: maps of exact sequences}
    \begin{tikzcd}[column sep = tiny]
        0 \arrow[r] & \Fil_{\ES, \kappa_{\calU}, \frakm_{\Pi}}^{3-i} \arrow[r]\arrow[d, two heads] & H_{\proket}^3(\calX_n^{\tor}, \scrO\!\!\scrD_{\kappa_{\calU}}^r)^{\leq h}_{\frakm_{\Pi}} \arrow[r]\arrow[d, two heads] & \image(\Res_{\kappa_{\calU}}) \arrow[r]\arrow[d, two heads] & 0\\
        & \Fil_{\ES, \kappa_{\calU}, \frakm_{\Pi}}^{3-i}\otimes_{R_{\calU, \frakm_k}}\Q_p \arrow[r]\arrow[d, hook] & H_{\proket}^3(\calX_n^{\tor}, \scrO\!\!\scrD_{\kappa_{\calU}}^r)^{\leq h}_{\frakm_{\Pi}}\otimes_{R_{\calU, \frakm_k}}\Q_p \arrow[r]\arrow[d, "\cong"] & \image(\Res_{\kappa_{\calU}})\otimes_{R_{\calU, \frakm_{k}}}\Q_p\arrow[r]\arrow[d, hook] & 0\\
        0 \arrow[r] & \Fil_{\ES, k, \frakm_{\Pi}}^{3-i}\arrow[r] & H_{\proket}^3(\calX_n^{\tor}, \scrO\!\!\scrD_{k}^r)^{\leq h}_{\frakm_{\Pi}} \arrow[r] & \image(\Res_k) \arrow[r] & 0
    \end{tikzcd}
\end{equation}
where the rows are exact sequences. Applying the Snake Lemma to the bottom two rows of \eqref{eq: maps of exact sequences}, we obtain an exact sequence 
\begin{align*}
&  \ker\left(\image(\Res_{\kappa_{\calU}}) \otimes_{R_{\calU, \frakm_k}}\Q_p \rightarrow \image(\Res_k) \right) \longrightarrow  \coker\left(\Fil_{\ES, \kappa_{\calU}, \frakm_{\Pi}}^{3-i}\otimes_{R_{\calU, \frakm_k}}\Q_p \rightarrow \Fil_{\ES, k, \frakm_{\Pi}}^{3-i} \right) \\ 
& \longrightarrow \coker\left(H_{\proket}^3(\calX_n^{\tor}, \scrO\!\!\scrD_{\kappa_{\calU}}^r)^{\leq h}_{\frakm_{\Pi}}\otimes_{R_{\calU, \frakm_k}}\Q_p \rightarrow H_{\proket}^3(\calX_n^{\tor}, \scrO\!\!\scrD_{k}^r)^{\leq h}_{\frakm_{\Pi}}\right).
\end{align*}
Since the first term and the third term are zero, the middle term is zero as well; namely, \[
    \Fil_{\ES, \kappa_{\calU}, \frakm_{\Pi}}^{3-i}\otimes_{R_{\calU, \frakm_k}}\Q_p \cong \Fil_{\ES, k, \frakm_{\Pi}}^{3-i}.
\]
This also implies that, in \eqref{eq: maps of exact sequences}, the middle row is isomorphic to the bottom row. 

It remains to show that $\Fil_{\ES, \kappa_{\calU}, \frakm_{\Pi}}^{3-i}$ is free of rank $i+1$. By Proposition \ref{Prop: recover classical ES}, we have $\dim_{\C_p} \Fil^{3-i}_{\ES, k, \frakm_{\Pi}}=i+1$, for $i=0,1,2,3$. Pick a $\C_p$-basis $\{v_1, v_2, v_3, v_4\}$ for $H^3_{\proket}(\calX_n^{\tor}, \scrO\!\!\scrD_k^r)_{\frakm_{\Pi}}^{\leq h}$ such that $\Fil^{3-i}_{\ES, k, \frakm_{\Pi}}$ is spanned by $\{v_1, \ldots, v_{i+1}\}$, for all $i=0,1,2,3$. Then we pick lifts $\widetilde{v}_1$, $\widetilde{v}_2$, $\widetilde{v}_3$, $\widetilde{v}_4$ in $H^3_{\proket}(\calX_n^{\tor}, \scrO\!\!\scrD_{\kappa_{\calU}}^r)_{\frakm_{\Pi}}^{\leq h}$ such that $\widetilde{v}_{i+1}$ lives in $\Fil^{3-i}_{\ES, \kappa_{\calU}, \frakm_{\Pi}}$, for $i=0,1,2,3$. By Nakayama's Lemma, $\widetilde{v}_1, \ldots, \widetilde{v}_{i+1}$ necessarily generate $\Fil^{3-i}_{\ES, \kappa_{\calU}, \frakm_{\Pi}}$. Consequently, it follows from the freeness of $H^3_{\proket}(\calX_n^{\tor}, \scrO\!\!\scrD_{\kappa_{\calU}}^r)_{\frakm_{\Pi}}^{\leq h}$ that $\Fil^{3-i}_{\ES, \kappa_{\calU}, \frakm_{\Pi}}$ is precisely the free $R_{\calU, \frakm_{k}}\widehat{\otimes}\C_p$-submodule of $H^3_{\proket}(\calX_n^{\tor}, \scrO\!\!\scrD_{\kappa_{\calU}}^r)_{\frakm_{\Pi}}^{\leq h}$ generated by $\widetilde{v}_1, \ldots, \widetilde{v}_{i+1}$, as desired.

As a byproduct, we know that $\Gr_{\ES, \kappa_{\calU}, \frakm_{\Pi}}^{3-i}$ is a free $R_{\calU, \frakm_{k}}\widehat{\otimes}\C_p$-module of rank 1 generated by the image of $\widetilde{v}_{i+1}$, for $i=0,1,2,3$, and that the specialisation map induces an isomorphism
\[\Gr^{3-i}_{\ES, \kappa_{\calU}, \frakm_{\Pi}}\otimes_{R_{\calU, \frakm_{k}}}\Q_p\cong \Gr^{3-i}_{\ES, k, \frakm_{\Pi}}.\]

\paragraph{\textbf{Step 4: Proof of (e).}} In this step, we show that there exists a canonical Hecke- and Galois-equivariant isomorphism \begin{equation}\label{eq: canonical ES iso on localised graded pieces}
    \Gr_{\ES, \kappa_{\calU}, \frakm_{\Pi}}^{3-i} \cong H_{\calZ_{n, \bfitw_i}}^{3-i}(\calX^{\tor, \bfitu_p}_{n, \bfitw_i}, \underline{\omega}_{n, r}^{\bfitw_3^{-1}\bfitw_i \kappa_{\calU}+k_{\bfitw_i}})^{\leq h}_{\frakm_{\Pi}}(\bfitw_i \kappa_{\calU}^{\cyc} -i).
\end{equation}

To this end, we first extract the following diagram from \eqref{eq: big ES diagram for family localised at a nice-enough point} \[
    \begin{tikzcd}
        H_{\proket}^3(\calX_n^{\tor}, \scrO\!\!\scrD_{\kappa_{\calU}}^r)_{\frakm_{\Pi}}^{\leq h}\\
        H_{\overline{\calX_{n, \leq \bfitw_i}^{\tor}}, \proket}^3(\calX_n^{\tor}, \scrO\!\!\scrD_{\kappa_{\calU}}^r)^{\leq h}_{\frakm_{\Pi}} \arrow[u, "f_i"]\arrow[r, "g_i"] & H_{\calZ_{n, \bfitw_i}}^{3-i}(\calX^{\tor, \bfitu_p}_{n, \bfitw_i}, \underline{\omega}_{n, r}^{\bfitw_3^{-1}\bfitw_i \kappa_{\calU}+k_{\bfitw_i}})^{\leq h}_{\frakm_{\Pi}}(\bfitw_i \kappa_{\calU}^{\cyc} -i)
    \end{tikzcd}.
\]
Observe from the proof of Proposition \ref{Prop: recover classical ES} that we have a commutative diagram \[
    \begin{tikzcd}[column sep = tiny]
         \scalemath{0.9}{ \Gr_{\ES, \kappa_{\calU}, \frakm_{\Pi}}^{3-i} }\arrow[dd, two heads] & \scalemath{0.9}{ H_{\overline{\calX_{n, \leq \bfitw_i}^{\tor}}, \proket}^3(\calX_n^{\tor}, \scrO\!\!\scrD_{\kappa_{\calU}}^r)^{\leq h}_{\frakm_{\Pi}} } \arrow[r, "g_i"]\arrow[d, two heads]\arrow[l, two heads] & \scalemath{0.9}{ H_{\calZ_{n, \bfitw_i}}^{3-i}(\calX^{\tor, \bfitu_p}_{n, \bfitw_i}, \underline{\omega}_{n, r}^{\bfitw_3^{-1}\bfitw_i \kappa_{\calU}+k_{\bfitw_i}})^{\leq h}_{\frakm_{\Pi}}(\bfitw_i \kappa_{\calU}^{\cyc} -i) }\arrow[dd, two heads]\\
         & \scalemath{0.9}{ H_{\overline{\calX_{n, \leq \bfitw_i}^{\tor}}, \proket}^3(\calX_n^{\tor}, \scrO\!\!\scrD_{\kappa_{\calU}}^r)^{\leq h}_{\frakm_{\Pi}} \otimes_{R_{\calU,\frakm_k}\widehat{\otimes}\C_p}\Q_p } \arrow[d, hook]\\
         \scalemath{0.9}{ \Gr_{\ES, k, \frakm_{\Pi}}^{3-i} } \arrow[rr, bend right = 10, "\cong"] & \scalemath{0.9}{ H_{\overline{\calX_{n, \leq \bfitw_i}^{\tor}}, \proket}^3(\calX_n^{\tor}, \scrO\!\!\scrD_{k}^r)^{\leq h}_{\frakm_{\Pi}} } \arrow[r, two heads, "\overline{g_i}"]\arrow[l, two heads] & \scalemath{0.9}{ H_{\calZ_{n, \bfitw_i}}^{3-i}(\calX^{\tor, \bfitu_p}_{n, \bfitw_i}, \underline{\omega}_{n, r}^{\bfitw_3^{-1}\bfitw_i k+k_{\bfitw_i}})^{\leq h}_{\frakm_{\Pi}}(\bfitw_i \kappa_{\calU}^{\cyc} -i) }
    \end{tikzcd}.
\]
In particular, we have a commutative diagram \[
    \begin{tikzcd}
        H_{\overline{\calX_{n, \leq \bfitw_i}^{\tor}}, \proket}^3(\calX_n^{\tor}, \scrO\!\!\scrD_{\kappa_{\calU}}^r)^{\leq h}_{\frakm_{\Pi}} \arrow[r, "g_i"]\arrow[d, two heads] & H_{\calZ_{n, \bfitw_i}}^{3-i}(\calX^{\tor, \bfitu_p}_{n, \bfitw_i}, \underline{\omega}_{n, r}^{\bfitw_3^{-1}\bfitw_i \kappa_{\calU}+k_{\bfitw_i}})^{\leq h}_{\frakm_{\Pi}}(\bfitw_i \kappa_{\calU}^{\cyc} -i)\arrow[d, two heads]\\
        H_{\overline{\calX_{n, \leq \bfitw_i}^{\tor}}, \proket}^3(\calX_n^{\tor}, \scrO\!\!\scrD_{\kappa_{\calU}}^r)^{\leq h}_{\frakm_{\Pi}} \otimes_{R_{\calU,\frakm_k}\widehat{\otimes}\C_p}\Q_p \arrow[r, two heads] & H_{\calZ_{n, \bfitw_i}}^{3-i}(\calX^{\tor, \bfitu_p}_{n, \bfitw_i}, \underline{\omega}_{n, r}^{\bfitw_3^{-1}\bfitw_i k+k_{\bfitw_i}})^{\leq h}_{\frakm_{\Pi}}(\bfitw_i \kappa_{\calU}^{\cyc} -i)
    \end{tikzcd},
\]
where the vertical maps are the specialisation maps. Therefore, by Nakayama's Lemma, $g_i$ is surjective.

Define $W_i \coloneq f_i(\ker g_i)$. Since $H_{\proket}^3(\calX_n^{\tor}, \scrO\!\!\scrD_{\kappa_{\calU}}^r)^{\leq h}_{\frakm_{\Pi}}$ is free of rank $4$ over $R_{\calU, \frakm_k}\widehat{\otimes}\C_p$, $W_i$ is finitely generated. Additionally, by definition, we have $\Fil_{\ES, \kappa_{\calU}, \frakm_{\Pi}}^{3-i+1} \subset W_i$, where $\Fil_{\ES, \kappa_{\calU}, \frakm_{\Pi}}^{3-i+1}$ is free of rank $i$. 

Recall from the proof of Proposition \ref{Prop: recover classical ES} that there is a canonical commutative diagram \[
    \begin{tikzcd}
        \Fil_{\ES, k, \frakm_{\Pi}}^{3-i} \arrow[rd, two heads]\\
        H_{\overline{\calX_{n,\leq \bfitw_i}^{\tor}}\proket}^{3}(\calX_n^{\tor}, \scrO\!\!\scrD_k^{r})_{\frakm_{\Pi}}^{\leq h} \arrow[r, two heads, "\overline{g_i}"]\arrow[u, "\overline{f_i}"] & H_{\calZ_{n, \bfitw_i}}^{3-i}(\calX^{\tor, \bfitu_p}_{n, \bfitw_i}, \underline{\omega}_{n, r}^{\bfitw_3^{-1}\bfitw_i k+k_{\bfitw_i}})^{\leq h}_{\frakm_{\Pi}}(\bfitw_i \kappa_{\calU}^{\cyc} -i)
    \end{tikzcd},
\]
where the top surjective arrow induces the classical Eichler--Shimura decomposition. Also recall that \[
    \ker(\Fil_{\ES, k, \frakm_{\Pi}}^{3-i} \rightarrow H_{\calZ_{n, \bfitw_i}}^{3-i}(\calX^{\tor, \bfitu_p}_{n, \bfitw_i}, \underline{\omega}_{n, r}^{\bfitw_3^{-1}\bfitw_i k+k_{\bfitw_i}})^{\leq h}_{\frakm_{\Pi}}(\bfitw_i \kappa_{\calU}^{\cyc} -i)) = \Fil_{\ES, k, \frakm_{\Pi}}^{3-i+1}
\] Hence, we must have\[
    W_i \otimes_{R_{\calU, \frakm_k}\widehat{\otimes}\C_p} \Q_p \subset \overline{f_i}(\ker \overline{g_i}) \subset \Fil_{\ES, k, \frakm_{\Pi}}^{3-i+1}.
\]
Therefore, by Nakayama's Lemma, the $R_{\calU, \frakm_k}\widehat{\otimes}\C_p$-module $W_i$ can be generated by at most $i$ elements. However, since $\Fil_{\ES, \kappa_{\calU}, \frakm_{\Pi}}^{3-i+1}$ is a free submodule of $W_i$ of rank $i$, we conclude that $W_i$ must be free of rank $i$ and must agree with $\Fil_{\ES, \kappa_{\calU}, \frakm_{\Pi}}^{3-i+1}$. 

Consequently, we arrive at Hecke- and Galois-equivariant morphisms \[
    \scalemath{0.8}{ \Gr_{\ES, \kappa_{\calU}, \frakm_{\Pi}}^{3-i} = \Fil_{\ES, \kappa_{\calU}, \frakm_{\Pi}}^{3-i}/W_i \twoheadleftarrow H_{\overline{\calX_{n, \leq \bfitw_i}^{\tor}}, \proket}^3(\calX_n^{\tor}, \scrO\!\!\scrD_{\kappa_{\calU}}^r)^{\leq h}_{\frakm_{\Pi}}/\ker g_i \cong H_{\calZ_{n, \bfitw_i}}^{3-i}(\calX^{\tor, \bfitu_p}_{n, \bfitw_i}, \underline{\omega}_{n, r}^{\bfitw_3^{-1}\bfitw_i \kappa_{\calU}+k_{\bfitw_i}})^{\leq h}_{\frakm_{\Pi}}(\bfitw_i \kappa_{\calU}^{\cyc} -i) }.
\]
Since the modules on both ends of the sequence are free of rank $1$, the surjection in the middle must be an isomorphism. This is the desired canonical Hecke- and Galois-equivariant isomorphism \eqref{eq: canonical ES iso on localised graded pieces}.\\

\paragraph{\textbf{Step 5: Spread out to a family.}} Now we spread out the local properties (a)--(e) above to a family and then achieve property (ii). 

Let $\calV$ be the connected component of $\wt^{-1}(\calU)$ that contains $x_{\Pi}$. We define a decreasing filtration $\Fil_{\ES, \calV}^{\bullet}$ on $H_{\proket}^3(\calX_n^{\tor}, \scrO\!\!\scrD_{\kappa_{\calU}}^r)^{\leq h}$ by \[
    \Fil_{\ES, \calV}^{3-i} \coloneq e_{\calV} \image\left( H^3_{\overline{\calX_{n, \leq \bfitw_i}^{\tor}}, \proket}(\calX_n^{\tor}, \scrO\!\!\scrD_{\kappa_{\calU}}^r)^{\leq h} \rightarrow H_{\proket}^3(\calX_n^{\tor}, \scrO\!\!\scrD_{\kappa_{\calU}}^r)^{\leq h} \right)
\]
and let $\Gr_{\ES, \calV}^{\bullet}$ denote the corresponding graded pieces. 

Up to shrinking $\calU$ and using the local properties (a)--(e) above, we can guarantee that
\begin{itemize}
\item $e_{\calV}H^3_{\proket}(\calX_{n}^{\tor}, \scrO\!\!\scrD_{\kappa_{\calU}}^r)^{\leq h}$ is free of rank 4 over $R_{\calU}\widehat{\otimes}\C_p$;
\item $\Fil_{\ES, \calV}^{3-i}$ is free of rank $i+1$ over $R_{\calU}\widehat{\otimes}\C_p$, for $i=0,1,2,3$;
\item $\Gr_{\ES, \calV}^{3-i}$ is free of rank 1 over $R_{\calU}\widehat{\otimes}\C_p$, for $i=0,1,2,3$;
\item $e_{\calV}H_{\calZ_{n, \bfitw_i}}^{3-i}(\calX^{\tor, \bfitu_p}_{n, \bfitw_i}, \underline{\omega}_{n, r}^{\bfitw_3^{-1}\bfitw_i \kappa_{\calU}+k_{\bfitw_i}})^{\leq h}(\bfitw_i \kappa_{\calU}^{\cyc} -i)$ is free of rank $1$ over $R_{\calU}\widehat{\otimes}\C_p$, for $i=0,1,2,3$;
\item there exists a canonical Hecke- and Galois-equivariant isomorphism \[
    \Gr_{\ES, \calV}^{3-i} \cong e_{\calV}H_{\calZ_{n, \bfitw_i}}^{3-i}(\calX^{\tor, \bfitu_p}_{n, \bfitw_i}, \underline{\omega}_{n, r}^{\bfitw_3^{-1}\bfitw_i \kappa_{\calU}+k_{\bfitw_i}})^{\leq h}(\bfitw_i \kappa_{\calU}^{\cyc} -i).
\]
\end{itemize}
These observations conclude (ii). \\

\paragraph{\textbf{Step 6: Decomposition.}} Finally, to achieve the desired decomposition, we argue as in \cite[Theorem 6.3.2]{DRW} (see also \cite[Theorem 6.1 (c)]{AIS-2015} or \cite[Theorem 5.14 (3)]{CHJ-2017}) inductively with respect to the filtration $\Fil^{\bullet}_{\mathrm{ES}, \calV}$. We sketch the proof for reader's convenience. 

Consider the Hecke- and Galois-equivariant short exact sequence \[
    0 \rightarrow \Fil^{i-1}_{\ES, \calV} \rightarrow \Fil^i_{\ES, \calV} \rightarrow \Gr_{\ES, \calV}^i \rightarrow 0.
\]
Let \[
    N_i \coloneq \Hom_{R_{\calU}}(\Gr_{\ES, \calV}^i, \Fil_{\ES, \calV}^{i-1}).
\]
The short exact sequence defines a class in $H^1(\Gal_{\Q_p}, N_i) \cong \mathrm{Ext}_{R_{\calU}[\Gal_{\Q_p}]}^1(\Gr_{\ES, \calV}^{i}, \Fil_{\ES, \calV}^{i-1})$. Let $\varphi_{\mathrm{Sen},i}$ be the Sen operator associated with $N_i$. We know from \cite[Proposition 2.3]{Kisin-2003} that $0 \neq \det \varphi_{\mathrm{Sen},i}\in R_{\calU}$ kills $H^1(\Gal_{\Q_p}, N_i)$. Therefore, after localising at this element, the short exact sequence split as semilinear $\Gal_{\Q_p}$-representations. Since the Galois-action commutes with the Hecke-actions, this splitting must be Hecke-stable. We then conclude by (once again) shrinking $\calV$ if necessary.
\end{proof}

\begin{Corollary}\label{Corollary: etaleness of the weight map}
    Let $\Pi = (\pi, \varphi_p)$ be a $p$-stabilisation of $\Pi$ that satisfies Assumption \ref{Assumption: Multiplicity One for cuspidal automorphic representations} and has small slope. Then, $\scrO_{\calE, x_{\Pi}}$ is free of rank $1$ over $R_{\calU, \frakm_k}\widehat{\otimes}\C_p$ and the weight map $\wt$ is \'etale at $x_{\Pi}$. 
\end{Corollary}
\begin{proof}
    By the proof of Theorem \ref{Theorem: overconvergent Eichler--Shimura decomposition}, there is a Hecke-equivariant decomposition \[
        H_{\proket}^3(\calX_n^{\tor}, \scrO\!\!\scrD_{\kappa_{\calU}}^r)^{\leq h}_{\frakm_{\Pi}} \cong \bigoplus_{i=0}^{3}\Gr_{\ES, \kappa_{\calU}, \frakm_{\Pi}}^{3-i}.
    \]
    for each $i=0,1,2,3$. This induces an injection \[
        \scrO_{\calE, x_{\Pi}} \hookrightarrow \End_{R_{\calU, \frakm_k}}(\Gr_{\ES, \kappa_{\calU}, \frakm_{\Pi}}^{3-i}).
    \]
On the other hand, since each $\Gr_{\ES, \kappa_{\calU}, \frakm_{\Pi}}^{3-i}$ is free of rank one over $R_{\calU, \frakm_k}\widehat{\otimes}\C_p$, we have \[\End_{R_{\calU, \frakm_k}}(\Gr_{\ES, \kappa_{\calU}, \frakm_{\Pi}}^{3-i}) \cong R_{\calU, \frakm_k}\widehat{\otimes}\C_p.\] One concludes that $\scrO_{\calE, x_{\Pi}} \cong R_{\calU, \frakm_k}\widehat{\otimes}\C_p$, as desired.
\end{proof}

\begin{Corollary}\label{Corollary: big Galois representation}
    Let $\Pi = (\pi, \varphi_p)$ be a $p$-stabilisation of $\Pi$ that satisfies Assumption \ref{Assumption: Multiplicity One for cuspidal automorphic representations} and has small slope. Let $\calV$ be a good finite-slope family of weight $(R_{\calU}, \kappa_{\calU})$ passing through $x_{\Pi}$ as in Theorem \ref{Theorem: overconvergent Eichler--Shimura decomposition}. Then, there exists a family of Galois representations \[
        \rho_{\calV}: \Gal_{\Q} \rightarrow \GL_4(R_{\calU})
    \] attached to $\calV$ such that \begin{enumerate}
        \item[(i)] $\rho_{\calV}$ is unramified at $\ell\nmid Np$ and the characteristic polynomial of the geometric Frobenius at $\ell$ agrees with the Hecke polynomial at $\ell$\footnote{ For the definition of the Hecke polynomials, we refer the readers to \cite[\S 3.1]{GT-TWGSp4}.}; 
        \item[(ii)] $\rho_{\calV}|_{\Gal_{\Q_p}}\widehat{\otimes} \C_p$ admits a Galois-stable decreasing filtration and has Hodge--Tate--Sen weight $(-3, \kappa_{\calU, 2}-2, \kappa_{\calU, 1}-1, \kappa_{\calU,1}+\kappa_{\calU, 2})$, ordered by the labeling of the graded pieces of the filtration. 
    \end{enumerate}
\end{Corollary}
\begin{proof}
    Let $(\kappa_{\calY}, R_{\calY})$ be an open small weight (i.e., a small weight that is also an open weight in the sense of Remark \ref{Remark: maps of weights to the weight space}) such that $\calU \hookrightarrow \calY \hookrightarrow \calW$. Define \[
        H_{\et}^3(X_{n, \overline{\Q}}, \scrD_{\kappa_{\calU}}^r) := H_{\et}^3(X_{n, \overline{\Q}}, \scrD_{\kappa_{\calY}}^r)\widehat{\otimes}R_{\calU} \quad \text{ and }\quad H_{\proket}^3(\calX_n^{\tor}, \scrD_{\kappa_{\calU}}^r) := H_{\proket}^3(\calX_n^{\tor}, \scrD_{\kappa_{\calY}}^r)\widehat{\otimes}R_{\calU}.
    \]
There is a sequence of isomorphisms \[
     \scalemath{0.9}{   H_{\proket}^3(\calX_n^{\tor}, \scrD_{\kappa_{\calU}}^r) \cong H^3(X_n(\C), D_{\kappa_{\calU}}^r) \cong H^3(X_n(\C), D_{\kappa_{\calY}}^r)\widehat{\otimes}R_{\calU} \cong H_{\et}^3(X_{n, \overline{\Q}}, \scrD_{\kappa_{\calY}}^r)\widehat{\otimes}R_{\calU} = H_{\et}^3(X_{n, \overline{\Q}}, \scrD_{\kappa_{\calU}}^r),}
    \]
    where the third isomorphism follows from Artin comparison. 
    Note that the composition of the isomorphisms is Galois-equivariant. In particular, $H_{\proket}^3(\calX_n^{\tor}, \scrD_{\kappa_{\calU}}^r)$ is equipped with a natural continuous action of $\Gal_{\Q}$. Observe that there is a natural morphism \[
        H_{\proket}^3(\calX_n^{\tor}, \scrD_{\kappa_{\calU}}^r) \rightarrow H_{\proket}^3(\calX_n^{\tor}, \scrO\!\!\scrD_{\kappa_{\calU}}^r).
    \]
    Choose $h$ as in Theorem \ref{Theorem: overconvergent Eichler--Shimura decomposition} and define \[
        e_{\calV}  H_{\proket}^3(\calX_n^{\tor}, \scrD_{\kappa_{\calU}}^r)^{\leq h} := \text{preimage of }e_{\calV}H_{\proket}^3(\calX_n^{\tor}, \scrO\!\!\scrD_{\kappa_{\calU}}^r)^{\leq h}.
    \]
    One sees from the construction that \[
        e_{\calV}  H_{\proket}^3(\calX_n^{\tor}, \scrD_{\kappa_{\calU}}^r)^{\leq h} \widehat{\otimes}\C_p = e_{\calV}H_{\proket}^3(\calX_n^{\tor}, \scrO\!\!\scrD_{\kappa_{\calU}}^r)^{\leq h},
    \]
    hence $e_{\calV}  H_{\proket}^3(\calX_n^{\tor}, \scrD_{\kappa_{\calU}}^r)^{\leq h}$ is free of rank $4$ over $R_{\calU}$. We then define $\rho_{\calV}$ to be the Galois representation \[
        \rho_{\calV}: \Gal_{\Q} \rightarrow \Aut_{R_{\calU}}(e_{\calV}  H_{\proket}^3(\calX_n^{\tor}, \scrD_{\kappa_{\calU}}^r)^{\leq h}).
    \]
    The second assertion then follows immediately from Theorem \ref{Theorem: overconvergent Eichler--Shimura decomposition}. 
    

    For (i), given $\ell \nmid Np$, let $P_{\varphi_{\ell}}$ (resp., $P_{\mathrm{Hecke}, \ell}$) be the characteristic polynomial of the geometric Frobenius at $\ell$ (resp., Hecke polynomial at $\ell$). Then, for any classical point $y$ with residue field $F_{y}$, let $P_{\varphi_{\ell}}|_{F_{y}}$ (resp., $P_{\mathrm{Hecke}, \ell}|_{F_{y}}$) be the base change of $P_{\varphi_{\ell}}$ (resp., $P_{\mathrm{Hecke}, \ell}$) to $F_{y}$. According to \cite[Theorem I]{Weissauer-4dimGalRep}, we have \[
        P_{\varphi_{\ell}}|_{F_{y}} = P_{\mathrm{Hecke}, \ell}|_{F_{y}}.
    \]
    However, since classical points are Zariski dense in $\calV$ (\cite[Theorem 5.4.4]{Urban-2011}), the desired assertion then follows. 
\end{proof}

\begin{Remark}\label{Remark: families of Galois representation}
    \begin{enumerate}
        \item[(i)] Compared with the result on Galois representations in \cite{DRW}, Corollary \ref{Corollary: big Galois representation} provides more information on the Hodge--Tate--Sen weight. 
        \item[(ii)] Corollary \ref{Corollary: big Galois representation} also implies that we can attach concrete Galois representations to overconvergent Siegel modular forms (if it lives in a nice enough family) without passing through pseudo-representations or determinants. More precisely, for any $y\in \calV$ that corresponds to a maximal ideal $\frakm_{y}\subset \scrO_{\calV}(\calV)$ with $\wt(y)= \kappa_{y}$, we have the Galois representation \[
            \rho_{y}: \Gal_{\Q} \rightarrow \GL_4(R_{\calU}/\frakm_{\kappa_{y}})
        \] obtained by $\rho_{\calV} \mod \frakm_{\kappa_{y}}$. Moreover, $\rho_{y}$ also satisfies the analogous (i) and (ii) in Corollary \ref{Corollary: big Galois representation}.
    \end{enumerate}
\end{Remark}

\subsection{The case of non-neat level}\label{subsection: non-neat level}
Often, one needs to work with levels that are not neat (e.g., modular forms of level $\Gamma_0(N)$). In this subsection, we briefly explain how to deduce results for the overconvergent cohomology groups of non-neat level from the results in the previous sections. The idea is choosing an auxiliary neat level and then taking group invariants; see for example \cite[Remark 8.3.1]{AIP-2015}.

Let $\Gamma$ be the same as before. Let $\Gamma'$ be a non-neat compact open subgroup of $\GSp_4(\A^{\infty, p}_{\Q})$. Suppose that $\Gamma'$ contains $\Gamma$ as a normal subgroup. Consider the compact open subgroup
\[
    \Gamma'_n \coloneq \Gamma'\Iw_{\GSp_4, n}^+ \subset \GSp_4(\A_{\Q}^{\infty}).
\]
Note that $\Gamma'/\Gamma$ is a finite group. By \cite[Theorem 4.3.4]{Zavyalov-Quotient}, we know that \[
    \calX_n' \coloneq \calX_n/(\Gamma'/\Gamma) \quad \text{ and }\quad \calX_n'^{\tor} \coloneq \calX_n^{\tor}/(\Gamma'/\Gamma)
\]
exist as adic spaces. Via the fixed isomorphism $\C_p \cong \C$, the $\C$-points of $\calX_n'$ agrees with the locally symmetric space
\[
    X'_{n}(\C) = \GSp_{4}(\Q) \backslash \GSp_{4}(\A_{\Q}^{\infty}) \times \bbH_2 / \Gamma_n'.
\]
Moreover, the natural morphism \[
    \varphi: \calX_n^{\tor} \rightarrow \calX_n'^{\tor}
\]
is a finite surjective morphism of adic spaces, and the fibres of $\varphi$ are exactly the $(\Gamma'/\Gamma)$-orbits.


However, in general, $\calX_n'^{\tor}$ is not smooth. It is also unclear whether $\calX_n'^{\tor}$ is necessarily an fs log adic space. Therefore, the constructions in the previous sections do not directly apply to $\calX_n'^{\tor}$. Nevertheless, there is an action of the finite group $\Gamma'/\Gamma$ on each $H^3_{\overline{\calX_{n, \leq \bfitw}^{\tor}}, \proket}(\calX_n^{\tor}, \scrO\!\!\scrD_{\kappa_{\calU}}^r)$ and we may simply view $H^3_{\overline{\calX_{n, \leq \bfitw}^{\tor}}, \proket}(\calX_n^{\tor}, \scrO\!\!\scrD_{\kappa_{\calU}}^r)^{\Gamma'/\Gamma}$ as a substitution of the desired overconvergent cohomology group `$H^3_{\overline{\calX_{n, \leq \bfitw}'^{\tor}}, \proket}(\calX_n'^{\tor}, \scrO\!\!\scrD_{\kappa_{\calU}}^r)$'. Similarly, we may consider $H_{\calZ_{n, \bfitw}, \proket}^3(\calX^{\tor, \bfitu_p}_{n, \bfitw}, \scrO\!\!\scrD_{\kappa_{\calU}}^r)^{\Gamma'/\Gamma}$ and $H_{\calZ_{n, \bfitw}}^{3-l(\bfitw)}(\calX^{\tor, \bfitu_p}_{n, \bfitw}, \underline{\omega}_{n, r}^{\bfitw_3^{-1}\bfitw\kappa_{\calU} + k_{\bfitw}})^{\Gamma'/\Gamma}$. Indeed, these finite group invariants only depend on $\Gamma'$; namely, they are independent of the choice of $\Gamma$.

\begin{Remark}
When $\Gamma'$ is the paramodular level, the space $H^0(\calX^{\tor, \bfitu_p}_{n, \bfitw}, \underline{\omega}_{n, r}^k)^{\Gamma'/\Gamma}$ is precisely the space of `overconvergent paramodular Siegel modular forms'. See also \cite[Remark 3.2.1]{LZ-BKGSp4XGL2}.
\end{Remark}

The following result is an immediate corollary of Theorem \ref{Theorem: big OES diagram}. The horizontal arrows in the diagram can be viewed as the \emph{overconvergent Eichler--Shimura morphisms of level $\Gamma'_n$}.

\begin{Theorem}\label{Theorem: big OES diagram non-neat}
Let $(R_{\calU}, \kappa_{\calU})$ be an affinoid weight and suppose $n\geq r\geq 1+r_{\calU}$. Then there is a natural Hecke- and Galois-equivariant diagram \[
        \begin{tikzcd}[column sep = tiny]
            \scalemath{0.8}{ \left(H_{\proket}^3(\calX_n^{\tor}, \scrO\!\!\scrD_{\kappa_{\calU}}^r)^{\Gamma'/\Gamma}\right)^{\fs} } \arrow[r] & \scalemath{0.8}{ \left(H^3_{\proket}(\calX_{n, \bfitw_3, (r,r)}^{\tor}, \scrO\!\!\scrD_{\kappa_{\calU}}^r)^{\Gamma'/\Gamma}\right)^{\fs} } \arrow[r] & \scalemath{0.8}{ \left(H^0(\calX_{n, \bfitw_3, (r,r)}^{\tor}, \underline{\omega}_{n, r}^{\kappa_{\calU} + (3,3)})^{\Gamma'/\Gamma}\right)^{\fs}(-3) }\\
            \scalemath{0.7}{ \left(H^3_{\overline{\calX_{n, \leq \bfitw_2}^{\tor}}, \proket}(\calX_n^{\tor}, \scrO\!\!\scrD_{\kappa_{\calU}}^r)^{\Gamma'/\Gamma}\right)^{\fs} } \arrow[r] \arrow[u] & \scalemath{0.7}{ \left(H_{\calZ_{n, \bfitw_2}, \proket}^3(\calX^{\tor, \bfitu_p}_{n, \bfitw_2}, \scrO\!\!\scrD_{\kappa_{\calU}}^r)^{\Gamma'/\Gamma}\right)^{\fs} } \arrow[r] & \scalemath{0.7}{ \left(H_{\calZ_{n, \bfitw_2}}^1(\calX^{\tor, \bfitu_p}_{n, \bfitw_2}, \underline{\omega}_{n, r}^{\bfitw_3^{-1}\bfitw_2\kappa_{\calU} + (3,1)})^{\Gamma'/\Gamma}\right)^{\fs}(\bfitw_2\kappa_{\calU}^{\cyc} - 2) }\\
            \scalemath{0.7}{ \left(H^3_{\overline{\calX_{n, \leq \bfitw_1}^{\tor}}, \proket}(\calX_n^{\tor}, \scrO\!\!\scrD_{\kappa_{\calU}}^r)^{\Gamma'/\Gamma}\right)^{\fs} } \arrow[r] \arrow[u] & \scalemath{0.7}{ \left(H_{\calZ_{n, \bfitw_1}, \proket}^3(\calX^{\tor, \bfitu_p}_{n, \bfitw_1}, \scrO\!\!\scrD_{\kappa_{\calU}}^r)^{\Gamma'/\Gamma}\right)^{\fs} } \arrow[r] & \scalemath{0.7}{ \left(H_{\calZ_{n, \bfitw_1}}^2(\calX^{\tor, \bfitu_p}_{n, \bfitw_1}, \underline{\omega}_{n, r}^{\bfitw_3^{-1}\bfitw_1\kappa_{\calU} + (2,0)})^{\Gamma'/\Gamma}\right)^{\fs}(\bfitw_1\kappa_{\calU}^{\cyc} - 1) }\\
            \scalemath{0.7}{ \left(H_{\overline{\calX_{n, \one_4}^{\tor}}, \proket}^3(\calX_n^{\tor}, \scrO\!\!\scrD_{\kappa_{\calU}}^r)^{\Gamma'/\Gamma}\right)^{\fs} } \arrow[r]\arrow[u] & \scalemath{0.7}{ \left(H_{\calZ_{n, \one_4}, \proket}^3(\calX^{\tor, \bfitu_p}_{n, \one_4}, \scrO\!\!\scrD_{\kappa_{\calU}}^r)^{\Gamma'/\Gamma}\right)^{\fs} } \arrow[r] & \scalemath{0.7}{ \left(H_{\calZ_{n, \one_4}}^3(\calX^{\tor, \bfitu_p}_{n, \one_4}, \underline{\omega}_{n, r}^{\bfitw_3^{-1}\kappa_{\calU}})^{\Gamma'/\Gamma}\right)^{\fs}(\kappa_{\calU}^{\cyc}) }
        \end{tikzcd}
        .
    \]
\end{Theorem}
\begin{proof}
It suffices to notice that the group action by $\Gamma'/\Gamma$ commutes with the Hecke- and Galois-actions.



\end{proof}

\begin{Remark}
Given Theorem \ref{Theorem: big OES diagram non-neat}, one can proceed and deduce analogues of Theorems \ref{Theorem: big OES diagram, parabolic version} and Theorem \ref{Theorem: overconvergent Eichler--Shimura decomposition} for non-neat levels. For example, in an analogue of Theorem \ref{Theorem: overconvergent Eichler--Shimura decomposition}, one implements an assumption similar to Assumption \ref{Assumption: Multiplicity One for cuspidal automorphic representations}, but replacing all $\Gamma$ therein with $\Gamma'$. We leave the details to the interested reader.
\end{Remark}

\begin{appendix}
    \section{Cohomology with supports}\label{section: cohomology with supports}

The goal of this appendix is to study the (pro-)Kummer \'etale cohomology with supports over an adic space. In \S \ref{subsection: basic coh with support}, we introduce the basic definitions of such a cohomology theory, as well as some basic properties. In \S \ref{subsection: coh with supp spectral sequence}, we study the spectral sequence induced from an stratification. Finally, in \S \ref{subsection: Banach sheaves and pro-Kummer \'etale cohomology with supports}, we focus on the situation where the coefficients of the cohomology theory are Banach sheaves (Definition \ref{Definition: proket Banach sheaves}). Our discussion is highly inspired by \cite[\S 2.5]{BP-HigherColeman}.  

Throughout \S \ref{section: cohomology with supports}, let $\calX$ be a locally noetherian fs log adic space over some affinoid field $\Spa(K, K^+)$ where $K$ is a complete non-archimedean field extension of $\Q_p$.

\subsection{Basic definitions and properties}\label{subsection: basic coh with support}

Let $\calZ \subset \calX$ be a closed topological subset. (Here, $\calZ$ is not necessarily an adic space itself.) We denote by $\calU := \calX \smallsetminus \calZ$ and write $\jmath: \calU \hookrightarrow \calX$ for the natural embedding. Since $\calU$ is open in $\calX$, it is naturally an adic space. We view $\calU$ as a log adic space equipped with the pullback log structure from $\calX$. For $\tau \in \{\mathrm{an}, \ket, \proket\}$, there are natural morphisms of sites \[
    \jmath_{\tau}: \calU_{\tau} \rightarrow \calX_{\tau}. 
\]
To simplify the notation, we often abuse the notation and write $\jmath$ instead of $\jmath_{\tau}$, when the underlying topology is clear. 

\begin{Definition}\label{Definition: cohomology with supp}
Let $\tau\in \{\text{an}, \ket, \proket\}$ and let $\scrF$ be an abelian sheaf on $\calX_{\tau}$. The \emph{$\tau$-cohomology of $\scrF$ with support in $\calZ$} is defined to be the mapping cone \[
        R\Gamma_{\calZ, \tau}(\calX, \scrF) := \mathrm{Cone}\left(R\Gamma_{\tau}(\calX, \scrF) \xrightarrow{\mathrm{res}} R\Gamma_{\tau}(\calU, \scrF|_{\calU_{\tau}}) \right)[-1].
    \]
The corresponding cohomology groups are denoted by $H^i_{\calZ, \tau}(\calX, \scrF)$.
\end{Definition}

\begin{Remark}\label{Remark: definition of cohomology with support}
    \begin{enumerate}
        \item[(i)] Equivalently, $R\Gamma_{\calZ, \tau}(\calX, -)$ can be defined as the right derived functor of the functor \[
            \Gamma_{\calZ, \tau}(\calX, -) := \ker\left(\Gamma(\calX, -) \xrightarrow{\mathrm{res}} \Gamma(\calU, -)\right)
        \]
        on the category of abelian sheaves on $\calX_{\tau}$.
        \item[(ii)] When $\tau = \an$, Definition \ref{Definition: cohomology with supp} is nothing but the classical cohomology with supports. Readers are referred to \cite[Expos\'e I]{SGA2} for more detailed discussion.
    \end{enumerate} 
\end{Remark}

We observe the following properties for cohomology with supports. \\

\paragraph{\textbf{Distinguished triangle.}} There is a distinguished triangle \begin{equation}\label{eq: distinguished triangle}
    R\Gamma_{\calZ, \tau}(\calX, \scrF) \rightarrow R\Gamma_{\tau}(\calX, \scrF) \rightarrow R\Gamma_{\tau}(\calU, \scrF|_{\calU_{\tau}}),
\end{equation}
which follows immediately from the definition. \\

\paragraph{\textbf{Corestriction.}} Suppose $\calZ_1 \subset \calZ_2 \subset \calX$ are two closed topological subspaces. There is a corestriction map \begin{equation}\label{eq: corestriction map}
    \mathrm{cores}: R\Gamma_{\calZ_1, \tau}(\calX, \scrF) \rightarrow R\Gamma_{\calZ_2, \tau}(\calX, \scrF).
\end{equation} 
Indeed, let $\calU_i := \calX \smallsetminus \calZ_i$. Then the corestriction map fits into a morphism of distinguished triangles \[
    \begin{tikzcd}
        R\Gamma_{\calZ_1, \tau}(\calX, \scrF)\arrow[d] \arrow[r] & R\Gamma_{\tau}(\calX, \scrF) \arrow[d, equal]\arrow[r] & R\Gamma_{\tau}(\calU_1, \scrF|_{\calU_{1, \tau}})\arrow[d]\\
        R\Gamma_{\calZ_2, \tau}(\calX, \scrF) \arrow[r] & R\Gamma_{\tau}(\calX, \scrF)\arrow[r] & R\Gamma_{\tau}(\calU_2, \scrF|_{\calU_{2, \tau}})
    \end{tikzcd},
\]
where the vertical arrow on the right-hand side is the restriction map. \\

\paragraph{\textbf{Pullbacks.}} Let $f: \calX' \rightarrow \calX$ be a log smooth morphism of locally noetherian fs log adic spaces over $\Spa(K, K^+)$. Let $\calZ \subset \calX$ and $\calZ' \subset \calX'$ be closed subspaces such that $f^{-1}(\calZ) \subset \calZ'$. In particular, we have $f(\calU')\subset \calU$ where $\calU=\calX\smallsetminus \calZ$ and $\calU'=\calX'\smallsetminus \calZ'$. Then there is a natural pullback map \begin{equation}\label{eq: pullback map}
    R\Gamma_{\calZ, \tau}(\calX, \scrF) \rightarrow R\Gamma_{\calZ', \tau}(\calX', f^{-1}\scrF). 
\end{equation} 
Indeed, the pullback map fits into a morphism of distinguished triangles \[
    \begin{tikzcd}
        R\Gamma_{\calZ, \tau}(\calX, \scrF) \arrow[r]\arrow[d] & R\Gamma_{\tau}(\calX, \scrF) \arrow[d]\arrow[r] & R\Gamma_{\tau}(\calU, \scrF|_{\calU_{\tau}})\arrow[d]\\
        R\Gamma_{\calZ', \tau}(\calX', f^{-1}\scrF) \arrow[r] & R\Gamma_{\tau}(\calX', f^{-1}\scrF) \arrow[r] & R\Gamma_{\tau}(\calU', f^{-1}\scrF|_{\calU'_{\tau}})
    \end{tikzcd}
\]
where the vertical arrows in the middle and on the right are the usual pullback maps on the cohomology groups without supports.\\

\paragraph{\textbf{Change of ambient spaces.}} Let $\calZ$ be a closed subset of $\calX$ and let $\calW\subset \calX$ be an open subspace of $\calX$ that contains $\calZ$. We equip $\calW$ with the pullback log structure from $\calX$; namely, the inclusion $\jmath:\calW\subset \calX$ is a strict open immersion of locally noetherian fs log adic spaces. Then the pullback map along $\jmath$ induces a quasi-isomorphism \begin{equation}\label{eq: change of ambient spaces}
    R\Gamma_{\calZ, \tau}(\calX, \scrF) \cong R\Gamma_{\calZ, \tau}(\calW, \scrF|_{\calW_{\tau}}).
\end{equation}
Indeed, there is an isomorphism
\[ \Gamma_{\calZ, \tau}(\calX, -)\cong  \Gamma_{\calZ, \tau}(\calW, -)\circ \jmath^{-1}\]
where $\jmath^{-1}:\mathrm{Sh}_{\mathrm{Ab}}(\calX_{\tau})\rightarrow\mathrm{Sh}_{\mathrm{Ab}}(\calW_{\tau})$ is the restriction map. It suffices to notice that $\jmath^{-1}$ is exact, hence sends injective sheaves to injective sheaves.

\begin{Remark}\label{Remark: change of ambient spaces}
Let $\calZ\subset \calX$ be a closed subset and let $\calX\subset \calX'$ be a strict open immersion of locally noetherian fs log adic spaces. Suppose $\scrF$ is an abelian sheaf on $\calX_{\tau}$. Inspired by (\ref{eq: change of ambient spaces}), we sometimes abuse the notation and write $R\Gamma_{\calZ, \tau}(\calX', \scrF)$, by which we mean $R\Gamma_{\calZ, \tau}(\calX, \scrF)$.
\end{Remark}

The following lemma is an analogue of \cite[Lemma 2.1.1]{BP-HigherColeman}.

\begin{Lemma}\label{Lemma: direct sum for cohomology with supports}
    Let $\calZ_1, \calZ_2 \subset \calX$ be two closed subsets such that $\calZ_1 \cap \calZ_2 = \emptyset$. Then the corestriction maps induces a quasi-isomorphism
\[
        R\Gamma_{\calZ_1, \tau}(\calX, \scrF) \oplus R\Gamma_{\calZ_2, \tau}(\calX, \scrF) \xrightarrow[]{\cong} R\Gamma_{\calZ_1 \cup \calZ_2, \tau}(\calX, \scrF).
    \]
\end{Lemma}

\begin{proof}
It suffices to observe that the map
\[\Gamma_{\calZ_1, \tau}(\calX, -)\oplus \Gamma_{\calZ_2, \tau}(\calX, -)\rightarrow \Gamma_{\calZ_1\cup \calZ_2, \tau}(\calX, -)\]
sending $(s_1, s_2)\mapsto s_1+s_2$ is an isomorphism.
\end{proof}

\subsection{A spectral sequence}\label{subsection: coh with supp spectral sequence}

The following spectral sequence is an analogue of \cite[p. 227]{Hartshorne-residue}. (Also see \cite[\S 2.3]{BP-HigherColeman}.)

\begin{Proposition}\label{Proposition: a spectral sequence for stratifications}
    Let $\calX$ be a locally noetherian fs log adic space as above. Consider a stratification \[
        \calX = \calZ_0 \supsetneq \calZ_1 \supsetneq \cdots \supsetneq \calZ_n \supsetneq \emptyset
    \]
    by closed subspaces of $|\calX|$. Then, for any abelian sheaf $\scrF$ on $\calX_{\tau}$, there is an $E_1$-spectral sequence \[
        E_1^{i,j} =  H_{\calZ_i \smallsetminus \calZ_{i+1}, \tau}^{i+j}(\calX \smallsetminus \calZ_{i+1}, \scrF) \Rightarrow H^{i+j}_{\tau}(\calX, \scrF).
    \]
\end{Proposition}
\begin{proof}
For $i=0,1,\ldots, n-1$, consider the corestriction map
\[\mathrm{cores}: R\Gamma_{\calZ_{i+1}, \tau}(\calX, \scrF)\rightarrow R\Gamma_{\calZ_i, \tau}(\calX, \scrF).\]
We claim that this map fits into a distinguished triangle
\[R\Gamma_{\calZ_{i+1}, \tau}(\calX, \scrF)\xrightarrow[]{\mathrm{cores}} R\Gamma_{\calZ_i, \tau}(\calX, \scrF)\rightarrow R\Gamma_{\calZ_i\smallsetminus\calZ_{i+1}, \tau}(\calX\smallsetminus\calZ_{i+1}, \scrF)\]
where the second arrow is given by the pullback map.

Consider the commutative diagram
\[
        \begin{tikzcd}
            R\Gamma_{\calZ_{i+1}, \tau}(\calX, \scrF) \arrow[r]\arrow[d, "\mathrm{cores}"] & R\Gamma_{\tau}(\calX, \scrF) \arrow[r]\arrow[d, equal] &  R\Gamma_{\tau}(\calX\smallsetminus \calZ_{i+1}, \scrF)\arrow[d] \\
            R\Gamma_{\calZ_i, \tau}(\calX, \scrF) \arrow[r] & R\Gamma_{\tau}(\calX, \scrF)\arrow[r]\arrow[d] & R\Gamma_{\tau}(\calX \smallsetminus \calZ_i, \scrF)\arrow[d]\\
             & 0 \arrow[r] & R\Gamma_{\calZ_i\smallsetminus\calZ_{i+1}, \tau}(\calX\smallsetminus\calZ_{i+1}, \scrF)[1]
        \end{tikzcd}
    \]
where the top two rows are distinguished triangles, so are the right two columns. By \cite[\href{https://stacks.math.columbia.edu/tag/05R0}{Tag 05R0}]{stacks-project}, the diagram completes into 
\[
        \begin{tikzcd}
            R\Gamma_{\calZ_{i+1}, \tau}(\calX, \scrF) \arrow[r]\arrow[d, "\mathrm{cores}"] & R\Gamma_{\tau}(\calX, \scrF) \arrow[r]\arrow[d, equal] &  R\Gamma_{\tau}(\calX\smallsetminus \calZ_{i+1}, \scrF)\arrow[d] \\
            R\Gamma_{\calZ_i, \tau}(\calX, \scrF) \arrow[r] \arrow[d]& R\Gamma_{\tau}(\calX, \scrF)\arrow[r]\arrow[d] & R\Gamma_{\tau}(\calX \smallsetminus \calZ_i, \scrF)\arrow[d]\\
           R\Gamma_{\calZ_i\smallsetminus\calZ_{i+1}, \tau}(\calX\smallsetminus\calZ_{i+1}, \scrF)\arrow[r]  & 0 \arrow[r] & R\Gamma_{\calZ_i\smallsetminus\calZ_{i+1}, \tau}(\calX\smallsetminus\calZ_{i+1}, \scrF)[1]
        \end{tikzcd}
    \]
where all rows and columns are distinguished triangle. One checks that the bottom left vertical arrow is necessarily given by the pullback map.

Putting all $i$'s together, we arrive at a diagram \[
        \begin{tikzcd}
            R\Gamma_{\calZ_1, \tau}(\calX, \scrF) \arrow[r]\arrow[rd, equal] &  R\Gamma_{\tau}(\calX, \scrF) \arrow[r] &  R\Gamma_{\tau}(\calX \smallsetminus \calZ_1, \scrF)\\
            R\Gamma_{\calZ_2, \tau}(\calX, \scrF) \arrow[r] & R\Gamma_{\calZ_1, \tau}(\calX, \scrF) \arrow[r] & R\Gamma_{\calZ_1\smallsetminus \calZ_2, \tau}(\calX\smallsetminus \calZ_2, \scrF)\\
            & \vdots &\\
            R\Gamma_{\calZ_{n}, \tau}(\calX, \scrF) \arrow[r] & R\Gamma_{\calZ_{n-1}, \tau}(\calX, \scrF) \arrow[r] & R\Gamma_{\calZ_{n-1}\smallsetminus \calZ_{n}, \tau}(\calX\smallsetminus \calZ_n, \scrF)
        \end{tikzcd}.
    \]
Then we simply take the spectral sequence associated with the filtered complex.
\end{proof}

\subsection{Banach sheaves and pro-Kummer \'etale cohomology with supports}\label{subsection: Banach sheaves and pro-Kummer \'etale cohomology with supports}

In this subsection, we discuss pro-Kummer \'etale cohomology with supports with coefficients in (limits of) Banach sheaves. In particular, we generalise results in \cite[\S 2.5]{BP-HigherColeman} to the pro-Kummer \'etale topology. We remark that our discussion is highly inspired by the work of Boxer--Pilloni, but we have to deal with the additional complication caused by the pro-Kummer \'etale topology. 

Throughout this subsection, we assume that $K$ is a complete field extension of $\Q_p$ in $\C_p$ and $K^+ = \calO_K$. We also fix an affinoid $(K, \calO_K)$-algebra $(R, R^{\circ})$ (in the sense of Tate). 

\begin{Definition}\label{Definition: proket Banach sheaves}
    \begin{enumerate}
        \item[(i)] A \emph{sheaf of Banach $\widehat{\scrO}_{\calX_{\proket}}\widehat{\otimes}R$-modules} is a sheaf of $\widehat{\scrO}_{\calX_{\proket}}\widehat{\otimes}R$-modules $\scrF$ such that \begin{itemize}
            \item for any quasicompact object $\calU \in \calX_{\proket}$, $\scrF(\calU)$ is a Banach $\widehat{\scrO}_{\calX_{\proket}}(\calU)\widehat{\otimes}R$-module; 
            \item there exists a pro-Kummer \'etale covering $\frakU = \{\calU_i\}_{i\in I}$ of log affinoid perfectoid objects in $\calX_{\proket}$ such that for any $\calU\in \frakU$ and any pro-Kummer \'etale map $\calV \rightarrow \calU$ with $\calV$ being log affinoid perfectoid, the natural map \[
                \scrF(\calU) \otimes_{\widehat{\scrO}_{\calX_{\proket}}(\calU)} \widehat{\scrO}_{\calX_{\proket}}(\calV) \rightarrow  \scrF(\calV)
            \] induces an isomorphism \[
                \scrF(\calU) \widehat{\otimes}_{\widehat{\scrO}_{\calX_{\proket}}(\calU)} \widehat{\scrO}_{\calX_{\proket}}(\calV) \xrightarrow{\cong}  \scrF(\calV).
            \]
        \end{itemize}
        We call such a pro-Kummer \'etale covering a \emph{pro-Kummer \'etale atlas} for $\scrF$.
        \item[(ii)] A sheaf of Banach $\widehat{\scrO}_{\calX_{\proket}}\widehat{\otimes}R$-modules $\scrF$ is \emph{ON-able} (resp., \emph{locally projective}) if there exists a pro-Kummer \'etale atlas $\frakU$ such that for any $\calU \in \frakU$, $\scrF(\calU)$ is an ON-able Banach $\widehat{\scrO}_{\calX_{\proket}}(\calU)\widehat{\otimes}R$-module (resp., a Banach $\widehat{\scrO}_{\calX_{\proket}}(\calU)\widehat{\otimes}R$-module satisfying property (Pr)) (in the sense of \cite{Buzzard_2007}). 
    \end{enumerate}
\end{Definition}

\begin{Lemma}\label{Lemma: log affinoid perfectoid objects provide a nice base}
    Let $\scrF$ be a locally projective sheaf of Banach $\widehat{\scrO}_{\calX_{\proket}}\widehat{\otimes}R$-modules and let $\frakU$ be a pro-Kummer \'etale atlas for $\scrF$. Then, for any log affinoid perfectoid object $\calU \in \frakU$, we have $H^i_{\proket}(\calU, \scrF) = 0$ for all $i>0$. 
\end{Lemma}
\begin{proof}
    By the definition of (Pr), it suffices to prove the assertion when $\scrF$ is ON-able over $\calU$. We choose a presentation  \[   
\scrF \cong \widehat{\bigoplus}_{j\in J}\left(\widehat{\scrO}_{\calX_{\proket}}|_{\calU}\widehat{\otimes}R\right) = \left(\varprojlim_n \bigoplus_{j\in J} \left(\widehat{\scrO}_{\calX_{\proket}}^+|_{\calU}\otimes R^{\circ}/p^n \right)\right)\Big[\frac{1}{p}\Big].
    \]
By \cite[Theorem 5.4.3]{Diao}, $H^i_{\proket}(\calU, \widehat{\scrO}_{\calX_{\proket}}^+\otimes R^{\circ}/p^n)$ is almost zero for all $i>0$. The desired vanishing then follows from an almost version of \cite[Lemma 3.18]{Scholze-2013}.
\end{proof}

\begin{Lemma}\label{Lemma: proket coh can be computed by Cech cohomology}
    Let $\scrF$ be a locally projective sheaf of $\widehat{\scrO}_{\calX_{\proket}}\widehat{\otimes}R$-modules and let $\calY \in \calX_{\proket}$. Let $\frakU = \{\calU_i: i\in I\}$ be a pro-Kummer \'etale atlas for $\scrF$. Then $R\Gamma_{\proket}(\calY, \scrF)$ is computed by the \v{C}ech complex associated with the covering $\{\calU_i \times_{\calX}\calY \rightarrow \calY\}_{i\in I}$. 
\end{Lemma}
\begin{proof}
    The assertion follows immediately from Lemma \ref{Lemma: log affinoid perfectoid objects provide a nice base}, \cite[Proposition 5.3.12]{Diao} and \cite[\href{https://stacks.math.columbia.edu/tag/03F7}{Tag 03F7}]{stacks-project}.
\end{proof}

For the rest of \S \ref{subsection: Banach sheaves and pro-Kummer \'etale cohomology with supports}, we make the following assumption on $\calX$.

\begin{Assumption}\label{Assumption: computing proket cohomology with support for Banach coefficients}
    There exists an element $\calX_{\infty} \in \calX_{\proket}$ such that 
    \begin{itemize}
        \item the map $\calX_{\infty} \rightarrow \calX$ is a pro-Kummer \'etale covering; 
        \item for any affinoid open $\calV \subset \calX$, its preimage $\calV_{\infty}$ in $\calX_{\infty}$ is a log affinoid perfectoid object in $\calX_{\proket}$.
    \end{itemize}
\end{Assumption}

\begin{Proposition}\label{Proposition: proket cohomology with supp for Banach sheaves}
     Suppose $\calX$ satisfies Assumption \ref{Assumption: computing proket cohomology with support for Banach coefficients}. Let $\scrF$ be a locally projective sheaf of $\widehat{\scrO}_{\calX_{\proket}}\widehat{\otimes}R$-modules. Suppose we are given the following subsets of $\calX$: \begin{itemize}
        \item an open subset $\calU \subset \calX$ such that it is a finite union of quasi-Stein spaces \footnote{For the definition of quasi-Stein spaces, see \cite[Definition 2.3]{Kiehl-ThmAB}. (Also see \cite[Definition 2.5.14]{BP-HigherColeman}.)}; 
        \item a closed subset $\calZ \subset \calX$ such that its complement is a finite union of quasi-Stein spaces. 
    \end{itemize}
    Then $R\Gamma_{\calU \cap \calZ, \proket}(\calU, \scrF)\in \mathrm{Pro}_{\Z\geq 0}(\mathrm{K}^{\proj}(\Ban(R)))$.
\end{Proposition}
\begin{proof}
    By construction, there is a distinguished triangle \[
        R\Gamma_{\calU \cap \calZ, \proket}(\calU, \scrF) \rightarrow R\Gamma_{\proket}(\calU, \scrF) \rightarrow R\Gamma_{\proket}(\calU \smallsetminus (\calU \cap \calZ), \scrF).
    \]
    Hence, it is enough to prove the assertion for $R\Gamma_{\proket}(\calU, \scrF)$ whenever $\calU$ is a finite union of quasi-Stein spaces. (Note that $\calU \smallsetminus (\calU \cap \calZ)$ is also a finite union of quasi-Stein spaces.)

    Write $\calU = \bigcup_{j=1}^n \calU_j$ where $\calU_j$'s are quasi-Stein spaces such that each $\calU_j = \bigcup_{i\in \Z_{>0}} \calU_{ji}$ is an increasing union of affinoid $\calU_{ji}$. Let $\calU_{ji, \infty}$ denote the preimage of $\calU_{ji}$ in $\calX_{\infty}$. By assumption, each $\calU_{ji, \infty}$ is a log affinoid perfectoid object in $\calX_{\infty}$. For each $i\in \Z_{>0}$, let $\calV_i:=\bigcup_{j=1}^n \calU_{ji}$. We claim that $R\Gamma_{\proket}(\calV_i, \scrF)$ is an object in $\mathrm{K}^{\proj}(\Ban(R))$. Indeed, by Lemma \ref{Lemma: proket coh can be computed by Cech cohomology}, $R\Gamma_{\proket}(\calV_i, \scrF)$ is computed by the \v{C}ech complex associated with the covering $\{\calU_{ji, \infty}\rightarrow \calV_i\}_{j=1}^n$. Since each term in the \v{C}ech complex is in $\Ban(R)$, we are done.
    
    To finish the proof, it suffices to observe that $R\Gamma_{\proket}(\calU, \scrF)=\lim_i R\Gamma_{\proket}(\calV_i, \scrF)$.
\end{proof}

\begin{Proposition}\label{Proposition: corestriction-restriction is a compact map}
    Suppose $\calX$ is proper over $\Spa(K, \calO_K)$ and Assumption \ref{Assumption: computing proket cohomology with support for Banach coefficients} is satisfied. Suppose we are given the following data: \begin{itemize}
        \item $\calU' \subset \calU$, two open subspaces of $\calX$ that are finite unions of quasi-Stein spaces; 
        \item $\calZ \subset \calZ'$, two closed subsets of $\calX$ whose complements in $\calX$ are finite unions of quasi-Stein spaces. 
    \end{itemize}
    Assume furthermore that \begin{itemize}
        \item there exists a quasicompact open subspace $\calU'' \subset \calX$ such that $\calU' \cap \calZ' \subset \calU''$ and $\overline{\calU''} \subset \calU$; 
        \item there exist closed subsets $\calZ'' \subset \calZ'''$ in $\calX$ with quasicompact complements in $\calX$ such that $\calU \cap \calZ \subset \calZ''$ and $\calZ'' \subset \mathring{\calZ'''} \subset \calZ'$.
    \end{itemize}
    (Here $\overline{\bullet}$ and $\mathring{\bullet}$ stand for the closure and the interior of $\bullet$, respectively.) Let $\phi: \scrF \rightarrow \scrG$ be a compact morphism of locally projective sheaves of $\widehat{\scrO}_{\calX_{\proket}}\widehat{\otimes}R$-modules. Then both $R\Gamma_{\calZ \cap \calU, \proket}(\calU, \scrF)$ and $R\Gamma_{\calZ'\cap \calU', \proket}(\calU', \scrG)$ lie in $\mathrm{Pro}_{\Z\geq 0}(\mathrm{K}^{\proj}(\Ban(R)))$ and the natural map 
\[
        R\Gamma_{\calZ \cap \calU, \proket}(\calU, \scrF) \rightarrow R\Gamma_{\calZ'\cap \calU', \proket}(\calU', \scrG)
    \]
    induced by $\phi$ is compact in the sense of Definition \ref{Definition: potent compact operators}.
\end{Proposition}
\begin{proof} 
The first statement follows from Proposition \ref{Proposition: proket cohomology with supp for Banach sheaves}. It remains to prove the second statement. Since $\phi: \scrF \rightarrow \scrG$ is compact, one reduces to show the natural map (obtained by corestriction and pullback) \[
        R\Gamma_{\calZ \cap \calU, \proket}(\calU, \widehat{\scrO}_{\calX_{\proket}}\widehat{\otimes}R) \rightarrow R\Gamma_{\calZ'\cap \calU', \proket}(\calU', \widehat{\scrO}_{\calX_{\proket}}\widehat{\otimes}R)
    \]
is a compact morphism. We split the proof into several steps.\\
    
\paragraph{\textbf{Step 1.}} Suppose $\calU' \subset \calU$ are quasicompact open subspaces of $\calX$ such that $\overline{\calU'} \subset \calU$. We claim that the restriction map \[
        R\Gamma_{\proket}(\calU, \widehat{\scrO}_{\calX_{\proket}}\widehat{\otimes}R) \rightarrow R\Gamma_{\proket}(\calU', \widehat{\scrO}_{\calX_{\proket}}\widehat{\otimes}R)
    \]
    is a compact morphism. 

By writing $\calU$ and $\calU'$ as unions of affinoid open subspaces, we may assume that $\calU$ and $\calU'$ are affinoid. Let $\calU_{\infty}$ and $\calU'_{\infty}$ be the pullbacks of $\calU$ and $\calU'$ along $\calX_{\infty}\rightarrow \calX$. By the proof of Proposition \ref{Proposition: proket cohomology with supp for Banach sheaves}, $R\Gamma_{\proket}(\calU, \widehat{\scrO}_{\calX_{\proket}}\widehat{\otimes}R)$ is computed by the \v{C}ech complex associated with the covering $\calU_{\infty}\rightarrow \calU$; similarly for $R\Gamma_{\proket}(\calU', \widehat{\scrO}_{\calX_{\proket}}\widehat{\otimes}R)$. We immediately reduce to show that for all $n\geq 1$, the restriction map
\[\widehat{\scrO}_{\calX_{\proket}}(\calU_{\infty}^{(n)})\rightarrow \widehat{\scrO}_{\calX_{\proket}}({\calU'}^{(n)}_{\infty})\]
is compact, where $\calU_{\infty}^{(n)}$ (resp., ${\calU'}^{(n)}_{\infty}$) is the $n$-fold fiber product $\calU_{\infty}\times_{\calU}\cdots \times_{\calU}\calU_{\infty}$ (resp., $\calU'_{\infty}\times_{\calU'}\cdots \times_{\calU'}\calU'_{\infty}$).
Write $\calX_{\infty}=\lim_i \calX_i$ and write $\calU_i:=\calU\times_{\calX}\calX_i$ (resp., $\calU'_i:=\calU'\times_{\calX}\calX_i$). Let $\calU_i^{(n)}$ (resp., ${\calU'}^{(n)}_i$) be the $n$-fold fiber product $\calU_i\times_{\calU}\cdots \times_{\calU}\calU_i$ (resp., $\calU'_i\times_{\calU'}\cdots \times_{\calU'}\calU'_i$). Following the proof of \cite[Lemma 2.5.23]{BP-HigherColeman}, we know that $\calU_i^{(n)}$ is relatively compact in ${\calU'}^{(n)}_i$, and hence $\scrO_{\calX_{\ket}}(\calU_i^{(n)})\rightarrow\scrO_{\calX_{\ket}}({\calU'}^{(n)}_i)$ is compact. Consequently, $\widehat{\scrO}_{\calX_{\proket}}(\calU_{\infty}^{(n)})\rightarrow \widehat{\scrO}_{\calX_{\proket}}({\calU'}^{(n)}_{\infty})$ is compact as it is the completed colimit of $\scrO_{\calX_{\ket}}(\calU_i^{(n)})\rightarrow\scrO_{\calX_{\ket}}({\calU'}^{(n)}_i)$.\\

    \paragraph{\textbf{Step 2.}} Suppose $\calU' \subset \calU$ are quasicompact open subspaces of $\calX$ and $\calZ \subset \calZ'$ are closed subsets of $\calX$ with quasicompact completments in $\calX$. We assume that $\overline{\calU'} \subset \calU$ and $\calZ \subset \mathring{\calZ'}$. Then the natural map \[
        R\Gamma_{\calU \cap \calZ, \proket}(\calU,\widehat{\scrO}_{\calX_{\proket}}\widehat{\otimes}R) \rightarrow R\Gamma_{\calU' \cap \calZ', \proket}(\calU',\widehat{\scrO}_{\calX_{\proket}}\widehat{\otimes}R)
    \] 
    is a compact morphism. 
    
    Indeed, by definition, we have a morphism of distinguished triangles \[
        \begin{tikzcd}
            \scalemath{0.9}{R\Gamma_{\calU \cap \calZ, \proket}(\calU, \widehat{\scrO}_{\calX_{\proket}}\widehat{\otimes}R)} \arrow[r]\arrow[d] & \scalemath{0.9}{R\Gamma_{\proket}(\calU, \widehat{\scrO}_{\calX_{\proket}}\widehat{\otimes}R)} \arrow[r]\arrow[d] & \scalemath{0.9}{R\Gamma_{\proket}(\calU \smallsetminus (\calU \cap \calZ), \widehat{\scrO}_{\calX_{\proket}}\widehat{\otimes}R)} \arrow[d]\\
            \scalemath{0.9}{R\Gamma_{\calU' \cap \calZ', \proket}(\calU', \widehat{\scrO}_{\calX_{\proket}}\widehat{\otimes}R)} \arrow[r] & \scalemath{0.9}{R\Gamma_{\proket}(\calU', \widehat{\scrO}_{\calX_{\proket}}\widehat{\otimes}R)} \arrow[r] & \scalemath{0.9}{R\Gamma_{\proket}(\calU' \smallsetminus (\calU' \cap \calZ'), \widehat{\scrO}_{\calX_{\proket}}\widehat{\otimes}R)}
        \end{tikzcd}.
    \]
    Hence, it is enough to show the compactness of the two vertical maps on the right-hand side. But these follow from Step 1 as $
        \overline{\calU'} \subset \calU$ and $\overline{\calU' \smallsetminus (\calU' \cap \calZ')} \subset \overline{\calU'} \smallsetminus (\overline{\calU'} \cap \mathring{\calZ'}) \subset \calU \smallsetminus (\calU \cap \calZ)$.\\

    \paragraph{\textbf{Step 3.}} Finally, we finish the proof by reducing to Step 2.

    We may write $\calU = \bigcup_{i\in \Z_{> 0}} \calU_i$ as an increasing union such that each $\calU_i$ is a quasicompact open subspace of $\calX$ (see, for example, the proof of Proposition \ref{Proposition: proket cohomology with supp for Banach sheaves}). Since $\calX$ is proper, one deduces that $\overline{\calU''} \subset \calU_n$ for some $n\in \Z_{>0}$. Thus, the morphism \[
        R\Gamma_{\calZ \cap \calU, \proket}(\calU, \widehat{\scrO}_{\calX_{\proket}}\widehat{\otimes}R) \rightarrow R\Gamma_{\calZ'\cap \calU', \proket}(\calU', \widehat{\scrO}_{\calX_{\proket}}\widehat{\otimes}R)
    \]
    factors as \[
        \begin{tikzcd}[column sep = tiny]
            \scalemath{0.9}{R\Gamma_{\calZ \cap \calU, \proket}(\calU, \widehat{\scrO}_{\calX_{\proket}}\widehat{\otimes}R)} \arrow[rr]\arrow[d] && \scalemath{0.9}{R\Gamma_{\calZ'\cap \calU', \proket}(\calU', \widehat{\scrO}_{\calX_{\proket}}\widehat{\otimes}R)}\arrow[d, equal]\\
            \scalemath{0.9}{R\Gamma_{\calZ'' \cap \calU_n, \proket}(\calU_n, \widehat{\scrO}_{\calX_{\proket}}\widehat{\otimes}R)} \arrow[r] & \scalemath{0.9}{R\Gamma_{\calZ''' \cap \calU'', \proket}(\calU'', \widehat{\scrO}_{\calX_{\proket}}\widehat{\otimes}R)} \arrow[r] & \scalemath{0.9}{R\Gamma_{\calZ' \cap \calU', \proket}(\calU'\cap \calU'', \widehat{\scrO}_{\calX_{\proket}}\widehat{\otimes}R)}
        \end{tikzcd},
    \]
    where the vertical identification on the right-hand side is given by \eqref{eq: change of ambient spaces}. Hence, it is enough to show that \[
        R\Gamma_{\calZ'' \cap \calU_n, \proket}(\calU_n, \widehat{\scrO}_{\calX_{\proket}}\widehat{\otimes}R) \rightarrow R\Gamma_{\calZ''' \cap \calU'', \proket}(\calU'', \widehat{\scrO}_{\calX_{\proket}}\widehat{\otimes}R)
    \]
    is compact. This follows from Step 2. 
\end{proof}

\subsection{Integral structures of Banach sheaves}\label{subsection: integral structure of Banach sheaves} 
The purpose of this subsection is to introduce the notion of \emph{integral structures} for locally projective Banach sheaves (in the sense of Definition \ref{Definition: proket Banach sheaves}) on the pro-Kummer \'etale site and to prove Lemma \ref{Lemma: higher cohomology of integral structure has bounded torsion}, which is used in the main body of the paper. A similar discussion for locally projective Banach sheaves on the analytic/{\'e}tale site can be found in \cite[\S 2.6]{BP-HigherColeman}.

We retain the setting of \S \ref{subsection: Banach sheaves and pro-Kummer \'etale cohomology with supports}.

\begin{Definition}\label{Definition: integral structure of Banach sheaves}
    Let $\scrF$ be a locally projective sheaf of Banach $\widehat{\scrO}_{\calX_{\proket}}\widehat{\otimes}R$-modules. A subsheaf $\scrF^+\subset \scrF$ of $\widehat{\scrO}_{\calX_{\proket}}^+ \widehat{\otimes} R^{\circ}$-modules is called an \emph{integral structure} of $\scrF$ if
\begin{enumerate}
        \item[(i)] the natural map $\scrF^+\otimes_{\calO_K}K \rightarrow \scrF$ is an isomorphism;
        \item[(ii)] there exists a pro-Kummer \'etale covering $\frakU = \{\calU_i\}_{i\in I}$ by log affinoid perfectoid objects in $\calX_{\proket}$ such that $\scrF^+(\calU_i)$ is a completion of a free $\widehat{\scrO}_{\calX_{\proket}}^+(\calU_i)\widehat{\otimes}R^{\circ}$-module and the canonical map \[
            \scrF^+(\calU_i)\otimes_{\widehat{\scrO}_{\calX_{\proket}}^+(\calU_i)\widehat{\otimes}R^{\circ}}\left(\widehat{\scrO}_{\calX_{\proket}}^+|_{\calU_i}\widehat{\otimes}R^{\circ}\right) \rightarrow \scrF^+|_{\calU_i}
        \]
        factors through an isomorphism
\[
            \scrF^+(\calU_i)\widehat{\otimes}_{\widehat{\scrO}_{\calX_{\proket}}^+(\calU_i)\widehat{\otimes}R^{\circ}}\left(\widehat{\scrO}_{\calX_{\proket}}^+|_{\calU_i}\widehat{\otimes}R^{\circ}\right) \xrightarrow[]{\cong} \scrF^+|_{\calU_i}.
\]
    \end{enumerate}
\end{Definition}

\begin{Lemma}\label{Lemma: higher cohomology of integral structure has bounded torsion}
    Let $\scrF$ be a sheaf of locally projective Banach $\widehat{\scrO}_{\calX_{\proket}}\widehat{\otimes}R$-modules on $\calX_{\proket}$. Suppose $\calU \in \calX_{\proket}$ is a log affinoid perfectoid object such that for any pro-Kummer \'etale map $\calV \rightarrow \calU$ with $\calV$ being log affinoid perfectoid, the natural map \[
                \scrF(\calU) \otimes_{\widehat{\scrO}_{\calX_{\proket}}(\calU)} \widehat{\scrO}_{\calX_{\proket}}(\calV) \rightarrow  \scrF(\calV)
            \] induces an isomorphism \[
                \scrF(\calU) \widehat{\otimes}_{\widehat{\scrO}_{\calX_{\proket}}(\calU)} \widehat{\scrO}_{\calX_{\proket}}(\calV) \xrightarrow{\cong}  \scrF(\calV).
            \]
Let $\scrF^+$ be an integral structure of $\scrF$. Then, there exists $N\in \Z_{\geq 0}$ such that $p^N$ annihilates $H^i_{\proket}(\calU, \scrF^+)$ for all $i>0$.
\end{Lemma}
\begin{proof}
    Let $\Spa(A, A^+)$ denote the affinoid perfectoid space associated with the log affinoid perfectoid object $\calU$. Then $M = \scrF(\calU)$ is a Banach $A\widehat{\otimes}R$-module satisfying property (Pr). In particular, there exists another Banach $A\widehat{\otimes}R$-module $N$ such that 
    \[M\oplus N \cong c_{A\widehat{\otimes}R}(J)\]
    for some index set $J$. Here, $c_{A\widehat{\otimes}R}(J)$ stands for the ON-able Banach $A\widehat{\otimes}R$-module with orthonormal basis $\{e_j\}_{j\in J}$; namely, $c_{A\widehat{\otimes}R}(J)$ consists of sums $\sum_{j\in J} a_j e_j$ with $a_j\in A\widehat{\otimes}R$ such that $|a_j|\rightarrow 0$ as $j\rightarrow \infty$. Let $c^+_{A\widehat{\otimes}R}(J)\subset c_{A\widehat{\otimes}R}(J)$ denote the $A^+ \widehat{\otimes}R^{\circ}$-submodule consisting of those $\sum_{j\in J} a_j e_j$ with $a_j\in A^+ \widehat{\otimes}R^{\circ}$.
    
Consider sheaves
\[\scrG:=c_{A\widehat{\otimes}R}(J) \widehat{\otimes}_{A\widehat{\otimes}R}\left(\widehat{\scrO}_{\calX_{\proket}}|_{\calU}\widehat{\otimes}R\right) \]
and
\[\scrG^+:=c^+_{A\widehat{\otimes}R}(J) \widehat{\otimes}_{A^+\widehat{\otimes}R^{\circ}}\left(\widehat{\scrO}^+_{\calX_{\proket}}|_{\calU}\widehat{\otimes}R^{\circ}\right).\]
Let
\[
        M^+ \coloneq c^+_{A\widehat{\otimes}R}(J) \cap M
\]
and let $\scrM^+$ be the sheafification of the presheaf 
\[M^+\widehat{\otimes}_{A^+\widehat{\otimes}R^{\circ}} \left(\widehat{\scrO}^+_{\calX_{\proket}}|_{\calU}\widehat{\otimes}R^{\circ}\right).\]
This is a subsheaf of $\scrF \cap \scrG^+$, where the intersection is taken inside $\scrG$. We claim that there exists $N'\in \Z_{\geq 0}$ such that $H^i_{\proket}(\calU, \scrM^+)$ is annihilated by $p^{N'}$ for all $i>0$. 

    To show the claim, consider \[
        \widetilde{M}^+ \coloneq \image\left( c^+_{A \widehat{\otimes}R}(J) \hookrightarrow c_{A \widehat{\otimes}R}(J) \twoheadrightarrow M\right).
    \]
We have $M^+ \subset\widetilde{M}^+$. Since $c^+_{A \widehat{\otimes}R}(J)$ is open in $c_{A \widehat{\otimes}R}(J)$, both $M^+$ and $\widetilde{M}^+$ are open in $M$. Hence, there exists $N'\in \Z_{\geq 0}$ such that $p^{N'}$ annihilates $\coker(M^+ \hookrightarrow \widetilde{M}^+)$. Therefore, $p^{N'}: M^+ \rightarrow M^+$ factors as \[
        M^+ \hookrightarrow c^+_{A \widehat{\otimes}R}(J) \twoheadrightarrow \widetilde{M}^+ \xrightarrow{p^{N'}} M^+.
    \]
    As a result, $p^{N'} : \scrM^+ \rightarrow \scrM^+$ factors as \[
        \scrM^+ \rightarrow \scrG^+ \rightarrow \scrM^+.
    \]
It follows from \cite[Theorem 5.4.3]{Diao} that $H^i_{\proket}(\calU, \scrG^+)$ is almost zero for all $i>0$, and the claim follows. 

To finish the proof, we may rescale and assume $M^+ \hookrightarrow \scrF^+(\calU)$, which yields an inclusion $\scrM^+ \hookrightarrow \scrF^+$. We claim that there exists $N''\in \Z_{\geq 0}$ such that $\coker(\scrM^+ \rightarrow \scrF^+)$ is annihilated by $p^{N''}$. Let $\frakV = \{\calV_i\}_{i\in I}$ be a pro-Kummer \'etale covering of $\calX$ by log affinoid perfectoid objects as in Definition \ref{Definition: integral structure of Banach sheaves} (ii). Let $\calU_i \coloneq \calU \times_{\calX}\calV_i$, then $\{\calU_i \rightarrow \calU\}_{i\in I}$ is a pro-Kummer \'etale covering of $\calU$ by log affinoid perfectoid objects in $\calU_{\proket}$. Since $\calU$ is quasi-compact, we may assume $I$ is finite. For each $i\in I$, there exists $N_i\in \Z_{\geq 0}$ such that the cokernel of the canonical map
\[
        M^+ \widehat{\otimes}_{A^+\widehat{\otimes}R^{\circ}}\left(\widehat{\scrO}_{\calX_{\proket}}^+(\calU_i)\widehat{\otimes}R^{\circ}\right) \rightarrow \scrF^+(\calU_i)
    \] is annihilated by $p^{N_i}$, because the image of the map is open. Therefore, if we put $N''= \sum_{i\in I} N_i$, we have $\coker(\scrM^+ \rightarrow \scrF^+)$ is annihilated by $p^{N''}$. 
    
    Finally, we conclude the proof by taking $N = N'+N''$.
\end{proof}
\end{appendix}

\bibliographystyle{amsalpha}
\bibliography{Reference}

\vspace{10mm}

\begin{tabular}{l}
    H.D.\\
    Tsinghua University,   \\
    Yau Mathematical Sciences Center\\
    Beijing, China\\
    \textit{E-mail address: }\texttt{hdiao@mail.tsinghua.edu.cn}
\end{tabular}
\\
\\

\begin{tabular}{l}
    G.R. \\
    Concordia University   \\
    Department of Mathematics and Statistics\\
    Montr\'{e}al, Qu\'{e}bec, Canada\\
    \textit{E-mail address: }\texttt{giovanni.rosso@concordia.ca}
\end{tabular}
\\
\\

\begin{tabular}{l}
    J.-F.W.\\
    School of Mathematics and Statistics\\
    University College Dublin\\
    Belfield, Dublin 4, Ireland\\
    \textit{E-mail address: }\texttt{ju-feng.wu@ucd.ie}
\end{tabular}

\end{document}